\documentclass{memo-l}

\usepackage{amsmath,amsthm,latexsym}
\usepackage{enumerate}
\usepackage{amssymb,amsfonts,mathrsfs}
\usepackage{multirow}
\usepackage{longtable}
\usepackage{float}
\usepackage{diagbox}
\usepackage{pifont}

\usepackage{graphicx} 
\usepackage{lipsum}   

\usepackage{tkz-graph}
\usepackage{xypic}
\usetikzlibrary{%
	arrows,%
	shapes.misc,
	shapes.arrows,%
	chains,%
	matrix,%
	positioning,
	scopes,%
	decorations.pathmorphing,
	shadows%
}

\usepackage{comment}

\usepackage{array}
\usepackage{dynkin-diagrams}

\usepackage{hyperref}
\usepackage{cleveref}

\renewcommand*{\arraystretch}{1.1}

\makeatletter
\renewcommand\subsection{\@startsection{subsection}{2}{\z@}%
	{-3.5ex \@plus -1ex \@minus-.2ex}%
	{2.3ex \@plus.2ex}%
	{\normalfont\large\bfseries}}
\makeatother

\usepackage[colorinlistoftodos,prependcaption,textsize=footnotesize,textwidth=0.7in]{todonotes}

\newcounter{todocounter}
\newcommand{\todonum}[1]{\stepcounter{todocounter}\todo{\thetodocounter: #1}}
\makeatletter
\providecommand\@dotsep{5}
\renewcommand{\listoftodos}[1][\@todonotes@todolistname]{%
	\@starttoc{tdo}{#1}}
\makeatother

\theoremstyle{plain}
\newtheorem{Thm}{Theorem}[chapter]
\newtheorem{Lem}[Thm]{Lemma}
\newtheorem{Conj}[Thm]{Conjecture}
\newtheorem{Cor}[Thm]{Corollary}
\newtheorem{Prop}[Thm]{Proposition}

\theoremstyle{definition}

\newtheorem{Rem}[Thm]{Remark}

\theoremstyle{remark}

\numberwithin{section}{chapter}
\numberwithin{equation}{chapter}

\newcommand{\C}{\mathbb{C}}
\newcommand{\Z}{\mathbb{Z}}

\DeclareMathOperator{\Image}{Im}

\DeclareMathOperator{\Ind}{Ind}

\DeclareMathOperator{\St}{St}
\DeclareMathOperator{\Span}{Span}

\newcommand{\divides}{\Big \vert}

\newcommand{\Card}[1]{\left\vert #1\right\vert} 
\newcommand{\Places}{\mathcal{P}} 
\DeclareMathOperator{\zfun}{\zeta} 
\DeclareMathOperator{\Lfun}{\mathcal{L}} 
\newcommand{\coset}[1]{\left[ #1 \right]}  
\newcommand{\FNorm}[1]{\left\vert #1 \right\vert} 

\newcommand{\sgn}{\operatorname{sgn}}

\newcommand{\Q}{\mathbb{Q}}

\newcommand{\R}{\mathbb{R}}
\newcommand{\N}{\mathbb{N}}

\newcommand{\mO}{\mathcal{O}}

\newcommand{\bk}[1]{\left(#1\right)} 
\newcommand{\bm}{\begin{multline*}}
\newcommand{\tu}{\end{multline*}}

\DeclareMathOperator{\Id}{\mathbf{1}} 
\DeclareMathOperator{\unif}{\varpi} 

\renewcommand{\check}[1]{#1 ^{\vee}} 
\DeclareMathOperator{\Real}{Re} 

\newcommand{\fun}[1]{\bar{\omega}_{#1}}
\newcommand{\jac}[3]{r^{#1}_{#2}\bk{#3}}

\newcommand{\mult}[2]{mult\bk{#1,#2}}
\newcommand{\s}[1]{s_{#1}}
\newcommand{\para}[1]{#1}

\newcommand{\piece}[1]{\left\{\begin{matrix} #1 \end{matrix}\right.} 
\newcommand{\set}[1]{\left\{ #1 \right\}} 
\newcommand{\mvert}{\mathrel{}\middle\vert\mathrel{}} 
\newcommand{\res}[1]{\Bigg\vert_{#1}}
\newcommand{\ldual}[1]{{^L}#1}
\newcommand{\Eisen}{\mathcal{E}}
\newcommand{\lmod}{\backslash}
\newcommand{\rmod}{/}
\newcommand{\gen}[1]{\left< #1 \right>}
\newcommand{\inner}[1]{\langle #1 \rangle}
\newcommand{\Stab}{\operatorname{Stab}}

\newcommand{\intl}{\, \int\limits}
\newcommand{\suml}{\, \sum\limits}
\newcommand{\prodl}{\, \prod\limits}

\newcommand{\Res}[1]{\operatorname{Res}_{#1}}
\newcommand{\ord}[2]{\operatorname{ord}_{#1}\bk{#2}}
\newcommand{\DPS}[1]{\operatorname{I}_{#1}}

\newcommand{\ResRep}[4]{\operatorname{Res}_{#1}\bk{#2,#3,#4}}
\newcommand{\ResRepTr}[3]{\operatorname{Res}_{#1}\bk{#2,#3}}

\renewcommand\AA{\mathbb{A}} 

\newcommand\GG{\mathbb{G}}

\newcommand\TT{\mathbb{T}}

\newcommand\ZZ{\mathbb{Z}}

\newcommand\bB{\mathbf{B}}

\newcommand\bG{\mathbf{G}}

\newcommand\bM{\mathbf{M}}
\newcommand\bN{\mathbf{N}}

\newcommand\bP{\mathbf{P}}

\newcommand\bT{\mathbf{T}}
\newcommand\bU{\mathbf{U}}

\newcommand\bX{\mathbf{X}}

\newcommand\bZ{\mathbf{Z}}

\newcommand\bCT{\mathbf{CT}}

\DeclareMathAlphabet{\mathcal}{OMS}{cmsy}{m}{n}

\newcommand\cA{\mathcal{A}}

\newcommand\cZ{\mathcal{Z}}

\newcommand\fraka{\mathfrak{a}}

\newcommand\frake{\mathfrak{e}}

\newcommand\frakg{\mathfrak{g}}
\newcommand\frakh{\mathfrak{h}}

\newcommand\frakJ{\mathfrak{J}}

\newcommand\frakR{\mathfrak{R}}

\newcommand{\cmark}{\text{\ding{51}}}
\newcommand{\xmark}{\text{\ding{55}}}

\newcommand{\esixchar}[6]{\renewcommand*{\arraystretch}{1} \begin{pmatrix}&& #2 && \\ #1 & #3 & #4 & #5 & #6 \end{pmatrix} }

\newcommand{\esevenchar}[7]{\renewcommand*{\arraystretch}{1} \begin{pmatrix}&& #2 && \\ #1 & #3 & #4 & #5 & #6 &#7 \end{pmatrix} }

\newcommand{\eeightchar}[8]{\renewcommand*{\arraystretch}{1} \begin{pmatrix}&& #2 &&&& \\ #1 & #3 & #4 & #5 & #6 &#7 & #8 \end{pmatrix} }

\newcommand{\equivclassespole}[1]{\Sigma_{#1}}

\makeatletter
\providecommand*{\cupdot}{%
	\mathbin{%
		\mathpalette\@cupdot{}%
	}%
}
\newcommand*{\@cupdot}[2]{%
	\ooalign{%
		$\m@th#1\cup$\cr
		\sbox0{$#1\cup$}%
		\dimen@=\ht0 %
		\sbox0{$\m@th#1\cdot$}%
		\advance\dimen@ by -\ht0 %
		\dimen@=.5\dimen@
		\hidewidth\raise\dimen@\box0\hidewidth
	}%
}

\providecommand*{\bigcupdot}{%
	\mathop{%
		\vphantom{\bigcup}%
		\mathpalette\@bigcupdot{}%
	}%
}
\newcommand*{\@bigcupdot}[2]{%
	\ooalign{%
		$\m@th#1\bigcup$\cr
		\sbox0{$#1\bigcup$}%
		\dimen@=\ht0 %
		\advance\dimen@ by -\dp0 %
		\sbox0{\scalebox{2}{$\m@th#1\cdot$}}%
		\advance\dimen@ by -\ht0 %
		\dimen@=.5\dimen@
		\hidewidth\raise\dimen@\box0\hidewidth
	}%
}
\makeatother

\makeindex

\begin{document}

\frontmatter

\title[Poles, Residues and Siegel-Weil Identities for $E_n$]{Poles, Residues and Siegel-Weil Identities of Degenerate Eisenstein Series on Split Exceptional Groups of Type $E_n$}


\author{Hezi Halawi}
\address{School of Mathematics, Ben Gurion University of the Negev, POB 653, Be’er Sheva 84105, Israel}
\email{halawi@post.bgu.ac.il}

\author{Avner Segal}
\address{Mathematics Department, Shamoon College of Engineering,
	56 Bialik St., Beer-Sheva 84100, Israel}
\email{avnerse@sce.ac.il}

\date{\today}

\subjclass[2010]{Primary }

\keywords{}


\begin{abstract}
	This manuscript has two goals:
	\begin{enumerate}
		\item To write an explicit description of the degenerate residual spectrum of the split, simple, simply-connected, exceptional groups of type $E_n$ (for $n=6,7,8$).
		\item To set a practical guide for similar calculations and, in particular, to describe various methods of ``computational representation theory'' relevant to the study of residues of automorphic Eisenstein series.
	\end{enumerate}
	
	In Part I we supply background information and notations from the theory of automorphic representations as well as concrete information on the exceptional groups of type $E_n$ and their representation theory over non-Archimedean local fields.
	
	In Part II we make a systematic study of the residual spectrum of these groups where each chapter is devoted to a certain aspect of theory, it begins with methodical section and continues with sections devoted to the results for each of these groups.
	
	We describe completely the residual (square-integrable and non-square integrable) spectrum of the groups of type $E_6$ and $E_7$ and an almost complete description in the case of the group of type $E_8$.
	We also study and list Siegel-Weil like identities between residual representations of these groups and list the Arthur parameters for their square-integrable residual representations.

\end{abstract}

\maketitle

\tableofcontents

\chapter*{Introduction}

This paper is a culmination of the authors joint project of studying the degenerate principal series of exceptional groups of type $E_n$, where  $n$ is $6$, $7$ or $8$.
In \cite{SDPS_E6}, \cite{SDPS_E7} and \cite{SDPS_E8} we studied the degenerate principal series representations of split, simple, simply-connected groups of type $E_n$ over non-Archimedean local fields;
in these papers, we listed the reducible ones and provide a rough description of their irreducible subrepresentations and quotients.
Here, we use the information collected in these works in order to give a description of the degenerate residual spectrum of the simple simply-connected groups of type $E_n$ over the ring of adeles.

More precisely, let $F$ be a number field and let $\AA$ be its ring of adeles.
For a reductive group $G$ over $F$, with center $\cZ_G$, we consider the space of automorphic functions $\cA\bk{G\bk{\AA}}$, as defined in \cite{MR1476501} for example.
A smooth representation $\pi$ of $G\bk{\AA}$ is said to be automorphic if it is a subquotient of $\cA\bk{G\bk{\AA}}$.
This notion is closely connected, but not identical, with the regular representation of $G\bk{\AA}$, namely the space $L^2\bk{\cZ_G\bk{\AA} G\bk{F}\lmod G\bk{\AA}}$ of square-integrable functions.

Indeed, any automorphic representation $\pi$ which consists of square-integrable functions, can be realized  (via the Casselman-Wallach Globalization Theorem) as a subquotient of $L^2\bk{\cZ_G\bk{\AA} G\bk{F}\lmod G\bk{\AA}}$.
Moreover, this space has a rough decomposition into three coarse blocks:
\begin{equation}
	\begin{split}
		L^2_G = L^2\bk{\cZ_G\bk{\AA} G\bk{F}\lmod G\bk{\AA}}
		& = L^2_{cusp.}\bk{\cZ_G\bk{\AA} G\bk{F}\lmod G\bk{\AA}} \\
		& \oplus L^2_{res.}\bk{\cZ_G\bk{\AA} G\bk{F}\lmod G\bk{\AA}} \\
		& \oplus L^2_{cont.}\bk{\cZ_G\bk{\AA} G\bk{F}\lmod G\bk{\AA}} ,
	\end{split}
\end{equation}
where
\begin{itemize}
	\item $L^2_{cusp.}=L^2_{cusp.}\bk{\cZ_G\bk{\AA} G\bk{F}\lmod G\bk{\AA}}$ denotes the cuspidal spectrum of $G\bk{\AA}$.
	
	\item $L^2_{res.}=L^2_{res.}\bk{\cZ_G\bk{\AA} G\bk{F}\lmod G\bk{\AA}}$ denotes the residual spectrum of $G\bk{\AA}$.
	
	\item $L^2_{cont.}=L^2_{cont.}\bk{\cZ_G\bk{\AA} G\bk{F}\lmod G\bk{\AA}}$ denotes the continuous spectrum of $G\bk{\AA}$.
	
	\item $L^2_{cusp.}\bk{\cZ_G\bk{\AA} G\bk{F}\lmod G\bk{\AA}} \oplus L^2_{res.}\bk{\cZ_G\bk{\AA} G\bk{F}\lmod G\bk{\AA}}$ forms the discrete spectrum of $G$, that is the maximal semi-simple subrepresentation of $L^2\bk{\cZ_G\bk{\AA} G\bk{F}\lmod G\bk{\AA}}$.
\end{itemize}

Constructing the square-integrable spectrum of an adelic group is a central problem in the theory of automorphic representations.
There is no general method of constructing the cuspidal spectrum of a group and most constructions of cuspidal representations rely on functorial lifts from other groups (whose cuspidal spectrum is not always fully understood either).
In the case of the residual and continuous spectrum the situation is slightly simpler.
Assuming knowledge of the cuspidal spectrum of Levi subgroup (this, again, is many times not fully understood) one can construct the residual and continuous spectrum using the theory of Eisenstein series devised by R. Langlands (a short account of the theory of Eisenstein series can be found in \Cref{Chap:Preliminaries}, for a complete picture the reader is advised to consider \cite{MR0579181,MR1361168}).
In this manuscript we attempt to construct a particular block within the residual spectrum.
To describe this, we first look at a finer decomposition of the square-integrable spectrum.

Denote the Weyl group of $G$ by $W$.
Given an irreducible subquotient $\pi$ of $L^2_G$, one may attach to it a $W$-conjugacy class of pairs $\coset{M,\mu}$ of a Levi subgroup $M$ of $G$ and a cuspidal representation $\mu$ of $M\bk{\AA}$ such that $\pi$ is an irreducible subquotient of $\Ind_{P\bk{\AA}}^{G\bk{\AA}} \bk{\mu}$, where $P$ is a parabolic subgroup of $G$ whose Levi subgroup is $M$ (the induction here is normalized).
Thus, one may write
\[
L^2_G = \bigoplus_{\coset{M,\mu}} L^2_{\coset{M,\mu}},
\]
where the sum runs over all $W$-conjugacy classes of cuspidal datum and $L^2_{\coset{M,\mu}}$ denotes the maximal subrepresentation of $L^2_G$ whose subquotients admit cuspidal data $\coset{M,\mu}$.

For a Levi subgroup $M$ of $G$ we denote by $L^2_{\coset{M}}$ the direct sum of all isotypic components $L^2_{\coset{M',\mu}}$ with $M'$ conjugate to $M$.
The cuspidal spectrum is then given by $L^2_{\coset{G}}$ while for proper Levi subgroups $M$ of $G$, $L^2_{\coset{M}}$ contributes to both $L^2_{res.}$ and $L^2_{cont.}$.

We write $L^2_{T,res.}$ for $L^2_{\coset{T}}\cap L^2_{res.}$.
It is a question of interest what portion of $L^2_{T,res.}$ can be attained from the degenerate residual spectrum.
It is not true in general that $L^2_{T,res.}$ is generated by pairs $\coset{T,\mu}$ where $\mu$ is associated with an induction from a $1$-dimensional representation of a proper maximal Levi subgroup.
For example, it is known that for symplectic groups this is the case and it is also expected that this case for the exceptional group of type $F_4$.

Let $P$ be a maximal parabolic subgroup of $G$ with a Levi subgroup $M$.
Let $\mu$ be a $1$-dimensional unitary representation of $M$ (the restriction of which to $T$ is automatically cuspidal) and let $\mu_s$ denote $\mu\cdot \Card{\omega_{M}}^s$, where $\omega_M$ is the fundamental weight associated with $M$ and $s\in\C$.
This gives rise to a family of induced representations $\DPS{P}\bk{s,\mu}=\Ind_{P\bk{\AA}}^{G\bk{\AA}}\bk{\mu_s}$.
The degenerate Eisenstein series is defined by the meromorphic continuation (to $\C$) of the series
\[
\Eisen_{P}\bk{s,\mu,f_s,g} = \suml_{\gamma\in P\bk{F}\lmod G\bk{F}} f_s\bk{\gamma g} \quad g\in G\bk{\AA},
\]
which absolutely converge for a ``nice'' section $f_s$ of $I_P\bk{\mu,s}$ in a half-plane $\Real\bk{s}\gg 0$.

Degenerate Eisenstein series are useful not only because they are defined using only one complex variables.
They are also a crucial to the Rankin-Selberg and Langlands-Shahidi methods of integral representations for automorphic $\Lfun$-functions.
See	\cite{MR2192819,MR2683009} for surveys on this to methods.
In particular, the residues of the series, which include the minimal representation of $G$ if it admits one, are key ingredients in the study of explicit lifts such as the $\theta$-lift.
See \cite{MR1159270} for a road-map of the application of Rankin-Selberg integrals to the study of $\theta$-lifts for classical groups and \cite{RallisSchiffmannPaper} for an application of these ideas for exceptional groups.

Representations constructed as residues of degenerate Eisenstein series form the \textbf{degenerate residual spectrum} of $G$.
The construction of this spectrum for the simply-connected groups of type $E_n$ is the aim of this manuscript.
In order to construct these representations, one needs to first determine those $\mu$ and $s_0$ such that $\Eisen_{P}\bk{s,\mu,f_s,g}$ admits a pole at $s=s_0$ and the maximal order of the pole there (for different sections $f_s$ and $g\in G\bk{\AA}$).
Then, one needs to determine whether the residual representation at this point is square-integrable (that is, appears in $L^2_G$) or not and to describe its structure using the irreducible constituents of $\DPS{P}\bk{s,\mu}$.

The structure of $\DPS{P}\bk{s,\mu}$ and its irreducible quotients can be deduced from the structure of the local degenerate principal series $\DPS{P}\bk{s,\mu}_v$, where $v$ is a place of $F$.
For the groups of type $E_n$, these constituents are described using the results of \cite{SDPS_E6,SDPS_E7,SDPS_E8} when $v$ is non-Archimedean.
For Archimedean $v$, the structure is currently unknown at large.

In \Cref{Part:Preliminaries} we supply the reader with background on the theory of Eisenstein series, the groups of type $E_n$ and the local degenerate principal series of these groups.

In \Cref{Part:DRS}, we study the degenerate residual spectrum of groups of type $E_n$ using a script implemented in the Sagemath environment \cite{sagemath}.
This part is structured as follows:
\begin{itemize}
	\item In \Cref{Chap:Poles_of_Eisen_ser} we classify the poles of $\Eisen_{P}\bk{s,\mu,f_s,g}$ for $\Real\bk{s}>0$.
	That is, we determine the location and order of poles as well as determining whether the residue is square integrable or not.
	
	\item In \Cref{Chap:Square_Integrable} we describe those degenerate residual representation which appear in $L^2_{dis.}\bk{G}$ - this is the so-called degenerate residual spectrum.
	
	\item In \Cref{Chap:Siegel_Weil} we study identities between residual representations arising from different maximal parabolic subgroups.
	These identities are in the spirit of \cite{MR1159270,MR946349} and the discussion follows ideas set in \cite{MR1174424}.
	
	\item In \Cref{Chap:Arthur} we attach, to each square-integrable residue found in \Cref{Chap:Square_Integrable}, its Arthur parameter.
	
	\item In \Cref{Chap:Non_square_Integrable} we conclude by describing those residual representations which are not square-integrable.
	
\end{itemize}

At this point, we would like to explain what is completely determined in this manuscript and what remains to be done:
\begin{itemize}
	\item All of the results in this manuscript are valid for $F=\Q$ and most are valid for other number fields.
	When a result is restricted to $F=\Q$, it would be mentioned in the text.
	
	\item We consider the analytic behavior of $\Eisen_{P}\bk{s,\mu,f_s,g}$ mostly for $\Real\bk{s}>0$.
	For $\Real\bk{s}=0$, the Eisenstein series is known to be analytic due to \cite{MR2767521} and for certain cases we use the special values there in order to determine the residues in points with $\Real\bk{s}>0$.
	
	The Eisenstein series at the half-plane $\Real\bk{s}<0$ would not contribute to the residual spectrum and hence we do not deal with the analytic behavior there.
	Bellow we remark on the poles at this half-plane.
	
	\item In what follows, we do not address the structure of $\DPS{P}\bk{s,\mu}_v$ for Archimedean $v$.
	In order to stay on firm grounds, it is possible to restrict the discussion characters $\mu$ which are unramified at Archimedean places and take sections only from the subspace  $\DPS{P}\bk{s,\mu}^0$ of $\DPS{P}\bk{s,\mu}$ which is generated by tensors $\otimes_{v\in\Places} f_{s,v}$ which are spherical for all Archimedean places, once determined.
	
	In this manuscript, however, we chose to write our results in a fashion that would make it easy to incorporate results on the structure of $\DPS{P}\bk{s,\mu}_v$ in Archimedean places.
	In particular, in our notations we assume that, for $s>0$, the structure of the maximal semi-simple quotient at Archimedean place is analogues to the structure at non-Archimedean places.

	\item For the groups of type $E_6$ and $E_7$ we completely determine the locations and orders of poles at $\Real\bk{s}>0$, the square-integrability of the residue as well as the structure of the residue.
	For the group of type $E_8$ we completely determine the location and order of poles and determine the residue for most case.
	In the remaining cases, we were unable to completely determine the structure of local representations due to size and complexity of the Weyl group of type $E_8$.
	The remaining cases would hopefully be resolved when stronger computers are more commonly available or if alternate methods are suggested.
	
\end{itemize}

\begin{Rem}
	The analytic behavior of the Eisenstein series for $\Real\bk{s}<0$ is determined by the functional equation which $\Eisen_{P}\bk{\mu_s,f_s,g}$ satisfy.
	However, the analysis in this domain relies heavily on the location of zeros of the Dedekind $\zfun$-function of $F$ and thus on particular properties of $F$.
	Furthermore, it is known that, many times, poles of Eisenstein series at the left half-plane may have unbounded orders .
	That is, by varying the section $f_s$ we may attain poles whose orders can be as large as we wish.
	In particular, no contribution to the residual spectrum can be constructed using the values and residues of Eisenstein series at these points.
	
\end{Rem}

Finally, we wish to put the results of this manuscript in the context of other studies of degenerate residual spectrum:
\begin{itemize}
	\item For symplectic groups, the degenerate residual spectrum was fully described in \cite{MR3437492}.
	
	\item For classical groups, the square-integrable degenerate residual spectrum can be read from 	\cite{MR1404331,MR1385286}
	
	\item For the exceptional group of type $G_2$, the square-integrable degenerate residual spectrum can be read from \cite{MR1426903,MR1485422,MW_AppendixIII}.
	
	\item For the exceptional gorup of type $F_4$, the part of the degenerate residual spectrum associated with $\mu=\Id$ was calculated in the first author's Ph.D. thesis.
\end{itemize}

\subsection*{Acknowledgements}

\Cref{Chap:Siegel_Weil} relies on results from the first authors M.Sc. thesis \cite{HeziMScThesis}.
Partial support was provided by grant 259/14 from the Israel Science Foundation.
The authors were also supported by the Junior Researcher Grant of Shamoon College of Engineering (SCE).

\mainmatter

\part{Preliminaries}
\label{Part:Preliminaries}

In this part we set notations and recall facts which will be used throughout this manuscript.
It comprises of three chapters:
\begin{enumerate}
	\item \Cref{Chap:Preliminaries} fix notations and general facts regarding the general theory of Eisenstein series for arbitrary groups.
	
	\item \Cref{Chap:Group_En} describes the structure of exceptional groups of type $E_n$.
	
	\item In \Cref{Chap:Local_DPS}, we recall information from \cite{SDPS_E6}, \cite{SDPS_E7} and \cite{SDPS_E8} regarding the local degenerate principal series representations at $p$-adic places.
\end{enumerate}

\chapter{Background Theory}
\label{Chap:Preliminaries}

In this chapter, we set notations and give a short account of certain aspects of the theory of Eisenstein series, degenerate principal series and standard intertwining operators.
A more comprehensive discussion of the material here can be found in \cite{MR1361168,MR0579181,MR546601,MR0207650}.
One can also consult the preliminaries sections in previous works of the second author (\cite[Section 2]{MR3803152} and \cite[Section 2]{MR4024536}).

Let $F$ be a number field with a set of places $\Places$ and a ring of adeles $\AA_F$.
We denote by $\Places_{fin.}$ the set of finite places, namely, all $v\in\Places$ such that $v\not\divides \infty$.
Given $v\in\Places$, we denote by $F_v$ the completion of $F$ with respect to $v$.
If $v\not\divides\infty$, we let $\mO_v$ denote the ring of integers of $F_v$ and fix a uniformizer $\unif_v$ of $F_v$. We further let $q_v$ denote the cardinality of the residue field $\mO_v\rmod \bk{\unif_v}$ of $F_v$.

\section{Notations from Analytic and Algebraic Number Theory}

In this section, we fix certain notations for Hecke $\Lfun$-functions and related objects.
We also recall certain consequences of the Dirichlet and Chebotarev density theorems.
For more details, the reader may consider \cite{MR1990377}, \cite{Lieberman_Hecle_K-functions} and \cite{MR1697859}.

\begin{itemize}
	\item Let $\chi= \otimes_{v\in\Places}\chi_v: F^\times\lmod\AA^\times \to \C^\times$ be a Hecke character and let $\psi=\otimes_{v\in \Places} \psi_v:F\lmod \AA \to \C^\times$ be an additive character.
	\item Let $\Lfun_{F_v}\bk{\chi_v,s}$, with $s\in\C$, denote the local $\Lfun$-function of $\chi_v$.
	\item The local $\Lfun$-factor $\Lfun_{F_v}\bk{\chi_v,s}$ satisfies a functional equation of the form
	\[
	\Lfun_{F_v}\bk{\chi_v,s} = \frac{\epsilon_{F_v}\bk{\chi_v,s,\psi_v}}{\gamma_{F_v}\bk{\chi_v,s,\psi_v}} L_{F_v}\bk{\chi_v^{-1},1-s},
	\]
	where $\epsilon_{F_v}\bk{\chi_v,s,\psi_v}$ is the local $\epsilon$-factor and $\gamma_{F_v}\bk{\chi_v,s,\psi_v}$ is the local $\gamma$-factor of $\chi_v$.
	Both of which depend on the fixed character $\psi_v$.
	
	\item The global $\epsilon$-factor
	\[
	\epsilon_{F}\bk{s,\chi} = \prodl_{v\in\Places} \epsilon_{F_v}\bk{\chi_v,s,\psi_v} ,
	\]
	however, is independent of $\psi$.
	The "global $\gamma$-factor", on the other hand is constant:
	\[
	\prodl_{v\in\Places} \gamma_{F_v}\bk{\chi_v,s,\psi_v} =1 .
	\]

	\item Let $\Lfun_F\bk{s,\chi}$ denote the completed Hecke $\Lfun$-function of $\chi$ given as the meromorphic continuation of the product (which converges for $\Real\bk{s}\gg 0$)
	\[
	\Lfun_F\bk{s,\chi} = \prodl_{v\in\Places} \Lfun_{F_v}\bk{\chi_v,s}
	\]
	
	\item As in \cite[Section 3.3]{MR2894271}, we may normalize the Hecke $\Lfun$-function $\Lfun_F\bk{s,\chi}$ so that it satisfies the functional equation
	\[
	\Lfun_F\bk{s,\chi} = \epsilon_\chi\Lfun_F\bk{\chi^{-1},1-s},
	\]
	where $\epsilon_\chi=\epsilon_{F}\bk{s,\chi}$ is the \emph{root number} of $\chi$ and it satisfies $\FNorm{\epsilon_\chi}=1$.

	\item In particular, for $\chi=\Id$, the trivial character, the global $\Lfun$-function turns out to be the Dedekind $\zeta$-function $\zfun_F\bk{s}$, normalized so that it satisfies the functional equation
	\[
	\zfun_F\bk{s}=\zfun_F\bk{1-s} .
	\]

\end{itemize}

By Class Field Theory, every class of Hecke characters on $\AA_F^\times$ of order $n$ corresponds to an Abelian Galois extension $K/F$ of order $n$.
On the other hand, by Chebotarev's density theorem, the set of primes in $K$ which split over $F$ in an algebraic number field extension $K/F$ of order $n$, has density $\frac{1}{n}$.
It follows that given a Hecke character $\chi$ of order $n$, the local character $\chi_v$ is trivial in $\frac{1}{n}$ of the places $v\in F$ (in terms of its Dirichlet density).
In fact, for any $k\divides n$, there are infinitely many places $v\in\Places$ such that $\chi_v$ is of order $k$.

\section{Groups and Characters}

Let $\bG$ denote a split reductive group defined over $F$.
Let $\bT$ be a maximal split torus in $\bG$ contained in a Borel subgroup $\bB$ of $\bG$ and let $\bN$ denote the unipotent radical of $\bB$.
Also, let $\bZ\bk{\bG}=\bZ\bk{\bT}$ denote the center of $\bG$.

We let $\Phi_\bG=\Phi\bk{\bG,\bT}$ denote the root system of $\bG$ with respect to $\bT$ and let $\Delta=\Delta\bk{\bG,\bT}=\set{\alpha_1,...,\alpha_n}$ be the set of simple roots together with a fixed labeling (we use the conventions of \cite{Bourbaki:2002} for labeling).
Let $\Phi_\bG^{+}$ denote the set of positive roots in $\Phi_\bG$ with respect to $\bB$.

Let $\check{\Phi_\bG}$ denote the set of co-roots of $\bG$ with respect to $\bT$, together with a bijection
\[
\begin{array}{ccc}
\Phi & \to & \check{\Phi_\bG} \\
\alpha & \mapsto & \check{\alpha}
\end{array}
\]
We denote the natural pairing between roots and co-roots by $\inner{\cdot,\cdot}:\Phi_\bG\times\check{\Phi_\bG} \to \Z$.

Let $W_\bG=W\bk{\bG,\bB}$ denote the Weyl group of $\bG$.
The Weyl group $W$ is generated by the simple reflections $w_\alpha$ along the simple roots $\alpha\in\Delta$.

For any $\Theta\subset\Delta$, we denote by $\bP_\Theta$ the standard (i.e. $\bP_\Theta\supset\bB$) parabolic subgroup of $\bG$ associated to $\Theta$.
We denote the Levi subgroup of $\bP_\Theta$ by $\bM_\Theta$ and its unipotent radical by $\bU_\Theta$.
In particular, for $\Theta_i=\Delta\setminus\set{\alpha_i}$, we write $\bP_i=\bP_{\Theta_i}$ and $\bM_i=\bM_{\Theta_i}$

Let $\Phi_{\bM_\Theta}\subset\Phi_\bG$ denote the root system of $\bM_\Theta$ with respect to $\bT$.
The set of positive roots of $\bM_\Theta$ with respect to $\bB\cap\bM_\Theta$ is $\Phi_{\bM_\Theta}^{+}=\Phi_\bG^{+}\cap\Phi_{\bM_\Theta}$ and the set of simple roots is $\Delta_{\bM_\Theta}=\Theta$.
Denote the half-sum of the roots in $\mathfrak{u}_\Theta=Lie\bk{\bU_\Theta}$ by $\rho_{\bP}$.

Fix a standard parabolic subgroup $\bP_\Theta$ as above.
Let $\mathfrak{a}_{\bM,\C}^\ast=X^\ast\bk{\bM}_F\otimes_\Z\C$, where $X^\ast\bk{\bM}_F$ denote the $F$-rational characters of $\bM$.
We use an additive notation for the elements of $\mathfrak{a}_{\bM,\C}^\ast$.
Note that $\inner{\cdot,\cdot}$ extends to map $\mathfrak{a}_{\bM,\C}^\ast\times\check{\Phi_\bG}\to\C$ linearly.

Also, let $W_\bM=W\bk{\bM,\bM\cap\bB}$ denote the Weyl group of $\bM$ and let $W\bk{\bM,\bG}$ denote the set of shortest representatives in $W$ of the coset space $W_\bM \lmod W$.

For any $\alpha_i\in\Delta$, let $\omega_{\alpha_i}$ denote the fundamental weight associated to $\alpha$, which satisfy
\[
\inner{\omega_{\alpha_i},\check{\alpha_j}} = \delta_{i,j} .
\]
The fundamental weights give rise to the isomorphism
\[
\begin{array}{ccc}
\C^n & \rightarrow & \mathfrak{a}_{\bT,\C}^\ast \\
\bar{s}=\bk{s_1,...,s_n} & \mapsto & \lambda_{\bar{s}} = \suml_{i=1}^n s_i \omega_{\alpha_i} .
\end{array}
\]

More generally, any element of $\lambda\in\mathfrak{a}^\ast_{\bM,\C}$ is of the form
\[
\lambda=\lambda_C+\suml_{i\not\in\Theta} s_i \omega_{\alpha_i},
\]
where $\bk{s_i}_{i\not\in\Theta}\in\C$ and $\lambda_C\in \mathfrak{a}^\ast_{\bZ\bk{\bG},\C}$.
Also, note that $ \mathfrak{a}^\ast_{\bG,\C} \cong \mathfrak{a}^\ast_{\bZ\bk{\bG},\C}$

For any group $G$, let $\Id_G$ denote the trivial representation of $G$ and if there is no source of confusion on the identity of $G$, it will simply be denoted by $\Id$.

Let $\bX_\bM$ denote the complex manifold of characters of $\bM\bk{\AA}$ trivial on $\bM\bk{F}$.
We use additive notation for $\bX_\bM$ too.
There is a natural embedding of $\mathfrak{a}^\ast_{\bM,\C}$ in $\bX_\bM$.
One can choose a direct sum complement $\bX_\bM = \mathfrak{a}^\ast_{\bM,\C} \oplus \bX_{\bM,0}$, where the characters in $\bX_{\bM,0}$ are of finite order, namely, 
\[
\forall \chi\in \bX_{\bM,0} \ \exists k\in\N:\ \chi^k=\Id_{\bM\bk{\AA}} .
\]
Furthermore, any $\chi\in\bX_\bM$ is of the form
\[
\chi = \chi_C+\suml_{i\not\in\Theta} \chi_i\circ \omega_{\alpha_i},
\]
where $\chi_i\in\bX_{\mathbf{GL_1}}$ for $i\not\in\Theta$ and $\chi_C\in\bX_{\bZ\bk{\bG}}$.

In what follows, it will be convenient to set the following notations:
\begin{itemize}
	\item For $\alpha\in\Phi^{+}$ and $\epsilon\in\R$, let $H_{\alpha}^\epsilon$ denote the affine hyper-plane in $\mathfrak{a}_{\bM,\C}^\ast$ given by
	\[
	H_{\alpha}^\epsilon = 
	\set{\lambda\in\fraka^\ast_\C \mvert \inner{\lambda,\check{\alpha}}=\epsilon} .
	\]
	
	\item For a standard Levi subgroup $\bM$ of $\bG$, we fix a shifted positive Weyl chamber
	\[
	\mathfrak{F}_M^{+}=\set{\lambda \mvert \Real\bk{\gen{\lambda-\rho_{\mathbf{B}}+\rho_{\mathbf{P}},\check{\alpha}}}>1 \quad \forall \alpha\in \Delta^{+}\setminus\Delta_M^{+}}.
	\]
	In particular, we write $\mathfrak{F}^{+}=\mathfrak{F}_\bT^{+}$.
	
	\item For $\alpha\in\Phi^{+}_\bG$, we fix the following affine functions on $\mathfrak{a}_{\bM,\C}^\ast$:
	\[
	l_\alpha^{\pm}\bk{\lambda} = \gen{\lambda,\check{\alpha}}\pm 1 .
	\]
\end{itemize}

\section{Representations}

In this manuscript, we will be interested in parabolic induction from one-dimensional representations.

We fix a maximal compact subgroup $\mathbf{K}=\prod_{v\in\Places}K_v$ of $\bG\bk{\AA}$ by taking $K_v=\bG\bk{\mO}$ for all $v\in\Places_{fin.}$ such that $\bG$ is defined over $\mO$ (this covers almost all places), by taking $K_v$ to be a hyper-special maximal compact subgroup of $\bG\bk{F_v}$ for other $v\in\Places_{fin.}$ and by taking $K_v$ to be any maximal compact subgroup of $\bG\bk{F_v}$ for $v\in\Places_{\infty}$ (note that all maximal compact subgroups are conjugate in this case).

Fix a standard Levi subgroup $\bM$ of $\bG$ and a Hecke character 
\[
\mu:\bM\bk{F}\lmod \bM\bk{\AA}\to\C^\times,
\]
we will usually assume that it is of finite order.
For $\lambda\in\mathfrak{a}_{\bM,\C}^\ast$, we consider the normalized parabolic induction
\[
\DPS{\bP}\bk{\lambda,\mu} = \Ind_{\bP\bk{\AA}}^{\bG\bk{\AA}} \bk{\lambda\otimes\mu}.
\]

A section $f_\lambda\in \DPS{\bP}\bk{\lambda,\mu}$ is called
\begin{itemize}
	\item
	\emph{Standard}  if $f_\lambda$ is independent of $\lambda$  when
	restricted to $\mathbf{K}$.
	\item \emph{Spherical} if it is standard and $\mathbf{K}$-invariant. 
	We further call
	such a section \emph{normalized} if $f_\lambda\bk{1}=1$. Since the space of spherical vectors
	in $\DPS{\bP}\bk{\lambda,\mu}$ is at most one-dimensional, the normalized spherical vector (if it exists) is unique.
	\item  \emph{Holomorphic} if $f_\lambda\bk{g}$ is a holomorphic function of $\lambda$ for any $g\in \bG\bk{\AA}$.

\end{itemize}

\begin{Rem}
	Note that any $\mathbf{K}$-finite holomorphic section is a finite combination of standard sections with coefficients in the ring of holomorphic functions on $\mathfrak{a}_{\bM,\C}^\ast$.
\end{Rem}

When $\bP$ is a (proper) maximal parabolic subgroup of $\bG$, $\DPS{\bP}\bk{\lambda,\mu}$ is called a \textbf{degenerate principal series representation}.

The global parabolic induction $\DPS{\bP}\bk{\lambda,\mu}$ is a restricted tensor product of local parabolic inductions
\[
\DPS{\bP}\bk{\lambda,\mu} = \otimes_{v\in\Places}' \DPS{\bP}\bk{\lambda,\mu}_v,
\]
where $\DPS{\bP}\bk{\lambda,\mu}_v$ is given as follows.
Write $\mu=\otimes_{v\in \Places}\mu_v$ and denote the local normalized induction by
\[
\DPS{\bP}\bk{\lambda,\mu}_v = \Ind_{\bP\bk{F_v}}^{\bG\bk{F_v}} \bk{\lambda\otimes \mu_v}
\]
We will use $\DPS{\bP}\bk{\lambda,\mu}_v$, $I_{\bP,v}\bk{\lambda,\mu}$ and $\DPS{\bP}\bk{\lambda,\mu_v}$ interchangeably to denote the same local induced representation.

In \Cref{Chap:Local_DPS} we recall the results of \cite{SDPS_E6}, \cite{SDPS_E7} and \cite{SDPS_E8} regarding the structure of $\DPS{\bP}\bk{\lambda,\mu}_v$ for $\bP$ a maximal parabolic subgroup of a simple group of type $E_n$ and $v\in\Places_{fin.}$.

Let $\bP=\bP_i$ be the maximal parabolic subgroup of $\bG$ such that $\Theta_\bM=\Delta\setminus\set{\alpha_i}$ and let $\chi\in\bX_{\mathbf{GL_1}}$, we write
\[
\mu_\chi = \chi\circ\omega_{\alpha_i} .
\]
We write
\[
\begin{array}{l}
\lambda_{s}^{\bP_i} = s\omega_{\alpha_i} -\rho_{\bB} + \rho_{\bP} , \quad 
\lambda_{s,\chi}^{\bP_i} = \mu_\chi+\lambda_{s}^{\bP_i} \\
\eta_{s}^{\bP_i} = s\omega_{\alpha_i} +\rho_{\bB} - \rho_{\bP} , \quad 
\eta_{s,\chi}^{\bP_i} = \mu_\chi+\eta_{s}^{\bP_i} .
\end{array}
\]
For simplicity, we will write $\DPS{\bP}\bk{s,\chi}$ for $\DPS{\bP}\bk{\mu_\chi,s\circ\omega_{\alpha_i}}$.
We note that
\begin{equation}
	\label{Eq:Natural_Embedding_and_Surjection_to_Borel}
	I_{\bP_i}\bk{s,\chi} \hookrightarrow i_\bB\bk{\mu_\chi,\lambda_{s}^{\bP_i}} = \Ind_{\bB\bk{\AA}}^{\bG\bk{\AA}} \bk{ \lambda_{s,\chi}^{\bP_i} } , \quad
	i_\bB\bk{\eta_{s}^{\bP_i},\mu_\chi} = \Ind_{\bB\bk{\AA}}^{\bG\bk{\AA}} \bk{ \eta_{s,\chi}^{\bP_i} }  \twoheadrightarrow \DPS{\bP_i}\bk{s,\chi}.
\end{equation}
We make similar notations in the local case.

We let $\DPS\bP\bk{s,\chi}^0$ denote the set of standard factorizable sections $\otimes_{v\in\Places} f_{s,v}$ in $\DPS\bP\bk{s,\chi}$ such that $f_{s,v}=f_{s,v}^0$ is spherical for all $v=\infty$.

Throughout, we denote the contragredient of a representation $\pi$ by $\pi^\ast$.
Also, when $\chi=\Id$, we drop it form our notations, namely, we write $\DPS{\bP_i}\bk{s}$ for $\DPS{\bP_i}\bk{s,\chi}$.
Similarly, if $\mu=\Id$ we write $\DPS\bB\bk{\lambda}$ for $\DPS\bB\bk{\Id,\lambda}$, etc.

Also, we note that for each $\DPS{\bP}\bk{\lambda,\mu}$ there exist a unique $\DPS{\bB}\bk{\widetilde{\lambda},\widetilde{\mu}}$ such that $\lambda\otimes\mu$ is the unique irreducible subrepresentation of $\Ind_{\bB\bk{\AA}\cup\bM\bk{\AA}}^{\bM\bk{\AA}} \bk{\widetilde{\lambda}\otimes\widetilde{\mu}}$. We call $\widetilde{\lambda}\otimes\widetilde{\mu}$ the \textbf{initial exponent} of $\DPS{\bP}\bk{\lambda,\mu}$.
Similarly we define initial exponents for the local representations $\DPS{\bP}\bk{\lambda,\mu}_v$.

\section{Intertwining Operators}

For $\lambda\in\mathfrak{a}_{\bT,\C}^\ast$, a unitary Hecke character $\mu:T\bk{F}\lmod T\bk{\AA}\to\C^\times$ and $w\in W$ we consider the standard intertwining operator given by the integral
\begin{equation}
M\bk{w,\lambda,\mu}f_\lambda\bk{g} = \intl_{\bN\bk{\AA}\cap w\bN\bk{\AA}w^{-1}\lmod \bN\bk{\AA}} f_\lambda\bk{w^{-1}ug} du.
\end{equation}
This integral converges in the shifted positive Weyl chamber $\mathfrak{F}^{+}$  to a holomorphic family of operators and admits a meromorphic continuation to $\mathfrak{a}_{\bT,\C}^\ast$, where
\[
\mathfrak{F}^{+}=\set{\lambda \mvert \Real\bk{\gen{\lambda-\rho_{\mathbf{B}},\check{\alpha}}}>0 \quad \forall \alpha\in \Phi^{+}}.
\]

At points of holomorphy, $M\bk{w,\lambda,\mu}$ defines an intertwining operator
\[
M\bk{w,\lambda,\mu} : \DPS{\bB}\bk{\lambda,\mu} \to \DPS{\bB}\bk{w\cdot\lambda,w\cdot\mu} .
\]

We note the following cocycle relation on the standard intertwining operators
\begin{Lem}
	For any $w,w'\in W$ we have
	\[
	M\bk{ww',\lambda,\mu} = M\bk{w,w'\cdot\mu, w'\cdot\lambda} \circ M\bk{w',\lambda,\mu} .
	\]
\end{Lem}

For a place $v\in\Places$, $\lambda\in\mathfrak{a}_{\bT,\C}^\ast$, a unitary character $\mu_v:\bT\bk{F_v}\to\C^\times$, a section $f_{\lambda,v}\in \DPS{\bB,v}\bk{\lambda,\mu}$ and $w\in W$ we consider the standard local intertwining operator given by the integral
\begin{equation}
\label{Eq:Local_Intertwining_Operator_Def}
M_v\bk{w,\lambda,\mu_v} f_{\lambda,v}\bk{g} = \intl_{\bN\bk{F_v}\cap w\bN\bk{F_v}w^{-1}\lmod \bN\bk{F_v}} f_{\lambda,v}\bk{w^{-1}ug} du .
\end{equation}
This integral converges absolutely to an analytic function in the shifted positive Weyl chamber $\mathfrak{F}^{+}$ and admits a meromorphic continuation to $\mathfrak{a}_{\bT,\C}^\ast$.
Furthermore, it holds that the global intertwining operator decomposes into a restricted tensor product of local operators
\[
M\bk{w,\lambda,\mu} = \bigotimes_{v\in\Places}\,' M_v\bk{w,\lambda,\mu_v} .
\]
Namely, given a pure tensor $f_\lambda = \otimes' f_{\lambda,v}$ it holds that
\begin{equation}
\label{Eq:Decomposition_of_global_intertwining_operator_to_local_ones}
M\bk{w,\lambda,\mu}f_\lambda = \bigotimes_{v\in\Places}\,' M_v\bk{w,\lambda,\mu_v}f_{\lambda,v}
\end{equation}

Assuming that $\mu_v$ is unramified and applying $M_v\bk{w,\lambda,\mu_v}$ to the normalized spherical section $f_{\lambda,v}^0\in \DPS{\bB,v}\bk{\lambda,\mu}$ yields the local Gindikin-Karpelevich formula
\begin{equation}
\label{Eq:Local_GK_Formula}
M_v\bk{w,\lambda,\mu_v} f_{\lambda,v}^0 = \bk{\prodl_{\gamma\in R\bk{w}} \frac{\Lfun_{F_v}\bk{\inner{\lambda,\check{\gamma}},\mu_v\circ\check{\gamma}}}{\Lfun_{F_v}\bk{\inner{\lambda,\check{\gamma}}+1,\mu_v\circ\check{\gamma}}}} f_{w\cdot\lambda,v}^0 ,
\end{equation}
where
\[
R\bk{w} = \set{\alpha\in\Phi^{+}_\bG \mvert w\cdot\alpha\notin \Phi^{+}_\bG} .
\]
Let \textbf{local Gindikin-Karpelevich factor} $J_v\bk{w,\lambda,\mu_v}$ be defined by
\begin{equation}
J_v\bk{w,\lambda,\mu_v}
= \prodl_{\gamma\in R\bk{w}} \frac{\Lfun_{F_v}\bk{\inner{\lambda,\check{\gamma}},\mu_v\circ\check{\gamma}}}{\Lfun_{F_v}\bk{\inner{\lambda,\check{\gamma}}+1,\mu_v\circ\check{\gamma}}} .
\end{equation}

It is convenient to define the normalized standard intertwining operator
\begin{equation}
\label{eq:Normalized_intertwining_operators}
N_v\bk{w,\lambda,\mu_v} =
\bk{\prodl_{\gamma\in R\bk{w}} 
\frac{ \Lfun_{F_v}\bk{\inner{\lambda,\check{\gamma}}+1,\mu_v\circ\check{\gamma}} }{ \Lfun_{F_v}\bk{\inner{\lambda,\check{\gamma}},\mu_v\circ\check{\gamma}} \epsilon_{F_v}\bk{\inner{\lambda,\check{\gamma}},\mu_v\circ\check{\gamma},\psi_v} } }
M\bk{w,\lambda,\mu_v} .
\end{equation}

The normalized intertwining operator satisfies the cocycle equation
\[
N_v\bk{ww',\lambda,\mu_v} = N_v\bk{w,w'\cdot\mu_v,w'\cdot\lambda} \circ N_v\bk{w',\lambda,\mu_v} \quad \forall w,w'\in W
\]
and, if $\mu_v$ is unramified,
\[
N_v\bk{w,\lambda,\mu_v} f_{\lambda,v}^0 = f_{w\cdot\lambda,v}^0 .
\]

Let $f\lambda\in \DPS\bB\bk{\lambda,\mu}$ be a pure tensor $f_\lambda=\otimes_{v\in\Places} f_{\lambda,v}$ and let $S\subset\Places$ be a finite subset such that $f_{\lambda,v}=f_{\lambda,v}^{0}$ for all $v\notin S$.
The global Gindikin-Karpelevich formula follows from \Cref{Eq:Local_GK_Formula},
\begin{equation}
\label{Eq:Global_GK_Formula}
\begin{split}
& M\bk{w,\lambda,\mu} f_{\lambda}
= \bk{\displaystyle\operatorname*{\otimes}_{v\in S} M_v\bk{w,\lambda,\mu_v} f_{\lambda,\nu} }\bigotimes \bk{\displaystyle\operatorname*{\otimes}_{v\notin S} J_v\bk{w,\lambda,\mu_v} f_{\lambda,\nu}^0} \\
& = J\bk{w,\lambda,\mu} \bk{\displaystyle\operatorname*{\otimes}_{v\in S}  J_v\bk{w,\lambda,\mu_v}^{-1} M_v\bk{w,\lambda,\mu_v}f_{\lambda,\nu} }\bigotimes \bk{\displaystyle\operatorname*{\otimes}_{v\notin S} f_{\lambda,\nu}^0} \\
& = \bk{\prodl_{\gamma\in R\bk{w}} \epsilon_{F}\bk{\gen{\lambda,\check{\alpha}},\mu\circ\check{\alpha}}} J\bk{w,\lambda,\mu} \bk{\displaystyle\operatorname*{\otimes}_{v\in S}  N_v\bk{w,\lambda,\mu_v}f_{\lambda,\nu} }\bigotimes \bk{\displaystyle\operatorname*{\otimes}_{v\notin S} f_{\lambda,\nu}^0} \ ,
\end{split}
\end{equation}
where the \textbf{global Gindikin-Karpelevich factor} $J\bk{w,\lambda,\mu}$ is defined by
\begin{equation}
J\bk{w,\lambda,\mu} = \prodl_{v\in\Places} J_v\bk{w,\lambda,\mu_v} 
= \prodl_{\gamma\in R\bk{w}} \frac{\Lfun_{F}\bk{\inner{\lambda,\check{\gamma}},\mu\circ\check{\gamma}}}{\Lfun_{F}\bk{\inner{\lambda,\check{\gamma}}+1,\mu\circ\check{\gamma}}} .
\end{equation}

This implies that the analytic behavior of $M\bk{w,\lambda,\mu} f_{\lambda}$ depends on the analytic behaviour of $J\bk{w,\lambda,\mu}$ and the $N_v\bk{w,\lambda,\mu_v}f_{\lambda,\nu}$ for $\nu\in S$.

Since the partially normalized intertwining operators
\[
\frac{1}{\prodl_{\gamma\in R\bk{w}} \Lfun_{F_v}\bk{\gen{\lambda,\check{\gamma}} ,\mu_v\circ\check{\gamma}}} M_v\bk{w,\lambda,\mu_v}
\]
are entire for all $\nu\in\Places$ (see \cite{MR944102} when $\nu\vert\infty$ and \cite{MR517138} when $\nu\not\vert\infty$), it follows that $N_v\bk{w_\alpha,\lambda,\mu_v}$ is holomorphic whenever $\Real\bk{\gen{\lambda,\check{\alpha}}}>-1$ and $\alpha\in\Delta$.

The holomorphicity of $N_v\bk{w,\lambda,\mu_v}f_{\lambda,\nu}$ will be discussed with more details in \Cref{Chap:Poles_of_Eisen_ser}.

We write $M_w\bk{s,\chi}$ for the intertwining operator $M\bk{w,\lambda_{s,\chi}}\res{\DPS\bP\bk{s,\chi}}$, where we think of $\DPS\bP\bk{s,\chi}$ as embedded in $\DPS{\bB}\bk{\mu_\chi,\lambda_{s}^{\bP}}$ following \Cref{Eq:Natural_Embedding_and_Surjection_to_Borel}.
Similarly, we write $N_w\bk{s,\chi}$ for the intertwining operator $N_v\bk{w,\lambda_{\chi_v,s}}\res{\DPS{\bP,v}\bk{s,\chi}}$.

\begin{Rem}
	Note that, for convenience, we sometimes denote $M\bk{w,\lambda\otimes\mu}$ for $M\bk{w,\lambda,\mu}$, $M_v\bk{w,\lambda\otimes\mu}$ for $M_v\bk{w,\lambda,\mu}$, etc.
\end{Rem}

\section{Eisenstein Series and the Constant Term Formula}

For $\mu\in\bX_{\bM,0}$ and a standard section $f_\lambda\in \DPS{\bP}\bk{\lambda,\mu}$ we form the associated Eisenstein series
\begin{equation}
\label{Eq:Degenerate_Eisenstein_series_definition}
\Eisen_{\bP}\bk{f,\lambda,\mu,g} = \suml_{\gamma\in \bP\bk{F}\lmod \bG\bk{F}} f_\lambda\bk{\gamma g} .
\end{equation}
This series converges for $\lambda\in \mathfrak{F}_M^{+}$ and admits a meromorphic continuation to $\mathfrak{a}_{\bM,\C}^\ast$.

At points $\lambda_0$ where $\Eisen_{\bP}\bk{f,\lambda,\mu,g}$ is holomorphic, the value $\Eisen_{\bP}\bk{f,\lambda_0,\mu,g}$ is an automorphic form as a function of $g$.
Moreover, singularities of $\Eisen_{\bP}\bk{f,\lambda,\mu,g}$ lie along hyperplanes in $\mathfrak{a}_{\bM,\C}^\ast$ and the iterated residues along these singularities give rise to automorphic forms too.
By residues here, and throughout this manuscript, we mean the lowest coefficient in the Laurent series and not necessarily the $-1$ one.

If $\bP$ is a maximal parabolic subgroup of $\bG$, we denote sections of $\DPS\bP\bk{s,\chi}$ by $f_s$ and the \textbf{degenerate Eisenstein series} $\Eisen_{\bP}\bk{f,\lambda_s^\bP,\mu_\chi,g}$ by $\Eisen_{\bP}\bk{f,s,\chi,g}$.

We have the following connection between $\Eisen_{\bP}$ and $\Eisen_{\bB}$ (see \cite[Proposition 2.7]{SegalResiduesD4} for a proof).
\begin{Prop}
	\label{Prop:Equality_of_Eisenstein_series}
	\begin{enumerate}
	\item For any $f\in \DPS{\bP}\bk{\lambda,\mu}$	it holds that
	\begin{equation}
	\label{Eq:Equality_of_Degenerate_Eisenstein_Series_on_P_and_on_B}
	\Eisen_{\bP}\bk{f,\lambda,\mu,g} = \Eisen_{\bB}\bk{f,\lambda\otimes\FNorm{\rho_{\bB}-\rho_{\bP}},\mu,g} .
	\end{equation}
	
	\item For any $f\in \DPS{\bP}\bk{\lambda,\mu}$	there exists a section $\widetilde{f}\in \DPS{\bB}\bk{\lambda\otimes\FNorm{\rho_{\bP}-\rho_{\bB}},\mu}$ such that
	$M\bk{w_{\bM,l},\lambda,\mu}f=\widetilde{f}$ and
	\begin{equation}
	\Eisen_{\bP}\bk{f,\lambda,\mu,g} =
	\lim\limits_{\lambda'\to\lambda}
	\bk{\prodl_{\alpha\in\Phi_\bM^{+}} \bk{\gen{\lambda',\check{\alpha}}-1}} \Eisen_{\bB}\bk{\tilde{f},\lambda'\otimes\FNorm{\rho_{\bP}-\rho_{\bB}},\mu,g} ,
	\end{equation}
	where here $\lambda\in\mathfrak{a}_{\bM,\C}^\ast$ and $\lambda'\in\mathfrak{a}_{\bT,\C}^\ast$.
	\end{enumerate}
\end{Prop}

\begin{Rem}
	Item \textit{(2)} can be understood as $\Eisen_{\bP}\bk{f,\lambda,\mu,g}$ being realized as an iterated residue of $\Eisen_{\bB} \bk{\tilde{f}, \lambda'\otimes\FNorm{\rho_{\bB}-\rho_{\bP}}, \mu, g}$.
	Namely, if we write $\Phi_\bM^{+}=\set{\beta_1,...,\beta_n}$
	\[
	\frakh_k = \set{\lambda'\in \mathfrak{a}_{\bT,\C}^\ast \mvert \gen{\lambda',\check{\beta_k}}=1},
	\]
	then the iterated residue 
	\begin{equation}
		\Res{\frakh_{n}\cap...\cap\frakh_1}
		\bk{\Res{\frakh_{n-1}\cap...\cap\frakh_1}
			\bk{ ... \Res{\frakh_1}
				\bk{\Eisen_{\bB}}}}
	\end{equation}
	equals (up to a scalar) to $\Eisen_{\bP}\bk{f,\lambda,\mu,g}$.
	
\end{Rem}

If $\bP=\bP_i$ is a maximal Levi subgroup of $\bG$, we denote
\[
\Eisen_{\bP}\bk{f,s,\chi,g} = \Eisen_{\bP}\bk{f,\lambda_{s}^{\bP},\mu_\chi,g} .
\]

The constant term of $\Eisen_\bP\bk{f,\lambda,\mu,g}$ along $\bB$ is given by
\[
\Eisen_\bP\bk{f,\lambda,\mu,g}_\bCT = \intl_{\bN\bk{F}\lmod \bN\bk{\AA}} \Eisen_\bP\bk{f,\lambda,\mu,ug} \, du .
\]
When restricted to $\bT\bk{\AA}$, this is an automorphic form on $\bT\bk{\AA}$.
However, it is also a function of $\bG\bk{\AA}$ and the map
\[
f_{\lambda_0} \mapsto \Eisen_\bP\bk{f,\lambda_0,\mu,\cdot}_\bCT
\]
is $\bG\bk{\AA}$-equivariant when $\Eisen_\bP\bk{f,\lambda,\mu,\cdot}_\bCT$ is holomorphic at $\lambda_0$.
At singular points, the iterated residue of the constant term is $\bG\bk{\AA}$-equivariant.

The constant term formula, computed as in \cite{MR1469105}, is given as follows:
\begin{equation}
\label{Eq:Constant_term}
\Eisen_\bP\bk{f,\lambda,\mu,g}_\bCT = \suml_{w\in W\bk{\bM,\bG}} M\bk{w,\lambda,\mu} f_\lambda \bk{g}.
\end{equation}

The analytic behavior of $\Eisen_{\bP}\bk{f,s,\chi}$ can be studied via that of $\Eisen_{\bP}\bk{f,s,\chi}_\bCT$.
The constant term formula can also be used to study the residual representation of $\Eisen_{\bP}\bk{f,s,\chi}$ at singular points.
More precisely, we have the following:
\begin{Prop}
	\label{NonvanishingifCTisNonvanishing}
	\label{Cor:Kernel_of_Series_is_kernel_of_CT}
	
	Let $\bP=\bP_i$ denote a maximal parabolic subgroup of $\bG$ and assume that $\Eisen_{\bP}\bk{f,s,\chi}$ admits a pole of order $m$ at $s_0\in\C$.
	
	\begin{enumerate}
		\item Let $f_s\in \DPS{\bP}\bk{s,\chi}$ be a holomorphic section, let
		\[
		\varphi\bk{g} = \lim_{s=s_0} \coset{\bk{s-s_0}^m \Eisen_{\bP}\bk{f,s,\chi,g}} .
		\]
		and let $\varphi_{\bCT}$ denote the constant term of $\varphi$ along $\bN$.
		Then, $\varphi\equiv 0$ if and only if $\varphi_{\bCT}\equiv 0$.
		
		\item As $\C$-vector spaces,
		\begin{equation}
		\label{Eq:Kernel_of_Series_is_kernel_of_CT}
		\begin{array}{l}
		Span_{\C}\set{\lim\limits_{s\to s_0} \bk{s-s_0}^m \Eisen_\bP\bk{f,s,\chi} \mvert f_s\in \DPS{\bP}\bk{s,\chi}} \\
		\cong Span_{\C} \set{\lim\limits_{s\to s_0} \bk{s-s_0}^m \suml_{w\in W\bk{\bM,\bG}} M\bk{w,\lambda_{s,\chi}^{\bP_i}}f_s\res{\bT\bk{\AA}} \mvert f_s\in \DPS{\bP}\bk{s,\chi}} .
		\end{array}
		\end{equation}
	\end{enumerate}
\end{Prop}

%

\section{Entire $W$-invariant Spherical Normalized Eisenstein Series}

We consider the following normalization of the spherical Eisenstein series as follows:
\begin{equation}
\label{qq:Normalized_Eisenstein_Series}
\Eisen_{B}^\sharp\bk{\lambda,g} = 
\coset{\prodl_{\alpha\in\Phi^+} \zfun\bk{l_\alpha^+\bk{\lambda}} l_\alpha^+\bk{\lambda} l_\alpha^-\bk{\lambda}}
\Eisen_{B}\bk{f^0_\lambda,\lambda,g} ,
\end{equation}
where $\lambda\in \mathfrak{a}_{\bT,\C}^\ast$ and
\[
l_\alpha^\pm\bk{\lambda} = \gen{\lambda,\check{\alpha}} \pm 1 .
\]

It holds that:
\begin{Prop}
	\label{Prop:W_inv_entire_Eisen_Ser}
	The normalized Eisenstein series $\Eisen_{B}^\sharp\bk{\lambda,g}$ is entire and $W$-invariant in the sense that for any $w\in W$ it holds that $\Eisen_{B}^\sharp\bk{w\cdot \lambda,g}=\Eisen_{B}^\sharp\bk{\lambda,g}$.
\end{Prop}
For a proof, see \cite[Appendix C]{gurevich_segal_2019}.

\section{An Extension to a Result of Keys-Shahidi}

We recall \cite[Proposition 6.3]{MR944102}.
For $\bG=SL_2$, let $w_0$ denote the non-trivial element in $W_\bG$.
It holds that
\[
M_{w_0}\bk{\overline{0}} = -Id ,
\]
where $\overline{0}$ denotes the zero vector in $\mathfrak{a}^\ast_{\bT,\C}$.
In \cite{MR944102}, it is stated and proven as result on an intertwining operator between automorphic principal series representations of $SL_2\bk{\AA}$.
Essentially, the proof boils down to the fact that the Eisenstein series $\Eisen_{\bP}\bk{f,\lambda,\Id,g}$ admits a simple zero at $\lambda=\overline{0}$ and hence also $\Eisen_\bP\bk{f,\lambda,\Id,g}_\bCT$.
It follows that
\[
\lim\limits_{\lambda \to \overline{0}} \Eisen_\bP\bk{f,\lambda,\Id,g}_\bCT = \lim\limits_{\lambda\to\overline{0}} \bk{I+M\bk{w_0,\lambda,\Id}} f = 0,
\]
from which we deduce the identity $M\bk{w_0,\overline{0},\Id}=-Id$.

We use \Cref{Prop:W_inv_entire_Eisen_Ser} to generalize this result.
\begin{Cor}
	\label{Cor:Extension_to_KS}
	The sum
	\[
	\suml_{w\in W_\bG} M\bk{w,\lambda,\Id}
	\]
	admits a zero of order at least $\Card{\Phi_\bG^{+}}$ at $\lambda = \overline{0}$.
\end{Cor}

\begin{proof}
	We note that
	\begin{itemize}
		\item $\Eisen_{B}^\sharp\bk{\lambda,g}$ is entire and hence, in particular, holomorphic at $\lambda=\overline{0}$.
		
		\item The factors $l_\alpha^+\bk{\lambda} l_\alpha^-\bk{\lambda}$ are holomorphic and non-zero at $\lambda=\overline{0}$.
		
		\item The factors $\zfun\bk{l_\alpha^+\bk{\lambda}}$ admit a simple pole at $\lambda=\overline{0}$.
	\end{itemize}
	Hence, $\Eisen_{B}\bk{f^0_\lambda,\lambda,g}$ admits a zero of order at least $\Card{\Phi_\bG^{+}}$.
	This holds for $\Eisen_{B}\bk{f^0_\lambda,\lambda,g}_\bCT$ two.
	The claim then follows from \Cref{Eq:Constant_term}.
\end{proof}

\section{Decomposition of Principal Series Representations \`a la \v Zampera}
Fix a place $v\in\Places$.
We recall the main result of \cite{SegalSingularities} which generalized \cite[Lemma 3.1 and 3.2]{MR1485422}.

Let $M$ be a Levi subgroup of $G$ with $\beta\in\Delta_M$, let
\[
\lambda_0 = w_{\beta}\cdot\bk{-\varpi_\beta + \suml_{\alpha\in\Delta\setminus\Delta_M} a_\alpha \circ \alpha}
\]
and let $w=\s{\beta} w' \s{\beta}\in W_M$ be a Weyl element so that
\begin{itemize}
	\item $w'\in\Stab_{W_M}\bk{\s{\beta}\cdot\lambda_0}$.
	\item $\coset{\s{\beta},w'}\neq 1$.
	\item $N_{w',v}\bk{\s{\beta}\cdot\lambda}$ is holomorphic at $\lambda_0$ and satisfies
	\[
	N_{w',v}\bk{\s{\beta}\cdot\lambda_0} = Id .
	\]
\end{itemize}
Also, fix a line $\lambda_s$ so that
\[
\begin{array}{ccc}
\gen{\lambda_s,\check{\beta}} = 1 & \Leftrightarrow & \lambda_s=\lambda_0 \\
\gen{w\cdot\lambda_s,\check{\beta}} = 1 & \Leftrightarrow & \lambda_s=\lambda_0 .
\end{array}
\]

Then, the following holds:
\begin{Prop}
	\label{Prop:Zampera}
	The operator $N_{w,v}\bk{\lambda_s}$ is holomorphic along $\lambda_s$ near $s_0=0$.
	Furthermore,
	\[
	i_T^M\lambda_0 = \bk{i_{M_{\beta}}^M \Omega_0} \oplus \bk{i_{M_{\set{\beta}}}^M St_{M_{\set{\beta}}}\otimes\Omega_0},
	\]
	where $\Omega_0\in\bX\bk{M_{\set{\beta}}}$ satisfy $r_T^{M_{\set{\beta}}}\Omega_0=\lambda_0$.
	Let
	\[
	\Xi = \lim\limits_{s\to 0} N_{w,v}\bk{\lambda_s}
	\]
	and assume that the limit
	\[
	\kappa = -\lim\limits_{s\to 0} \frac{l_\beta^{-}\bk{\lambda_s}}{l_\beta^{-}\bk{w\cdot\lambda_s}}
	\]
	exists.
	If $\kappa\neq 1$, then $i_{M_{\beta}}^M \Omega_0$ is the $1$-eigenspace of $\Xi$ and $i_{M_{\set{\beta}}}^M St_{M_{\set{\beta}}}\otimes\Omega_0$ is the $\kappa$-eigenspace of $T$.
	
\end{Prop}

\section{General Remarks Regarding Notations}
\label{Sec:General_Remarks_and_Notations}

Throughout this manuscript, many calculations involve both local and global objects.
Also, there are many calculations which involve the irreducible constituents of the local degenerate principal series.
In order to do these efficiently, we set a few general guidelines for notations:
\begin{itemize}
	\item Also, in any vector space $V$ over $\C$, we denote the zero vector by $\bar{0}$.
	
	\item For an algebraic group $\GG$ we write $G_v$ for $\GG\bk{F_v}$ (similarly, $T_v$ for $\TT\bk{F_v}$ etc.).
	
	\item The spaces of unramified character of $\bG_\AA$ and $G_v$ admit natural isomorphism to each other, this allows us to identify them in our notations.
	
	\item For a local degenerate principal series $\pi_v=\DPS{P}\bk{\chi,s_0}_v$, with $\Real\bk{s_0}>0$, $\pi$ admits a unique irreducible subrepresentation which will be denoted by $\pi_{0,v}$.
	
	We associate its irreducible quotients as eigenspaces of certain intertwining operators (its ``\emph{R-group}'').
	Namely, we denote its irreducible quotients by $\pi_{\eta,v}$ where $\eta$ is the relevant eigenvalue.
	In particular, when $\pi$ admits a unique irreducible quotient, it will be denoted by $\pi_{1,v}$.
	As a byproduct, if $\chi_v$ is unramified, then $\pi_{1,v}$ is the unique unramified subquotient of $\pi_v$.
	
	\item Other irreducible subquotients of $\pi$, if exist, will be denoted by $\sigma_i$, where $i\in\N$ is determined arbitrarily.
	
	\item In chapters dealing only with local representations, we fix a place $v\in\Places_{fin.}$ and drop $v$ from the notations.
	That is, in those chapters, $\DPS{P}\bk{\chi,s_0}$ stands for $\DPS{P}\bk{\chi,s_0}_v$, $\pi_\eta$ stands for $\pi_{\eta,v}$ etc.
	This is stated at the beginning of any such chapter.
	
	\item For a subquotient $\sigma$ of $\pi$, we denote by $\widetilde{\sigma}$ the inverse image of $\sigma$ in $\pi$.
	
	\item For a complex function $h\bk{s}$, we say that $h\bk{s}$ admits a pole of order $n$ at $s=s_0$ if
	\[
	\lim\limits_{s\to s_0}\bk{s-s_0}^n h\bk{s}\in\C^\times .
	\]
	We denote this number by $\operatorname{ord}_{s=s_0}h\bk{s}$.

	\item If $\bP=\bP_i$ is a maximal parabolic subgroup, we also write $\Eisen_{P_i}\bk{\chi}$ or $\Eisen_{P_i}\bk{s,\chi}$ for $\Eisen_{\bP_i}\bk{\cdot,\lambda_s^{\bP_i},\mu_\chi,g}$.
	
	\item If $\Eisen_{P_i}\bk{s,\chi}$ admits a pole of order $k$ at $s=s_0$, we write $\Res{s=s_0} \Eisen_{\bP_i}\bk{\chi}$ or $\ResRep{\bG}{i}{s_0}{ord\bk{\chi}}$ for the residue of $\Eisen_{P_i}\bk{s,\chi}$ at $s=s_0$, that is the representation spanned by
	\[
	\lim\limits_{s\to s_0} \bk{s-s_0}^k \Eisen_{\bP_i}\bk{f,\lambda_s^{\bP_i},\mu_\chi,g} \quad f\in \DPS{\bP_i}\bk{s,\chi} .
	\]
	
	\item Throughout this manuscript, we separate our calculations into cases denoted by triples $\coset{i,s_o,k}$.
	This triple stands for the residue $\ResRep{G}{i}{s_0}{ord\bk{\chi}}$ of the degenerate principal series $\DPS{P_i}\bk{\chi,s_0}$, where $ord\bk{\chi}=k$ (and note that the results are presented in a unified way for different characters $\chi$ of the same order).
	
	\item If $\chi=\Id$ we would simply use $\Res{s=s_0} \Eisen_{\bP_i}$ and $\ResRepTr{\bG}{i}{s_0}$ for the residual representation.
	
\end{itemize}

\section{Assumptions Made Throughout This Paper}

Through out \Cref{Part:DRS}, we study the poles of the degenerate principal of groups of type $E_n$ under certain constraints and assumptions which we list bellow.
It should be pointed that most of these assumptions seems to be of a technical nature and could be removed by further investigation along the same lines laid down in this manuscript.

\begin{enumerate}
	\item Since there is currently no complete description of the degenerate principal series of $G_\R$ and $G_\C$ for $G$ of type $E_6$, $E_7$ and $E_8$, we use notations which assume that the structure of the local degenerate principal series \underline{at Archimedean places} is analogues to \underline{the non-Archimedean case}.
	Alternatively, we may assume, as done in other works, that for Archimedean $v$, $\chi_v=\Id_v$ and any section $f_{s}=\otimes_{v\in\Places} f_{s,v}$ satisfy $f_{s,v}\in\widetilde{\pi_{1,v}}$ for all $v\divides \infty$.
	
	\item For the evaluation of the order of the poles of $\Eisen\bk{f,s}$ at $s_0$, when $\chi=\Id$ and $I_{P}\bk{s_0}$ is generated by $f_{s_0}^0$,  we make use of \Cref{Eq:Global_GK_Formula} in a way which assumes that $\zfun_F\bk{\frac{1}{2}}\neq 0$.
	When working over a number field $F$ for which $\zfun_F\bk{\frac{1}{2}}= 0$, this require a more careful analysis of the Gindikin-Karpelevich factors in \Cref{Subsection:Methods_Reduction_of_order_of_poles} as this might reduce the order of poles of certain global intertwining operators.
	
	\item When $\chi$ is of order higher than $2$, its epsilon factor $\epsilon_F\bk{s,\chi}$ is possibly non-trivial.
	In such a case, for the simplicity of calculation, we assumed that $\epsilon_F\bk{s,\chi}$ is a constant, which is true when $F=\Q$, but could be false for certain $F$ and $\chi$.
	This could mostly affect the application of item (2) in \Cref{Subsection:Methods_Reduction_of_order_of_poles}.
	
\end{enumerate}
\chapter{Split Exceptional Groups of Type $E_n$}
\label{Chap:Group_En}

In this part, we present the split, simply-connected simple groups of type $E_n$.

Let $G$ be the split, simple, simply-connected group of type $E_n$. In this section we describe the structure of $G$ and set notations for the rest of the section. We fix a Borel subgroup $\para{B}$ and a maximal split torus $T\subset \para{B}$. 
It holds that
\[
\Card{\Phi_G} = \piece{72,& n=6 \\ 126,& n=7 \\ 240,& n=8 .}
\]
We label the simple roots $\Delta_{G}$ and the Dynkin diagram of $G$ as follows:

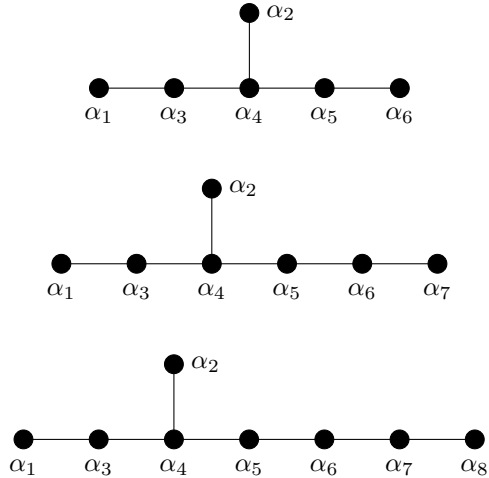
\[\begin{tikzpicture}[scale=0.5]
\draw (-1,0) node[anchor=east]{};
\draw (0 cm,0) -- (8 cm,0);
\draw (4 cm, 0 cm) -- +(0,2 cm);
\draw[fill=black] (0 cm, 0 cm) circle (.25cm) node[below=4pt]{$\alpha_1$};
\draw[fill=black] (2 cm, 0 cm) circle (.25cm) node[below=4pt]{$\alpha_3$};
\draw[fill=black] (4 cm, 0 cm) circle (.25cm) node[below=4pt]{$\alpha_4$};
\draw[fill=black] (6 cm, 0 cm) circle (.25cm) node[below=4pt]{$\alpha_5$};
\draw[fill=black] (8 cm, 0 cm) circle (.25cm) node[below=4pt]{$\alpha_6$};
\draw[fill=black] (4 cm, 2 cm) circle (.25cm) node[right=3pt]{$\alpha_2$};
\end{tikzpicture}\]

\[\begin{tikzpicture}[scale=0.5]
\draw (-1,0) node[anchor=east]{};
\draw (0 cm,0) -- (10 cm,0);
\draw (4 cm, 0 cm) -- +(0,2 cm);
\draw[fill=black] (0 cm, 0 cm) circle (.25cm) node[below=4pt]{$\alpha_1$};
\draw[fill=black] (2 cm, 0 cm) circle (.25cm) node[below=4pt]{$\alpha_3$};
\draw[fill=black] (4 cm, 0 cm) circle (.25cm) node[below=4pt]{$\alpha_4$};
\draw[fill=black] (6 cm, 0 cm) circle (.25cm) node[below=4pt]{$\alpha_5$};
\draw[fill=black] (8 cm, 0 cm) circle (.25cm) node[below=4pt]{$\alpha_6$};
\draw[fill=black] (10 cm, 0 cm) circle (.25cm) node[below=4pt]{$\alpha_7$};
\draw[fill=black] (4 cm, 2 cm) circle (.25cm) node[right=3pt]{$\alpha_2$};
\end{tikzpicture}\]

\[\begin{tikzpicture}[scale=0.5]
\draw (-1,0) node[anchor=east]{};
\draw (0 cm,0) -- (12 cm,0);
\draw (4 cm, 0 cm) -- +(0,2 cm);
\draw[fill=black] (0 cm, 0 cm) circle (.25cm) node[below=4pt]{$\alpha_1$};
\draw[fill=black] (2 cm, 0 cm) circle (.25cm) node[below=4pt]{$\alpha_3$};
\draw[fill=black] (4 cm, 0 cm) circle (.25cm) node[below=4pt]{$\alpha_4$};
\draw[fill=black] (6 cm, 0 cm) circle (.25cm) node[below=4pt]{$\alpha_5$};
\draw[fill=black] (8 cm, 0 cm) circle (.25cm) node[below=4pt]{$\alpha_6$};
\draw[fill=black] (10 cm, 0 cm) circle (.25cm) node[below=4pt]{$\alpha_7$};
\draw[fill=black] (12 cm, 0 cm) circle (.25cm) node[below=4pt]{$\alpha_8$};
\draw[fill=black] (4 cm, 2 cm) circle (.25cm) node[right=3pt]{$\alpha_2$};
\end{tikzpicture}\]

\[
M_j = M_{\Delta\setminus\set{\alpha_j}}, \quad
L_j = M_{\set{\alpha_j}} .
\]

The group $G$ is generated by symbols 
\[\set{x_{\alpha}(r) \: : \: \alpha \in \Phi_{G} ,r \in F}\]
subject to the Chevalley relations as in \cite[Section 6]{MR0466335}.

The Weyl groups of type $E_n$ have the following orders:
\[
\Card{W_{E_n}} = \piece{
51,840,& n=6 \\
2,903,040,& n=7 \\
696,729,600,& n=8 
}
\]
We denote the generators of $W=W_{E_n}$ by $w_1,...,w_n$, these are given by a simple reflection along $\alpha_1,...,\alpha_n$ (respectively).

The Weyl groups of maximal Levi subgroups of $G$ have the following orders:
\begin{longtable}[H]{|c|c|c|c|c|c|c|c|c|}
	\hline
	\diagbox{$n$}{$\Card{W_{M_i}}$}& $i=1$ & $i=2$ & $i=3$ & $i=4$ & $i=5$ & $i=6$ & $i=7$ & $i=8$ \\ \hline
	$n=6$ & 1,920 & 720 & 240 & 72 & 240 & 1,920 && \\ \hline
	$n=7$ & 23,040 & 5,040 & 1,440 & 288 & 720 & 2,520 & 51,840 & \\ \hline
	$n=8$ & 322.560 & 40,320 & 10,080 & 1,440 & 2,880 & 11,520 & 103,680 & 2,903,040 \\ \hline
\end{longtable}
While they have the following indices:
\begin{longtable}[H]{|c|c|c|c|c|c|c|c|c|}
	\hline
	 \diagbox{$n$}{$\Card{W_{E_n}\rmod W_{M_i}}$}& $i=1$ & $i=2$ & $i=3$ & $i=4$ & $i=5$ & $i=6$ & $i=7$ & $i=8$ \\ \hline
	 $n=6$ & 27 & 72 & 216 & 720 & 216 & 27 && \\ \hline
	 $n=7$ & 126 & 576 & 2,016 & 10,080 & 4,032& 1,512& 56 & \\ \hline
	 $n=8$ & 2,160 & 17,280 & 69,120 & 483,840 & 241,920 & 60,480 & 6,720 & 240 \\ \hline
\end{longtable}

Later on, we also denote $W^{M_i,T}=W_{E_n}\rmod W_{M_i}$

The following table is useful to translate between the normalization of the complex variable $s$ used in this paper based on fundamental weights and the one, sometimes used, which is based on the modulus function $\delta_P$ of the parabolic subgroup.
We note that the modulus function satisfies:
\[
\delta_{P_i} = \FNorm{\omega_{\alpha_i}}^{\inner{\check{\alpha_i},\rho_{M_i}}}
\]
and list the possible values of $\inner{\check{\alpha_i},\rho_{M_i}}$ in the following table.

\begin{longtable}[H]{|c|c|c|c|c|c|c|c|c|}
	\hline
	\diagbox{$n$}{$\inner{\check{\alpha_i},\rho_{M_i}}$}& $i=1$ & $i=2$ & $i=3$ & $i=4$ & $i=5$ & $i=6$ & $i=7$ & $i=8$ \\ \hline
	$n=6$ & 12 & 11 & 9 & 7 & 9 & 12 && \\ \hline
	$n=7$ & 17 & 14 & 11 & 8 & 10& 13& 18 & \\ \hline
	$n=8$ & 23 & 17 & 13 & 9 & 11 & 14 & 19 & 29 \\ \hline
\end{longtable}

Throughout this manuscript, we will use a convenient presentation for characters $\lambda\in \mathfrak{a}_{T,\C}^\ast$.
Namely, we write:
\begin{itemize}
	\item For $n=6$:
	\[\esixchar{\Omega_1}{\Omega_2}{\Omega_3}{\Omega_4}{\Omega_5}{\Omega_6}=\sum_{i=1}^{6} \Omega_{\alpha_{i}} \circ \fun{\alpha_i}.\]
	\item For $n=7$:
	\[\esevenchar{\Omega_1}{\Omega_2}{\Omega_3}{\Omega_4}{\Omega_5}{\Omega_6}{\Omega_7}=\sum_{i=1}^{7} \Omega_{\alpha_{i}} \circ \fun{\alpha_i}.\]
	\item For $n=8$:
	\[\eeightchar{\Omega_1}{\Omega_2}{\Omega_3}{\Omega_4}{\Omega_5}{\Omega_6}{\Omega_7}{\Omega_8}=\sum_{i=1}^{8} \Omega_{\alpha_{i}} \circ \fun{\alpha_i}.\]
\end{itemize}

Finally, we note that the dual complex group $\check{G}$ of the simply-connected group $G$ of type $E_n$ is the complex adjoint group of type $E_n$.
We use the same notations for the root system of $\check{G}$ as for that of $G$.
\chapter{Local Degenerate Principal Series}
\label{Chap:Local_DPS}

In this part, we recall information on the local degenerate principal series of groups of type $E_n$.

\section{Reducibility and Length of the Socle and Co-Socle - $p$-adic Places}
For this section, fix a finite place $v\in\Places$.
In this section we quote from \cite{SDPS_E6,SDPS_E7,SDPS_E8} the list of reducible degenerate principle series of the groups $E_6\bk{F_\nu}$, $E_7\bk{F_\nu}$ and $E_8\bk{F_\nu}$ as well as the length of their socle and co-socle (that is, their maximal semi-simple subrepresentation and quotient).

\subsection{$E_6$}
Assume that $G$ of type $E_6$.
\begin{Thm}
	The representations $\DPS{P_i}\bk{s,\chi}_v$ is reducible if and only if the triple $\bk{i,s,ord\bk{\chi}}$ is in the following list:
	\begin{enumerate}
		\item $i=1$: $\bk{1,\pm 6, 1}$ and $\bk{1,\pm 3, 1}$.
		\item $i=2$: $\bk{2,\pm \frac{11}{2}, 1}$, $\bk{2,\pm \frac{7}{2}, 1}$, $\bk{2,\pm \frac{5}{2}, 1}$, $\bk{2,\pm \frac{1}{2}, 1}$ and $\bk{2,\pm \frac{1}{2}, 2}$.
		\item $i=3$: $\bk{3,\pm \frac{9}{2}, 1}$, $\bk{3,\pm \frac{7}{2}, 1}$, $\bk{3,\pm \frac{5}{2}, 1}$, $\bk{3,\pm \frac{3}{2}, 1}$ and $\bk{3,\pm \frac{3}{2}, 2}$.
		\item $i=4$: $\bk{4,\pm \frac{7}{2}, 1}$, $\bk{4,\pm \frac{5}{2}, 1}$, $\bk{4,\pm \frac{3}{2}, 1}$, $\bk{4,\pm \frac{3}{2}, 2}$, $\bk{4,\pm 1, 1}$, $\bk{4,\pm 1, 2}$, $\bk{4,\pm \frac{1}{2}, 1}$, $\bk{4,\pm \frac{1}{2}, 2}$ and $\bk{4,\pm \frac{1}{2}, 3}$.
	\end{enumerate}
	The case of $i=5$ is identical to $i=3$ and the case $i=6$ is identical to that of $i=1$.

	In all cases, $\DPS{P_i}\bk{s,\chi}_v$ admits irreducible socle and co-socle, with the exception of the case $\bk{4,\pm \frac{1}{2}, 3}$.
	The representation $\bk{4,-\frac{1}{2}, 3}$ admits a socle of length $3$ and, by contragredience, $\bk{4,\frac{1}{2}, 3}$ admits a co-socle of length $3$.
\end{Thm}

\begin{Rem}
	We point out that the image of $\DPS{P_1}\bk{s,\chi}_v$ under the Iwahori-Matsumoto involution (in the sense of \cite[Remark 2.2.5]{MR1341660}, the authors also use the term \textbf{inverted representation} in \cite{SDPS_E8}) is $\DPS{P_6}\bk{-s,\overline{\chi}}_v$.
	That is, $\DPS{P_1}\bk{s,\chi}_v$ and $\DPS{P_6}\bk{-s,\overline{\chi}}_v$ share the same composition factors, while their composition series is "reversed".
	Similarly for $\DPS{P_3}\bk{s,\chi}_v$ and $\DPS{P_5}\bk{s,\chi}_v$.
	For other maximal parabolic subgroups of $E_6$ and for any parabolic subgroup of $E_7$ and $E_8$ it holds that the Iwahori-Matsumoto involution of $\DPS{P_i}\bk{s,\chi}_v$ is $\DPS{P_i}\bk{\overline{\chi},-s}_v$.
\end{Rem}

\subsection{$E_7$}
Assume that $G$ of type $E_7$.
\begin{Thm}
	The representations $\DPS{P_i}\bk{s,\chi}_v$ is reducible if and only if the triple $\bk{i,s,ord\bk{\chi}}$ is in the following list:
	\begin{enumerate}
		\item $i=1$: $\bk{1,\pm \frac{17}{2}, 1}$, $\bk{1,\pm \frac{11}{2}, 1}$, $\bk{1,\pm \frac{7}{2}, 1}$, $\bk{1,\pm \frac{1}{2}, 1}$ and $\bk{1,\pm \frac{1}{2}, 2}$.
		
		\item $i=2$: $\bk{2,\pm 7, 1}$, $\bk{2,\pm 5, 1}$, $\bk{2,\pm 4, 1}$, $\bk{2,\pm 3, 1}$, $\bk{2,\pm 2, 1}$, $\bk{2,\pm 2, 2}$, $\bk{2,\pm 1, 1}$ and $\bk{2,0, 2}$.
		
		\item $i=3$: $\bk{3,\pm \frac{11}{2}, 1}$, $\bk{3,\pm \frac{9}{2}, 1}$, $\bk{3,\pm \frac{7}{2}, 1}$, $\bk{3,\pm \frac{5}{2}, 1}$, $\bk{3,\pm \frac{5}{2}, 2}$, $\bk{3,\pm \frac{3}{2}, 1}$, $\bk{3,\pm \frac{3}{2}, 2}$, $\bk{3,\pm \frac{1}{2}, 1}$, $\bk{3,\pm \frac{1}{2}, 2}$ and $\bk{3,\pm \frac{1}{2}, 3}$.
		
		\item $i=4$: $\bk{4,\pm 4, 1}$, $\bk{4,\pm 3, 1}$, $\bk{4,\pm 2, 1}$, $\bk{4,\pm 2, 2}$, $\bk{4,\pm \frac{3}{2}, 1}$, $\bk{4,\pm \frac{3}{2}, 2}$, $\bk{4,\pm 1, 1}$, $\bk{4,\pm 1, 2}$, $\bk{4,\pm 1, 3}$, $\bk{4,\pm \frac{2}{3}, 1}$, $\bk{4,\pm \frac{2}{3}, 3}$, $\bk{4,\pm \frac{1}{2}, 1}$, $\bk{4,\pm \frac{1}{2}, 2}$, $\bk{4,\pm \frac{1}{2}, 4}$ and $\bk{4,0, 2}$.
		
		\item $i=5$: $\bk{5,\pm 5, 1}$, $\bk{5,\pm 4, 1}$, $\bk{5,\pm 3, 1}$, $\bk{5,\pm 2, 1}$, $\bk{5,\pm 2, 2}$, $\bk{5,\pm \frac{3}{2}, 1}$, $\bk{5,\pm \frac{3}{2}, 2}$, $\bk{5,\pm 1, 1}$, $\bk{5,\pm 1, 2}$ and $\bk{5,\pm 1, 3}$.
		
		\item $i=6$: $\bk{6,\pm \frac{13}{2}, 1}$, $\bk{6,\pm \frac{11}{2}, 1}$, $\bk{6,\pm \frac{7}{2}, 1}$, $\bk{6,\pm \frac{5}{2}, 1}$, $\bk{6,\pm \frac{5}{2}, 2}$, $\bk{6,\pm \frac{1}{2}, 1}$ and $\bk{6,\pm \frac{1}{2}, 2}$.
		
		\item $i=7$: $\bk{7,\pm 9, 1}$, $\bk{7,\pm 5, 1}$, $\bk{7,\pm 1, 1}$ and $\bk{7,0, 2}$
	\end{enumerate}
	
	In all cases, $\DPS{P_i}\bk{s,\chi}_v$ admits irreducible socle and co-socle, with the following exceptions: $\bk{2,\pm 2,2}$, $\bk{2,\pm 1,1}$, $\bk{2,0,2}$, $\bk{4,\pm \frac{1}{2},4}$, $\bk{4,0,2}$, $\bk{5,\pm 2,2}$ and $\bk{7,0,2}$.
	In these cases, when $s\leq0$, $\DPS{P_i}\bk{s,\chi}_v$ admits a socle of length $2$ and an irreducible co-socle, while when $s\geq0$, $\DPS{P_i}\bk{s,\chi}_v$ admits a co-socle of length $2$ and an irreducible socle.
\end{Thm}

\subsection{$E_8$}

Assume that $G$ is of type $E_8$.

\begin{Thm}
	The representations $\DPS{P_i}\bk{s,\chi}_v$ is reducible if and only if the triple $\bk{i,s,ord\bk{\chi}}$ is in the following list:
	\begin{enumerate}
		\item $i=1$: $\bk{1,\pm \frac{23}{2}, 1}$, $\bk{1,\pm \frac{17}{2}, 1}$, $\bk{1,\pm \frac{13}{2}, 1}$, $\bk{1,\pm \frac{11}{2}, 1}$, $\bk{1,\pm \frac{7}{2}, 1}$,$\bk{1,\pm \frac{5}{2}, 1}$, $\bk{1,\pm \frac{1}{2}, 1}$, $\bk{1,\pm \frac{7}{2}, 2}$ and $\bk{1,\pm \frac{1}{2}, 2}$.
		
		\item $i=2$: $\bk{2,\pm \frac{17}{2}, 1}$, $\bk{2,\pm \frac{13}{2}, 1}$, $\bk{2,\pm \frac{11}{2}, 1}$, $\bk{2,\pm \frac{9}{2}, 1}$, $\bk{2,\pm \frac{7}{2}, 1}$, $\bk{2,\pm \frac{5}{2}, 1}$, $\bk{2,\pm \frac{3}{2}, 1}$, $\bk{2,\pm \frac{1}{2}, 1}$, $\bk{2,\pm \frac{7}{2}, 2}$, $\bk{2,\pm \frac{5}{2}, 2}$, $\bk{2,\pm \frac{3}{2}, 2}$, $\bk{2,\pm \frac{1}{2}, 2}$ and $\bk{2,\pm \frac{3}{2}, 3}$.
		
		\item $i=3$: $\bk{3,\pm \frac{13}{2}, 1}$, $\bk{3,\pm \frac{11}{2}, 1}$, $\bk{3,\pm \frac{9}{2}, 1}$, $\bk{3,\pm \frac{7}{2}, 1}$, $\bk{3,\pm \frac{5}{2}, 1}$, $\bk{3,\pm 2, 1}$, $\bk{3,\pm \frac{3}{2}, 1}$, $\bk{3,\pm \frac{7}{6}, 1}$, $\bk{3,\pm 1, 1}$, $\bk{3,\pm \frac{1}{2}, 1}$, $\bk{3,\pm \frac{7}{2}, 2}$, $\bk{3,\pm \frac{5}{2}, 2}$, $\bk{3,\pm 2, 2}$, $\bk{3,\pm \frac{3}{2}, 2}$, $\bk{3,\pm 1, 2}$, $\bk{3,\pm \frac{1}{2}, 2}$, $\bk{3,\pm \frac{3}{2}, 3}$, $\bk{3,\pm \frac{7}{6}, 3}$ and $\bk{3,\pm 1, 4}$.
		
		\item $i=4$: $\bk{4,\pm \frac{9}{2}, 1}$, $\bk{4,\pm \frac{7}{2}, 1}$, $\bk{4,\pm \frac{5}{2}, 1}$, $\bk{4,\pm 2, 1}$, $\bk{4,\pm \frac{3}{2}, 1}$, $\bk{4,\pm \frac{7}{6}, 1}$, $\bk{4,\pm 1, 1}$, $\bk{4,\pm \frac{5}{6}, 1}$, $\bk{4,\pm \frac{3}{4}, 1}$, $\bk{4,\pm \frac{1}{2}, 1}$, $\bk{4,\pm \frac{3}{10}, 1}$, $\bk{4,\pm \frac{5}{2}, 2}$,$\bk{4,\pm 2, 2}$, $\bk{4,\pm \frac{3}{2}, 2}$, $\bk{4,\pm 1 ,2}$, $\bk{4,\pm \frac{3}{4}, 2}$, $\bk{4,\pm \frac{1}{2}, 2}$, $\bk{4,\pm \frac{3}{2}, 3}$,$\bk{4,\pm \frac{7}{6}, 3}$, $\bk{4,\pm \frac{5}{6}, 3}$, $\bk{4,\pm \frac{1}{2}, 3}$, $\bk{4,\pm 1, 4}$, $\bk{4,\pm \frac{3}{4}, 4}$, $\bk{4,\pm \frac{1}{2}, 4}$, $\bk{4,\pm \frac{1}{2}, 5}$, $\bk{4,\pm \frac{3}{10}, 5}$ and $\bk{4,\pm \frac{1}{2}, 6}$.
		
		\item $i=5$: $\bk{5,\pm \frac{11}{2}, 1}$, $\bk{5,\pm \frac{9}{2}, 1}$, $\bk{5,\pm \frac{7}{2}, 1}$, $\bk{5,\pm \frac{5}{2}, 1}$, $\bk{5,\pm 2, 1}$, $\bk{5,\pm \frac{3}{2}, 1}$, $\bk{5,\pm \frac{7}{6}, 1}$, $\bk{5,\pm 1, 1}$, $\bk{5,\pm \frac{5}{6}, 1}$, $\bk{5,\pm \frac{1}{2}, 1}$, $\bk{5,\pm \frac{5}{2}, 2}$, $\bk{5,\pm 2, 2}$, $\bk{5,\pm \frac{3}{2}, 2}$, $\bk{5,\pm 1, 2}$, $\bk{5,\pm \frac{1}{2}, 2}$, $\bk{5,\pm \frac{3}{2}, 3}$, $\bk{5,\pm \frac{7}{6}, 3}$, $\bk{5,\pm \frac{5}{6}, 3}$, $\bk{5,\pm \frac{1}{2}, 3}$, $\bk{5,\pm 1, 4}$, $\bk{5,\pm \frac{1}{2}, 4}$ and $\bk{5,\pm \frac{1}{2}, 5}$.
		
		\item $i=6$: $\bk{6,\pm 7, 1}$, $\bk{6,\pm 6, 1}$, $\bk{6,\pm 5, 1}$, $\bk{6,\pm 4, 1}$, $\bk{6,\pm 3, 1}$, $\bk{6,\pm \frac{5}{2}, 1}$, $\bk{6,\pm 2, 1}$, $\bk{6,\pm 1, 1}$, $\bk{6,\pm \frac{1}{2}, 1}$, $\bk{6, 0, 1}$, $\bk{6,\pm 3, 2}$, $\bk{6,\pm \frac{5}{2}, 2}$, $\bk{6,\pm 2, 2}$, $\bk{6,\pm 1, 2}$, $\bk{6,\pm \frac{1}{2}, 2}$, $\bk{6,\pm 0, 2}$, $\bk{6,\pm 2, 3}$, $\bk{6,\pm 1, 3}$ and $\bk{6,\pm \frac{1}{2}, 4}$.
		
		\item $i=7$: $\bk{7,\pm \frac{19}{2}, 1}$, $\bk{7,\pm \frac{17}{2}, 1}$, $\bk{7,\pm \frac{11}{2}, 1}$, $\bk{7,\pm \frac{9}{2}, 1}$, $\bk{7,\pm \frac{5}{2}, 1}$, $\bk{7,\pm \frac{3}{2}, 1}$, $\bk{7,\pm \frac{1}{2}, 1}$, $\bk{7,\pm \frac{9}{2}, 2}$, $\bk{7,\pm \frac{5}{2}, 2}$, $\bk{7,\pm \frac{1}{2}, 2}$ and $\bk{7,\pm \frac{1}{2}, 3}$.
		
		\item $i=8$: $\bk{8,\pm \frac{29}{2}, 1}$, $\bk{8,\pm \frac{19}{2}, 1}$, $\bk{8,\pm \frac{11}{2}, 1}$, $\bk{8,\pm \frac{1}{2}, 1}$ and $\bk{8,\pm \frac{1}{2}, 2}$
		
	\end{enumerate}
	
	As for the length of the socle and co-socle of these representations we have:
	
	\begin{itemize}
		\item In the following cases, when $s\leq0$, $\DPS{P_i}\bk{s,\chi}_v$ admits a socle of length $2$ and an irreducible co-socle, while when $s\geq0$, $\DPS{P_i}\bk{s,\chi}_v$ admits a co-socle of length $2$ and an irreducible socle:
		$\bk{1,\pm \frac{5}{2},1}$, $\bk{3,\pm\frac{1}{2},2}$, $\bk{6,0,1}$, $\bk{6,0,2}$ and $\bk{7,\pm \frac{3}{2},1}$.
	
		\item Furthermore, in the cases $\bk{2,-\frac{1}{2},1}$ and $\bk{5,-\frac{1}{2},1}$ the socle is irreducible while the co-socle is either irreducible or of length $2$. Similarly, for $\bk{2,\frac{1}{2},1}$ and $\bk{5,\frac{1}{2},1}$ the co-socle is irreducible and the socle is either irreducible or of length $2$.
	
		\item In all other cases, $\DPS{P_i}\bk{s,\chi}_v$ admits irreducible socle and co-socle.
	\end{itemize}
	
\end{Thm}

\section{Lengths of Certain Representations}
\label{Sec:Local_DPS_Lengths}

In this section we add information on the structure of particular degenerate principle series.
These are the cases where $\pi=\DPS{P}\bk{s,\chi}_v$ does not admit a unique irreducible quotient and the cases associated with non-square-integrable residues that cannot be calculated using a Siegel-Weil identity.
This is required for the calculations performed in \Cref{Chap:Square_Integrable} and \Cref{Chap:Non_square_Integrable}.
Proofs of the statements made in this section can be found in \Cref{Appendix:Chap_Local_Calculations}.
We fix a finite place $v\in\Places$ and drop it from our notations.

\subsection{$E_6$}
\label{Subsec:Local_Lengths_E6}

\begin{enumerate}
	
	\item The following representations are of length $2$, they have a unique irreducible subrepresentation, which will be denoted by $\pi_0$, and a unique irreducible quotient, which will be denoted by $\pi_1$:
	\[
	\DPS{P_4}\bk{1,\Id}, \DPS{P_4}\bk{1,\chi}, \DPS{P_4}\bk{\frac12,\chi},
	\]
	where $\chi$ is a character of order $2$ of $F^\times$.
	
	\item The representation $\pi=\DPS{P_4}\bk{\frac{1}{2}}$ is of length $3$.
	Its composition series is given as follows:
	\begin{itemize}
		\item It admits a unique irreducible subrepresentation $\pi_0$.
		\item It admits a unique irreducible quotient $\pi_1$ (which is spherical).
		\item $\pi$ contains a third irreducible subquotient $\sigma_1$.
		This is the unique irreducible subrepresentation of $\pi\rmod\pi_0$.
		Namely
		\[
		\pi_1 = \bk{\pi\rmod \pi_0}\rmod \sigma_1
		\]
	\end{itemize}

	\item Let $\chi$ be a character of order $3$ of $F^\times$ and let $\mu_3$ denote the group of cube roots of unity.
	The representation $\pi=\DPS{P_4}\bk{\chi,\frac{1}{2}}$ has length $4$.
	We denote its unique irreducible subrepresentation by $\pi_0$.\\
	The quotient $\pi\rmod\pi_0=\oplus_{\eta\in\mu_3}\pi_{\eta}$ is semi-simple of length $3$.
	Note that $\Stab_W\bk{\lambda_{a.d.}}=\gen{w_{3165}}$
	It is possible to distinguish between the $\pi_{\eta}$ by their eigenvalues under the action of $N_{w_{3165}}\bk{\lambda_{a.d.}}$.
	Namely,
	\[
	N_{w_{3165}}\bk{\lambda_{a.d.}}v = \eta v \quad \forall v \in \pi_{\eta} \quad \forall \eta\in\mu_3 .
	\]
	Here,
	\[
	\lambda_{a.d.}=\esixchar{\chi}{\chi}{\chi}{-1+\chi}{\chi}{\chi} .
	\]

\end{enumerate}

\subsection{$E_7$}
\label{Subsec:Local_Lengths_E7}

\begin{enumerate}
	
	\item
	The following representations are of length $2$, they have a unique irreducible subrepresentation, which will be denoted by $\pi_0$, and a unique irreducible quotient, which will be denoted by $\pi_1$:
	\begin{itemize}
		\item $\DPS{P_2}\bk{2}$, $\DPS{P_4}\bk{\frac12}$, $\DPS{P_4}\bk{\frac23}$, $\DPS{P_4}\bk{\frac32}$ and $\DPS{P_5}\bk{\frac32}$.
		
		\item $\DPS{P_4}\bk{\frac12,\chi}$, $\DPS{P_4}\bk{\frac32,\chi}$, $\DPS{P_5}\bk{1,\chi}$ and $\DPS{P_5}\bk{\frac32,\chi}$, where $\chi$ is a character of order $2$ of $F^\times$.
		
		\item $\DPS{P_4}\bk{\frac23,\chi}$, where $\chi$ is a character of order $3$ of $F^\times$.
	\end{itemize}

	\item Let $\chi$ be a character of order $2$ of $F^\times$.
	The representations $\DPS{P_2}\bk{2,\chi}$ and $\pi=\DPS{P_5}\bk{2,\chi}$ have length $3$.
	In each case, we denote the unique irreducible subrepresentation by $\pi_0$.
	
	The quotient $\pi\rmod\pi_0=\pi_1\oplus\pi_{-1}$ is semi-simple of length $2$.	
	Note that $\Stab_W\bk{\lambda_{a.d.}}=\gen{w_{257}}$
	It is possible to distinguish between $\pi_1$ and $\pi_{-1}$ by their eigenvalues under the action of $N_{w_{257}}\bk{\lambda_{a.d.}}$.
	Namely,
	\[
	N_{w_{257}}\bk{\lambda_{a.d.}}v = \eta v \quad \forall v \in \pi_{\eta},\quad \eta\in\set{1,-1} .
	\]
	Here,
	\[
	\lambda_{a.d.}=\esevenchar{-1}{\chi}{\chi}{-1}{\chi}{-1}{\chi}.
	\]
	
	\item Let $\chi$ be a character of order $4$ of $F^\times$.
	The representation $\pi=\DPS{P_4}\bk{\chi,\frac{1}{2}}$ has length $3$.
	We denote the unique irreducible subrepresentation by $\pi_0$.
	
	The quotient $\pi\rmod\pi_0=\pi_1\oplus\pi_{-1}$ is semi-simple of length $2$.
	Note that $\Stab_W\bk{\lambda_{a.d.}}=\gen{w_{257}}$
	It is possible to distinguish between $\pi_1$ and $\pi_{-1}$ by their eigenvalues under the action of $N_{w_{257}}\bk{\lambda_{a.d.}}$.
	Namely,
	\[
	N_{w_{257}}\bk{\lambda_{a.d.}}v = \eta v \quad \forall v \in \pi_{\eta},\quad \eta\in\set{1,-1} .
	\]
	Here,
	\[
	\lambda_{a.d.}=\esevenchar{-\frac{1}{2}+3\chi}{2\chi}{2\chi}{-\frac{1}{2}+\chi}{2\chi}{-\frac{1}{2}+3\chi}{2\chi}.
	\]

	\item The representation $\pi=\DPS{P_2}\bk{1}$ has length $3$, with a unique irreducible subrepresentation $\pi_0$ and the quotient $\pi\rmod\pi_0=\pi_1\oplus\pi_{-1}$ is semi-simple of length $2$.
	
	Furthermore, it holds that
	\[
	N_{u_0^{-1} w_0 u_0}\bk{\lambda_0} f_\epsilon \quad \forall f_\epsilon\in\pi_\epsilon,\ \epsilon=\pm 1 ,
	\]
	where
	\begin{itemize}
		\item $	\lambda_0 = \lambda_{1}^{P_2} = \esevenchar{-1}{7}{-1}{-1}{-1}{-1}{-1}$.
		\item $u_0=w\coset{1,3,4,2}$.
		\item $w_0=w\coset{5, 4, 3, 2, 4, 5, 6, 7, 5, 4, 3, 2, 4, 5, 6, 5, 4, 3, 2, 4, 5}$.
	\end{itemize}
\end{enumerate}

\subsection{$E_8$}

\begin{enumerate}
	\item 
	The following representations are of length $2$, they have a unique irreducible subrepresentation, which will be denoted by $\pi_0$, and a unique irreducible quotient, which will be denoted by $\pi_1$:
	\begin{itemize}
		\item $\DPS{P_3}\bk{\frac76}$, $\DPS{P_3}\bk{2}$, $\DPS{P_4}\bk{\frac{3}{10}}$, $\DPS{P_4}\bk{\frac34}$, $\DPS{P_4}\bk{\frac76}$, $\DPS{P_4}\bk{2}$, $\DPS{P_5}\bk{\frac56}$, $\DPS{P_5}\bk{\frac76}$, $\DPS{P_5}\bk{2}$ and $\DPS{P_6}\bk{\frac52}$.
		
		\item $\DPS{P_3}\bk{2,\chi}$, $\DPS{P_4}\bk{\frac34,\chi}$, $\DPS{P_4}\bk{2,\chi}$, $\DPS{P_5}\bk{2,\chi}$, $\DPS{P_6}\bk{2,\chi}$ and $\DPS{P_6}\bk{\frac52,\chi}$, where $\chi$ is a character of order $2$ of $F^\times$.
		
		\item $\DPS{P_3}\bk{\frac76,\chi}$, $\DPS{P_4}\bk{\frac76,\chi}$, $\DPS{P_5}\bk{\frac12,\chi}$, $\DPS{P_5}\bk{\frac56,\chi}$ and $\DPS{P_5}\bk{\frac76,\chi}$, where $\chi$ is a character of order $3$ of $F^\times$.
		
		\item $\DPS{P_4}\bk{\frac12,\chi}$ and $\DPS{P_4}\bk{\frac34,\chi}$, where $\chi$ is a character of order $4$ of $F^\times$.
		
		\item $\DPS{P_4}\bk{\frac{3}{10},\chi}$, where $\chi$ is a character of order $5$ of $F^\times$.
	\end{itemize}

	\item The representation $\pi=\DPS{P_1}\bk{\frac52}$ has length $3$, with a unique irreducible subrepresentation $\pi_0$ and the quotient $\pi\rmod\pi_0=\pi_1\oplus\pi_{-1}$ is semi-simple of length $2$.
	
	\item The representation $\pi=\DPS{P_7}\bk{\frac32}$ has length $3$, with a unique irreducible subrepresentation $\pi_0$ and the quotient $\pi\rmod\pi_0=\pi_1\oplus\pi_{-1}$ is semi-simple of length $2$.

	\item Let $\chi$ be a character of order $2$ of $F^\times$.
	The representations $\DPS{P_3}\bk{\frac12,\chi}$ have length $3$.
	In each case, we denote the unique irreducible subrepresentation by $\pi_0$.
	The quotient $\pi\rmod\pi_0=\pi_1\oplus\pi_{-1}$ is semi-simple of length $2$.
\end{enumerate}

\section{A Little on the Degenerate Principal Series at Archimedean Places}

If $P$ has an Abelian nilradical, the structure of $\DPS{P}\bk{s,\Id}$ was already studied in
\begin{itemize}
	\item \cite{MR1329899}, when $P$ is conjugate to its opposite parabolic (the case of $\bk{E_7,P_7}$).
	
	\item \cite[Theorem 3.3]{MR3165233}, when $P$ is not conjugate to its opposite parabolic (the case of $\bk{E_6,P_1}$).

\end{itemize}

Throughout this manuscript, the global sections which are spherical at Archimedean places,
This can, of course, be easily replaced by sections which are in the representation generated by the spherical vector.
Furthermore, one can remove this restriction, once enough information is known regarding the structure of the Archimedean degenerate principal series.
In fact, non of the formulas in the following chapters need to be altered under the following conjecture:
\begin{Conj}
	The structure of the co-socle in the Archimedean places is analogous to the non-Archimedean places.
	Namely, for $\Real\bk{s}>0$, $v_1$ an Archimedean place of $F$ and $v_2$ a non-Archimedean place of $F$:
	\begin{itemize}
		\item The length of the maximal semi-simple quotient of $\DPS{P,v_1}\bk{s,\Id}$ is the same as that of $\DPS{P,v_2}\bk{s,\Id}$.
		
		\item Assume that $F_{v_1}=\R$.
		Then, the length of the maximal semi-simple quotient of $\DPS{P,v_1}\bk{s,sgn}$ equals that of $\DPS{P,v_2}\bk{s,\chi}$, where $\chi$ has order $2$.
	\end{itemize}
\end{Conj}

\part{The Degenerate Residual Spectrum of Groups of Type $E_n$}
\label{Part:DRS}

In this part we study the degenerate residual spectrum of groups of type $E_n$.
Each chapter deals with a certain aspect in the description of the degenerate residual spectrum of these groups, it begins with a section describing the methods used to study this particular aspect, followed by three sections dedicated to each of the types $E_n$ ($n=6,7,8$).
We wish to point out that the methods described in the method sections are applicable for other groups and not only those of type $E_n$.

Through out this part, let $G$ be a group of type $E_n$ as in \Cref{Chap:Group_En}, as well as other notations set in \Cref{Chap:Group_En}.

We list the four chapters in this part:
\begin{enumerate}
	\item \Cref{Chap:Poles_of_Eisen_ser} studies the set of poles of an Eisenstein series, their orders and which of the residues is square-integrable.
	
	\item \Cref{Chap:Square_Integrable} studies the square-integrable residues of Eisenstein series.
	
	\item \Cref{Chap:Siegel_Weil} describes the Siegel-Weil identities between Eisenstein series.
	
	\item In \Cref{Chap:Arthur} we attach Arthur parameters to the square-integrable residues found in \Cref{Chap:Poles_of_Eisen_ser} and \Cref{Chap:Square_Integrable}.
	
	\item \Cref{Chap:Non_square_Integrable} studies the non-square-integrable residues of Eisenstein series.
\end{enumerate}

We also wish to point out that some of the calculations in this part rely on local calculations performed in \Cref{Appendix:Chap_Local_Calculations}.

\chapter{Poles of Degenerate Eisenstein Series}
\label{Chap:Poles_of_Eisen_ser}

In this part, we study the poles of the degenerate Eisenstein series.
That is, we describe the set of poles of $\Eisen_{P_i} \bk{f,s,\chi,g}$ and determine the order of these poles.

\section{The Method}
\label{Section:Poles_Method}

\subsection{The Constant Term at a Point and Equivalence Classes}

Let $f_s=\otimes_{v\in\Places} f_{s,v}$ be a factorizable standard section and assume that $f_{s,v}=f_{s,v}^0$ for $v\notin S$, where $S\subset\Places$ is a finite set of primes of $F$.

It follow from \Cref{Eq:Constant_term} and \Cref{Eq:Global_GK_Formula} that
\begin{equation}
	\label{eq::CT_first_form}
	\begin{split}
		\Eisen_{P_i} \bk{f,s,\chi,g}_T 
		& = \suml_{w\in W^{M,T}} M_w\bk{\lambda_{s,\chi}^{P_i}} f_s\bk{g} \\
		& = \suml_{w\in W^{M,T}} C_w\bk{\lambda_{s,\chi}^{P_i}} \otimes_{v\notin S}f_{s,v}\bk{g_v} \otimes_{v\in S} N_{w,v}\bk{\lambda_{s,\chi}^{P_i}}f_{s,v}\bk{g_v} \\
	\end{split}
\end{equation}
Each summand may have a pole of a certain order and, by summation, the orders of these poles can be reduced, that is different summands may cancel each others poles.

We consider the constant term $\Eisen_{P_i} \bk{f,s,\chi,g}_T$ as a representation of $T$; as such, each summand on the right-hand-side of \Cref{eq::CT_first_form} lies within a single isotypic component.
Thus, cancellations of poles may happen only within the same isotypic component.
Thus, given $\chi$ and $s_0\in\C$, we define an equivalence relation $\sim_{s_0,\chi}$ on $W^{M,T}$ as follows:
\[
w\sim_{s_0,\chi}w' \Leftrightarrow w\cdot\lambda_{s_0} = w'\cdot\lambda_{s_0} .
\]
Thus, we may rewrite \Cref{eq::CT_first_form} by summing over equivalence classes as follows
\begin{equation}
	\label{eq::CT_second_form}
	\Eisen_{P_i} \bk{f,s,\chi,g}_T = 
	\suml_{\coset{w'}\in W^{M,T}\rmod\sim_{s_0,\chi}}  \kappa_{\coset{w'}}\bk{f,g},
\end{equation}
where
\begin{equation}
	\label{eq::sum_over_eq_cl}
	\begin{split}
		\kappa_{\coset{w'}}\bk{f,s,g} 
		& = \suml_{w\in\coset{w'}} M_w\bk{\lambda_{s,\chi}^{P_i}} f_s\bk{g} \\
		& = \suml_{w\in\coset{w'}} C_w\bk{\lambda_{s,\chi}^{P_i}} \otimes_{v\notin S}f_{s,v}\bk{g_v} \otimes_{v\in S} N_{w,v}\bk{\lambda_{s,\chi}^{P_i}}f_{s,v}\bk{g_v} .
	\end{split}
\end{equation}

We note that the poles of each summand can arise either from the Gindikin-Karpelevich factor $C_w\bk{\lambda_{s,\chi}^{P}}$ or from the local intertwining operators $N_{w,v}\bk{\lambda_{s,\chi}^{P}}f_{s,v}\bk{g_v}$.
We will show that, for $\Real\bk{s}>0$, the local factors do not contribute to the order.
That is, the order of the pole of each summand is determined by the order of $C_w\bk{\lambda_{s,\chi}^{P}}$.

We also point out that the order of $M_w\bk{s,\chi}$ is constant on an equivalence class $\coset{w'}$.
We write $\equivclassespole{s_0,\chi,k}$ for the set of equivalence classes in $W^{M,T}\rmod\sim_{s_0,\chi}$ which satisfy:
\[
\ord{s=s_0}{M_w\bk{s,\chi}} \geq k \quad \forall w\in\coset{w'} .
\]
We say that a pole of order $k$ of $\Eisen_{P_i} \bk{f,s,\chi,g}$ is canceled or reduced, if $\equivclassespole{i,s_0,\chi,k}\neq \emptyset$, while $\ord{s=s_0}{\Eisen_{P_i} \bk{f,s,\chi,g}}<k$.

\subsection{Holomorphicity of Local Normalized Intertwining Operators}

In our analysis of the poles of degenerate Eisenstein series and their residues, we make use of the following property of local intertwining operators:

\hbox{}
\textbf{\underline{Property A:}}
The operator $N_{w,v}\bk{\lambda}$ is well-defined and holomorphic at $\lambda_{s,\chi}^{P_i}$ for all $w\in W^{M_i,T}$ and $s\in\C$ such that $\Real\bk{s}>0$.
\hbox{}\\

It is known that \textbf{Property A} holds due to \cite{Zhang1997,Arthur1989,MR581582,Kostant1978,Vogan1978,Muic2001,MR1341660}.

\begin{Cor}
	\begin{enumerate}
		\item The order of $M_w\bk{\lambda_{s,\chi}^{P_i}} f_s\bk{g}$ at $s_0$ equals that of $C_w\bk{\lambda_{s,\chi}^{P_i}}$.
		\item If $\Eisen_{P_i} \bk{f,s,\chi,g}$ admits a pole at $s_0\in\C$, with $\Real\bk{s_0}>0$, then there exists $w\in W^{M,T}$ such that $C_w\bk{\lambda_{s,\chi}^{P_i}}$ admits a pole at $s_0$.

	\end{enumerate}
\end{Cor}

Thus, in order to find all of the poles of $\Eisen_{P_i} \bk{f,s,\chi,g}$ with $\Real\bk{s_0}>0$, one can proceed with the following steps:

\begin{enumerate}
	\item Find all $\chi$, $s_0$ and $w\in W^{M,T}$ such that $C_w\bk{\lambda_{s,\chi}^{P_i}}$ admits a pole at $s_0$.
	
	\item Bound the order of $\Eisen_{P_i} \bk{f,s,\chi,g}$ from above by $n$.
	We use two methods to perform this step:
	\begin{enumerate}
		\item For $\chi$ and $s_0$ and $w$ as above, determine the order of $\kappa_{\coset{w}}$ and let $n$ be the maximum of the orders of such poles using a reduction of orders of poles, as described in \Cref{Subsection:Methods_Reduction_of_order_of_poles}.
		
		\item In case that $\chi$ is trivial and $\DPS{P}\bk{s_0}$ is generated by its spherical vector (i.e. $\DPS{P}\bk{s_0}$ admits a unique irreducible quotient), we follow \cite{MR1174424} and use a method described in \Cref{Subsection:Poles_Method_Normalized_Eis_ser}.
	\end{enumerate}
	
	\item Prove that some equivalence class $\coset{w'}\in W^{M,T}\rmod\sim_{s_0,\chi}$ realizes a pole of order $n$.
	That is, show that
	\[
	\ord{s=s_0}{\kappa_{\coset{w'}}} = n .
	\]
	Usually, this can be done using a singleton $\set{w}\in W^{M,T}\rmod\sim_{s_0,\chi}$ such that the order of $C_w\bk{\lambda_{s,\chi}^{P}}$ at $s_0$ is $n$.
	
	\item Finally, we wish to determine the square-integrability of the residual representation.
	We do this using Langlands' criterion as detailed in \Cref{Subsection:Square-Integrability_criterion}.
\end{enumerate}

\begin{Rem}
	\label{Rem:Order_of_trivial_eq_class}
	We point out that $\set{1}\in W^{M,T}\rmod\sim_{s_0,\chi}$ for any $\Real\bk{s_0}>0$ and that $\ord{s=s_0}{\kappa_{\set{1}}}=0$.
	Hence, for any $\Real\bk{s_0}>0$, we have $\ord{s=s_0}{\Eisen_{P_i}\bk{\chi}}\geq 0$.
\end{Rem}

\subsection{Reduction of Orders of Poles}
\label{Subsection:Methods_Reduction_of_order_of_poles}

Considering the sum $\kappa_{\coset{w'}}\bk{f,s,g}$ in \Cref{eq::sum_over_eq_cl}, it is possible that each summand $M_w\bk{\lambda_{s,\chi}^{P_i}} f_s\bk{g}$ admits a pole of order $n$, yet $\kappa_{\coset{w'}}\bk{f,s,g}$ itself will admit a pole of a lower order or even be holomorphic.

The main reason for the cancellation of these poles is as follows.
Assume that $w'$ is the shortest element in $\coset{w'}$ and let $\lambda=w'\cdot\lambda_{s_0,\chi}^{P_i}$.

We note that the coset $\coset{w'}$ can be written as follows
\[
\coset{w'} = \set{uw' \mvert u\in \Stab_W\bk{\lambda} \cap \bk{W^{M,T} \cdot \bk{w'}^{-1}} } 
\]
We assume that $\Stab_W\bk{\lambda} \cap \bk{W^{M,T} \cdot \bk{w'}^{-1}}$ is isomorphic to $\bk{\bZ\rmod2\bZ}^k$, or in other words, it can be written as
\begin{equation}
	\label{Eq:Equivalence_class_as_a_shifted_product_of_Z/2Z}
	\coset{w'} = \set{uw' \mvert u \in \prodl_{i=1}^k \set{1,u_i} }.
\end{equation}

Finally, we assume that
\begin{equation}
	\label{Eq:Global_interwining_operator_acts_as_-1}
	M_{u_i}\bk{\lambda} \varphi = -\varphi \qquad \forall\varphi\in \Res{s=s_0} \coset{M_{w}\bk{\lambda_{s,\chi}^{P_i}}\bk{\DPS{P_i}\bk{s,\chi}}} .
\end{equation}
For $\varphi\in \Res{s=s_0} \coset{M_{w}\bk{\lambda_{s,\chi}^{P_i}}\bk{\DPS{P_i}\bk{s,\chi}}}$, it holds that
\[
\prodl_{i=1}^k \bk{\Id+M_{u_i}\bk{\lambda}}\varphi \in o\bk{\bk{s-s_0}^k} .
\]

It follows that
\[
\max\set{\ord{s=s_0}{M_w\bk{s,\chi}}} = n \quad \Rightarrow \quad \ord{s=s_0}{\kappa_{\coset{w'}}\bk{f,s,g}} \leq n-k .
\]

We list bellow a couple methods of validating \Cref{Eq:Global_interwining_operator_acts_as_-1}:
\begin{enumerate}
	\item Assume that $u_i$ can be written as $u_i=u_0^{-1}\cdot z\cdot u_0$, where $N_{u_0}\bk{\lambda}$ is an isomorphism and $z=s_{j_1}\cdot...\cdot s_{j_l}$, where
	\[
	\inner{u_0\cdot\lambda,\check{\alpha_{j_1}}} = ... = \inner{u_0\cdot\lambda,\check{\alpha_{j_l}}} = 0 .
	\]
	Then, by \cite[Proposition 6.3]{MR944102},
	\[
	M_z\bk{u_0\cdot\lambda} = \bk{-1}^l Id .
	\]
	If $l$ is odd, it follows that
	\[
	\begin{split}
		M_{u_i}\bk{\lambda} 
		& = M_{u_0^{-1}}\bk{z\cdot u_0 \cdot\lambda} \circ N_z\bk{u_0\cdot\lambda} \circ M_{u_0}\bk{\lambda} \\
		& =	- M_{u_0^{-1}}\bk{u_0 \cdot\lambda} \circ Id \circ M_{u_0}\bk{\lambda} \\
		& =	- M_{u_0^{-1}}\bk{u_0 \cdot\lambda} \circ M_{u_0}\bk{\lambda} = -Id .
	\end{split}
	\]
	
	\item Assume that $u_i$ can be written as $u_i=u_0^{-1}\cdot z\cdot u_0$, where $M_{u_0}\bk{\lambda}$ is an isomorphism and that $M_z\bk{u\cdot\lambda}$ satisfies the conditions of \Cref{Prop:Zampera}.
	Then, $N_z\bk{u\cdot\lambda}_v$ is holomorphic and diagonalizable at $s=s_0$ for all $v\in\Places$.
	Further assume that
	\[
	C_z\bk{u\cdot\lambda} \underset{s\to s_0}{\longrightarrow} -1 
	\]
	and that $N_{u_0w',v}\bk{\lambda_{s_0,\chi}}\bk{\DPS{P}\bk{s_0,\chi}_v}$ is contained within the $1$-eigenspace of $N_z\bk{u_0\cdot\lambda}$.
	Then,
	\[
	M_z\bk{u_0\cdot\lambda}=-Id.
	\]
	One concludes that $M_{u_i}\bk{\lambda}=-Id$ as in the previous item.
	
	We now explain how, given $v\in\Places_{fin.}$, to check that $N_{u_0w',v}\bk{\lambda_{s_0,\chi}^{P_i}}\bk{\DPS{P_i}\bk{s_0,\chi}_v}$ is, indeed, contained within the $1$-eigenspace of $N_z\bk{u_0\cdot\lambda_{s_0,\chi}^{P_i}}$, under the assumption that $\pi_v=\DPS{P_i}\bk{s_0,\chi}_v$ admits a unique irreducible quotient $\pi_{1,v}$.
	
	Let $\lambda_{0,v}$ denote the initial exponent of $\pi_v$ and let $\lambda_1=u_0w'\cdot\lambda_0$.
	We write $z=s_j z' s_j$, where $\inner{\lambda_{1},\check{\alpha_j}}=\pm 1$.
	By \Cref{Prop:Zampera}, $i_T^G\lambda_1$ decomposes as follows
	\[
	i_{T_v}^{G_v}\lambda_1 = \underbrace{i_{L_{j,v}}^{G_v}\bk{\tau_{tr.}}}_{1-eigenspace} \oplus \underbrace{i_{L_{j,v}}^{G_v}\bk{\tau_{St.}}}_{\epsilon-eigenspace},
	\]
	where $\tau$ is a $1$-dimensional representation of $L_{j,v}$, $\epsilon<0$ and
	\[
	\tau_{tr.}=\tau\otimes \Id_{L_{j,v}},\quad 
	\tau_{St.}=\tau\otimes \St_{L_{j,v}} .
	\]
	Now, since $N_{u_0w'}\bk{\pi}$ is a quotient of $\pi$, it also admits a unique irreducible quotient.
	It follows that $N_{u_0w'}\bk{\pi} \subset i_{L_{j,v}}^{G_v}\bk{\tau_{tr.}}$ or $N_{uw'}\bk{\pi} \subset i_{L_{j,v}}^{G_v}\bk{\tau_{St.}}$.
	It thus remains to prove that $N_{u_0w'}\bk{\pi} \subset i_{L_{j,v}}^{G_v}\bk{\tau_{tr.}}$, this would follow from the following calculation (if true):
	\begin{itemize}
		\item $s_j\cdot\lambda_1\leq r_T^G\bk{i_{L_{j,v}}^{G_v}\bk{\tau_{tr.}}}$.
		\item $s_j\cdot\lambda_1\nleq r_T^G\bk{i_{L_{j,v}}^{G_v}\bk{\tau_{St.}}}$.
		\item $s_j\cdot\lambda_1\leq r_T^G\pi_{1,v}$.
	\end{itemize}
	While the first two items follow from \Cref{Prop:Zampera}, the third follows from a branching rule calculation (see \Cref{Appendix:Branching_Rules} for more details).
\end{enumerate}

\begin{Rem}
	In some cases, we use a variation of the above method.
	Namely, we replace \Cref{Eq:Equivalence_class_as_a_shifted_product_of_Z/2Z} with the assumption that
	\[
	\coset{w'} \cdot \bk{w'}^{-1} \subseteq \prodl_{i=1}^k \set{1,u_i}
	\]
	and that for any $u\in \prodl_{i=1}^k \set{1,u_i}$ such that $uw'\notin W^{M,T}$, it holds that $M_{uw}$ vanishes on the degenerate principal series $\DPS{P_i}\bk{s,\chi}$ to the $l^{th}$ order.
	Hence, for any $f_s\in \DPS{P}\bk{s,\chi}$ it holds that
	\[
	\suml_{w\in\coset{w'}} M_{w}f_s = \suml_{u\in \prodl_{i=1}^k \set{1,u_i}} M_{uw'}f_s
	\]
	Under these assumptions, it is possible to apply the above methods on the right-hand-side instead of the left one, taking the order $l$ of the zeros of the added intertwining operators into account.
\end{Rem}

\vbox{}

We would also like to mention a slightly more complex variation of the application of \Cref{Prop:Zampera} mentioned above.
We consider the following situation:
\begin{itemize}
	\item Fix $v\in\Places_{fin.}$, we assume that $\pi_v$ admits a unique irreducible quotient $\pi_{1,v}$.
	\item $\lambda=w'\cdot\lambda_{s_0,\chi}^{P_i}$.
	\item $\coset{w'} = \set{w',u w'}$.
	\item The Weyl element $u$ can be written in the form $u=u_0^{-1}\cdot z\cdot u_0$, where $N_{u_0}\bk{\lambda_{s_0,\chi}^{P_i}}$ is an isomorphism such that $\lambda_1=u_0\cdot\lambda$ and $z$ satisfy the following conditions:
	\begin{itemize}
		\item $z\in W_L$, where $L=L_\Theta$, $\Theta\subset\Delta$, is a standard Levi subgroup of $G$ of type $A_5$.
		\item Furthermore, denote the elements of $\Theta$ by $\alpha_{j_1},...,\alpha_{j_5}$ such that the sub-Dynkin diagram of $\Delta$ induced by $\Theta$ is
		
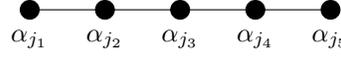
\[\begin{tikzpicture}[scale=0.5]
			\draw (-1,0) node[anchor=east]{};
			\draw (0 cm,0) -- (8 cm,0);
			\draw[fill=black] (0 cm, 0 cm) circle (.25cm) node[below=4pt]{$\alpha_{j_1}$};
			\draw[fill=black] (2 cm, 0 cm) circle (.25cm) node[below=4pt]{$\alpha_{j_2}$};
			\draw[fill=black] (4 cm, 0 cm) circle (.25cm) node[below=4pt]{$\alpha_{j_3}$};
			\draw[fill=black] (6 cm, 0 cm) circle (.25cm) node[below=4pt]{$\alpha_{j_4}$};
			\draw[fill=black] (8 cm, 0 cm) circle (.25cm) node[below=4pt]{$\alpha_{j_5}$};
		\end{tikzpicture}\]
		\item $z=w_{j_2}w_{j_4}tw_{j_2}w_{j_4}$, where $t\in\gen{w_{j_1},w_{j_3},w_{j_5}}$.
		\item The restriction $\mu$ of $\lambda_1$ to $L$, under the isomorphism with the $A_5$ root system, written with respect to the basis of fundamental weights, is given by
		\[
		\mu=\bk{-1,1,-1,1,-1} .
		\]
	\end{itemize}
\end{itemize}

In this case, we argue that, $N_{uw',v}\bk{\lambda_{s_0,\chi}^{P_i}}\bk{\DPS{P_i}\bk{s_0,\chi}_v}$ is contained within the $1$-eigenspace of $N_z\bk{u\cdot\lambda}$ and hence one may continue as above.

Indeed, by \cite[Example 5.5]{SegalSingularities}, the representation $i_{T_v}^{G_v}\lambda_1$ decomposes into a direct sum of the $\pm1$-eigenspaces of $N_{w_{j_2}w_{j_1}w_{j_2}}$ and $N_{w_{j_4}w_{j_5}w_{j_4}}$ respectively.
Namely,
\[
i_{T_v}^{G_v}\lambda_1 = \oplus_{\delta,\epsilon\in\set{-1,1}} \pi_{\delta,\epsilon},
\]
where $\pi_{\delta,\epsilon}$ is the $\delta$-eigenspace of $N_{w_{j_2}w_{j_1}w_{j_2}}$ and the $\epsilon$-eigenspace of $N_{w_{j_4}w_{j_5}w_{j_4}}$.

Write $\lambda_2=w_{j_4}w_{j_2}\cdot \lambda_1$, this is an anti-dominant exponent when restricted to $L$ and it appears with full multiplicity in $\pi_{1,1}$.
One may conclude that $N_z\bk{\lambda_1}$ acts on $\pi_{1,1}$ as $\Id$.

Thus, if $\lambda_2$ is an exponent of $\pi_{1,v}$, we may conclude that it is the image $N_{u_0w'}\bk{\lambda_{s_0,\chi}\bk{\pi}^{P_i}}$ and it is a subrepresentation of $\pi_{1,1}$ and hence $N_{u}\bk{\lambda}$ acts on it as $\Id$.

\begin{Rem}
	Similar arguments can be made for $L_\Theta$ of type $A_6$ and $D_6$, all of which have been considered in \cite[Section 5.2]{SegalSingularities}.
\end{Rem}

\subsection{Bounds on the Order of the Pole for trivial $\chi$}
\label{Subsection:Poles_Method_Normalized_Eis_ser}

Here we note an alternative method for bounding the order of $\Eisen_{P_i} \bk{f,s,\chi,g}$ in the case that $\chi=\Id$.

We recall that the spherical normalized Eisenstein series,
\[
\Eisen_{B}^\sharp\bk{\lambda,g} = 
\coset{\prodl_{\alpha\in\Phi^+} \zfun\bk{l_\alpha^+\bk{\lambda}} l_\alpha^+\bk{\lambda} l_\alpha^-\bk{\lambda}}
\Eisen_{B}\bk{f^0_\lambda,\lambda,g} ,
\]
given in \Cref{qq:Normalized_Eisenstein_Series}, is entire and $W_G$-invariant.
Following \cite{MR1174424}, we consider the value of $\Eisen_{B}^\sharp\bk{\eta_s^{P_i},g}$.

\begin{Prop}
	It holds that
	\[
	\Eisen_{B}^\sharp\bk{w\cdot \eta_s^{P_i},g} = G_{P_i}\bk{s} \cdot \Eisen_{P_i}^{G} \bk{f^0,s,\chi,g},
	\]
	where $G_{P_i}\bk{s}$ is holomorphic on $\Real\bk{s}>0$.
	In particular, with $\Real\bk{s_0}$, the order of the pole of $\Eisen_{P_i}^{G} \bk{f^0,s,\chi,g}$ at $s_0$ is bounded from above by the order of the zero of $G_{P_i}\bk{s}$ at $s_0$.
\end{Prop}

\begin{proof}
	This follows from \Cref{Prop:Equality_of_Eisenstein_series}, the term $G_{P_i}\bk{s}$ is given by the elements of the normalizing factor other than thus associated with the hyperplanes $\frakh_k$.
\end{proof}

In particular, at points $s_0$ where $\Real\bk{s_0}>0$ and $\DPS{P_i}\bk{s_0}$ is generated by the spherical vector (i.e., it admits a unique irreducible quotient), it holds that $\Eisen_{P_i}^{G} \bk{f,s,\chi,g}$ admits a pole only if $\Eisen_{P_i}^{G} \bk{f^0,s,\chi,g}$ does.
Furthermore, if $\Eisen_{P_i}^{G} \bk{f^0,s,\chi,g}$ admits a pole, then $G_{P_i}\bk{s_0}=0$.
Thus, under these assumptions, the set of poles of $\Eisen_{P_i}^{G} \bk{f,s,\chi,g}$ in $\Real\bk{s_0}>0$ is contained by the set of zeros of $G_{P_i}\bk{s}$ there.
Let
\[
N_\epsilon\bk{\lambda} = \set{\alpha\in\Phi^{+} \mvert \inner{\lambda,\check{\alpha}}=\epsilon}
\]
and let $d_{P_i}\bk{s_0}$ denote the order of the zero of $G_{P_i}\bk{s}$ at $s_0$.
The following proposition determines the value of $d_{P_i}\bk{\eta_s}$ when $F=\Q$ and $s\in\R$.
The assumption that $F=\Q$ is made only so that the denominator of the Gindikin-Karpelevich factors is non-vanishing for $\Real\bk{s}>0$.

\begin{Prop}
	If $F=\Q$ and $s\in\R$, then
	\[
	d_{P_i}\bk{\lambda_s^{P_i}} = \Card{N_1\bk{\lambda_s^{P_i}}}-\Card{N_0\bk{\lambda_s^{P_i}}}-\bk{n-1} .
	\]
\end{Prop}

\begin{proof}
	Since $\Delta_M\subseteq N_1\bk{\lambda_s^{P_i}}$, it follows that
	\[
	\begin{split}
		G_{P_i}\bk{s} = & \bk{\prodl_{\alpha\in N_{-1} \bk{\lambda_s^{P_i}}} \zfun\bk{l_\alpha^+\bk{\lambda}} l_\alpha^+\bk{\lambda} l_\alpha^-\bk{\lambda}} \\
		& \times \bk{\prodl_{\alpha\in N_{0} \bk{\lambda_s^{P_i}}} \zfun\bk{l_\alpha^+\bk{\lambda}} l_\alpha^+\bk{\lambda} l_\alpha^-\bk{\lambda}} \\
		& \times \bk{\prodl_{\alpha\in N_{1} \bk{\lambda_s^{P_i}} \setminus \Delta_M } \zfun\bk{l_\alpha^+\bk{\lambda}} l_\alpha^+\bk{\lambda} l_\alpha^-\bk{\lambda}} \\
		& \times \bk{\prodl_{\alpha\in \Phi^{+} \setminus N_{1,0,-1} \bk{\lambda_s^{P_i}}} \zfun\bk{l_\alpha^+\bk{\lambda}} l_\alpha^+\bk{\lambda} l_\alpha^-\bk{\lambda}} ,
	\end{split}
	\]
	where
	\[
	N_{\epsilon_1,...,\epsilon_l}\bk{\lambda} = \bigcup_{i=1}^l N_{\epsilon_i}\bk{\lambda} .
	\]
	On the other hand, it holds that
	\[
	\begin{split}
		& \prodl_{\alpha\in N_{0} \bk{\lambda_s^{P_i}}} l_\alpha^+\bk{\lambda} l_\alpha^-\bk{\lambda} = \bk{-1}^{\Card{N_{0} \bk{\lambda_s^{P_i}}}} \\
		& \prodl_{\alpha\in N_{1} \bk{\lambda_s^{P_i}} \setminus \Delta_M } l_\alpha^+\bk{\lambda} l_\alpha^-\bk{\lambda} = \coset{2\zfun\bk{2}}^{ \Card{N_{0} \bk{\lambda_s^{P_i}}} -\bk{n-1}} \\
		& \prodl_{\alpha\in N_{-1} \bk{\lambda_s^{P_i}}} l_\alpha^-\bk{\lambda} = \bk{-2}^{\Card{N_{-1} \bk{\lambda_s^{P_i}}}} .
	\end{split}
	\]
	Note, in particular, that $\Card{\Delta_M}=n-1$.
	
	Hence,
	\[
	\begin{split}
		G_{P_i}\bk{s} = & \bk{\prodl_{\alpha\in N_{-1} \bk{\lambda_s^{P_i}}} \zfun\bk{l_\alpha^+\bk{\lambda}} l_\alpha^+\bk{\lambda}} \times \bk{\prodl_{\alpha\in N_{0} \bk{\lambda_s^{P_i}}} \zfun\bk{l_\alpha^+\bk{\lambda}} } \\
		& \times \bk{\prodl_{\alpha\in N_{1} \bk{\lambda_s^{P_i}} \setminus \Delta_M } \zfun\bk{l_\alpha^+\bk{\lambda}} } \times \bk{\prodl_{\alpha\in \Phi^{+} \setminus N_{1,0,-1} \bk{\lambda_s^{P_i}}} \zfun\bk{l_\alpha^+\bk{\lambda}} l_\alpha^+\bk{\lambda} l_\alpha^-\bk{\lambda} } \\
		& \times 
		\zfun\bk{2}^{ \Card{N_{0} \bk{\lambda_s^{P_i}}} -\bk{n-1}} \times
		\bk{-1}^{\Card{N_{0,-1} \bk{\lambda_s^{P_i}}}} \times
		2^{\Card{N_{1,-1} \bk{\lambda_s^{P_i}}} -\bk{n-1} } .
	\end{split}
	\]
	Under our assumptions, it turns out the the following terms are holomorphic and non-zero
	\[
	\begin{split}
		& \prodl_{\alpha\in N_{-1} \bk{\lambda_s^{P_i}}} \zfun\bk{l_\alpha^+\bk{\lambda}} l_\alpha^+\bk{\lambda} , \\
		& \prodl_{\alpha\in \Phi^{+} \setminus N_{1,0,-1} \bk{\lambda_s^{P_i}}} \zfun\bk{l_\alpha^+\bk{\lambda}} l_\alpha^+\bk{\lambda} l_\alpha^-\bk{\lambda}, \\
		& \zfun\bk{2}^{ \Card{N_{0} \bk{\lambda_s^{P_i}}} -\bk{n-1}} \times
		\bk{-1}^{\Card{N_{0,-1} \bk{\lambda_s^{P_i}}}} \times
		2^{\Card{N_{1,-1} \bk{\lambda_s^{P_i}}} -\bk{n-1} } .
	\end{split}
	\]
	Hence, the order of the zero of $G_{P_i}\bk{s}$ at $s=s_0$ is that of
	\[
	\bk{\prodl_{\alpha\in N_{0} \bk{\lambda_s^{P_i}}} \zfun\bk{l_\alpha^+\bk{\lambda}} } \times \bk{\prodl_{\alpha\in N_{1} \bk{\lambda_s^{P_i}} \setminus \Delta_M } \zfun\bk{l_\alpha^+\bk{\lambda}} }
	\]
	and hence
	\[
	d_{P_i}\bk{\lambda_s^{P_i}} = \Card{N_1\bk{\lambda_s^{P_i}}}-\Card{N_0\bk{\lambda_s^{P_i}}} - \bk{n-1} .
	\]
\end{proof}

In particular, it follows that:
\begin{Cor}
	Under the same assumptions, there exist a non-zero constant $C\in\C^\times$ such that 
	\[
	\Eisen_{B}^\sharp\bk{\lambda_{s_0}^{P_i},g} = C \cdot \Res{s=s_0} \Eisen_{P_i} \bk{f_{\lambda_s^{P_i}}^0,s,\Id,g} .
	\]
\end{Cor}

We conclude this discussion be formulating a test for the existence and order of a pole.
For a point $s_0$ such that $C_w\bk{\lambda_{s,\chi}^{P_i}}$ admits a pole at $s_0$ for some $w\in W^{M,T}$, $\DPS{P_i}\bk{s_0}$ is generated by its spherical vector and $d_{P_i}\bk{\lambda_s}>0$, one expects that the order of $\Eisen_{P_i}^{G} \bk{f,s,\chi,g}$ at $s_0$ equals $d_{P_i}\bk{\lambda_s}>0$.
This can usually be verified by finding a singleton $\set{w}\in W^{M,T}\rmod\sim_{s_0,\chi}$ such that the order of $\ord{s_0}{C_w\bk{\lambda_{s,\chi}^{P_i}}}=d_{P_i}\bk{\lambda_s}$.

\subsection{Square-Integrability of the Residue}
\label{Subsection:Square-Integrability_criterion}

In order to determine whether the residual representation $\Res{s=s_0} \Eisen_{P}\bk{\chi}$ is square-integrable, that is admits an embedding into $L^2\bk{Z_\AA G_F\lmod G_\AA,\chi}$, we will use the Langlands' square-integrability criterion \cite[page 104]{MR0579181}.
Assume that $\Eisen_{P}\bk{\chi}$ admits a pole of order $k$ at $s=s_0$ and let
\[
\Sigma^{\bP,0}_{\bk{s_0,\chi,m}} = 
\set{ \coset{w'} \in \Sigma^{\bP}_{\bk{s_0,\chi,m}} \mvert
	\lim\limits_{s\to s_0} \bk{s-s_0}^k \kappa_{\coset{w'}}\bk{f,s,g} \not\equiv 0}
\]
\begin{Prop}[Langlands' Square-integrability Criterion]
	\label{Prop:Square_integrability_of_residue}
	The residual representation $\Res{s=s_0} \Eisen_{P}\bk{\chi}$ is square-integrable if and only if
	\[
	\Real\bk{\gen{w'\cdot \lambda_{s_0,\chi}^{P},\check{\alpha}}}<0 \quad \forall \alpha\in\Delta_G
	\]
	for all $\coset{w'}\in \Sigma^{\bP,0}_{\bk{s_0,\chi,m}}$.
\end{Prop}

\section{Results for $E_6$}

\begin{Thm}
	The orders of poles of the degenerate Eisenstein series $\Eisen_{P_i}^{G} \bk{f,s,\chi,g}$ are given in the tables bellow.
	The tables also state whether the residue is square-integrable or not.
	
	\begin{longtable}[H]{|c|c|c|c|}
		\hline
		$ord\bk{\chi}$ & $P_1$, $P_6$ & $3$ & $6$ \\ \hline
		\multirow{2}{*}{1} & Pole Order & 1 & 1 \\ 
		& $L^2$ & $\cmark$ & $\cmark$ \\ \hline
	\end{longtable}

	\begin{longtable}[H]{|c|c|c|c|c|c|}
		\hline
		$ord\bk{\chi}$ & $P_2$ & $\frac{1}{2}$ & $\frac{5}{2}$ & $\frac{7}{2}$ & $\frac{11}{2}$ \\ \hline
		\multirow{2}{*}{1} & Pole Order & 1 & 1 & 1 & 1 \\ 
		& $L^2$ & $\cmark$ & $\xmark$ & $\cmark$ & $\cmark$ \\ \hline
		\multirow{2}{*}{2} & Pole Order & 1 & 0 & 0 & 0 \\ 
		& $L^2$ & $\cmark$ &  &  &  \\ \hline
	\end{longtable}
	
	\begin{longtable}[H]{|c|c|c|c|c|c|}
		\hline
		$ord\bk{\chi}$ & $P_3$, $P_5$ & $\frac{3}{2}$ & $\frac{5}{2}$ & $\frac{7}{2}$ & $\frac{9}{2}$ \\ \hline
		\multirow{2}{*}{1} & Pole Order & 2 & 1 & 1 & 1 \\ 
		& $L^2$ & $\cmark$ & $\xmark$ & $\xmark$ & $\cmark$ \\ \hline
		\multirow{2}{*}{2} & Pole Order & 1 & 0 & 0 & 0 \\ 
		& $L^2$ & $\cmark$ &  &  &  \\ \hline
	\end{longtable}

	\begin{longtable}[H]{|c|c|c|c|c|c|c|}
		\hline
		$ord\bk{\chi}$ & $P_4$ & $\frac{1}{2}$ & $1$ & $\frac{3}{2}$ & $\frac{5}{2}$ & $\frac{7}{2}$ \\ \hline
		\multirow{2}{*}{1} & Pole Order & 2 & 1 & 3 & 2 & 1 \\ 
		& $L^2$ & $\xmark$ & $\xmark$ & $\cmark$ & $\cmark$ & $\cmark$ \\ \hline
		\multirow{2}{*}{2} & Pole Order & 1 & 1 & 1 & 0 & 0 \\ 
		& $L^2$ & $\xmark$ & $\xmark$ & $\cmark$  & &  \\ \hline
		\multirow{2}{*}{3} & Pole Order & 1 & 0 & 0 & 0 & 0 \\ 
		& $L^2$ & $\cmark$ &  &  & &  \\ \hline
	\end{longtable}
	
\end{Thm}

\begin{proof}
	We start by noting the symmetry between the cases $i=1$ and $i=6$ and between $i=3$ and $i=5$.
	Thus, we only consider the cases $1\leq i\leq 4$.
	
	In each case $\bk{P_i,s_0,ord\bk{\chi}}$, we calculate the maximal order of pole of any particular intertwining operator.
	Namely,
	\begin{equation}
		\label{Eq:E6_highest_int_op_pole_order}
		\max\set{ord\bk{M_w\bk{s,\chi}} \mvert w\in W^{M,T}}
	\end{equation}
	These are given in the following tables:
	
	\begin{longtable}[H]{|c|c|c|c|c|c|c|c|}
		\hline
		$ord\bk{\chi}$ & $P_1$, $P_6$ & $1$ & $2$ & $3$ & $4$ & $5$ & $6$ \\ \hline
		1 & Pole Order & 2 & 2 & 2 & 1 & 1 & 1 \\ \hline
	\end{longtable}
	
	\begin{longtable}[H]{|c|c|c|c|c|c|c|c|}
		\hline
		$ord\bk{\chi}$ & $P_2$ & $\frac{1}{2}$ & $\frac{3}{2}$ & $\frac{5}{2}$ & $\frac{7}{2}$ & $\frac{9}{2}$ & $\frac{11}{2}$ \\ \hline
		1 & Pole Order & 4 & 3 & 3 & 2 & 1 & 1 \\ \hline
		2 & Pole Order & 1 & 0 & 0 & 0 & 0 & 0 \\ \hline
	\end{longtable}

	\begin{longtable}[H]{|c|c|c|c|c|c|c|c|}
		\hline
		$ord\bk{\chi}$ & $P_3$, $P_5$ & $\frac{1}{2}$ & $1$ & $\frac{3}{2}$ & $\frac{5}{2}$ & $\frac{7}{2}$ & $\frac{9}{2}$ \\ \hline
		1 & Pole Order & 5 & 1 & 5 & 3 & 2 & 1 \\ \hline
		2 & Pole Order & 1 & 1 & 1 & 0 & 0 & 0 \\ \hline
	\end{longtable}

	\begin{longtable}[H]{|c|c|c|c|c|c|c|c|}
		\hline
		$ord\bk{\chi}$ & $P_4$ & $\frac{1}{6}$ & $\frac{1}{2}$ & $1$ & $\frac{3}{2}$ & $\frac{5}{2}$ & $\frac{7}{2}$ \\ \hline
		1 & Pole Order & 1 & 9 & 2 & 6 & 3 & 1 \\ \hline
		2 & Pole Order & 0 & 3 & 2 & 1 & 0 & 0 \\ \hline
		3 & Pole Order & 1 & 1 & 0 & 0 & 0 & 0 \\ \hline
	\end{longtable}
	
	In the case of trivial $\chi$, the order of the Eisenstein series $\Eisen_{P_i}^{G} \bk{f,s,\chi,g}$ is bounded by the orders given in the statement of the theorem using the method described in \Cref{Subsection:Poles_Method_Normalized_Eis_ser}.
	On the other hand, these orders are realized by a singleton $\set{w}\in W^{M,T}\rmod\sim_{s_0,\chi}$ such that the order of $C_w\bk{\lambda_{s,\chi}^{P_i}}$ at $s_0$ is as required.
	
	We now consider the case of non-trivial $\chi$.
	There are a couple of distinct cases that should be taken cared of here:
	\begin{enumerate}
		\item If the order given by \Cref{Eq:E6_highest_int_op_pole_order} equals that of the statement of the theorem, it is enough to provide an equivalence class $\coset{w'}\in W^{M,T}\rmod\sim_{s_0,\chi}$ which realizes a pole of this order.
		The relevant cases are: $\bk{2, 1/2, 2}$, $\bk{3, 3/2, 2}$ and $\bk{4, 3/2, 2}$.
		
		In the case $\bk{4,1/2,3}$, there is no singleton equivalence class which supports a pole.
		However, in \Cref{Chap:Square_Integrable} we calculate the residue of the $\kappa_{\coset{w'}}\bk{f,s,g}$ which support a simple pole.
		In particular, the Eisenstein series admits a simple pole there.
		
		\item If the order given by \Cref{Eq:E6_highest_int_op_pole_order} is lower than that in the statement of the theorem, then one uses the method described in \Cref{Subsection:Methods_Reduction_of_order_of_poles}, in order to provide a lower bound to $\ord{s=s_0}{\Eisen_{P_i}^{G} \bk{f,s,\chi,g}}$ which would agree with the order stated in the theorem.
		There are, then, two cases:
		\begin{enumerate}
			\item If $\ord{s=s_0}{\Eisen_{P_i}^{G} \bk{f,s,\chi,g}} = 0$, then $\Eisen_{P_i}^{G} \bk{f,s,\chi,g}$ is holomorphic.
			Consider \Cref{Rem:Order_of_trivial_eq_class} to see that it is also non-vanishing.
			The relevant cases are: $\bk{3, 1/2, 2}$, $\bk{3, 1, 2}$ and $\bk{4, 1/6, 3}$.
			These are also the cases where $\DPS{P_i}\bk{s_0,\chi}_v$ is irreducible for any $v\in\Places$.
			
			\item Otherwise, there exists a singleton equivalence class which supports a pole of the prescribed order, as in item (1).
			The relevant cases are $\bk{4, 1/2, 2}$ and $\bk{4, 1, 2}$.
		\end{enumerate}

	\end{enumerate}
	
	In order to determine the square-integrability of the residue, we employ \Cref{Prop:Square_integrability_of_residue}.
	
\end{proof}

\section{Results for $E_7$}

\begin{Thm}
	The orders of poles of the degenerate Eisenstein series $\Eisen_{P_i}^{G} \bk{f,s,\chi,g}$ are given in the tables bellow.
	The tables also state whether the residue is square-integrable or not.
	\begin{longtable}[H]{|c|c|c|c|c|c|}
		\hline
		$ord\bk{\chi}$ & $P_1$& $\frac{1}{2}$ & $\frac{7}{2}$ & $\frac{11}{2}$ & $\frac{17}{2}$ \\ \hline
		\multirow{2}{*}{1} & Pole Order & 1 & 1 & 1 & 1 \\ 
		& $L^2$ & $\cmark$ & $\cmark$ & $\cmark$ & $\cmark$ \\ \hline
		\multirow{2}{*}{2} & Pole Order & 1 & 0 & 0 & 0 \\ 
		& $L^2$ & $\cmark$ &  &  &  \\ \hline
	\end{longtable}

	\begin{longtable}[H]{|c|c|c|c|c|c|c|c|}
		\hline
		$ord\bk{\chi}$ & $P_2$& $1$ & $2$ & $3$ & $4$ & $5$ & $7$ \\ \hline
		\multirow{2}{*}{1} & Pole Order & 1 & 1 & 1 & 1 & 1 & 1 \\ 
		& $L^2$ & $\cmark$ & $\xmark$ & $\xmark$ & $\xmark$ & $\cmark$ & $\cmark$  \\ \hline
		\multirow{2}{*}{2} & Pole Order & 0 & 1 & 0 & 0 & 0 & 0 \\ 
		& $L^2$ &  & $\cmark$ &&&&  \\ \hline
	\end{longtable}

	\begin{longtable}[H]{|c|c|c|c|c|c|c|c|}
		\hline
		$ord\bk{\chi}$ & $P_3$& $\frac{1}{2}$ & $\frac{3}{2}$ & $\frac{5}{2}$ & $\frac{7}{2}$ & $\frac{9}{2}$ & $\frac{11}{2}$ \\ \hline
		\multirow{2}{*}{1} & Pole Order & 2 & 2 & 2 & 1 & 1 & 1 \\ 
		& $L^2$ & $\cmark$ & $\cmark$ & $\cmark$ & $\xmark$ & $\xmark$ & $\cmark$ \\ \hline
		\multirow{2}{*}{2} & Pole Order & 1 & 1 & 1 & 0 & 0 & 0 \\ 
		& $L^2$ & $\cmark$ & $\cmark$ & $\cmark$ &  &  &  \\ \hline
		\multirow{2}{*}{3} & Pole Order & 1 & 0 & 0 & 0 & 0 & 0 \\ 
		& $L^2$ & $\cmark$ &  &  &  &  &  \\ \hline
	\end{longtable}
	
	\begin{longtable}[H]{|c|c|c|c|c|c|c|c|c|}
		\hline
		$ord\bk{\chi}$ & $P_4$& $\frac{1}{2}$ & $\frac{2}{3}$ & $1$ & $\frac{3}{2}$ & $2$ & $3$ & $4$ \\ \hline
		\multirow{2}{*}{1} & Pole Order & 1 & 1 & 4 & 1 & 3 & 2 & 1 \\ 
		& $L^2$ & $\xmark$ & $\xmark$ & $\cmark$ & $\xmark$ & $\cmark$ & $\cmark$ & $\cmark$  \\ \hline
		\multirow{2}{*}{2} & Pole Order & 1 & 0 & 2 & 1 & 1 & 0 & 0 \\ 
		& $L^2$ & $\xmark$ &  & $\cmark$ & $\xmark$ & $\cmark$ &  &   \\ \hline
		\multirow{2}{*}{3} & Pole Order & 0 & 1 & 1 & 0 & 0 & 0 & 0 \\ 
		& $L^2$ &  & $\xmark$ & $\cmark$ &  &  &  &   \\ \hline
		\multirow{2}{*}{4} & Pole Order & 1 & 0 & 0 & 0 & 0 & 0 & 0 \\ 
		& $L^2$ & $\cmark$ &  &  &  &  &  &   \\ \hline
	\end{longtable}

	\begin{longtable}[H]{|c|c|c|c|c|c|c|c|}
		\hline
		$ord\bk{\chi}$ & $P_5$& $1$ & $\frac{3}{2}$ & $2$ & $3$ & $4$ & $5$ \\ \hline
		\multirow{2}{*}{1} & Pole Order & 3 & 1 & 2 & 2 & 1 & 1 \\ 
		& $L^2$ & $\cmark$ & $\xmark$ & $\xmark$ & $\cmark$ & $\xmark$ & $\cmark$  \\ \hline
		\multirow{2}{*}{2} & Pole Order & 1 & 1 & 1 & 0 & 0 & 0 \\ 
		& $L^2$ & $\xmark$ & $\xmark$ & $\cmark$ &  &  &   \\ \hline
		\multirow{2}{*}{3} & Pole Order & 1 & 0 & 0 & 0 & 0 & 0 \\ 
		& $L^2$ & $\cmark$ &  &  &  &  &   \\ \hline
	\end{longtable}

	\begin{longtable}[H]{|c|c|c|c|c|c|c|}
		\hline
		$ord\bk{\chi}$ & $P_6$ & $\frac{1}{2}$ & $\frac{5}{2}$ & $\frac{7}{2}$ & $\frac{11}{2}$ & $\frac{13}{2}$ \\ \hline
		\multirow{2}{*}{1} & Pole Order & 1 & 2 & 1 & 1 & 1 \\ 
		& $L^2$ & $\cmark$ & $\cmark$ & $\cmark$ & $\xmark$ & $\cmark$ \\ \hline
		\multirow{2}{*}{2} & Pole Order & 1 & 1 & 0 & 0 & 0 \\ 
		& $L^2$ & $\cmark$ & $\cmark$ & & & \\ \hline
	\end{longtable}

	\begin{longtable}[H]{|c|c|c|c|c|}
		\hline
		$ord\bk{\chi}$ & $P_7$ & $1$ & $5$ & $9$ \\ \hline
		\multirow{2}{*}{1} & Pole Order & 1 & 1 & 1 \\ 
		& $L^2$ & $\cmark$ & $\cmark$ & $\cmark$ \\ \hline
	\end{longtable}
	
\end{Thm}

\begin{Rem}
	In the case of $E_7$ and $P=P_7$, our results agree with the results of \cite{hanzer_savin_2019}.
\end{Rem}

\begin{proof}
	The proof in this case is similar to that of $E_6$ with a few exceptional cases.
	We start with the cases that can be completely sorted using the methods of \Cref{Section:Poles_Method} and then deal with the exceptional cases.
	
	In each case $\bk{P_i,s_0,ord\bk{\chi}}$, we list the maximal order of pole of any particular intertwining operator in the following table:
	
	\begin{longtable}[H]{|c|c|c|c|c|c|c|c|c|c|c|}
		\hline
		$ord\bk{\chi}$ & $P_1$& $\frac{1}{2}$ & $\frac{3}{2}$ & $\frac{5}{2}$ & $\frac{7}{2}$ & $\frac{9}{2}$ & $\frac{11}{2}$ & $\frac{13}{2}$ & $\frac{15}{2}$ & $\frac{17}{2}$ \\ \hline
		1 & Pole Order & 4 & 3 & 3 & 3 & 2 & 2 & 1 & 1 & 1 \\ \hline
		2 & Pole Order & 1 & 0 & 0 & 0 & 0 & 0 & 0 & 0 & 0 \\ \hline
	\end{longtable}

	\begin{longtable}[H]{|c|c|c|c|c|c|c|c|c|c|c|}
		\hline
		$ord\bk{\chi}$ & $P_2$ & $\frac{1}{2}$ & $1$ & $\frac{3}{2}$ & $2$ & $3$ & $4$ & $5$ & $6$ & $7$ \\ \hline
		1 & Pole Order & 1 & 6 & 1 & 5 & 4 & 3 & 2 & 1 & 1 \\ \hline
		2 & Pole Order & 1 & 1 & 1 & 1 & 0 & 0 & 0 & 0 & 0 \\ \hline
	\end{longtable}

	\begin{longtable}[H]{|c|c|c|c|c|c|c|c|c|c|c|}
		\hline
		$ord\bk{\chi}$ & $P_3$ & $\frac{1}{6}$ & $\frac{1}{2}$ & $1$ & $\frac{3}{2}$ & $2$ & $\frac{5}{2}$ & $\frac{7}{2}$ & $\frac{9}{2}$ & $\frac{11}{2}$ \\ \hline
		1 & Pole Order & 1 & 9 & 2 & 7 & 1 & 5 & 3 & 2 & 1 \\ \hline
		2 & Pole Order & 0 & 3 & 2 & 2 & 1 & 1 & 0 & 0 & 0 \\ \hline
		3 & Pole Order & 1 & 1 & 0 & 0 & 0 & 0 & 0 & 0 & 0 \\ \hline
	\end{longtable}

	\begin{longtable}[H]{|c|c|c|c|c|c|c|c|c|c|c|}
		\hline
		$ord\bk{\chi}$ & $P_4$ & $\frac{1}{4}$ & $\frac{1}{3}$ & $\frac{1}{2}$ & $\frac{2}{3}$ & $1$ & $\frac{3}{2}$ & $2$ & $3$ & $4$ \\ \hline
		1 & Pole Order & 1 & 2 & 5 & 2 & 11 & 2 & 6 & 3 & 1 \\ \hline
		2 & Pole Order & 1 & 0 & 5 & 0 & 4 & 2 & 1 & 0 & 0 \\ \hline
		3 & Pole Order & 0 & 2 & 0 & 2 & 1 & 0 & 0 & 0 & 0 \\ \hline
		4 & Pole Order & 1 & 0 & 1 & 0 & 0 & 0 & 0 & 0 & 0 \\ \hline
	\end{longtable}

	\begin{longtable}[H]{|c|c|c|c|c|c|c|c|c|c|c|}
		\hline
		$ord\bk{\chi}$ & $P_5$ & $\frac{1}{3}$ & $\frac{1}{2}$ & $\frac{2}{3}$ & $1$ & $\frac{3}{2}$ & $2$ & $3$ & $4$ & $5$ \\ \hline
		1 & Pole Order & 1 & 3 & 1 & 10 & 2 & 6 & 4 & 2 & 1 \\ \hline
		2 & Pole Order & 0 & 3 & 0 & 3 & 2 & 1 & 0 & 0 & 0 \\ \hline
		3 & Pole Order & 1 & 0 & 1 & 1 & 0 & 0 & 0 & 0 & 0 \\ \hline
	\end{longtable}

	\begin{longtable}[H]{|c|c|c|c|c|c|c|c|c|c|c|}
		\hline
		$ord\bk{\chi}$ & $P_6$ & $\frac{1}{2}$ & $1$ & $\frac{3}{2}$ & $2$ & $\frac{5}{2}$ & $\frac{7}{2}$ & $\frac{9}{2}$ & $\frac{11}{2}$ & $\frac{13}{2}$ \\ \hline
		1 & Pole Order & 6 & 1 & 5 & 1 & 5 & 3 & 2 & 2 & 1 \\ \hline
		2 & Pole Order & 2 & 1 & 1 & 1 & 1 & 0 & 0 & 0 & 0 \\ \hline
	\end{longtable}

	\begin{longtable}[H]{|c|c|c|c|c|c|c|c|c|c|c|}
		\hline
		$ord\bk{\chi}$ & $P_7$ &$1$ & $2$ & $3$ & $4$ & $5$ & $6$ & $7$ & $8$ & $9$ \\ \hline
		1 & Pole Order & 3 & 2 & 2 & 2 & 2 & 1 & 1 & 1 & 1 \\ \hline
	\end{longtable}
	
	For all cases with trivial $\chi$, the proof continues as in the case of $E_6$, with the exception of the case $\DPS{P_2}\bk{1}$, which is not generated by its spherical vector.
	This case will be dealt with bellow.
	
	We now deal with the cases with non-trivial $\chi$:
	\begin{itemize}
		\item In the following cases, there exists a singleton equivalence class which realize a pole of the same order as prescribed in the theorem and the above tables:
		$\bk{1,1/2,2}$, $\bk{3,1/2,3}$, $\bk{3,5/2,2}$, $\bk{4, 1, 3}$, $\bk{4, 2, 2}$ and $\bk{5, 1, 3}$.
		
		\item In the following cases, the pole in these tables cancels completely and the Eisenstein series is holomorphic at the point:
		$\bk{2,1/2,2}$, $\bk{2,1,2}$, $\bk{2,3/2,2}$, $\bk{3,1,2}$, $\bk{3,2,2}$, $\bk{3,1/6,3}$, $\bk{4,1/4,2}$, $\bk{4,1/3,3}$, $\bk{4,1/4,4}$, $\bk{5,1/3,3}$, $\bk{5,2/3,3}$, $\bk{6,1,2}$, $\bk{6,3/2,2}$ and $\bk{6,2,2}$.

		These are precisely the cases $[i,s,ord\bk{\chi}]$ in the above tables where $\DPS{P_i}\bk{s,\chi}_v$ is irreducible at infinitely many $v\in\Places$.
		
		\item In the following cases, the pole reduces to a simple pole and a simple pole is realized by a singleton equivalence class:
		$\bk{3,1/2,2}$, $\bk{3, 3/2, 2}$, $\bk{4,1,2}$, $\bk{4, 3/2, 2}$, $\bk{4, 2/3, 3}$, $\bk{5,1,2}$, $\bk{5, 3/2, 2}$, $\bk{6, 1/2, 2}$ and $\bk{6, 5/2, 2}$.
		
		\item The cases $\bk{2,2,2}$, $\bk{5,2,2}$ and $\bk{4,1/2,4}$ are dealt with in  \Cref{Chap:Square_Integrable}.
	\end{itemize}

	These leaves the case of $\bk{2,1,1}$ to be treated.

	In the case of $\bk{2,1,1}$ the method outlined in \Cref{Section:Poles_Method} requires a few modifications.
	We recall that $\DPS{P_2}\bk{1}$ has the following structure
	\[
	\pi_{s,v} \hookrightarrow \DPS{P_2}\bk{1}_v \twoheadrightarrow \pi_{1,v}\oplus\pi_{-1,v},
	\]
	where $\pi_0$ is the unique irreducible spherical subquotient and $\pi_{-1}$ is irreducible but not spherical.
	Furthermore,
	\[
	\dim_\C\pi_{-1,v}^{\frakJ_v} = 15 .
	\]
	Let
	\[
	\lambda_{0,v} = \esevenchar{-1}{5}{-1}{-1}{-1}{-1}{-1}
	\]
	denote the initial exponent of $\DPS{P_2}\bk{-1}$.
	Then, we may write
	\[
	r_{T_v}^{G_v}\pi_{-1,v} = \oplus_{\lambda\in Y} w\cdot\lambda_{0,v},
	\]
	where $Y\subset W$ is independent of $v$ and satisfies $\Card{Y}=15$.
	In particular, we note that $\Card{\coset{w}}=2$ for $w\in Y$.
	Furthermore,
	\[
	M_w\bk{\lambda_{s}^{P_2}} f_s\bk{g}
	= C_w\bk{\lambda_{s}^{P_2}} \otimes_{v\notin S}f_{s,v}\bk{g_v} \otimes_{v\in S} N_{w,v}\bk{\lambda_{s}^{P_2}}f_{s,v}\bk{g_v} 
	\]
	and
	\[
	N_{w,v}\bk{\lambda_{-1}^{P_2}}f_{s,v}\bk{g_v} \in i_{T_v}^{G_v}w\cdot\lambda_0 = i_{T_v}^{G_v}w\cdot\lambda_{-1}^{P_2} .
	\]
	Hence, $\pi_{1,v}$ appears in the image of $N_{w,v}\bk{\lambda_{s}^{P}}$ only for $w\in W^{M,T}$ such that $w\in Y$.
	In particular, one checks that for any other $w\in W^{M,T}$ such that $\kappa_{\coset{w}}$ admits a pole at $s_0=1$, it holds that $\Image\bk{N_{w,v}\bk{\lambda_{1}^{P}}}=\pi_{1,v}$.
	
	Hence, the order of $\kappa_{\coset{w}}$ at $s_0=1$ for $w\in W^{M,T} \setminus Y$ such that $\kappa_{\coset{w}}$ admits a pole there, is bounded by the order of $\kappa_{\coset{w}}\bk{f^0,g}$ there.
	Since
	\[
	\kappa_{\coset{w}}\bk{f^0,g,s} = \suml_{w\in W^{M,T}} C_w\bk{\lambda_{s,\chi}^{P_2}},
	\]
	the order of $\kappa_{\coset{w}}\bk{f^0,g}$ at $s_0=1$ can be determined by a standard calculation.
	Furthermore, one may use the method described in \Cref{Subsection:Poles_Method_Normalized_Eis_ser}, in order to show that $\kappa_{\coset{w}}\bk{f^0,g}$ admits at most a simple pole at $s_0=1$.

\end{proof}

\section{Results for $E_8$}

\begin{Thm}
	The orders of poles of the degenerate Eisenstein series $\Eisen_{P_i}^{G} \bk{f,s,\chi,g}$ are given in the tables bellow.
	The tables also state whether the residue is square-integrable or not.
\begin{longtable}[H]{|c|c|c|c|c|c|c|c|c|}
	\hline
	$ord\bk{\chi}$ & $P_1$& $\frac{1}{2}$ & $\frac{5}{2}$ & $\frac{7}{2}$ & $\frac{11}{2}$ & $\frac{13}{2}$ & $\frac{17}{2}$ & $\frac{23}{2}$ \\ \hline
	\multirow{2}{*}{1} & Pole Order & 1 & 1 & 1 & 1 & 1 & 1 & 1 \\ 
	& $L^2$ & $\cmark$ & $\cmark$ & $\cmark$ & $\xmark$ & $\cmark$ & $\cmark$ & $\cmark$  \\ \hline
	\multirow{2}{*}{2} & Pole Order & 1 & 0 & 1 & 0 & 0 & 0 & 0 \\ 
	& $L^2$ & \cmark &  & \cmark &  &  &  &   \\ \hline
\end{longtable}

\begin{longtable}[H]{|c|c|c|c|c|c|c|c|c|c|}
	\hline
	$ord\bk{\chi}$ & $P_2$& $\frac{1}{2}$ & $\frac{3}{2}$ & $\frac{5}{2}$ & $\frac{7}{2}$ & $\frac{9}{2}$ & $\frac{11}{2}$ & $\frac{13}{2}$ & $\frac{17}{2}$ \\ \hline
	\multirow{2}{*}{1} & Pole Order & 1${}^{\ast}$ & 2 & 2 & 2 & 1 & 1 & 1 & 1 \\ 
	& $L^2$ & $\cmark$ & $\cmark$ & $\cmark$ & $\cmark$ & $\xmark$ & $\xmark$ & $\cmark$ & $\cmark$  \\ \hline
	\multirow{2}{*}{2} & Pole Order & 1 & 1 & 1 & 1 & 0 & 0 & 0 & 0 \\ 
	& $L^2$ & \cmark & \cmark & \cmark & \cmark &  &  &  &  \\ \hline
	\multirow{2}{*}{3} & Pole Order & 0 & 1 & 0 & 0 & 0 & 0 & 0 & 0 \\ 
	& $L^2$ &  & \cmark &  &  &  &  &  &  \\ \hline
\end{longtable}

\begin{longtable}[H]{|c|c|c|c|c|c|c|c|c|c|c|c|}
	\hline
	$ord\bk{\chi}$ & $P_3$& $\frac{1}{2}$ & $1$ & $\frac{7}{6}$ & $\frac{3}{2}$ & $2$ & $\frac{5}{2}$ & $\frac{7}{2}$ & $\frac{9}{2}$ & $\frac{11}{2}$ & $\frac{13}{2}$ \\ \hline
	\multirow{2}{*}{1} & Pole Order & 1 & 1 & 1 & 3 & 1 & 2 & 2 & 1 & 1 & 1 \\ 
	& $L^2$ & $\xmark$ & $\xmark$ & $\xmark$ & $\cmark$ & $\xmark$ & $\cmark$ & $\cmark$ & $\xmark$ & $\xmark$ & $\cmark$  \\ \hline
	\multirow{2}{*}{2} & Pole Order & 1 & 1 & 0 & 1 & 1 & 1 & 1 & 0 & 0 & 0 \\ 
	& $L^2$ & \cmark & \xmark &  & \xmark & \xmark & \cmark & \cmark &  &  &  \\ \hline
	\multirow{2}{*}{3} & Pole Order & 0 & 0 & 1 & 1 & 0 & 0 & 0 & 0 & 0 & 0 \\ 
	& $L^2$ &  &  & \xmark & \cmark &  &  &  &  &  &  \\ \hline
	\multirow{2}{*}{4} & Pole Order & 0 & 1 & 0 & 0 & 0 & 0 & 0 & 0 & 0 & 0 \\ 
	& $L^2$ &  & \cmark &  &  &  &  &  &  &  &  \\ \hline
\end{longtable}

\begin{longtable}[H]{|c|c|c|c|c|c|c|c|c|c|c|c|c|}
	\hline
	$ord\bk{\chi}$ & $P_4$& $\frac{3}{10}$ & $\frac{1}{2}$ & $\frac{3}{4}$ & $\frac{5}{6}$ & $1$ & $\frac{7}{6}$ & $\frac{3}{2}$ & $2$ & $\frac{5}{2}$ & $\frac{7}{2}$ & $\frac{9}{2}$ \\ \hline
	\multirow{2}{*}{1} & Pole Order & 1 & 5 & 1 & 1 & 2 & 1 & 4 & 1 & 3 & 2 & 1 \\ 
	& $L^2$ & $\xmark$ & $\cmark$ & $\xmark$ & $\xmark$ & $\xmark$ & $\xmark$ & $\cmark$ & $\xmark$ & $\cmark$ & $\cmark$ & $\cmark$  \\ \hline
	\multirow{2}{*}{2} & Pole Order & 0 & 3 & 1 & 0 & 2 & 0 & 2 & 1 & 1 & 0 & 0  \\
	& $L^2$ &  & \cmark & \xmark &  & \xmark &  & \cmark & \xmark & \cmark &  &  \\ \hline
	\multirow{2}{*}{3} & Pole Order & 0 & 2 & 0 & 1 & 0 & 1 & 1 & 0 & 0 & 0 & 0  \\
	& $L^2$ &  & \cmark &  & \xmark &  & \xmark & \cmark &  &  &  &  \\ \hline
	\multirow{2}{*}{4} & Pole Order & 0 & 1 & 1 & 0 & 1 & 0 & 0 & 0 & 0 & 0 & 0  \\
	& $L^2$ &  & \xmark & \xmark &  & \cmark &  &  &  &  &  &  \\ \hline
	\multirow{2}{*}{5} & Pole Order & 1 & 1 & 0 & 0 & 0 & 0 & 0 & 0 & 0 & 0 & 0  \\
	& $L^2$ & \xmark & \cmark &  &  &  &  &  &  &  &  &  \\ \hline
	\multirow{2}{*}{6} & Pole Order & 0 & 1 & 0 & 0 & 0 & 0 & 0 & 0 & 0 & 0 & 0  \\
	& $L^2$ &  & \cmark &  &  &  &  &  &  &  &  &  \\ \hline
\end{longtable}

\begin{longtable}[H]{|c|c|c|c|c|c|c|c|c|c|c|c|}
	\hline
	$ord\bk{\chi}$ & $P_5$& $\frac{1}{2}$ & $\frac{5}{6}$ & $1$ & $\frac{7}{6}$ &  $\frac{3}{2}$ & $2$ & $\frac{5}{2}$ & $\frac{7}{2}$ & $\frac{9}{2}$ & $\frac{11}{2}$ \\ \hline
	\multirow{2}{*}{1} & Pole Order & 4${}^{\ast}$ & 1 & 2 & 1 & 4 & 1 & 3 & 2 & 1 & 1 \\ 
	& $L^2$ & $\cmark$ & $\xmark$ & $\xmark$ & $\xmark$ & $\cmark$ & $\xmark$ & $\cmark$ & $\cmark$ & $\xmark$ & $\cmark$  \\ \hline
	\multirow{2}{*}{2} & Pole Order & 2 & 0 & 2 & 0 & 2 & 1 & 1 & 0 & 0 & 0 \\ 
	& $L^2$ & \cmark &  & \cmark &  & \cmark & \xmark & \cmark &  &  &   \\ \hline
	\multirow{2}{*}{3} & Pole Order & 1 & 1 & 0 & 1 & 1 & 0 & 0 & 0 & 0 & 0 \\ 
	& $L^2$ & \xmark & \xmark &  & \xmark & \cmark &  &  &  &  &   \\ \hline
	\multirow{2}{*}{4} & Pole Order & 1 & 0 & 1 & 0 & 0 & 0 & 0 & 0 & 0 & 0 \\ 
	& $L^2$ & \cmark &  & \cmark &  &  &  &  &  &  &   \\ \hline
	\multirow{2}{*}{5} & Pole Order & 1 & 0 & 0 & 0 & 0 & 0 & 0 & 0 & 0 & 0 \\ 
	& $L^2$ & \cmark &  &  &  &  &  &  &  &  &   \\ \hline
\end{longtable}

\begin{longtable}[H]{|c|c|c|c|c|c|c|c|c|c|c|c|}
	\hline
	$ord\bk{\chi}$ & $P_6$& $\frac{1}{2}$ & $1$ & $2$ & $\frac{5}{2}$ & $3$ & $4$ & $5$ & $6$ & $7$ \\ \hline
	\multirow{2}{*}{1} & Pole Order & 1 & 2 & 3 & 1 & 2 & 1 & 1 & 1 & 1 \\ 
	& $L^2$ & $\xmark$ & $\cmark$ & $\cmark$ & $\xmark$ & $\cmark$ & $\xmark$ & $\xmark$ & $\xmark$ & $\cmark$  \\ \hline
	\multirow{2}{*}{2} & Pole Order & 1 & 1 & 1 & 1 & 1 & 0 & 0 & 0 & 0 \\ 
	& $L^2$ & \cmark & \cmark & \xmark & \xmark & \cmark &  &  &  &   \\ \hline
	\multirow{2}{*}{3} & Pole Order & 0 & 1 & 1 & 0 & 0 & 0 & 0 & 0 & 0 \\ 
	& $L^2$ &  & \cmark & \cmark &  &  &  &  &  &   \\ \hline
	\multirow{2}{*}{4} & Pole Order & 1 & 0 & 0 & 0 & 0 & 0 & 0 & 0 & 0 \\ 
	& $L^2$ & \cmark &  &  &  &  &  &  &  &   \\ \hline
\end{longtable}

\begin{longtable}[H]{|c|c|c|c|c|c|c|c|c|}
	\hline
	$ord\bk{\chi}$ & $P_7$& $\frac{1}{2}$ & $\frac{3}{2}$ & $\frac{5}{2}$ & $\frac{9}{2}$ & $\frac{11}{2}$ & $\frac{17}{2}$ & $\frac{19}{2}$ \\ \hline
	\multirow{2}{*}{1} & Pole Order & 2 & 1 & 1 & 2 & 1 & 1 & 1 \\ 
	& $L^2$ & $\cmark$ & $\cmark$ & $\cmark$ & $\cmark$ & $\cmark$ & $\xmark$ & $\cmark$  \\ \hline
	\multirow{2}{*}{2} & Pole Order & 1 & 0 & 1 & 1 & 0 & 0 & 0 \\ 
	& $L^2$ & \cmark &  & \cmark & \cmark &  &  &   \\ \hline
	\multirow{2}{*}{3} & Pole Order & 1 & 0 & 0 & 0 & 0 & 0 & 0 \\ 
	& $L^2$ & \cmark &  &  &  &  &  &   \\ \hline
\end{longtable}

\begin{longtable}[H]{|c|c|c|c|c|c|}
	\hline
	$ord\bk{\chi}$ & $P_8$& $\frac{1}{2}$ & $\frac{11}{2}$ & $\frac{19}{2}$ & $\frac{29}{2}$ \\ \hline
	\multirow{2}{*}{1} & Pole Order & 1 & 1 & 1 & 1 \\ 
	& $L^2$ & $\cmark$ & $\cmark$ & $\cmark$ & $\cmark$ \\ \hline
	\multirow{2}{*}{2} & Pole Order & 1 & 0 & 0 & 0 \\ 
	& $L^2$ & \cmark &  &  &  \\ \hline
\end{longtable}

With the exception of the cases denoted by an asterisk, $\DPS{P_2}\bk{\frac{1}{2}}$ and $\DPS{P_5}\bk{\frac{1}{2}}$, for whom the structure of the maximal semi-simple quotient is not completely determined and thus the result cited above is under the assumption that it is irreducible.
\end{Thm}

\begin{proof}
	As in the previous cases, we start by recording the maximal order of pole of any particular intertwining operator, for each case $\bk{P_i,s_0,ord\bk{\chi}}$, in the following table:
	
	\begin{longtable}[H]{|c|c|c|c|c|c|c|c|c|c|c|c|c|c|c|c|c|}
		\hline
		$ord\bk{\chi}$ & $P_1$& $\frac{1}{2}$ & $1$ & $\frac{3}{2}$ & $2$ & $\frac{5}{2}$ & $3$ & $\frac{7}{2}$ & $\frac{9}{2}$ & $\frac{11}{2}$ & $\frac{13}{2}$ & $\frac{15}{2}$ & $\frac{17}{2}$ & $\frac{19}{2}$ & $\frac{21}{2}$ & $\frac{23}{2}$ \\ \hline
		1 & Pole Order & 7 & 1 & 6 & 1 & 6 & 1 & 5 & 4 & 4 & 3 & 2 & 2 & 1 & 1 & 1 \\ \hline
		2 & Pole Order & 2 & 1 & 1 & 1 & 1 & 1 & 1 & 0 & 0 & 0 & 0 & 0 & 0 & 0 & 0 \\ \hline
	\end{longtable}

	\begin{longtable}[H]{|c|c|c|c|c|c|c|c|c|c|c|c|c|c|c|c|c|}
		\hline
		$ord\bk{\chi}$ & $P_2$ & $\frac{1}{6}$ & $\frac{1}{2}$ & $\frac{5}{6}$ & $1$ & $\frac{7}{6}$ & $\frac{3}{2}$ & $2$ & $\frac{5}{2}$ & $3$ & $\frac{7}{2}$ & $\frac{9}{2}$ & $\frac{11}{2}$ & $\frac{13}{2}$ & $\frac{15}{2}$ & $\frac{17}{2}$ \\ \hline
		1 & Pole Order & 1 & 11 & 1 & 3 & 1 & 10 & 2 & 8 & 1 & 6 & 4 & 3 & 2 & 1 & 1 \\ \hline
		2 & Pole Order & 0 & 4 & 0 & 3 & 0 & 3 & 2 & 2 & 1 & 1 & 0 & 0 & 0 & 0 & 0 \\ \hline
		3 & Pole Order & 1 & 1 & 1 & 0 & 1 & 1 & 0 & 0 & 0 & 0 & 0 & 0 & 0 & 0 & 0 \\ \hline
	\end{longtable}

	\begin{longtable}[H]{|c|c|c|c|c|c|c|c|c|c|c|c|c|c|c|c|c|}
		\hline
		$ord\bk{\chi}$ & $P_3$ & $\frac{1}{6}$ & $\frac{1}{4}$ & $\frac{1}{2}$ & $\frac{3}{4}$ & $\frac{5}{6}$ & $1$ & $\frac{7}{6}$ & $\frac{3}{2}$ & $2$ & $\frac{5}{2}$ & $3$ & $\frac{7}{2}$ & $\frac{9}{2}$ & $\frac{11}{2}$ & $\frac{13}{2}$ \\ \hline
		1 & Pole Order & 2 & 1 & 14 & 1 & 2 & 5 & 2 & 11 & 3 & 7 & 1 & 5 & 3 & 2 & 1 \\ \hline
		2 & Pole Order & 0 & 1 & 6 & 1 & 0 & 5 & 0 & 4 & 3 & 2 & 1 & 1 & 0 & 0 & 0 \\ \hline
		3 & Pole Order & 2 & 0 & 2 & 0 & 2 & 0 & 2 & 1 & 0 & 0 & 0 & 0 & 0 & 0 & 0 \\ \hline
		4 & Pole Order & 0 & 1 & 1 & 1 & 0 & 1 & 0 & 0 & 0 & 0 & 0 & 0 & 0 & 0 & 0 \\ \hline
	\end{longtable}

	\begin{longtable}[H]{|c|c|c|c|c|c|c|c|c|c|c|c|c|c|c|c|c|}
		\hline
		$ord\bk{\chi}$ & $P_4$ & $\frac{1}{10}$ & $\frac{1}{6}$ & $\frac{1}{4}$ & $\frac{3}{10}$ & $\frac{1}{3}$ & $\frac{1}{2}$ & $\frac{3}{4}$ & $\frac{5}{6}$ & $1$ & $\frac{7}{6}$ & $\frac{3}{2}$ & $2$ & $\frac{5}{2}$ & $\frac{7}{2}$ & $\frac{9}{2}$ \\ \hline
		1 & Pole Order & 2 & 5 & 3 & 2 & 1 & 21 & 2 & 3 & 6 & 2 & 11 & 2 & 6 & 3 & 1 \\ \hline
		2 & Pole Order & 0 & 1 & 3 & 0 & 1 & 10 & 2 & 0 & 6 & 0 & 4 & 2 & 1 & 0 & 0 \\ \hline
		3 & Pole Order & 0 & 5 & 0 & 0 & 1 & 5 & 0 & 3 & 0 & 2 & 1 & 0 & 0 & 0 & 0 \\ \hline
		4 & Pole Order & 0 & 0 & 3 & 0 & 0 & 3 & 2 & 0 & 1 & 0 & 0 & 0 & 0 & 0 & 0 \\ \hline
		5 & Pole Order & 2 & 0 & 0 & 2 & 0 & 1 & 0 & 0 & 0 & 0 & 0 & 0 & 0 & 0 & 0 \\ \hline
		6 & Pole Order & 0 & 1 & 0 & 0 & 1 & 1 & 0 & 0 & 0 & 0 & 0 & 0 & 0 & 0 & 0 \\ \hline
	\end{longtable}

	\begin{longtable}[H]{|c|c|c|c|c|c|c|c|c|c|c|c|c|c|c|c|c|}
		\hline
		$ord\bk{\chi}$ & $P_5$ & $\frac{1}{10}$ & $\frac{1}{6}$ & $\frac{1}{4}$ & $\frac{3}{10}$ & $\frac{1}{2}$ & $\frac{3}{4}$ & $\frac{5}{6}$ & $1$ & $\frac{7}{6}$ & $\frac{3}{2}$ & $2$ & $\frac{5}{2}$ & $\frac{7}{2}$ & $\frac{9}{2}$ & $\frac{11}{2}$ \\ \hline
		1 & Pole Order & 1 & 4 & 2 & 1 & 20 & 1 & 3 & 6 & 2 & 12 & 2 & 7 & 4 & 2 & 1 \\ \hline
		2 & Pole Order & 0 & 0 & 2 & 0 & 8 & 1 & 0 & 6 & 0 & 4 & 2 & 1 & 0 & 0 & 0 \\ \hline
		3 & Pole Order & 0 & 4 & 0 & 0 & 4 & 0 & 3 & 0 & 2 & 1 & 0 & 0 & 0 & 0 & 0 \\ \hline
		4 & Pole Order & 0 & 0 & 2 & 0 & 2 & 1 & 0 & 1 & 0 & 0 & 0 & 0 & 0 & 0 & 0 \\ \hline
		5 & Pole Order & 1 & 0 & 0 & 1 & 1 & 0 & 0 & 0 & 0 & 0 & 0 & 0 & 0 & 0 & 0 \\ \hline
	\end{longtable}

	\begin{longtable}[H]{|c|c|c|c|c|c|c|c|c|c|c|c|c|c|c|c|c|}
		\hline
		$ord\bk{\chi}$ & $P_6$ & $\frac{1}{4}$ & $\frac{1}{3}$ & $\frac{1}{2}$ & $\frac{2}{3}$ & $1$ & $\frac{4}{3}$ & $\frac{3}{2}$ & $\frac{5}{3}$ & $2$ & $\frac{5}{2}$ & $3$ & $4$ & $5$ & $6$ & $7$ \\ \hline
		1 & Pole Order & 1 & 2 & 5 & 2 & 12 & 1 & 3 & 1 & 10 & 2 & 6 & 4 & 3 & 2 & 1 \\ \hline
		2 & Pole Order & 1 & 0 & 5 & 0 & 4 & 0 & 3 & 0 & 3 & 2 & 1 & 0 & 0 & 0 & 0 \\ \hline
		3 & Pole Order & 0 & 2 & 0 & 2 & 2 & 1 & 0 & 1 & 1 & 0 & 0 & 0 & 0 & 0 & 0 \\ \hline
		4 & Pole Order & 1 & 0 & 1 & 0 & 0 & 0 & 0 & 0 & 0 & 0 & 0 & 0 & 0 & 0 & 0 \\ \hline
	\end{longtable}

	\begin{longtable}[H]{|c|c|c|c|c|c|c|c|c|c|c|c|c|c|c|c|c|}
		\hline
		$ord\bk{\chi}$ & $P_7$ & $\frac{1}{6}$ & $\frac{1}{2}$ & $1$ &$\frac{3}{2}$ & $2$ & $\frac{5}{2}$ & $3$ & $\frac{7}{2}$ & $4$ & $\frac{9}{2}$ & $\frac{11}{2}$ & $\frac{13}{2}$ & $\frac{15}{2}$ & $\frac{17}{2}$ & $\frac{19}{2}$ \\ \hline
		1 & Pole Order & 1 & 9 & 2 & 7 & 2 & 6 & 1 & 5 & 1 & 5 & 3 & 2 & 2 & 2 & 1 \\ \hline
		2 & Pole Order & 0 & 3 & 2 & 2 & 2 & 2 & 1 & 1 & 1 & 1 & 0 & 0 & 0 & 0 & 0 \\ \hline
		3 & Pole Order & 1 & 1 & 0 & 0 & 0 & 0 & 0 & 0 & 0 & 0 & 0 & 0 & 0 & 0 & 0 \\ \hline
	\end{longtable}

	\begin{longtable}[H]{|c|c|c|c|c|c|c|c|c|c|c|c|c|c|c|c|c|}
		\hline
		$ord\bk{\chi}$ & $P_8$& $\frac{1}{2}$ & $\frac{3}{2}$ & $\frac{5}{2}$ & $\frac{7}{2}$ & $\frac{9}{2}$ & $\frac{11}{2}$ & $\frac{13}{2}$ & $\frac{15}{2}$ & $\frac{17}{2}$ & $\frac{19}{2}$ & $\frac{21}{2}$ & $\frac{23}{2}$ & $\frac{25}{2}$ & $\frac{27}{2}$ & $\frac{29}{2}$ \\ \hline
		1 & Pole Order & 4 & 3 & 3 & 3 & 3 & 3 & 2 & 2 & 2 & 2 & 1 & 1 & 1 & 1 & 1 \\ \hline
		2 & Pole Order & 1 & 0 & 0 & 0 & 0 & 0 & 0 & 0 & 0 & 0 & 0 & 0 & 0 & 0 & 0 \\ \hline
	\end{longtable}

	For all cases with trivial $\chi$, the proof continues as in the case of $E_6$ and $E_7$, with the exception of the cases $\DPS{P_1}\bk{\frac{5}{2}}$ and $\DPS{P_7}\bk{\frac{3}{2}}$ which are not generated by its spherical vector, these cases will be dealt with bellow.
	
	We now deal with the cases with non-trivial $\chi$:
	\begin{itemize}
		\item In the following cases, there exists a singleton equivalence class which realize a pole of the same order as prescribed in the theorem and the above tables:
		$\bk{1,7/2,2}$, $\bk{2,7/2,2}$, $\bk{2,3/2,3}$, $\bk{3,7/2,2}$, $\bk{3,3/2,3}$, $\bk{3,1,4}$, $\bk{4,5/2,2}$, $\bk{4,3/2,3}$, $\bk{4,1,4}$, $\bk{4,1/2,5}$, $\bk{4,1/2,6}$, $\bk{5,5/2,2}$, $\bk{5,3/2,3}$, $\bk{5,1,4}$, $\bk{5,1/2,5}$, $\bk{6,3,2}$, $\bk{6,2,3}$, $\bk{6,1/2,4}$, $\bk{7,9/2,2}$, $\bk{7,1/2,3}$ and $\bk{8,1/2,2}$.
		.
		
		\item In the following cases, the pole in these tables cancels completely and the Eisenstein series is holomorphic at the point, sorted by the value of $i$:
		\begin{enumerate}
			\item $\bk{1,1,2}$, $\bk{1,3/2,2}$, $\bk{1,2,2}$, $\bk{1,5/2,2}$ and $\bk{1,3,2}$.
			\item $\bk{2,1,2}$, $\bk{2,3,2}$, $\bk{2,2,2}$, $\bk{2,1/6,3}$, $\bk{2,1/2,3}$, $\bk{2,5/6,3}$ and $\bk{2,7/6,3}$.
			\item $\bk{3,3,2}$, $\bk{3,3/4,2}$, $\bk{3,1/4,2}$, $\bk{3,3,2}$, $\bk{3,1/6,3}$, $\bk{3,1/2,3}$ and $\bk{3,5/6,3}$.
			\item $\bk{4,1/6,2}$, $\bk{4,1/4,2}$, $\bk{4,1/3,2}$, $\bk{4,1/6,3}$, $\bk{4,1/3,3}$, $\bk{4,1/4,4}$, $\bk{4,1/10,5}$, $\bk{4,1/6,6}$ and $\bk{4,1/3,6}$.
			\item $\bk{5,1/4,2}$, $\bk{5,3/4,2}$, $\bk{5,1/6,3}$, $\bk{5,1/4,4}$, $\bk{5,3/4,4}$, $\bk{5,1/10,5}$ and $\bk{5,3/10,5}$.
			\item $\bk{6,1/4,2}$, $\bk{6,3/2,2}$, $\bk{6,4/3,3}$, $\bk{6,1/3,3}$, $\bk{6,2/3,3}$, $\bk{6,5/3,3}$ and $\bk{6,1/4,4}$.
			\item $\bk{7,1,2}$, $\bk{7,3/2,2}$, $\bk{7,2,2}$, $\bk{7,3,2}$, $\bk{7,7/2,2}$, $\bk{7,4,2}$ and $\bk{7,1/6,3}$.
		\end{enumerate}
		
%
		
		These are precisely the cases $[i,s,ord\bk{\chi}]$ in the above tables where $\DPS{P_i}\bk{s,\chi}_v$ is irreducible at infinitely many $v\in\Places$.
		
		\item In the following cases, the pole reduces to a simple pole and a simple pole is realized by a singleton equivalence class, sorted by the value of $i$:
		\begin{enumerate}
			\item $\bk{1,1/2,2}$.
			\item $\bk{2,1/2,2}$, $\bk{2,3/2,2}$, $\bk{2,5/2,2}$ and $\bk{2,7/2,2}$.
			\item $\bk{3,1/2,2}$, $\bk{3,1,2}$, $\bk{3,3/2,2}$, $\bk{3,2,2}$, $\bk{3,5/2,2}$ and $\bk{3,7/6,3}$.
			\item $\bk{4,1/2,2}$, $\bk{4,1,2}$, $\bk{4,1/2,4}$, $\bk{4,3/4,2}$, $\bk{4,3/2,2}$, $\bk{4,1,2}$, $\bk{4,2,2}$, $\bk{4,1/2,3}$, $\bk{4,5/6,3}$, $\bk{4,7/6,3}$, $\bk{4,3/4,4}$ and $\bk{4,3/10,5}$.
			\item $\bk{5,1/2,2}$, $\bk{5,1,2}$, $\bk{5,3/2,2}$, $\bk{5,2,2}$, $\bk{5,3/2,3}$, $\bk{5,5/6,3}$, $\bk{5,7/6,3}$ and $\bk{5,1/2,4}$.
			\item $\bk{6,1/2,2}$, $\bk{6,1,2}$, $\bk{6,2,2}$, $\bk{6,5/2,2}$ and $\bk{6,1,3}$.
			\item $\bk{7,1/2,2}$ and $\bk{7,5/2,2}$
		\end{enumerate}
		
%
%
		
		We point out that the case of $\bk{3,1/2,2}$ follows in a similar way to the case $\bk{4,1/2,3}$ for the group $E_6$.

	\end{itemize}
	
\end{proof}

It should be noted that the results of this section settles a conjecture made by David Ginzburg and Joseph Hundley in \cite{MR3359720}.

\begin{Conj}[\cite{MR3359720}, Conjecture 2]
	For $\Real\bk{s}>0$, the possible poles of the Eisenstein series $\Eisen_{P_2}^{E_8}\bk{s,\chi}$ are as follows.
	\begin{itemize}
		\item If $\chi$ is trivial, then the Eisenstein series $\Eisen_{P_2}\bk{s,\chi}$ can have a double pole at the points $\frac{3}{2}$, $\frac{5}{2}$ and $\frac{7}{2}$.
		At the points $\frac{1}{2}$, $\frac{9}{2}$, $\frac{11}{2}$, $\frac{13}{2}$ and $\frac{17}{2}$ it can have a simple pole.
		\item If $\chi$ is nontrivial quadratic, then the Eisenstein series can have simple poles at
		$\frac{1}{2}$, $\frac{3}{2}$, $\frac{5}{2}$ and $\frac{7}{2}$.
		\item If $\chi$ is nontrivial cubic, then the Eisenstein series can have a simple pole at $\frac{3}{2}$.
		\item If the order of $\chi$ exceeds 3 then the Eisenstein series is holomorphic in $\Real\bk{s}>0$.
	\end{itemize}
\end{Conj}

\begin{Rem}
	Note that the normalization of the variable $s$ is slightly different in \cite{MR3359720}.
	If $t_0$ is a pole in the notations of \cite{MR3359720}, then the corresponding value $s_0$ in our notations is
	\[
	s_0= 17\times\bk{t_0-\frac{1}{2}} .
	\]
\end{Rem}
\chapter{Square-Integrable Degenerate Residual Spectrum}
\label{Chap:Square_Integrable}

In this section we study the square-integrable residues of degenerate Eisenstein series.
That is, if $\Eisen_{P_i}\bk{\chi,s}$ admits a pole of order $k$ at $s=s_0$, we describe
\[
\ResRep{G}{i}{s_0}{\chi} = \operatorname{Span}_{\C} \set{\lim\limits_{s\to s_0} \bk{s-s_0}^k \Eisen_{P_i}\bk{f,\lambda_s^{P_i},\mu_\chi,g} \mvert f_s\in \DPS{P_i}\bk{\chi,s}} .
\]

\section{The Method}

We first note that the square-integrable spectrum is semi-simple.
Thus, if $\ResRep{G}{i}{s_0}{\chi}$ is square integrable, then it can be written as a unitary direct sum
\begin{equation}
	\ResRep{G}{i}{s_0}{\chi} = \widehat{\oplus} \sigma_i,
\end{equation}
where $\sigma_i = \otimes_{v\in\Places} \sigma_{i,v}$ and $\sigma_{i,v}$ is an irreducible quotient of $\DPS{P_i}\bk{s_0,\chi}_v$.

\begin{Cor}
	If $\DPS{P_i}\bk{s_0,\chi}_v$ admits a unique irreducible quotient $\pi_{1,v}$ for all $v\in\Places$ and $\ResRep{G}{i}{s_0}{\chi}$ is square-integrable, then
	\[
	\ResRep{G}{i}{s_0}{\chi} = \pi_1 = \otimes_{v\in \Places} \pi_{1,v} .
	\]
\end{Cor}

Thus, it remains to deal with those square-integrable residues $\ResRep{G}{i}{s_0}{\chi}$ where $\DPS{P_i}\bk{s_0,\chi}_v$ admits a reducible maximal semi-simple quotient for some places $v\in\Places$.
We give a rough outline of the method which will be used in the various cases, while for each case we will detail the calculation bellow.

Note that, in this case, we have
\[
\ResRep{G}{i}{s_0}{\chi} \subseteq \otimes_{v\in\Places} \bk{\suml_{\gamma\bk{v}\in S_v} \pi_{\gamma\bk{v},v}},
\]
where $\suml_{\gamma\in S_v} \pi_{\gamma,v}$ is the maximal semi-simple quotient of $\DPS{P_i}\bk{\chi,s}$ and $S_v$ is a set parametrizing its irreducible quotients as will be detailed bellow.
Thus, in order to determine $\ResRep{G}{i}{s_0}{\chi}$, we wish to determine for which sequence $\bk{\gamma\bk{v}\in S_v}_{v\in\Places}$, the irreducible quotient $\pi_\gamma=\otimes_{v\in \Places} \pi_{\gamma\bk{v},v}$ appears in $\ResRep{G}{i}{s_0}{\chi}$.
In order to do this, we assume that $\Eisen_{P_i}\bk{s,\chi}$
admits of pole of order $k$ at $s=s_0$ and consider the equivalence classes within $\equivclassespole{s_0,\chi,k}$.

For $\coset{w'}\in\equivclassespole{s_0,\chi,k}$, we assume that $w'$ is the shortest element in $\coset{w'}$ and write $\varsigma = \set{u \mvert uw'\in \coset{w'}}$.
Usually, it happens that (for $\coset{w'}$ that contributes to the residue) $\varsigma$ is a finite group isomorphic to $\bk{\Z\rmod2\Z}^l$. Namely, there is a set of commuting elements of order $2$, $u_i\in\varsigma$, such that
\[
\varsigma = \prodl_{i=1}^l \set{1,u_i} .
\]
Thus, one is able to write
\[
\kappa_{\coset{w'}}\bk{f,s} = \coset{\prodl_{i=1}^k \bk{\Id+M_{u_i}\bk{w'\cdot\lambda_s^{P_i}}}} M_{w'}\bk{\lambda_s^{P_i}}\bk{f} .
\]

In this case, we wish to perform the following two calculations:
\begin{enumerate}
	\item To determine the image of $\lim_{s\to s_0} M_{w'}\bk{\lambda_s^{P_i}}$. Usually, this turns out to equal the maximal semi-simple quotient $\otimes_{v\in\Places} \bk{\suml_{\gamma\bk{v}\in S_v} \pi_{\gamma\bk{v},v}}$ of $\DPS{P_i}\bk{\chi,s}$.
	
	\item To determine the eigenvalues of the intertwining operators $M_{u_i}\bk{w'\cdot\lambda_s^{P_i}}$ on the various summands $\pi_\gamma=\otimes_{v\in \Places} \pi_{\gamma\bk{v},v}$, for a sequence $\bk{\gamma\bk{v}\in S_v}_{v\in\Places}$.
	This is done by an analysis of a local calculation of the eigenvalues and eigenspaces of  $N_{u_i,v}\bk{w'\cdot\lambda_s^{P_i}}$ for $v\in\Places_{fin.}$ and a calculation of the value of
	\[
	\lim_{s\to s_0} C_{u_i}\bk{w'\cdot \lambda_{s,\chi}^{P_i}}.
	\]
	We point out that we set up our notations of elements in $S_v$ so that they indicate the eigenvalues of the $N_{u_i,v}\bk{w'\cdot\lambda_s^{P_i}}$ on $\pi_{\gamma\bk{v},v}$.
\end{enumerate}

\section{Results for $E_6$}

First, we state the results for square-integrable residues when $\DPS{P_i}\bk{s_0,\chi}$ admits a unique irreducible quotient:
\begin{Prop}
	For the following triples $\bk{i,s_0,ord\bk{\chi}}$, $\ResRep{E_6}{i}{s_0}{\chi}$ is isomorphic to the unique irreducible quotient $\pi_1$ of $\DPS{P_i}\bk{s_0,\chi}$:
	\[
	\begin{split}
	& \bk{1,3,1}, \bk{1,6,1}, \\
	& \bk{2,\frac12,1}, \bk{2,\frac12,2}, \bk{2,\frac72,1}, \bk{2,\frac{11}{2},1} \\
	& \bk{3,\frac32,1}, \bk{3,\frac32,2}, \bk{3,\frac92,1}, \\
	& \bk{4,\frac32,1}, \bk{4,\frac32,2}, \bk{4,\frac52,1}, \bk{4,\frac72,1} \\
	& \bk{5,\frac32,1}, \bk{5,\frac32,2}, \bk{5,\frac92,1}, \\
	& \bk{6,3,1}, \bk{6,6,1} .
	\end{split}
	\]
\end{Prop}

We now turn to the only case, for $E_6$, where the maximal semi-simple quotient of $\DPS{P_i}\bk{s_0,\chi}$ is not irreducible.
Namely, we consider the case $\bk{4,\frac12,3}$.
We fix a Hekce character $\chi=\otimes_{v\in \Places}\chi_v$ of order $3$.
We recall that:
\begin{itemize}
	\item For a place $v$ such that $\chi_v=\Id$, $\DPS{P_4}\bk{\frac12,\chi}_v$ admits a unique irreducible quotient $\pi_{1,v}$.
	\item For a place $v$ such that $\chi_v\neq\Id$, $\DPS{P_4}\bk{\frac12,\chi}_v$ admits a maximal semi-simple quotient of length $3$, which will be written (in the notations of \Cref{Subsec:Local_Lengths_E6}) by $\oplus_{\kappa\in\mu_3} \pi_{\kappa,v}$, where $\mu_3$ denotes the group of cube roots of unity.
	We point out that the subscripts in $\pi_{\kappa,v}$ can be fixed so that
	\[
	N_{w_{3165},v}\bk{\lambda_{v, a.d.}} f = \kappa f \quad \forall f\in\pi_{\kappa,v} ,
	\]
	where
	\[
	\lambda_{a.d.} = \esixchar{\chi}{\chi}{\chi}{-1+\chi}{\chi}{\chi}
	\]
	is anti-dominant exponent of $\DPS{P_4}\bk{\frac12,\chi}_v$.
\end{itemize}

That is, here
\[
S_v = \piece{1,& \chi_v=\Id \\ \mu_3,& \chi_v\neq\Id .}
\]

\begin{Prop}
	\label{Prop:E6_4_1-2_3_Residue}
	For a Hecke character $\chi$ of order $3$ and a fixed $\kappa\in\widehat{\mu_3}\setminus\set{\Id}$, it holds that:
	\[
	\ResRep{G}{4}{\frac{1}{2}}{\chi} = \bigoplus_{\stackrel{S=S_{\kappa}\cupdot S_{\kappa^2}\subset\Places_\chi}{\Card{S_\kappa}\equiv \Card{S_{\kappa^2}}\pmod{3} }} \bk{\bigotimes_{v\notin S} \pi_{1,v}\bigotimes_{v\in S_\kappa} \pi_{\kappa,v} \bigotimes_{v\in S_{\kappa^2}} \pi_{\kappa^2,v}} ,
	\]
	where
	\[
	\Places_\chi = \set{v\in\Places \mvert \chi_v\neq \Id} .
	\]
\end{Prop}

\begin{proof}
	Let
	\[
	\lambda_{a.d.} = \esixchar{\chi}{\chi}{\chi}{-1+\chi}{\chi}{\chi} \quad
	\lambda_0 = \esixchar{-1}{-1}{-1}{3+\chi}{-1}{-1}
	\]
	and let $w_0\in W^{M_4,T}$ be the shortest element such that $\lambda_{a.d.}=w_0\cdot\lambda_0$.
	It holds that:
	\begin{itemize}
		\item If $M_w\bk{s,\chi}$ admits a pole at $s_0=\frac{1}{2}$, then this pole is simple.
		
		\item Let $\coset{w'}$ be an equivalence classes modulo $\sim_{s_0,\chi}$ such that $M_w\bk{s,\chi}$, with $w\in\coset{w'}$, admits a pole at $s_0=\frac{1}{2}$.
		Let $w$ denote the shortest element in $\coset{w'}$.
		Then
		\[
		\coset{w'} = \set{\overline{w'} u w_0 \mvert u\in \Stab_W\bk{\lambda_{a.d.}}},
		\]
		where $\overline{w'}=w w_0^{-1}$.
		
		\item $M_{\overline{w'}}\bk{\lambda}$ is non-singular at $\lambda_{a.d.}$ (both in the analytic and algebraic sense).
		
		\item For any $u\in \Stab_W\bk{\lambda_{a.d.}}$, the normalizing factor $C_u\bk{\lambda}$ is holomorphic at $\lambda_{a.d.}$ and satisfies
		\[
		C_u\bk{\lambda_{a.d.}} = 1 .
		\]
		
		\item For any $v\in\Places_{fin.}$ it holds that 
		\[
		N_{w_0}\bk{\DPS{P_4}\bk{\chi,\frac{1}{2}}_v} = \piece{\pi_{1,v}, & \chi_v=\Id \\ \oplus_{\epsilon\in\mu_3} \pi_{\epsilon,v} , & \chi_v\neq \Id .}
		\]

	\end{itemize}

	It then follows that
	\[
	\begin{split}
	\ResRep{G}{4}{\frac{1}{2}}{\chi}_N 
	& = \Span_\C\set{ \Res{s=\frac{1}{2}}\suml_{\coset{w'}\in \equivclassespole{s_0,\chi,1}} \kappa_{\coset{w'}}\bk{f,g} \mvert f\in \DPS{P_4}\bk{s,\chi}^0 } \\
	& = \Span_\C\set{\suml_{\coset{w'}\in \equivclassespole{s_0,\chi,1}} M_{\overline{w'}} \circ \bk{ \suml_{u\in \Stab_W\bk{\lambda_{a.d.}} } M_u } \circ \Res{s=\frac{1}{2}}\bk{M_{w_0}} \bk{f} \mvert f\in \DPS{P_4}\bk{s,\chi}^0 } .
	\end{split}
	\]

	Hence, in order to determine $\ResRep{G}{4}{\frac{1}{2}}{\chi}$, it is enough to determine the kernel of $\suml_{u\in \Stab_W\bk{\lambda_{a.d.}} } M_u$ when restricted to the image of $\Res{s=\frac{1}{2}}\bk{M_{w_0}}$.
	This image is given by
	\[
	\Res{s=\frac{1}{2}}\bk{M_{w_0}} \bk{\DPS{P_4}\bk{\chi,\frac{1}{2}}} = \bigoplus_{S=S_{\kappa}\cupdot S_{\kappa^2}\subset\Places_\chi} \bk{\bigotimes_{v\notin S} \pi_{1,v}\bigotimes_{v\in S_\kappa} \pi_{\kappa,v} \bigotimes_{v\in S_{\kappa^2}} \pi_{\kappa^2,v}} .
	\]
	Hence, for a pure tensor $\phi=\otimes\phi_v$ with $\phi_v\in\pi_{\epsilon,v}$ for $v\in S_\epsilon$, it holds that
	\[
	\suml_{u\in \Stab_W\bk{\lambda_{a.d.}} } M_u  \phi = \bk{1+\kappa^{\Card{S_\kappa} + 2\Card{S_{\kappa^2}} }  +\kappa^{2\Card{S_\kappa} + \Card{S_{\kappa^2}} } } \phi .
	\]
	Let $\tau=\kappa^{\Card{S_\kappa} + 2\Card{S_{\kappa^2}} }$ and note that
	\[
	1+\kappa^{\Card{S_\kappa} + 2\Card{S_{\kappa^2}} }  +\kappa^{2\Card{S_\kappa} + \Card{S_{\kappa^2}} } 
	 = \piece{\frac{\tau^3-1}{\tau-1},& \tau\neq 1 \\ 3,& \tau=1}  = \piece{0,& \Card{S_\kappa} \not\equiv \Card{S_{\kappa^2}} \pmod{3}  \\ 3, & \Card{S_\kappa} \equiv \Card{S_{\kappa^2}} \pmod{3} .}
	\]
	In particular,
	\[
	\suml_{u\in \Stab_W\bk{\lambda_{a.d.}} } M_u  \phi \neq 0 \quad \Longleftrightarrow \quad \Card{S_\kappa}\equiv \Card{S_{\kappa^2}}\pmod{3} 
	\]
	and if $\Card{S_\kappa}\equiv \Card{S_{\kappa^2}}\pmod{3}$, then
	\[
	\suml_{u\in \Stab_W\bk{\lambda_{a.d.}} } M_u
	\]
	is a non-zero scalar on the image of $N_{w_0}$.
	This completes the proof.
	
\end{proof}

\section{Results for $E_7$}

First, we state the results for square-integrable residues when $\DPS{P_i}\bk{s_0,\chi}$ admits a unique irreducible quotient:
\begin{Prop}
	For the following triples $\bk{i,s_0,ord\bk{\chi}}$, $\ResRep{E_7}{i}{s_0}{\chi}$ is isomorphic to the unique irreducible quotient $\pi_1$ of $\DPS{P_i}\bk{s_0,\chi}$:
	\[
	\begin{split}
		& \bk{1,\frac12,1}, \bk{1,\frac12,2}, \bk{1,\frac72,1}, \bk{\frac{17}{2},1} \\
		& \bk{2,5,1}, \bk{2,7,1} \\
		& \bk{3,\frac12,1}, \bk{3,\frac12,2}, \bk{3,\frac12,4}, \bk{3,\frac32,1}, \bk{3,\frac32,2}, \bk{3,\frac52,1}, \bk{3,\frac52,2}, \bk{3,\frac{11}{2},1} \\
		& \bk{4,1,1}, \bk{4,1,2}, \bk{4,1,3}, \bk{4,2,1}, \bk{4,2,2}, \bk{4,3,1}, \bk{4,4,1} \\
		& \bk{5,1,1}, \bk{5,1,3}, \bk{5,3,1}, \bk{5,5,1} \\
		& \bk{6,\frac12,1}, \bk{6,\frac12,2}, \bk{6,\frac52,1}, \bk{6,\frac52,2}, \bk{6,\frac72,1}, \bk{6,\frac{13}{2},1} \\
		& \bk{7,1,1}, \bk{7,5,1}, \bk{7,9,1} .
	\end{split}
	\]
\end{Prop}

We now turn to cases where the maximal semi-simple quotient of $\DPS{P_i}\bk{s_0,\chi}$ is not irreducible.
These cases are:
\[
\bk{2,1,1}, \bk{2,2,2}, \bk{4,\frac12,4}, \bk{5,2,2}.
\]

Both cases $\bk{2,2,2}$ and $\bk{5,2,2}$ have the same maximal semi-simple quotient:
\begin{itemize}
	\item At a place $v\in\Places_{fin.}$ such that $\chi_v=\Id$, the representation $\DPS{P_i}\bk{2,\chi}_v$ admits a unique irreducible quotient $\pi_{1,v}$.
	\item At a place $v\in\Places_{fin.}$ such that $\chi_v \neq \Id$, the representation $\DPS{P_i}\bk{2,\chi}_v$ admits a maximal semi-simple quotient of length $2$, say $\pi_{-1,v} \otimes_{v\notin S}\pi_{1,v}$ (in the notations of \Cref{Subsec:Local_Lengths_E7}).
\end{itemize}

\begin{Prop}
	For a Hecke character $\chi$ of order $2$, it holds that:
	\[
	\ResRep{G}{2}{2}{\chi} = \ResRep{G}{5}{2}{\chi} = \bigoplus_{\stackrel{S\subset\Places_\chi}{\Card{S}\equiv 0 \pmod{2} }} \otimes_{v\in S}\pi_{-1,v} \otimes_{v\notin S}\pi_{1,v} ,
	\]
	where
	\[
	\Places_\chi = \set{v\in\Places \mvert \chi_v\neq \Id} .
	\]
\end{Prop}

\begin{proof}
	In both cases, an anti-dominant exponent of $\DPS{P_i}\bk{\chi,2}$ is given by
	\[
	\lambda_{a.d.} = \esevenchar{-1}{\chi}{\chi}{-1}{\chi}{-1}{\chi} .
	\]
	It holds that
	\[
	\Stab_W\bk{\lambda_{a.d.}} = \set{1,s_2s_5s_7} .
	\]
	
	Let $\lambda_0=\lambda_{\chi,2}^{P_i}$ denote the initial exponent of $\DPS{P_i}\bk{\chi,2}$ and let $w_0\in W^{M,T}$ be the shortest element which satisfy $\lambda_{a.d.}=w_0\cdot\lambda_0$.
	That is
	\[
	\lambda_0 = \piece{
	\esevenchar{-1}{4+\chi}{-1}{-1}{-1}{-1}{-1},& i=2 \\
	\esevenchar{-1}{-1}{-1}{-1}{2+\chi}{-1}{-1},& i=5
	}
	\]
	\[
	w_0=\piece{w[2, 4, 3, 1, 5, 4, 2, 3, 4, 5, 6, 5, 4, 2, 3, 1, 4, 3, 5, 7, 6, 5, 4, 2, 3, 1, 4, 3, 5, 4, 6, 7],& i=2 \\
	\begin{split}
			w[& 5, 4, 2, 3, 1, 4, 3, 5, 4, 6, 5, 4, 2, 3, 1, 4, 3, 5, 4, 2, 6, 5, \\
			 & 7, 6, 5, 4, 2, 3, 1, 4, 3, 5, 4, 2, 6, 5, 4, 3, 1, 7, 6, 5, 4, 7]
	\end{split},& i=5	
	}
	\]
	
	One checks that:
	\begin{itemize}
		\item If $M_w\bk{s,\chi}$ admits a pole at $s_0=2$, then this pole is simple.
		
		\item Let $\coset{w'}$ be an equivalence class modulo $\sim_{s_0,\chi}$ such that $M_w\bk{s,\chi}$, with $w\in\coset{w'}$, admits a pole at $s_0=2$.
		Let $w$ be the shortest element in $\coset{w'}$.
		Then
		\[
		\coset{w'} = \set{\overline{w'} u w_0 \mvert u\in \Stab_W\bk{\lambda_{a.d.}}},
		\]
		where $\overline{w'}=w w_0^{-1}$.
		
		\item $M_{\overline{w'}}\bk{\lambda}$ is non-singular at $\lambda_{a.d.}$ (both in the analytic and algebraic sense).
		
		\item The normalizing factor $C_{s_2s_5s_7}\bk{\lambda}$ is holomorphic at $\lambda_{a.d.}$ and satisfies
		\[
		\begin{split}
		C_{s_2s_5s_7}\bk{\lambda_{a.d.}} 
		& = \lim_{z\to 2} \frac{\Lfun\bk{z-1,\chi}}{\Lfun\bk{z-2,\chi}} \\
		& = \lim_{z\to 2} \frac{\Lfun\bk{2-z,\chi}}{\Lfun\bk{z-2,\chi}} \\
		& = \lim_{z\to 2} \frac{\epsilon_F\bk{2-z,\chi} \Lfun\bk{z-2,\overline{\chi}}}{\Lfun\bk{2-z,\chi}} \\
		& = \lim_{z\to 2} \frac{\Lfun\bk{z-2,\chi}}{\Lfun\bk{z-2,\chi}} = 1 ,
		\end{split}
		\]
		since the global $\epsilon$-factor of a quadratic character is trivial (see \cite[page 37]{MR1421575}) and $\Lfun\bk{s,\chi}$ is holomorphic at $0$.
		
		\item It holds that
		\[
		N_{w_0,v}\bk{\chi_v,2}\bk{\DPS{P_i}\bk{\chi_v,2}} = \piece{
		\pi_{1,v},& v\notin\Places_\chi \\
		\pi_{-1,v}\oplus\pi_{1,v},& v\in\Places_\chi} .
		\]
		
		\item The subscript $\pm1$ in the decomposition of the maximal semi-simple quotient of $\DPS{P_i}\bk{2,\Id}_v$, so that for any $v\in\Places_{fin.}$, it holds that
		\[
		N_{s_2s_5s_7,v}\bk{\lambda_{a.d.}}\varphi = \epsilon\varphi \quad 
		\forall \varphi\in\pi_{\epsilon,v} .
		\]

	\end{itemize}

	Let
	\[
	\pi_S = \otimes_{v\in S}\pi_{-1,v} \otimes_{v\notin S}\pi_{1,v} .
	\]
	By a similar calculation to that performed in \Cref{Prop:E6_4_1-2_3_Residue}, we conclude that
	\[
	\bk{I+M_{s_2s_5s_7}}\varphi = \bk{1+\bk{-1}^{\Card{S}}} \varphi \quad
	\forall \varphi=\otimes_{v\in \Places} \varphi_v \in \pi_S.
	\]
	Hence
	\[
	\suml_{u\in \Stab_W\bk{\lambda_{a.d.}} } M_u  \varphi \neq 0 \quad \Longleftrightarrow \quad \Card{S}\equiv 0 \pmod{2} 
	\]
	and if $\Card{S}\equiv 0 \pmod{2}$, then
	\[
	\suml_{u\in \Stab_W\bk{\lambda_{a.d.}} } M_u
	\]
	is non-zero a scalar on the image of $N_{w_0}$.
	This completes the proof.
	
\end{proof}

In the case $\bk{4,\frac12,4}$, the local maximal semi-simple quotient is given as follows:
\begin{itemize}
	\item At a place $v\in\Places_{fin.}$ such that $\chi_v^2=\Id$, the representation $\DPS{P_4}\bk{\frac12,\chi}_v$ admits a unique irreducible quotient $\pi_{1,v}$.
	\item At a place $v\in\Places_{fin.}$ such that $\chi_v^2 \neq \Id$, the representation $\DPS{P_4}\bk{\frac12,\chi}_v$ admits a maximal semi-simple quotient of length $2$, say $\pi_{-1,v} \otimes_{v\notin S}\pi_{1,v}$ (in the notations of \Cref{Subsec:Local_Lengths_E7}).
\end{itemize}

\begin{Prop}
	For a Hecke character $\chi$ of order $4$, it holds that:
	\[
	\ResRep{G}{4}{1/2}{\chi} = \bigoplus_{\stackrel{S\subset\Places}{\Card{S}\equiv 0 \pmod{2} }} \bk{\otimes_{v\in S}\pi_{-1,v} \otimes_{v\notin S}\pi_{1,v}} ,
	\]
	where
	\[
	\Places_\chi = \set{v\in\Places \mvert \chi_v^2\neq \Id}.
	\]
\end{Prop}

\begin{proof}
	The proof in this case is essentially identical (mutatis mutandis) to the previous one, with the following changes:
	\begin{itemize}
		\item An anti-dominant exponent of $\DPS{P_i}\bk{\chi,2}$ is given by
		\[
		\lambda_{a.d.} = 
		\esevenchar{-\frac{1}{2}+3\chi}{2\chi}{2\chi}{-\frac{1}{2}+\chi}{2\chi}{-\frac{1}{2}+3\chi}{2\chi} .
		\]
		\item $\Stab_W\bk{\lambda_{a.d.}} = \set{1,s_2s_5s_7}$.
		\item $\lambda_0 = \esevenchar{-1}{-1}{-1}{\frac{5}{2}+\chi}{-1}{-1}{-1}$.
		\item The shortest element $w_0\in W^{M_4,T}$ such that $\lambda_{a.d.} = w_0\cdot \lambda_0$ is given by
		\[
		\begin{split}
			w_0=w[&4, 2, 3, 1, 4, 3, 5, 4, 2, 3, 1, 4, 3, 5, 6, 5, 4, 2, \\
			& 3, 1, 4, 3, 5, 6, 7, 6, 5, 4, 2, 3, 1, 4, 3, 5, 4, 6, 7]
		\end{split}
		\]
	\end{itemize}
\end{proof}

\begin{Prop}
	\label{Prop:Residue_E7_2_1_1}
	It holds that
	\[
	\ResRep{G}{2}{1}{1} = \bigoplus_{\stackrel{S\subset\Places_\chi}{\Card{S}\equiv 0 \pmod{2} }} \pi_S ,
	\]
	where (in the notations of \Cref{Subsec:Local_Lengths_E7}) 
	\[
	\pi_S = \otimes_{v\in S}\pi_{-1,v} \otimes_{v\notin S}\pi_{1,v} .
	\]
\end{Prop}

\begin{proof}
	We first recall that $\Eisen_{P_2} \bk{f,s,\chi,g}$	admits a simple pole at this point and that the maximal semi-simple quotient of $\DPS{P_2}\bk{1}_v$, for a non-Archimedean place $v$, is of length $2$, say $\pi_{1,v}\oplus\pi_{-1,v}$.
	We note that there are two types of equivalence classes in $\equivclassespole{1,\Id,1}$:
	\begin{itemize}
		\item Equivalence classes $\coset{w'}$ such that
		\[
		N_{w',v}\bk{\lambda_{1}^{P}}\bk{\DPS{P_2}\bk{1}_v} = \pi_{1,v}
		\]
		for any non-Archimedean place $v$.
		Since there are such equivalence classes which are singletons (say $w\coset{5, 4, 2, 3, 1, 4, 3, 5, 4, 2}$, for example) then $\pi_1=\otimes_{v\notin \Places}\pi_{1,v}\subseteq \ResRep{E_7}{2}{1}{1}$
		\item Equivalence classes $\coset{w'}$ such that
		
		\begin{equation}
			\label{Eq:E7_2_1_1_eq_classes_with_pi_-1_in_image}
			N_{w',v}\bk{\lambda_{1}^{P}}\bk{\DPS{P_2}\bk{1}_v} = \pi_{1,v} \oplus \pi_{-1,v}
		\end{equation}
		
		There exactly $15$ such equivalence classes (corresponding to each exponent in the Jacquet functor $r_T^G\pi_{-1,v}$).
		The calculation of $\ResRep{G}{2}{1}{1}$ relies on the following analysis of the local intertwining operators $N_w$, for $w\in\coset{w'}$, and the term $\kappa_{\coset{w'}}\bk{f,s,g}$ in the constant term \Cref{eq::CT_second_form}, for such equivalence classes.
	\end{itemize}
	Fix a non-Archimedean place $v$.
	Let
	\[
	\lambda_{1}^{P_2} = \esevenchar{-1}{7}{-1}{-1}{-1}{-1}{-1}
	\]
	be the initial exponent of $\DPS{P_2}\bk{1}_v$ and let
	\[
	\lambda_{-1}^{P_2} = \esevenchar{-1}{5}{-1}{-1}{-1}{-1}{-1}
	\]
	be the initial exponent of $\DPS{P_2}\bk{-1}_v$.
	Also, we let
	\[
	\mu=\esevenchar{-2}{-1}{-1}{-1}{3}{-1}{-1}
	\]
	denote a third exponent of interest.
	
	We also fix the following Weyl elements:
	\begin{itemize}
		\item $w_1=w\coset{5, 4, 3, 1, 6, 5, 4, 3, 7, 6, 5, 4, 2}$. We note that $w_1\cdot \lambda_{1}^{P_2} = \lambda_{-1}^{P_2}$
		\item $u_0=w\coset{1,3,4,2}$. We note that $u_0\cdot \lambda_{-1}^{P_2} = \mu$
		\item $w_0=\coset{5, 4, 3, 2, 4, 5, 6, 7, 5, 4, 3, 2, 4, 5, 6, 5, 4, 3, 2, 4, 5}$. We note that $w_0\cdot\mu=\mu$.
	\end{itemize}

	Any equivalence class $\set{w',w''}\in \equivclassespole{1,\Id,1}$, with $len\bk{w'}<len\bk{w''}$, such that 
	\[
	N_{w',v}\bk{\lambda_{1}^{P_2}}\bk{\DPS{P_2}\bk{1}_v} = \pi_{1,v} \oplus \pi_{-1,v}
	\]
	has the following form:
	\[
	w'=uw_1,\quad w'' =  w_2 u w_1 ,
	\]
	where $w_2$ can be written as follows:
	\[
	w_2 = u\bk{u_0^{-1}w_0u_0}u^{-1}
	\]
	such that, in particular, $N_{u^{-1}}\bk{\mu}$ and $N_{u}\bk{u^{-1}\cdot\mu}$ are isomorphisms.
	It is thus enough to analyze $\kappa_{\coset{w_1}}$, namely in the case
	\[
	w'=w_1,\quad w'' = u_0^{-1} w_0 u_0 w_1 .
	\]
	In \Cref{Appendix:Local_data_for_global_SI_resisudes} we prove that
	\[
	N_{u_0^{-1} w_0 u_0}\bk{\lambda_{-1}^{P_2}} f_\epsilon \quad \forall f_\epsilon\in\pi_\epsilon,\ \epsilon=\pm 1 .
	\]
	
	Following the same line of argument as in previous proofs, one checks that for any coset $\coset{w'}=\set{w,w'}$, satisfying \Cref{Eq:E7_2_1_1_eq_classes_with_pi_-1_in_image}, it holds that
	\[
	\frac{C_w\bk{\lambda_1}}{C_{w'}\bk{\lambda_1}} = 1 .
	\]
	Hence, for $f=\otimes f_v\in \pi_S$, with $S\subset\Places$, it holds that
	\[
	\kappa_{\coset{w'}}\bk{f,s,g} = \coset{1+\bk{-1}^{\Card{S}}} f ,
	\]
	which is non-vanishing if and only if $\Card{S}\equiv 1 \pmod{2}$.
	From which the claim follows.
\end{proof}

\section{Results for $E_8$}

First, we state the results for square-integrable residues when $\DPS{P_i}\bk{s_0,\chi}$ admits a unique irreducible quotient:
\begin{Prop}
	\label{Prop:SI_residue_UIQ_E8}
	For the following triples $\bk{i,s_0,ord\bk{\chi}}$, $\ResRep{E_8}{i}{s_0}{\chi}$ is isomorphic to the unique irreducible quotient $\pi_1$ of $\DPS{P_i}\bk{s_0,\chi}$:
	\[
	\begin{split}
		& \bk{1,\frac12,1}, \bk{1,\frac12,2}, \bk{1,\frac72,1}, \bk{1,\frac72,2}, \bk{1,\frac{13}{2},1}, \bk{1,\frac{17}{2},1}, \bk{1,\frac{23}{2},1} \\
		& \bk{2,\frac12,2}, \bk{2,\frac32,1}, \bk{2,\frac32,2}, \bk{2,\frac32,3}, \bk{2,\frac52,1}, \\ 
		& \bk{2,\frac52,2}, \bk{2,\frac72,1}, \bk{2,\frac72,2}, \bk{2,\frac{13}{2},1}, \bk{2,\frac{17}{2},1} \\
		& \bk{3,1,4}, \bk{3,\frac32,1}, \bk{3,\frac32,3}, \bk{3,\frac52,1}, \bk{3,\frac52,2},  \bk{3,\frac72,1}, \bk{3,\frac72,2}, \bk{3,\frac{13}{2},1} \\
		& \bk{4,\frac12,1} ,\bk{4,\frac12,3}, \bk{4,\frac12,5}, \bk{4,\frac12,6}, \bk{4,1,4}, \bk{4,\frac32,1}, \bk{4,\frac32,2}, \bk{4,\frac32,3}, \\
		& \bk{4,\frac52,1}, \bk{4,\frac52,2}, \bk{4,\frac72,1}, \bk{4,\frac92,1} \\
		& \bk{5,\frac12,2}, \bk{5,\frac12,4}, \bk{5,\frac12,5}, \bk{5,1,2}, \bk{5,1,4}, \bk{5,\frac32,1}, \bk{5,\frac32,2},\bk{5,\frac32,3}, \\
		& \bk{5,\frac52,1}, \bk{5,\frac52,2}, \bk{5,\frac72,1}, \bk{5,\frac{11}{2},1} \\
		& \bk{6,\frac12,1}, \bk{6,1,1}, \bk{6,1,2}, \bk{6,1,3}, \bk{6,2,1}, \bk{6,2,3}, \bk{6,3,1}, \bk{6,3,2}, \bk{6,7,1} \\
		& \bk{7,\frac12,1}, \bk{7,\frac12,2}, \bk{7,\frac12,3}, \bk{7,\frac52,1}, \bk{7,\frac52,2}, \bk{7,\frac92,1}, \bk{7,\frac92,2}, \bk{7,\frac{11}{2},1}, \bk{7,\frac{19}{2},1} \\
		& \bk{8,\frac12,1}, \bk{8,\frac12,2}, \bk{8,\frac{11}{2},1}, \bk{8,\frac{19}{2},1}, \bk{8,\frac{29}{2},1} .
	\end{split}
	\]
\end{Prop}

We now turn to cases where the maximal semi-simple quotient of $\DPS{P_i}\bk{s_0,\chi}$ is not irreducible.
These cases are
\[
\bk{1,\frac52,1}, \bk{3,\frac12,2}, \bk{7,\frac32,1} .
\]

There are also two cases for which the length of the local socle was not determined in \cite{SDPS_E8}: $\bk{2,\frac12,1}$ and $\bk{5,\frac12,1}$.
For both cases, the local maximal semi-simple is either irreducible or of length $2$.

\begin{Prop}
	\label{Prop:Res_E8_1_52_1}
	It holds that
	\[
	\ResRep{G}{1}{\frac52}{1} = \bigoplus_{\stackrel{S\subset\Places}{\Card{S}\equiv 0 \pmod{2} }} \pi_S ,
	\]
	where (in the notations of \Cref{Subsec:Local_Lengths_E7}) 
	\[
	\pi_S = \otimes_{v\in S}\pi_{-1,v} \otimes_{v\notin S}\pi_{1,v} .
	\]
\end{Prop}

\begin{proof}
	This follows from \Cref{Prop:Residue_E7_2_1_1} and \Cref{Appendix:Local_data_for_global_SI_resisudes} by using properties of induction in stages through the Levi subgroup $M_8$, which is of type $E_7$.
	Indeed, we recall from \cite{SDPS_E8} that $\DPS{P_1}\bk{\frac52}_v$ has a maximal semi-simple quotient of length $2$ and all of its irreducible constituents are factors of the induction from $M_8$ of the irreducible constituents of $\DPS{P_2}^{M_8}\bk{1}_v$.
	In particular, one notes that in equivalence classes $\coset{w'}$ such that
	\[
		N_{w',v}\bk{\lambda_{\frac52}^{P_2}}\bk{\DPS{P_1}\bk{1}_v} = \pi_{1,v} \oplus \pi_{-1,v}
	\]
	the shortest element $w'\in\coset{w'}$ is of the form
	\[
	w' = u^{-1} w_0 u w_1,
	\]
	where
	\begin{itemize}		
		\item $w_1$ is given by
		\[
		w_1=w[6, 5, 4, 2, 3, 4, 5, 6, 7, 6, 5, 4, 2, 3, 4, 5, 6, 8, 7, 6, 5, 4, 2, 3, 1, 4, 3, 5, 4, 2, 6, 5, 4, 3, 1]
		\]
		\item $w_1$ satisfies $w_1\cdot\lambda_0=\lambda_1$, where
		\[
		\lambda_0 = \lambda_{\frac52}^{P_2} = \eeightchar{13}{-1}{-1}{-1}{-1}{-1}{-1}{-1},\quad
		\lambda_1 = \lambda_{-\frac52}^{P_2} = \eeightchar{8}{-1}{-1}{-1}{-1}{-1}{-1}{-1}.
		\]
		\item $N_u\bk{\lambda_1}$ is an isomorphism.
		\item $w_0$ is as given in \Cref{Prop:Residue_E7_2_1_1}.
		Thus, the associated intertwining operator acts on $\pi_{\pm1}$ as $\pm\Id$.
	\end{itemize}
\end{proof}

Let $\chi=\otimes_{v\in \Places} \chi_v$ be a quadratic Hecke character, that is $\chi^2=\Id\neq\chi$.
We consider $\pi=\DPS{P_3}\bk{\chi,\frac12} = \otimes_{v\in \Places}\pi_v$.
Recall that at places where $\chi_v=\Id_v$, the local representation $\pi_v$ admits a unique irreducible quotient $\pi_{1,v}$, while at places where $\chi_v\neq\Id_v$, the maximal semi-simple quotient $\pi_{1,v}\oplus\pi_{-1,v}$ of $\pi_v$ has length $2$.
This decomposition, as in the previous case, factors through $\DPS{P_{2,8}}^{M_8}\bk{1}_v$.

A similar argument to the one made above shows

\begin{Prop}
	It holds that
	\[
	\ResRep{G}{3}{\frac12}{2} = \bigoplus_{\stackrel{S\subset\Places}{\Card{S}\equiv 1 \pmod{2} }} \pi_S ,
	\]
	where
	\[
	\Places_\chi = \set{v\in\Places \mvert \chi_v\neq \Id} .
	\]
	and
	\[
	\pi_S = \otimes_{v\in S}\pi_{-1,v} \otimes_{v\notin S}\pi_{1,v} .
	\]
\end{Prop}

The residue in the case $\bk{7,\frac32,1}$ can be determined in a similar fashion.

\begin{Prop}
	\label{Prop:Res_E8_7_32_1}
	It holds that
	\[
	\ResRep{G}{7}{\frac32}{1} = \bigoplus_{\stackrel{S\subset\Places}{\Card{S}\equiv 0 \pmod{2} }} \pi_S ,
	\]
	where (in the notations of \Cref{Subsec:Local_Lengths_E7}) 
	\[
	\pi_S = \otimes_{v\in S}\pi_{-1,v} \otimes_{v\notin S}\pi_{1,v} .
	\]
	
\end{Prop}

The differences between this case and the previous ones are as follows:
\begin{itemize}
	\item The role of $\DPS{P_{2,8}}^{M_8}\bk{1}_v$ is replaced by $\DPS{P_{1,4}}^{M_1}\bk{0}_v$, which is a unitary representation of length $2$ of the Levi subgroup $M_1$ of type $D_7$ by 	\cite{MR2017065}.
	\item The exponents $\lambda_0$ and $\lambda_1$ are given by
	\[
	\lambda_0 = \lambda_{\frac32}^{P_7} = \eeightchar{-1}{-1}{-1}{-1}{-1}{-1}{10}{-1},\quad
	\lambda_1 = \lambda_{-\frac32}^{P_7} = \eeightchar{-1}{-1}{-1}{-1}{-1}{-1}{7}{-1}.
	\]
	\item $w_1$ is given by
	\[
	w_1=w[3, 1, 4, 2, 3, 4, 5, 4, 2, 3, 1, 4, 3, 6, 5, 4, 2, 3, 1, 4, 3, 5, 4, 6, 5, 7, 6, 8, 7]
	\]
	
	\item $w_0$ is given by
	\[
	w_0=w\coset{4, 3, 2, 4, 5, 6, 7, 8, 4, 3, 2, 4, 5, 6, 7, 4, 3, 2, 4, 5, 6, 4, 3, 2, 4, 5, 4, 3, 2, 4}.
	\]
	Namely, it is the shortest representative of the longest coset in $W_{M_1}^{T,M_{1,4}}$ and thus, by \cite{MR2363302}, the associated intertwining operator  acts on $\pi_{\pm 1}$ as $\pm\Id$.
\end{itemize}

\chapter{Siegel-Weil Identities}
\label{Chap:Siegel_Weil}

In this chapter, we study Siegel-Weil type identities between residues of different degenerate Eisenstein series.

\section{Method and Results}
\label{Sec:Siegel-Weil_Methods_and_Results}

We consider the following scenario:
\begin{itemize}
	\item $P_i=M_i\cdot U_i$ and $P_j=M_j\cdot U_j$ are two standard maximal parabolic subgroups of $G$.
	\item $\Omega$ and $\Xi$ are real $1$-dimensional representations of $M$ and $L$ respectively.
	\item Assume that $\eta_{s_0}^{P_i}$ and $\eta_{t_0}^{P_j}$ lie in the same Weyl orbit.
	Namely, there exist $\widetilde{w}\in W$ such that $\eta_{t_0}^{P_j} = \widetilde{w}\cdot \eta_{s_0}^{P_i}$.
	\item Assume that $\DPS{P_i}\bk{s_0}$ and $\DPS{P_j}\bk{t_0}$ are generated by their spherical vectors.
\end{itemize}

We apply \Cref{Prop:W_inv_entire_Eisen_Ser} to prove an identity between the residues of the Eisenstein series $\Eisen_{P_i}$ and $\Eisen_{P_j}$.

Indeed, by the $W$-invariance of $\Eisen_{B}^\sharp$ (see \Cref{Prop:W_inv_entire_Eisen_Ser}), it holds that
\[
\Eisen_{B}^\sharp\bk{\eta_{s_0}^{P_i},g} = 
\Eisen_{B}^\sharp\bk{\eta_{t_0}^{P_j},g} 
\quad \forall g\in G\bk{\AA} .
\]
On the other hand, recall that
\[
\Res{\lambda=\eta_s^{P_i}} \Eisen_{B}\bk{f,\lambda,g} 
= \lim_{\lambda\to\eta_s} \coset{\bk{\prodl_{\stackrel{k=1}{k\neq i}}^n l^{+}_{\alpha_i}\bk{\lambda}} \Eisen_{B}\bk{f,\lambda,g} }
= \Eisen_{P_i}\bk{f,s,\chi,g} .
\]
Furthermore, $\eta_{s_0}$ lies on the intersection of the line
\[
\set{\eta_s\mvert s\in\C} = \bigcap_{\stackrel{k=1}{k\neq i}}^n H_{\alpha_i}^{+}
\]
with other hyperplanes $H_\alpha^{+}$ for other $\alpha\in\Phi^{+}$.
In particular, it holds that there exists a constant $C_{i,s_0}\in\C^\times$ such that
\begin{equation}
\Eisen_{B}^\sharp\bk{\eta_{s_0}^{P_i},g} = C_{i,s_0} \Res{s=s_0} \Eisen_{P_i}\bk{f^0_{P_i},s,g}.
\end{equation}
Similarly, there exists a constant $C_{j,t_0}\in\C^\times$ such that
\begin{equation}
\Eisen_{B}^\sharp\bk{\eta_{t_0}^{P_j},g} = C_{j,t_0} \Res{s=t_0} \Eisen_{P_j}\bk{f^0_{P_j},s,g}.
\end{equation}

We conclude that there exists $C_{i,j,s_0,t_0}$ such that
\begin{equation}
\Res{s=s_0} \Eisen_{P_i}\bk{f^0_{P_i},s,g}
= C_{i,j,s_0,t_0} \Res{s=t_0} \Eisen_{P_j}\bk{f^0_{P_j},s,g}.
\end{equation}
We point out that the constant here can be interpreted as the action of an intertwining operator.
More precisely, let $\widetilde{w}=\widetilde{w}_{i,j}$ be the Weyl element given by
\[
\widetilde{w}_{i,j} = w_{0,j} w_{0,j\to i} w_{0,i} ,
\]
where:
\begin{itemize}
	\item $w_{0,i}$ and $w_{0,j}$ are the longest elements in $W_{M_i}$ and $W_{M_j}$ (respectively)
	\item $w_{0,j\to i}$ is the shortest representative of the class in $W_{M_j}\rmod W_{M_i\cap M_j}$ of the longest element in $W$.
\end{itemize}
In this case, we note that
\begin{enumerate}
	\item $\eta_{t_0}^{P_j} = \widetilde{w}_{i,j}\cdot \eta_{s_0}^{P_i}$.
	\item ${M_{\widetilde{w}_{i,j}}\bk{f^0_{\eta_{s_0}^{P_i}}}} = C_{i,j,s_0,t_0} \cdot  f^0_{\eta_{t_0}^{P_j}}$ .
\end{enumerate}

This implies the following:
\begin{Prop}
	It holds that
	\begin{equation}
	\Res{s=s_0} \Eisen_{P_i}\bk{f,s,g}
	= C_{i,j,s_0,t_0} \Res{s=t_0} \Eisen_{P_j}\bk{M_{\widetilde{w}_{i,j}}\bk{\eta_s^{P_i}} f,s,g} \quad \forall f\in \DPS{P}\bk{s} .
	\end{equation}
	Furthermore,
	\begin{equation}
		\label{eq:Siegel_Weil_order_equation}
		\ord{s=s_0}{\Eisen_{P_i}} = \ord{s=s_0}{M_{\widetilde{w}_{i,j}}\bk{\eta_s^{P_i}}} + \ord{s=t_0}{\Eisen_{P_j}} .
	\end{equation}
	In particular, if $\Res{s=t_0} \Eisen_{P_j}$ is irreducible, then so is $\Res{s=s_0} \Eisen_{P_i}$ and
	\begin{equation}
		\ResRepTr{G}{i}{s_0} = \ResRepTr{G}{j}{t_0} .
	\end{equation}
\end{Prop}

By virtue of \Cref{eq:Siegel_Weil_order_equation}, the order of the pole of $\Eisen_{P_i}$ at $s_0$ is greater than that of $\Eisen_{P_j}$ at $t_0$.
It should also be pointed out that if $\Eisen_{P_j}$ is holomorphic at $t_0$, then $\Res{s=t_0} \Eisen_{P_j}$ is the spanned by the evaluation of $\Eisen_{P_j}\bk{f,t,g}$ at $t_0$.

In what follows, we depict lists of Siegel-Weil identities in diagrams, where such an identity will appear as the following diagram:

\begin{figure}[h!]
\begin{center}
	\begin{tikzpicture}
	[thin,scale=2]
	\tikzset{
		text style/.style={
			sloped, 
			text=black,
			font=\tiny,
			above
		}
	}
	\matrix (m) [matrix of math nodes, row sep=1em, column sep=3em]
	{
	|[name=e1]| \text{I}_{\para{P}_i}\bk{s_0}&  |[name=e2]| \text{I}_{\para{P}_j}\bk{t_0}
	\\};
	\draw[->]      
	(e1)  edge [->]
	node[text style,above]{$\widetilde{w}_{i,j}$}  (e2);
	\end{tikzpicture}
\end{center}	
\end{figure}
Note that the diagrams in \Cref{Thm:Siegel_Weil-identities} are organized so that an arrow between $\DPS{P_i}\bk{s_0}$ and $\DPS{P_j}\bk{t_0}$ implies that 
\[
\ord{s=s_0}{\Eisen_{P_i}} = \ord{s=t_0}{\Eisen_{P_j}} + 1 .
\]
Namely, $M_{\widetilde{w}_{i,j}}\bk{\eta_s^{P_i}}$ admits a simple pole.

\begin{Thm}
	\label{Thm:Siegel_Weil-identities}
	The groups of type $E_6$, $E_7$ and $E_8$ admit the Siegel-Weil identities depicted in 
	figures	\ref{digram:E6_SW_SI} through \ref{digram:E8_SW_NSI} which appear in \Cref{Sec:Siegel-Weil_Diagrams}.

\end{Thm}

\begin{Rem}
	In the diagrams of \Cref{Thm:Siegel_Weil-identities}, we have also added Siegel-Weil data $\bk{\DPS{P_i}\bk{s_0}, \DPS{P_j}\bk{t_0}, \widetilde{w}_{i,j}}$ such that $\DPS{P_j}\bk{t_0}$ is not generated by a spherical vector.
	In these cases, $\DPS{P_i}\bk{s_0}$ is still generated by a spherical vector and hence, $\ResRepTr{G}{i}{s_0}$ is isomorphic to the irreducible spherical subrepresentation of $\ResRepTr{G}{j}{t_0}$.
	The relevant cases are:
	\[
	\bk{G,j,t_0} = \bk{E_7,2,1}, \bk{E_8,1,\frac52}, \bk{E_8,3,\frac12} \bk{E_8,7,\frac32}
	\]
	
	This also possibly includes the case $\bk{E_8, 2,\frac12}$ and $\bk{E_8, 5,\frac12}$.
\end{Rem}

\allowdisplaybreaks
\begin{Rem}
	By \cite{MR2123125}, the minimal representation of these groups, is the unique irreducible quotient of the following representations:
	\[
	\piece{\DPS{P_1}\bk{3} = \DPS{P_6}\bk{3},& G=E_6 \\
	\DPS{P_7}\bk{5},& G=E_7 \\
	\DPS{P_8}\bk{\frac{19}{2}},& G=E_8}.
	\]
\end{Rem}

\begin{Cor}
	\label{Cor:Minimal_representation}
	The minimal representation of $G$ is given by the following residual representations:
	\[
	\tau=\piece{
		\ResRepTr{E_6}{1}{3},
		\ResRepTr{E_6}{6}{3},
		\ResRepTr{E_6}{2}{\frac72},
		\ResRepTr{E_6}{4}{\frac52}, & G=E_6 \\
		\ResRepTr{E_7}{7}{5},
		\ResRepTr{E_7}{1}{\frac{11}{2}},
		\ResRepTr{E_7}{2}{5},
		\ResRepTr{E_7}{4}{3}, & G=E_7 \\
		\ResRepTr{E_8}{8}{\frac{19}{2}},
		\ResRepTr{E_8}{1}{\frac{17}{2}},
		\ResRepTr{E_8}{2}{\frac{13}{2}},
		\ResRepTr{E_8}{4}{\frac72}, & G=E_8
	}
	\]

\end{Cor}

\begin{Cor}
	Other than the identities involving the minimal representation, we have the following Siegel-Weil identities between square-integrable residual representations of $G$:
	\begin{enumerate}
		\item If $G=E_6$, then
		\[
		\ResRepTr{E_6}{2}{\frac12} = \ResRepTr{E_6}{3}{\frac32} = \ResRepTr{E_6}{4}{\frac32} = \ResRepTr{E_6}{5}{\frac32} .
		\]
		\item If $G=E_7$, then
		\[
		\begin{split}
			& \ResRepTr{E_7}{1}{\frac12} = \ResRepTr{E_7}{3}{\frac52} = \ResRepTr{E_7}{4}{2} = \ResRepTr{E_7}{6}{\frac52} \\
			& \ResRepTr{E_7}{1}{\frac72} = \ResRepTr{E_7}{5}{3} = \ResRepTr{E_7}{6}{\frac72} = \ResRepTr{E_7}{7}{1} \\
			& \ResRepTr{E_7}{3}{\frac12} = \ResRepTr{E_7}{4}{1} = \ResRepTr{E_7}{5}{1} .
		\end{split}
		\]
		Also, it holds that
		\[
		\ResRepTr{E_7}{3}{\frac32} = \ResRepTr{E_7}{6}{\frac12} = \pi_1,
		\]
		where $\pi_1$ is irreducible spherical quotient of $\DPS{P_2}\bk{1}$.
		For more on $\DPS{P_2}\bk{1}$ and $\ResRepTr{E_7}{2}{1}$, see \Cref{Prop:Residue_E7_2_1_1}.
		
		\item If $G=E_8$, then
		\[
		\begin{split}
			&\ResRepTr{E_8}{4}{\frac12} = \ResRepTr{E_8}{5}{\frac12} \\
			&\ResRepTr{E_8}{2}{\frac12} = \ResRepTr{E_8}{6}{1} \\
			&\ResRepTr{E_8}{4}{\frac32} = \ResRepTr{E_8}{6}{2} = \ResRepTr{E_8}{7}{\frac12} \\
			&\ResRepTr{E_8}{2}{\frac32} = \ResRepTr{E_8}{3}{\frac32} = \ResRepTr{E_8}{5}{\frac32} \\
			&\ResRepTr{E_8}{3}{\frac72} = \ResRepTr{E_8}{4}{\frac52} = \ResRepTr{E_8}{7}{\frac92} = \ResRepTr{E_8}{8}{\frac12} \\
			&\ResRepTr{E_8}{1}{\frac{13}{2}} = \ResRepTr{E_8}{5}{\frac72} = \ResRepTr{E_8}{7}{\frac{11}{2}} = \ResRepTr{E_8}{8}{\frac{11}{2}} \\
			&\ResRepTr{E_8}{1}{\frac72} = \ResRepTr{E_8}{2}{\frac72} = \ResRepTr{E_8}{5}{\frac52} = \ResRepTr{E_8}{6}{3} .
		\end{split}
		\]
		Also, it holds that:
		\begin{itemize}
			\item 
			\[
				\ResRepTr{E_8}{3}{\frac52} = 	\ResRepTr{E_8}{7}{\frac52} = \pi_1,
			\]
			where $\pi_1$ is the irreducible spherical quotient of $\DPS{P_1}\bk{\frac52}$.
			For more on $\DPS{P_1}\bk{\frac52}$ and $\ResRepTr{E_8}{1}{\frac52}$, see \Cref{Prop:Res_E8_1_52_1}.
	
			\item 
			\[
				\ResRepTr{E_8}{1}{\frac12} = \ResRepTr{E_8}{2}{\frac52} = \pi_1,	
			\]
			where $\pi_1$ is the irreducible spherical quotient of $\DPS{P_7}\bk{\frac32}$.
			For more on $\DPS{P_7}\bk{\frac32}$ and $\ResRepTr{E_8}{7}{\frac32}$, see \Cref{Prop:Res_E8_7_32_1}.
		\end{itemize}
	\end{enumerate}
\end{Cor}

\begin{Rem}
	The following residual representations are the next-to-minimal representations of $G$ (see \cite[Table 2]{MR4273169} and \cite[Table 6]{MR3267116})
	\[
	\piece{
		\ResRepTr{E_6}{1}{0},
		\ResRepTr{E_6}{5}{0},
		\ResRepTr{E_6}{2}{\frac52}, & G=E_6 \\
		\ResRepTr{E_7}{1}{\frac72},
		\ResRepTr{E_7}{5}{3},
		\ResRepTr{E_7}{6}{\frac72},
		\ResRepTr{E_7}{7}{1}, & G=E_7 \\
		\ResRepTr{E_7}{1}{\frac{13}{2}},
		\ResRepTr{E_7}{5}{\frac72},
		\ResRepTr{E_7}{7}{\frac{11}{2}},
		\ResRepTr{E_7}{8}{\frac{11}{2}}, & G=E_8
	}
	\]
	Note that in the case of $E_6$, this residue is non-square integrable.
	This is, actually, not surprising since the next-to-minimal nilpotent orbit in $\frake_6$ is $2A_1$ and its dual $D_5$ is not distinguished (see \Cref{Chap:Arthur} for more about nilpotent orbits and the square-integrable spectrum).
\end{Rem}

\section{Identities for Non-trivial $\chi$}

In this section, we consider identities satisfied by residual representations with non-unique $\chi$.
It seems that the mechanism described in \Cref{Sec:Siegel-Weil_Methods_and_Results} does not admit a generalization to these Eisenstein series.
Thus, we describe these identities only as an  a posteriori result of \Cref{Chap:Square_Integrable} and \Cref{Chap:Non_square_Integrable}.
In particular, we prove only an isomorphism between the residual representations but not an equality of automorphic realizations (though, it is reasonable that this is also true). 

\begin{Thm}
	We list identities between degenerate residual residues $\ResRep{G}{i}{s_0}{\chi}$ of groups of type $E_n$; we sort the list by the type of $G$ and the order of $\chi$.
	\begin{enumerate}
		\item For $G$ of type $E_6$ and a quadratic Hecke character $\chi$, it holds that
		\[
		\Res{s=1} \Eisen_{P_4}\bk{\chi}
		\cong \DPS{P_3}\bk{0,\chi} \cong \DPS{P_5}\bk{0,\chi} .
		\]
		
		\item For $G$ of type $E_7$ and a quadratic Hecke character $\chi$, it holds that
		\begin{itemize}
			\item $\Res{s=2} \Eisen_{P_2}\bk{\chi} \cong \Res{s=2} \Eisen_{P_5}\bk{\chi}$.
			
			\item $\Res{s=\frac32} \Eisen_{P_3}\bk{\chi} \cong \Res{s=\frac12} \Eisen_{P_6}\bk{\chi}$.
			
			\item $\Res{s=1} \Eisen_{P_4}\bk{\chi} \cong \Res{s=\frac12} \Eisen_{P_3}\bk{\chi}$.
			
			\item $\Res{s=\frac32} \Eisen_{P_4}\bk{\chi} \cong \DPS{P_6}\bk{1,\chi}$.
			
			\item $\Res{s=\frac32} \Eisen_{P_5}\bk{\chi} \cong \DPS{P_2}\bk{\frac12,\chi}$.
			
		\end{itemize}
	
		\item For $G$ of type $E_7$ and a cubic Hecke character $\chi$, it holds that
		\[
		\Res{s=\frac23} \Eisen_{P_4}\bk{\chi}
		 = \DPS{P_5}\bk{\chi^2,\frac13} .
		\]
	
		\item For $G$ of type $E_8$ and a quadratic Hecke character $\chi$, it holds that
		\begin{itemize}
			\item $\Res{s=\frac12} \Eisen_{P_1}\bk{\chi} \cong \Res{s=\frac52} \Eisen_{P_2}\bk{\chi}$.
			
			\item $\Res{s=\frac12} \Eisen_{P_2}\bk{\chi} \cong \Res{s=1} \Eisen_{P_6}\bk{\chi}$.
			
			\item $\Res{s=\frac52} \Eisen_{P_3}\bk{\chi} \cong \Res{s=\frac52} \Eisen_{P_7}\bk{\chi}$.
			
			\item $\Res{s=\frac32} \Eisen_{P_4}\bk{\chi} \cong \Res{s=\frac12} \Eisen_{P_7}\bk{\chi}$.
			
			\item $\Res{s=1} \Eisen_{P_4}\bk{\chi} \cong \Res{s=1} \Eisen_{P_3}\bk{\chi}$.
			
			\item $\Res{s=\frac32} \Eisen_{P_5}\bk{\chi} \cong \Res{s=\frac32} \Eisen_{P_2}\bk{\chi}$.
			
			\item $\Res{s=1} \Eisen_{P_5}\bk{\chi} \cong \Res{s=\frac12} \Eisen_{P_6}\bk{\chi}$.
			
			\item $\Res{s=2} \Eisen_{P_3}\bk{\chi} \cong \DPS{P_7}\bk{1,\chi}$.
			
			\item $\Res{s=2} \Eisen_{P_4}\bk{\chi} \cong \DPS{P_7}\bk{3,\chi}$.
			
			\item $\Res{s=\frac34} \Eisen_{P_4}\bk{\chi} \cong \DPS{P_3}\bk{\frac14,\chi}$.
			
			\item $\Res{s=2} \Eisen_{P_5}\bk{\chi} \cong \DPS{P_1}\bk{1,\chi}$.
			
			\item $\Res{s=\frac52} \Eisen_{P_6}\bk{\chi} \cong \DPS{P_1}\bk{2,\chi}$.
			
		\end{itemize}
	
		\item For $G$ of type $E_8$ and a cubic Hecke character $\chi$, it holds that
		
		\begin{itemize}
			\item $\Res{s=\frac76} \Eisen_{P_5}\bk{\chi} \cong \DPS{P_2}\bk{\frac16,\chi}$.
		\end{itemize}
		
	\end{enumerate}
\end{Thm}

\begin{proof}
	In all the cases, listed here, where $\Res{s=s_0} \Eisen_{P_i}\bk{\chi}$ is non-square integrable, it is isomorphic to an irreducible degenerate principal series representation $\DPS{P_j}\bk{t_0,\chi}$.
	This irreducible representation can be realized as the value of $\Eisen_{P_j}\bk{\chi}$ at $s=t_0$.
	These equalities are listed and explained in \Cref{Chap:Non_square_Integrable}.
	Note that since $\Eisen_{P_j}\bk{\chi}$ is holomorphic and non-vanishing at $t_0$, its value there is its residue according to the definition of residue in \Cref{Chap:Preliminaries}.
	
	The square-integrable cases are those where $\Res{s=s_0} \Eisen_{P_i}\bk{\chi} = \Res{s=t_0} \Eisen_{P_j}\bk{\chi}$ and both residues are square-integrable.
	In all the listed cases, both $\DPS{P_i}\bk{s_0,\chi}$ and $\DPS{P_j}\bk{t_0,\chi}$ admit isomorphic (unique) irreducible quotients $\pi_1$.
	Thus
	\[
	\Res{s=s_0} \Eisen_{P_i}\bk{\chi} \cong \Res{s=t_0} \Eisen_{P_j}\bk{\chi} = \pi_1 .	
	\]
	
	We remark that in the statement of the theorem, the residue on the left hand side is the one whose pole is of higher order.
	
\end{proof}

\begin{Rem}
	We note that if, for a cubic Hecke character $\chi$, it would be known that $\Res{s=\frac56} \Eisen_{P_4}\bk{\chi}$ is irreducible, the the identity $\Res{s=\frac56} \Eisen_{P_4}\bk{\chi} \cong \DPS{P_6}\bk{\frac13,\chi}$ would similarly follow.	
\end{Rem}

\newpage

\section{Diagrams}
\label{Sec:Siegel-Weil_Diagrams}

\begin{figure}[h!]
	\begin{center}
		\begin{tikzpicture}[thin,scale=2]
			\tikzset{
				text style/.style={
					sloped, 
					text=black,
					font=\tiny,
					above
				}
			}
			\matrix (m) [matrix of math nodes, row sep=1em, column sep=3em]
			{ 		
				&|[name=e72]| \text{I}_{\para{P}_3}(\frac{3}{2})
				&&& |[name=e22]| \text{I}_{\para{P}_1}(3)
				\\
				|[name=e71]| \text{I}_{\para{P}_4}(\frac{3}{2}) 
				&& |[name=e74]| \text{I}_{\para{P}_2}(\frac{1}{2})
				&  |[name=e21]| \text{I}_{\para{P}_4}(\frac{5}{2})
				&  |[name=e23]| \text{I}_{\para{P}_2}(\frac{7}{2})
				\\
				&|[name=e73]| \text{I}_{\para{P}_5}(\frac{3}{2}) 
				&&& |[name=e24]| \text{I}_{\para{P}_6}(3)
				\\
			};
			\draw[->]  
			(e21)  edge [->]
			node[text style,above]{$\widetilde{w}_{4,1}$}  (e22)
			(e21)  edge [->]
			node[text style,above]{$\widetilde{w}_{4,2}$}  (e23)
			(e21)  edge [->]
			node[text style,above]{$\widetilde{w}_{4,6}$}  (e24)
			
			(e71)  edge [->]
			node[text style,above]{$\widetilde{w}_{4,3}$}  (e72)
			(e71)  edge [->]
			node[text style,above]{$\widetilde{w}_{4,5}$}  (e73)
			(e72)  edge [->]
			node[text style,above]{$\widetilde{w}_{3,2}$}  (e74)
			(e73)  edge [->]
			node[text style,above]{$\widetilde{w}_{5,2}$}  (e74)
			;
			
		\end{tikzpicture}
	\end{center}
	\caption{Siegel-Weil Identities for $E_6$ - Square Integrable Residues}
	\label{digram:E6_SW_SI}
\end{figure}

\begin{figure}[h!]
	\begin{center}
		\begin{tikzpicture}[thin,scale=2]
			\tikzset{
				text style/.style={
					sloped, 
					text=black,
					font=\tiny,
					above
				}
			}
			\matrix (m) [matrix of math nodes, row sep=1em, column sep=3em]
			{ 		
				|[name=e31]| \text{I}_{\para{P}_3}(\frac{7}{2})
				&  |[name=e32]| \text{I}_{\para{P}_1}(4) 
				&  |[name=e41]| \text{I}_{\para{P}_3}(\frac{5}{2})
				&  |[name=e42]| \text{I}_{\para{P}_6}(1)
				&  |[name=e51]| \text{I}_{\para{P}_5}(\frac{5}{2})
				&  |[name=e52]| \text{I}_{\para{P}_1}(1)

				\\
				& & & |[name=e12]| \text{I}_{\para{P}_1}(0)
				& & |[name=e84]| \text{I}_{\para{P}_3}(0)
				
				\\
				|[name=e61]| \text{I}_{\para{P}_5}(\frac{7}{2})
				&  |[name=e62]| \text{I}_{\para{P}_6}(4)
				&  |[name=e11]| \text{I}_{\para{P}_2}(\frac{5}{2})
				& & |[name=e92]| \text{I}_{\para{P}_4}(1)
				
				\\
				& & & |[name=e13]| \text{I}_{\para{P}_6}(0)
				& & |[name=e104]| \text{I}_{\para{P}_5}(0)
				\\
			};
			\draw[->]      
			(e11)  edge [->]
			node[text style,above]{$\widetilde{w}_{2,1}$}  (e12)
			(e11)  edge [->]
			node[text style,above]{$\widetilde{w}_{2,6}$}  (e13)
			
			(e31)  edge [->]
			node[text style,above]{$\widetilde{w}_{3,1}$}  (e32)
			(e41)  edge [->]
			node[text style,above]{$\widetilde{w}_{3,6}$}  (e42)
			(e51)  edge [->]
			node[text style,above]{$\widetilde{w}_{5,1}$}  (e52)
			(e61)  edge [->]
			node[text style,above]{$\widetilde{w}_{5,6}$}  (e62)

			(e92)  edge [->]
			node[text style,above]{$\widetilde{w}_{4,3}$}  (e84)
			(e92)  edge [->]
			node[text style,above]{$\widetilde{w}_{4,5}$}  (e104)
			;
			
		\end{tikzpicture}
	\end{center}
	\caption{Siegel-Weil Identities for $E_6$ - Non-square Integrable Residues}
	\label{digram:E6_SW_NSI}
\end{figure}

\begin{figure}[h!]
	\begin{center}
		\begin{tikzpicture}[thin,scale=2]
			\tikzset{
				text style/.style={
					sloped, 
					text=black,
					font=\tiny,
					above
				}
			}
			\matrix (m) [matrix of math nodes, row sep=1em, column sep=3em]
			{
				&  |[name=p3522]| \text{I}_{\para{P}_3}(\frac{5}{2})
				\\
				|[name=p414]| \text{I}_{\para{P}_4}(2)
				&& |[name=p1134]| \text{I}_{\para{P}_1}(\frac{1}{2}) &
				|[name=p418]| \text{I}_{\para{P}_4}(1)
				&  |[name=p5110]| \text{I}_{\para{P}_5}(1)
				&  |[name=p3122]| \text{I}_{\para{P}_3}(\frac{1}{2})
				&
				\\
				&  |[name=p6526]| \text{I}_{\para{P}_6}(\frac{5}{2})
				&  
				
				\\
				& |[name=p1734]| \text{I}_{\para{P}_1}(\frac{7}{2})
				& & |[name=p2114]| \text{I}_{\para{P}_2}(1)
				& & |[name=p11134]| \text{I}_{\para{P}_1}(\frac{11}{2})
				\\
				|[name=p5310]| \text{I}_{\para{P}_5}(3)
				& |[name=p6726]| \text{I}_{\para{P}_6}(\frac{7}{2})
				& |[name=p3322]| \text{I}_{\para{P}_3}(\frac{3}{2})
				& & |[name=p438]| \text{I}_{\para{P}_4}(3)
				& |[name=p2514]| \text{I}_{\para{P}_2}(5)
				\\
				& |[name=p711]| \text{I}_{\para{P}_7}(1)
				& &  |[name=p6126]| \text{I}_{\para{P}_6}(\frac{1}{2})
				& &  |[name=p7518]| \text{I}_{\para{P}_7}(5)
				\\};
			
			\draw[->]      
			(p414)  edge [->]
			node[text style,above]{$\widetilde{w}_{4,3}$}  (p3522)
			
			(p414)  edge [->]
			node[text style,above]{$\widetilde{w}_{4,6}$}  (p6526)
			
			(p6526)  edge [->]
			node[text style,above]{$\widetilde{w}_{6,1}$}  (p1134)
			
			(p3522)  edge [->]
			node[text style,above]{$\widetilde{w}_{3,1}$}  (p1134)
			
			(p418)  edge [->]
			node[text style,above]{$\widetilde{w}_{4,5}$}  (p5110)
			
			(p5110)  edge [->]
			node[text style,above]{$\widetilde{w}_{5,3}$}  (p3122)
			
			(p5310)  edge [->]
			node[text style,above]{$\widetilde{w}_{5,1}$}  (p1734)
			
			(p5310)  edge [->]
			node[text style,above]{$\widetilde{w}_{5,7}$}  (p711)		
			
			(p5310)  edge [->]
			node[text style,above]{$\widetilde{w}_{5,6}$}  (p6726)
			
			(p3322)  edge [->]
			node[text style,above]{$\widetilde{w}_{3,2}$}  (p2114)
			
			(p3322)  edge [->]
			node[text style,above]{$\widetilde{w}_{3,6}$}  (p6126)
			
			(p438)  edge [->]
			node[text style,above]{$\widetilde{w}_{4,1}$}  (p11134)
			
			(p438)  edge [->]
			node[text style,above]{$\widetilde{w}_{4,2}$}  (p2514)
			
			(p438)  edge [->]
			node[text style,above]{$\widetilde{w}_{4,7}$}  (p7518)
			;
		\end{tikzpicture}
	\end{center}
	\caption{Siegel-Weil Identities for $E_7$ - Square Integrable Residues}
	\label{digram:E7_SW_SI}
\end{figure}

\begin{figure}[h!]
	\begin{center}
		\begin{tikzpicture}[thin,scale=2]
			\tikzset{
				text style/.style={
					sloped, 
					text=black,
					font=\tiny,
					above
				}
			}
			\matrix (m) [matrix of math nodes, row sep=1em, column sep=3em]
			{ 		
				|[name=p2314]| \text{I}_{\para{P}_2}(3)
				&  |[name=p1334]| \text{I}_{\para{P}_1}(\frac{3}{2})
				&  |[name=p227]| \text{I}_{\para{P}_2}(4)
				&  |[name=p719]| \text{I}_{\para{P}_7}(2)
				&  |[name=p3722]| \text{I}_{\para{P}_3}(\frac{7}{2})
				&  |[name=p716]| \text{I}_{\para{P}_7}(3)
				
				\\
				|[name=p3922]| \text{I}_{\para{P}_3}(\frac{9}{2})
				&  |[name=p11334]| \text{I}_{\para{P}_1}(\frac{13}{2})
				&  |[name=p423]| \text{I}_{\para{P}_4}(\frac{2}{3})
				&  |[name=p513]| \text{I}_{\para{P}_5}(\frac{1}{3})
				&  |[name=p432]| \text{I}_{\para{P}_4}(\frac{3}{2})
				&  |[name=p611]| \text{I}_{\para{P}_6}(1)
				
				\\
				|[name=p532]| \text{I}_{\para{P}_5}(\frac{3}{2})
				&  |[name=p212]| \text{I}_{\para{P}_2}(\frac{1}{2})
				&  |[name=p515]| \text{I}_{\para{P}_5}(2)
				&  |[name=p217]| \text{I}_{\para{P}_2}(2)
				&  |[name=p525]| \text{I}_{\para{P}_5}(4)
				&  |[name=p713]| \text{I}_{\para{P}_7}(6)
				
				\\
				&&|[name=p61126]| \text{I}_{\para{P}_6}(\frac{11}{2})
				&  |[name=p7718]| \text{I}_{\para{P}_7}(7)

				\\};
			
			\draw[->]      
			(p2314)  edge [->]
			node[text style,above]{$\widetilde{w}_{2,1}$}  (p1334)
			
			(p227)  edge [->]
			node[text style,above]{$\widetilde{w}_{2,7}$}  (p719)
			
			(p3922)  edge [->]
			node[text style,above]{$\widetilde{w}_{3,1}$}  (p11334)
			
			(p515)  edge [->]
			node[text style,above]{$\widetilde{w}_{5,2}$}  (p217)
			
			(p525)  edge [->]
			node[text style,above]{$\widetilde{w}_{5,7}$}  (p713)
			
			(p61126)  edge [->]
			node[text style,above]{$\widetilde{w}_{6,7}$}  (p7718)
			
			(p3722)  edge [->]
			node[text style,above]{$\widetilde{w}_{3,7}$}  (p716)
			
			(p423)  edge [->]
			node[text style,above]{$\widetilde{w}_{4,5}$}  (p513)
			
			(p432)  edge [->]
			node[text style,above]{$\widetilde{w}_{4,6}$}  (p611)
			
			(p532)  edge [->]
			node[text style,above]{$\widetilde{w}_{5,2}$}  (p212)
			
			;
		\end{tikzpicture}
	\end{center}
	\caption{Siegel-Weil Identities for $E_7$ - Non-square Integrable Residues}
	\label{digram:E7_SW_NSI}
\end{figure}

\begin{figure}[h!]
	\begin{center}
		\begin{tikzpicture}[thin,scale=2]
			\tikzset{
				text style/.style={
					sloped, 
					text=black,
					font=\tiny,
					above
				}
			}
			\matrix (m) [matrix of math nodes, row sep=1em, column sep=3em]
			{ 	
				|[name=p416]| \text{I}_{\para{P}_4}(\frac{3}{2})
				&  |[name=p617]| \text{I}_{\para{P}_6}(2)
				&  |[name=p712]| \text{I}_{\para{P}_7}(\frac{1}{2})
				&  |[name=p5322]| \text{I}_{\para{P}_5}(\frac{3}{2})
				&  |[name=p3326]| \text{I}_{\para{P}_3}(\frac{3}{2})
				&  |[name=p2334]| \text{I}_{\para{P}_2}(\frac{3}{2})
				
				\\
				& &
				& |[name=p112]| \text{I}_{\para{P}_1}(\frac{1}{2})
				&	&  |[name=p7538]| \text{I}_{\para{P}_7}(\frac{5}{2})
				
				\\
				|[name=p4118]| \text{I}_{\para{P}_4}(\frac{1}{2})
				&  |[name=p5122]| \text{I}_{\para{P}_5}(\frac{1}{2}) &  |[name=p2534]| \text{I}_{\para{P}_2}(\frac{5}{2})
				& & |[name=p3526]| \text{I}_{\para{P}_3}(\frac{5}{2})
				&  &  
				
				\\
				&&
				& |[name=p7338]| \text{I}_{\para{P}_7}(\frac{3}{2})
				&  &  |[name=p152]| \text{I}_{\para{P}_1}(\frac{5}{2})

				\\
				&
				|[name=p11746]| \text{I}_{\para{P}_1}(\frac{17}{2})
				& &&
				& |[name=p11346]| \text{I}_{\para{P}_1}(\frac{13}{2})
				\\
				|[name=p4718]| \text{I}_{\para{P}_4}(\frac{7}{2})
				&
				|[name=p21334]| \text{I}_{\para{P}_2}(\frac{13}{2})
				& 
				|[name=p6114]| \text{I}_{\para{P}_6}(1)
				&  |[name=p2134]| \text{I}_{\para{P}_2}(\frac{1}{2}) 
				& |[name=p5722]| \text{I}_{\para{P}_5}(\frac{7}{2})
				& |[name=p8112]| \text{I}_{\para{P}_8}(\frac{11}{2})
				\\
				&
				|[name=p81958]| \text{I}_{\para{P}_8}(\frac{19}{2})		
				& 
				&&& |[name=p71138]| \text{I}_{\para{P}_7}(\frac{11}{2})

				\\
				& |[name=p2734]| \text{I}_{\para{P}_2}(\frac{7}{2})
				&&&
				|[name=p3726]| \text{I}_{\para{P}_3}(\frac{7}{2})
				\\
				|[name=p5522]| \text{I}_{\para{P}_5}(\frac{5}{2}) & &
				|[name=p1746]| \text{I}_{\para{P}_1}(\frac{7}{2})
				& |[name=p4518]| \text{I}_{\para{P}_4}(\frac{5}{2})
				&& |[name=p812]| \text{I}_{\para{P}_8}(\frac{1}{2})

				\\
				& |[name=p6314]| \text{I}_{\para{P}_6}(3)
				&&&
				|[name=p7938]| \text{I}_{\para{P}_7}(\frac{9}{2})
				\\};
			
			\draw[->]      
			(p4518)  edge [->]
			node[text style,above]{$\widetilde{w}_{4,7}$}  (p7938)
			(p7938)  edge [->]
			node[text style,above]{$\widetilde{w}_{7,8}$}  (p812)
			(p3726)  edge [->]
			node[text style,above]{$\widetilde{w}_{3,8}$}  (p812)
			(p4518)  edge [->]
			node[text style,above]{$\widetilde{w}_{4,3}$}  (p3726)
			
			(p4718)  edge [->]
			node[text style,above]{$\widetilde{w}_{4,8}$}  (p81958)
			(p4718)  edge [->]
			node[text style,above]{$\widetilde{w}_{4,1}$}  (p11746)
			(p4718)  edge [->]
			node[text style,above]{$\widetilde{w}_{4,2}$}  (p21334)
			
			(p5722)  edge [->]
			node[text style,above]{$\widetilde{w}_{5,8}$}  (p8112)
			(p5722)  edge [->]
			node[text style,above]{$\widetilde{w}_{5,7}$}  (p71138)
			(p5722)  edge [->]
			node[text style,above]{$\widetilde{w}_{5,1}$}  (p11346)
			
			(p5522)  edge [->]
			node[text style,above]{$\widetilde{w}_{5,2}$}  (p2734)
			(p5522)  edge [->]
			node[text style,above]{$\widetilde{w}_{5,6}$}  (p6314)
			(p2734)  edge [->]
			node[text style,above]{$\widetilde{w}_{2,1}$}  (p1746)
			(p6314)  edge [->]
			node[text style,above]{$\widetilde{w}_{6,1}$}  (p1746)
			
			(p2534)  edge [->]
			node[text style,above]{$\widetilde{w}_{2,7}$}  (p7338)
			(p2534)  edge [->]
			node[text style,above]{$\widetilde{w}_{2,1}$}  (p112)
			
			(p3526)  edge [->]
			node[text style,above]{$\widetilde{w}_{3,7}$}  (p7538)
			(p3526)  edge [->]
			node[text style,above]{$\widetilde{w}_{3,1}$}  (p152)
			
			(p416)  edge [->]
			node[text style,above]{$\widetilde{w}_{4,6}$}  (p617)
			(p617)  edge [->]
			node[text style,above]{$\widetilde{w}_{6,7}$}  (p712)
			
			(p5322)  edge [->]
			node[text style,above]{$\widetilde{w}_{5,3}$}  (p3326)
			(p3326)  edge [->]
			node[text style,above]{$\widetilde{w}_{3,2}$}  (p2334)
			
			(p4118)  edge [->]
			node[text style,above]{$\widetilde{w}_{4,5}$}  (p5122)
			
			(p6114)  edge [->]
			node[text style,above]{$\widetilde{w}_{6,2}$}  (p2134)
			;
		\end{tikzpicture}
	\end{center}
	\caption{Siegel-Weil Identities for $E_8$ - Square Integrable Residues}
	\label{digram:E8_SW_SI}
\end{figure}

\begin{figure}[h!]
	\begin{center}
		\begin{tikzpicture}[thin,scale=2]
			\tikzset{
				text style/.style={
					sloped, 
					text=black,
					font=\tiny,
					above
				}
			}
			\matrix (m) [matrix of math nodes, row sep=1em, column sep=3em]
			{ 		
				|[name=p11146]| \text{I}_{\para{P}_1}(\frac{11}{2})
				&  |[name=p8558]| \text{I}_{\para{P}_8}(\frac{5}{2})
				&  |[name=p292]| \text{I}_{\para{P}_2}(\frac{9}{2})
				&  |[name=p832]| \text{I}_{\para{P}_8}(\frac{3}{2})
				&  |[name=p21134]| \text{I}_{\para{P}_2}(\frac{11}{2})
				&  |[name=p81358]| \text{I}_{\para{P}_8}(\frac{13}{2})
				
				\\
				|[name=p312]| \text{I}_{\para{P}_3}(\frac{1}{2})
				&  |[name=p60]| \text{I}_{\para{P}_6}(0) 
				&  |[name=p376]| \text{I}_{\para{P}_3}(\frac{7}{6})
				&  |[name=p256]| \text{I}_{\para{P}_2}(\frac{5}{6})
				& |[name=p32]| \text{I}_{\para{P}_3}(2)
				&  |[name=p71]| \text{I}_{\para{P}_7}(1)
				
				\\
				|[name=p31126]| \text{I}_{\para{P}_3}(\frac{11}{2})
				&  |[name=p11946]| \text{I}_{\para{P}_1}(\frac{19}{2})
				& |[name=p3926]| \text{I}_{\para{P}_3}(\frac{9}{2})
				&  |[name=p81558]| \text{I}_{\para{P}_8}(\frac{15}{2})
				&  |[name=p4310]| \text{I}_{\para{P}_4}(\frac{3}{10})
				&  |[name=p5110]| \text{I}_{\para{P}_5}(\frac{1}{10})
				
				\\
				|[name=p434]| \text{I}_{\para{P}_4}(\frac{3}{4})
				&  |[name=p314]| \text{I}_{\para{P}_3}(\frac{1}{4})
				&  |[name=p456]| \text{I}_{\para{P}_4}(\frac{5}{6})
				&  |[name=p613]| \text{I}_{\para{P}_6}(\frac{1}{3})
				&  |[name=p41]| \text{I}_{\para{P}_4}(1)
				&  |[name=p31]| \text{I}_{\para{P}_3}(1)
				
				\\
				|[name=p476]| \text{I}_{\para{P}_4}(\frac{7}{6})
				&  |[name=p643]| \text{I}_{\para{P}_6}(\frac{4}{3})
				&  |[name=p42]| \text{I}_{\para{P}_4}(2)
				&  |[name=p73]| \text{I}_{\para{P}_7}(3)
				&  |[name=p556]| \text{I}_{\para{P}_5}(\frac{5}{6})
				&  |[name=p316]| \text{I}_{\para{P}_3}(\frac{1}{6})

				\\
				|[name=p51]| \text{I}_{\para{P}_5}(1)
				&  |[name=p612]| \text{I}_{\para{P}_6}(\frac{1}{2})
				&  |[name=p576]| \text{I}_{\para{P}_5}(\frac{7}{6})
				&  |[name=p216]| \text{I}_{\para{P}_2}(\frac{1}{6})
				&  |[name=p52]| \text{I}_{\para{P}_5}(2)
				&  |[name=p11]| \text{I}_{\para{P}_1}(1)
				
				\\
				|[name=p5922]| \text{I}_{\para{P}_5}(\frac{9}{2})
				&  |[name=p82158]| \text{I}_{\para{P}_8}(\frac{21}{2})
				&  |[name=p652]| \text{I}_{\para{P}_6}(\frac{5}{2})
				&  |[name=p12]| \text{I}_{\para{P}_1}(2)
				&  |[name=p64]| \text{I}_{\para{P}_6}(4)
				&  |[name=p872]| \text{I}_{\para{P}_8}(\frac{7}{2})
				
				\\
				|[name=p6514]| \text{I}_{\para{P}_6}(5)
				&  |[name=p71338]| \text{I}_{\para{P}_7}(\frac{13}{2})
				&  |[name=p637]| \text{I}_{\para{P}_6}(6)
				&  |[name=p82358]| \text{I}_{\para{P}_8}(\frac{23}{2})
				& |[name=p71738]| \text{I}_{\para{P}_7}(\frac{17}{2})
				&  |[name=p82558]| \text{I}_{\para{P}_8}(\frac{25}{2})
				
				\\};
			
			\draw[->]      
			(p11146)  edge [->]
			node[text style,above]{$\widetilde{w}_{1,8}$}  (p8558)
			
			(p21134)  edge [->]
			node[text style,above]{$\widetilde{w}_{2,8}$}  (p81358)
			
			(p3926)  edge [->]
			node[text style,above]{$\widetilde{w}_{3,8}$}  (p81558)
			
			(p5922)  edge [->]
			node[text style,above]{$\widetilde{w}_{5,8}$}  (p82158)
			
			(p637)  edge [->]
			node[text style,above]{$\widetilde{w}_{6,8}$}  (p82358)
			
			(p71738)  edge [->]
			node[text style,above]{$\widetilde{w}_{7,8}$}  (p82558)

			(p6514)  edge [->]
			node[text style,above]{$\widetilde{w}_{6,7}$}  (p71338)
			
			
			(p292)  edge [->]
			node[text style,above]{$\widetilde{w}_{2,8}$}  (p832)
			
			(p376)  edge [->]
			node[text style,above]{$\widetilde{w}_{3,2}$}  (p256)
			
			(p32)  edge [->]
			node[text style,above]{$\widetilde{w}_{3,7}$}  (p71)
			
			(p4310)  edge [->]
			node[text style,above]{$\widetilde{w}_{4,5}$}  (p5110)
			
			(p434)  edge [->]
			node[text style,above]{$\widetilde{w}_{4,3}$}  (p314)
			
			(p312)  edge [->]
			node[text style,above]{$\widetilde{w}_{3,6}$}  (p60)
			
			(p456)  edge [->]
			node[text style,above]{$\widetilde{w}_{4,6}$}  (p613)
			
			(p476)  edge [->]
			node[text style,above]{$\widetilde{w}_{4,6}$}  (p643)
			
			(p42)  edge [->]
			node[text style,above]{$\widetilde{w}_{4,7}$}  (p73)
			
			(p556)  edge [->]
			node[text style,above]{$\widetilde{w}_{5,3}$}  (p316)
			
			(p576)  edge [->]
			node[text style,above]{$\widetilde{w}_{5,2}$}  (p216)
			
			(p52)  edge [->]
			node[text style,above]{$\widetilde{w}_{5,1}$}  (p11)
			
			(p652)  edge [->]
			node[text style,above]{$\widetilde{w}_{6,1}$}  (p12)
			
			(p64)  edge [->]
			node[text style,above]{$\widetilde{w}_{6,8}$}  (p872)
			
			(p41)  edge [->]
			node[text style,above]{$\widetilde{w}_{4,3}$}  (p31)
			
			(p51)  edge [->]
			node[text style,above]{$\widetilde{w}_{5,6}$}  (p612)
			
			(p31126)  edge [->]
			node[text style,above]{$\widetilde{w}_{3,1}$}  (p11946)
			
			;
		\end{tikzpicture}
	\end{center}
	\caption{Siegel-Weil Identities for $E_8$ - Non-square Integrable Residues}
	\label{digram:E8_SW_NSI}
\end{figure}

\chapter[Arthur Parameters]{Arthur Parameters of the Square-Integrable Degenerate Residual Representations}
\label{Chap:Arthur}

In this chapter, we wish to compare the results of \Cref{Chap:Square_Integrable} with the Arthur conjectures  for $L^2_{\coset{\bT}}$.
We start by a brief description of Arthur parameters and conjectures.
Please see \cite{MR1021499} for more details.
Also, the reader may wish to consider \cite{MR1847140,MR1426903,MR2262172} for other detailed examples.

\section{Arthur Parameters and the Arthur Conjectures}

Arthur's conjectures aims at a precise description of the discrete spectrum $L^2_{dis.}\bk{ G\bk{F}\lmod G\bk{\AA}}$ of $G\bk{\AA}$.
More precisely, the conjectures first imply a rough decomposition
\[
L^2_{dis.}\bk{ Z_\AA G_F\lmod G\bk{\AA}} = \widehat{\oplus}_\psi L^2\bk{ G\bk{F} \lmod G\bk{\AA}}_\psi,
\]
where $\psi$ runs through the set of global A-parameters, that is the set of $\check{G}$-conjugacy classes of admissible maps
\[
\psi : \Lfun_F \times SL_2\bk{\C} \to \check{G} ,
\]
where $\Lfun_F$ is the conjectural Langlands group of the number field $F$.

The conjectures, then give a description of the space $L^2_\psi = L^2\bk{ G\bk{F}\lmod G\bk{\AA}}_\psi$.

First, we describe the local constituents $\pi_v$ of the irreducible constituents $\pi=\otimes_{v\in \Places}\pi_v$ of $L^2_\psi$.
Since the local Weil-Deligne group $W_{F_v}$ embeds into $\Lfun_F$ (according to the conjectured properties of $\Lfun_F$), the global A-parameter $\psi$ gives rise to the local A-parameter
\[
\psi_v = \psi\res{W_{F_v}\times SL_2\bk{\C}} .
\]
Denote $C_{\psi_v} = Cent\bk{im \psi, \check{G}}$ and let $S_{\psi_v} = C_{\psi_v}\rmod C_{\psi_v}^0$.
$S_{\psi_v}$ is called the component group associated to $\psi$.
Note that $S_{\psi_v}$ is a finite group.

Arthur's conjectures state, among other things, that:
\begin{itemize}
	\item There exists a set of irreducible representations $\Pi_{\psi_v}$, called the local A-packet of $\psi_v$, with a bijection
	\begin{equation}
		\label{Eq:Local_Arthur_cor}
		\begin{split}
			\widehat{S_{\psi_v}} & \to  \Pi_{\psi_v} \\
			\eta_v & \mapsto  \pi_{\eta_v}
		\end{split}
	\end{equation}
	
	\item If $\eta_v = \Id_v$, then $\pi_{\eta_v}$ is unramified with Satake parameter
	\begin{equation}
		\label{Eq:Satake_parameter_for_Arthur_parameter}
		\sigma_{\psi_v} = \psi_v \bk{Fr_v, \begin{pmatrix} q_v^{\frac12} & \\ & q_v^{-\frac12} \end{pmatrix}},
	\end{equation}
	where $Fr_v$ is the Frobenius element of $W_{F_v} \hookrightarrow \Lfun_{F_v}$.
	
	\item The global A-packet of $\psi$, defined by
	\[
	\Pi_\psi = \set{\pi=\otimes_{v\in\Places} \pi_{\eta_v} \mvert \pi_{\eta_v}\in \Pi_{\psi_v}, \ \eta_v=\Id_v \text{ for almost all $v$} },
	\]
	is in bijection with with the set of irreducible representations of the compact set $S_{\psi,\AA} = \prodl_{v\in\Places} S_{\psi_v}$.
	We write this bijection as follows
	\[
	\begin{split}
		\widehat{S_{\psi,\AA}} & \to \Pi_\psi \\
		\eta = \otimes_{v\in \Places} \eta_v & \mapsto \pi_\eta =  \otimes_{v\in\Places} \pi_{\eta_v} .
	\end{split}
	\]
	Note that $S_{\psi}$ embeds diagonally into $S_{\psi,\AA}$.
	
	\item Arthur attached a quadratic character $\epsilon_\psi$ of $S_\psi$ such that, for $\eta\in \widehat{S_{\psi,\AA}}$, the multiplicity of $\eta$ in $L^2\bk{S_{\psi,\AA}}$ and of $\pi_\eta$ in $L^2_{dis.}$ is given by
	\begin{equation}
		\label{Eq:Arthurs_multiplicity_formula}
		m_\eta = \frac{1}{\Card{S_\psi}} \bk{\suml_{s\in S_\psi} \epsilon_\psi \bk{s} \cdot \eta\bk{s}} .
	\end{equation}
	In particular, if $\mu=1$, then $\epsilon_\psi=1$.
	\item In particular, if $S_{\psi}$ is finite (in which case, $\psi$ is called \emph{elliptic}), then $\pi_\eta=\otimes \pi_{\eta_v}$ occurs in the discrete spectrum if and only if
	\begin{equation}
		\label{Eq:Arthur_test_for_discrete_series}
		\suml_{s\in S_\psi} \prodl_{v\in\Places}\eta_v\bk{s_v}  \neq 0 ,
	\end{equation}
	where we write $s=\otimes s_v$.
\end{itemize}

In order to analyze an Arthur parameter $\psi: \Lfun_F \times SL_2\bk{\C} \to \check{G}$, we need to describe the two commuting maps
\[
\begin{split}
	& \phi:SL_2\bk{\C}\to\check{G}, \quad \phi\bk{g}=\psi\bk{1,g} \\
	& \mu:\Lfun_F \to \check{G} , \quad \mu\bk{x} = \psi\bk{x,1} .
\end{split}
\]
First of all, we note that, by the Jacobson-Morozov Theorem (see \cite{MR1503258,MR0007750}), the $\check{G}$-conjugacy class of $\phi$ is determined by $X = \phi\begin{pmatrix} 1&1\\0&1 \end{pmatrix}$.

After fixing $X$ and $\phi$, we note that $\mu$ is, in fact, a map $\mu:\Lfun_F \to C_\phi$, where $C_\phi = Cent\bk{im\ \phi, \check{G}}$.

We also write $C_X = Cent\bk{X, \check{G}}$ and let $U^X$ denote the unipotent radical of $U_X$.
By \cite[Proposition 2.4]{MR782556}, it holds that
\begin{itemize}
	\item $C_\phi$ is reductive and $C_X = C_\phi \cdot U^X$ is a Levi decomposition.
	\item The inclusion $C_\phi \subseteq C_X$ induces an isomorphism on the component groups, $C_\phi\rmod C_\phi^0 Z_{\check{G}} \cong C_X\rmod C_X^0 Z_{\check{G}}$.
\end{itemize}

For $X$ as above, let $\gen{X, Y, H}$ be an $SL_2$-triple as in the Jacobson-Morozov theorem such that $H=\phi\begin{pmatrix} q&\\&q^{-1} \end{pmatrix}$.

We denote by $\mu_k$ the cyclic group $\ZZ\rmod\bk{k\ZZ}$ and by $S_k$ we denote the symmetric group of order $k!$.

Also, it should be added that all data regarding nilpotent orbits and their centralizers, can be read from the tables in \cite[page 402-407]{MR1266626}.

\section[Trivial $\chi$]{Unipotent Arthur Parameters for Degenerate Residual Representations of Groups of Type $E_n$ - Trivial $\chi$}

If $\chi=\Id$, \Cref{Eq:Satake_parameter_for_Arthur_parameter} implies that $\mu=\Id$, that is $\psi\res{\Lfun_F}=\Id$.
Hence, $\psi$ is determined by $X = \phi\begin{pmatrix} 1&1\\0&1 \end{pmatrix}$.
In other words, it is determined by the Satake parameter $\sigma_{\psi_v}$, which gives rise to the dominant weight $\lambda_{dom.} = \sigma_{\psi_v}$ of $T$.
That is, $\pi_{\Id_v}$ is the unique irreducible quotient of $i_{T\bk{F_v}}^{G\bk{F_v}}\lambda_{dom.}$.

It should be pointed out that, since we consider $\check{G}$-conjugacy classes of maps $\phi$, they are parameterized by distinguished nilpotent orbit $\check{G}$ and $\lambda_{dom.}$ can be read directly from the Weighted Dynkin diagrams.
More precisely, Given a unipotent orbit $\mO$, the theory of nilpotent attaches to it a Jacobson-Morozov triple $\gen{X, Y, H}$ and a labeled Dynkin diagram.
The semi-simple element $H=\phi\begin{pmatrix}q&\\&q^{-1}\end{pmatrix}$ is associated both with the weighted diagram and with $\sigma_{\psi_v}$.
By definition, the weight $\epsilon_i$ on the node $\alpha_i\in\Delta$, in the weighted Dynkin diagram, is given by $\epsilon_i = \gen{\alpha_i,H}$.
Thus, it holds that
\[
\lambda_{dom.} = \frac12\suml_{\alpha_i\in\Delta} \epsilon_i \fun{i} .
\]
Also, note that a distinguished orbit is even and hence $\epsilon_i\in\set{0,2}$.

\subsection{Case of $E_6$}

For the group of type $E_6$, we list the distinguished nilpotent orbits in $\check{E_6}$, the associated Satake parameter and data $\bk{i,s_0}$ such that $\ResRep{E_6}{i}{s_0}{\Id}$ realizes parts of $L^2_\psi$.
We also list the type of the component group $S_\psi$.

\begin{center}
	\begin{longtable}[h]{|c|c|c|c|}
			\hline
			Orbit & $\lambda_{dom.}$ & $\bk{i,s_0}$ & $S_\psi$
			\\ \hline \endhead \hline
			$E_6$ & $\esixchar{1}{1}{1}{1}{1}{1}$ & 
			\begin{tabular}{c}
				$\bk{1,6}$, $\bk{2,\frac{11}{2}}$, $\bk{3,\frac92}$, \\ $\bk{4,\frac72}$, $\bk{5,\frac92}$, $\bk{6,6}$
			\end{tabular}
			& 1 \\ \hline
			$E_6\bk{a_1}$ & $\esixchar{1}{1}{1}{0}{1}{1}$ & $\bk{1,3}$, $\bk{2,\frac72}$, $\bk{4,\frac52}$, $\bk{6,3}$ & 1 \\ \hline
			$E_6\bk{a_3}$ & $\esixchar{1}{0}{0}{1}{0}{1}$ & $\bk{2,\frac12}$, $\bk{3,\frac32}$, $\bk{4,\frac32}$, $\bk{5,\frac32}$ & $S_2$ \\ \hline
	\end{longtable}
	\label{Table:E6_Arthur}
\end{center}

\begin{itemize}
	\item The residual representation associated with the orbit $E_6$ is the trivial one.
	This residue is irreducible, in accordance with Arthur's conjectures.
	\item The residual representation associated with the orbit $E_6\bk{a_1}$ is the minimal one.
	This residue is irreducible, in accordance with Arthur's conjectures.
	\item The residual representation associated with the orbit $E_6\bk{a_3}$ is irreducible.
	According to Arthur's conjectures, the local packet $\Pi_{\psi_v} = \set{\pi_{1,v}, \pi_{sgn,v}}$ contains two irreducible representations.
	By \Cref{Eq:Arthurs_multiplicity_formula}, both will appear in the discrete spectrum.
	Indeed, for a finite subset $S\subset\Places$, let $\eta_S=\otimes_{v\in S} \sgn_v \otimes_{v\notin S} \Id_v$.
	Then,
	\[
	m_\eta = \piece{0, & \Card{S}\equiv 1 \pmod{2} \\ 
		1, & \Card{S}\equiv 0 \pmod{2}} .
	\]
	However, only $S_\emptyset$ appears in the degenerate principal series.
	By Kim's conjecture, the other $\pi_{\eta_S}$ will appear in the residual spectrum, but only as residues of Eisenstein series induced from smaller parabolic subgroups.
\end{itemize}

Note that, as opposed to the following cases, the next to minimal representation does not appear here.
This has to do with the fact that the next-to-minimal nilpotent orbit, $D_5$, is not distinguished and the next-to-minimal residual representation is not square-integrable.

\subsection{Case of $E_7$}

For the group of type $E_7$, we list the data in a similar fashion to the $E_6$ case.
The added information appearing here is that we mark the cases with a reducible residue by an asterisk.

\begin{center}
	\begin{longtable}[h]{|c|c|c|c|}
		\hline
		Orbit & $\lambda_{dom.}$ & $\bk{i,s_0}$ & $S_\psi$
		\\ \hline \endhead \hline
		$E_7$ & $\esevenchar{1}{1}{1}{1}{1}{1}{1}$ & 
		\begin{tabular}{c}
			$\bk{1,\frac{17}{2}}$, $\bk{2,7}$, $\bk{3,\frac{11}{2}}$, $\bk{4,4}$, \\
			$\bk{5,5}$, $\bk{6,\frac{13}{2}}$, $\bk{7,9}$
		\end{tabular}
		 & 1 \\ \hline
		$E_7\bk{a_1}$ & $\esevenchar{1}{1}{1}{0}{1}{1}{1}$
		& $\bk{1,\frac{11}{2}}$, $\bk{2,5}$, $\bk{4,3}$, $\bk{7,5}$ & 1 \\ \hline
		$E_7\bk{a_2}$ & $\esevenchar{1}{1}{1}{0}{1}{0}{1}$
		& $\bk{1,\frac72}$, $\bk{5,3}$, $\bk{6,\frac72}$, $\bk{7,1}$ & 1 \\ \hline
		$E_7\bk{a_3}$ & $\esevenchar{1}{0}{0}{1}{0}{1}{1}$
		& $\bk{1,\frac12}$, $\bk{3,\frac52}$, $\bk{4,2}$, $\bk{6,\frac52}$ & $S_2$ \\ \hline
		$E_7\bk{a_4}$ & $\esevenchar{1}{0}{0}{1}{0}{0}{1}$
		& $\bk{3,\frac32}^\ast$, $\bk{6,\frac12}^\ast$, $\bk{2,1}$ & $S_2$ \\ \hline
		$E_7\bk{a_5}$ &  $\esevenchar{0}{0}{0}{1}{0}{0}{1}$
		& $\bk{3,\frac12}$, $\bk{4,1}$, $\bk{5,1}$ & $S_3$ \\ \hline
	\end{longtable}
	\label{Table:E7_Arthur}
\end{center}

\begin{itemize}
	\item The residual representation associated with the orbit $E_7$ is the trivial one.
	This residue is irreducible, in accordance with Arthur's conjectures.
	\item The residual representation associated with the orbit $E_7\bk{a_1}$ is the minimal one.
	This residue is irreducible, in accordance with Arthur's conjectures.
	\item The residual representation associated with the orbit $E_7\bk{a_2}$ is the next-to-minimal one.
	This residue is irreducible, in accordance with Arthur's conjectures.
	\item In the case of the orbit $E_7\bk{a_3}$, one can carry an analysis similar to that of the orbit $E_6\bk{a_3}$ of $E_6$.
	\item The residual representation associated with the orbit $E_7\bk{a_4}$ is not irreducible, while the residual representations at $\bk{3,\frac32}$ and $\bk{6,\frac12}$ are.
	By Arthur's conjecture, the local packets $\Pi_{\psi_v}=\set{\pi_{1,v}, \pi_{\sgn,v}}$ should consist of two irreducible representation.
	Together with the multiplicities
	\[
	m_\eta = \piece{0, & \Card{S}\equiv 1 \pmod{2} \\ 
		1, & \Card{S}\equiv 0 \pmod{2}} ,
	\]
	where $\eta_S=\otimes_{v\in S} \sgn_v \otimes_{v\notin S} \Id_v$.
	Indeed, that has been proved to be the case in \Cref{Prop:Residue_E7_2_1_1}.
	That is, $L^2_\psi=\ResRep{E_7}{2}{1}{\Id}$.
	In particular, the eigenvalue calculated in \Cref{Appendix:Local_data_for_global_SI_resisudes} could be inferred from the multiplicity formula.
	
	We point out that $\ResRep{E_7}{3}{\frac32}{\Id}$ and $\ResRep{E_7}{6}{\frac12}{\Id}$ only realize $\pi_{\Id}$.

	\item The residual representation associated with the orbit $E_7\bk{a_5}$ is irreducible.
	However, here $S_\psi=S_3$ and Arthur's and Kim's conjectures expect that $\Pi_{\psi_v}=\set{\pi_{\Id,v}, \pi_{\sgn,v}, \pi_{\tau,v}}$.
	Here, $\tau$ denotes the irreducible $2$-dimensional representation of $S_3$.
	For $\dot{S}=S_{\sgn}\cupdot S_\tau$, write
	\[
	\eta_{\dot{S}} = \otimes_{v\notin \dot{S}} \Id_v \otimes_{v\in S_{\sgn}} \sgn_v \otimes_{v\in S_\tau} \tau_v.
	\]
	Then, the expected multiplicity of $\eta_{\dot{S}}$ in $L^2_{dis.}$ is
	\[
	m_{\eta_{\dot{S}}} = \piece{
	0, & \Card{S_\tau}=0,\ \Card{S_{\sgn}}\equiv 1 \pmod{2} \\
	1, & \Card{S_\tau}=0,\ \Card{S_{\sgn}}\equiv 0 \pmod{2} \\
	0, & \Card{S_\tau} = 1 \\
	\frac{2^{\Card{S_{\tau}}}+2\cdot\bk{-1}^{\Card{S_{\tau}}}}{6}, & \Card{S_{\tau}} > 1 . }
	\]
	It is interesting to point out that for $\Card{S_\tau}>1$, the multiplicities $m_\eta$ forms an unbounded sequence of numbers.
	Such a sequence is known to occur in the cuspidal spectrum of $G_2$ (see \cite{MR2192822,MR1918673} for more information).
	Such a sequence is also expected to appear in the cuspidal spectrum of $E_7$.
\end{itemize}

\subsection{Case of $E_8$}

For the group of type $E_8$, we list the data in a similar fashion to the $E_6$ and $E_7$ cases.
We also added the superscript $?$ to the representations $\bk{2,\frac12}$ and $\bk{5,\frac12}$ as the length of the maximal semi-simple quotient was not determined in this case.

\begin{center}
	\begin{longtable}[h]{|c|c|c|c|}
		\hline
		Orbit & $\lambda_{dom.}$ & $\bk{i,s_0}$ & $S_\psi$
		\\ \hline \endhead \hline
		$E_8$ & $\eeightchar{1}{1}{1}{1}{1}{1}{1}{1}$ & 
		\begin{tabular}{c}
			$\bk{1,\frac{23}{2}}$, $\bk{2,\frac{17}{2}}$, $\bk{3,\frac{13}{2}}$, $\bk{4,\frac{9}{2}}$, \\ $\bk{5,\frac{11}{2}}$, $\bk{6,7}$, $\bk{7,\frac{19}{2}}$, $\bk{8,\frac{29}{2}}$
		\end{tabular}
		& 1 \\ \hline
		$E_8\bk{a_1}$ & $\eeightchar{1}{1}{1}{0}{1}{1}{1}{1}$ 
		& $\bk{1,\frac{17}{2}}$, $\bk{2,\frac{13}{2}}$, $\bk{4,\frac72}$, $\bk{8,\frac{19}{2}}$
		& 1 \\ \hline
		$E_8\bk{a_2}$ & $\eeightchar{1}{1}{1}{0}{1}{0}{1}{1}$ 
		& $\bk{1,\frac{13}{2}}$, $\bk{5,\frac{7}{2}}$, $\bk{7,\frac{11}{2}}$, $\bk{8,\frac{11}{2}}$
		& 1 \\ \hline
		$E_8\bk{a_3}$ & $\eeightchar{1}{0}{0}{1}{0}{1}{1}{1}$ 
		& $\bk{3,\frac72}$, $\bk{4,\frac52}$, $\bk{7,\frac92}$, $\bk{8,\frac12}$
		& $S_2$ \\ \hline
		$E_8\bk{a_4}$ & $\eeightchar{1}{0}{0}{1}{0}{1}{0}{1}$ & $\bk{1,\frac72}$, $\bk{2,\frac72}$, $\bk{5,\frac52}$, $\bk{6,3}$
		& $S_2$ \\ \hline
		$E_8\bk{b_4}$ & $\eeightchar{1}{0}{0}{1}{0}{0}{1}{1}$ & $\bk{3,\frac52}^\ast$, $\bk{7,\frac52}^\ast$, $\bk{1,\frac52}$
		& $S_2$ \\ \hline
		$E_8\bk{a_5}$ & $\eeightchar{1}{0}{0}{1}{0}{0}{1}{0}$ & $\bk{1,\frac12}^\ast$, $\bk{2,\frac52}^\ast$, $\bk{7,\frac32}$
		& $S_2$ \\ \hline
		$E_8\bk{b_5}$ & $\eeightchar{0}{0}{0}{1}{0}{0}{1}{1}$ & $\bk{4,\frac32}$, $\bk{6,2}$, $\bk{7,\frac12}$
		& $S_3$ \\ \hline
		$E_8\bk{a_6}$ & $\eeightchar{0}{0}{0}{1}{0}{0}{1}{0}$ & $\bk{2,\frac32}$, $\bk{3,\frac32}$, $\bk{5,\frac32}$
		& $S_3$ \\ \hline
		$E_8\bk{b_6}$ & $\eeightchar{0}{0}{0}{1}{0}{0}{0}{1}$ & $\bk{2,\frac12}^{?}$, $\bk{6,1}$
		& $S_3$ \\ \hline
		$E_8\bk{a_7}$ & $\eeightchar{0}{0}{0}{0}{1}{0}{0}{0}$ & $\bk{4,\frac12}$, $\bk{5,\frac12}^{?}$
		& $S_5$ \\ \hline
	\end{longtable}
	\label{Table:E8_Arthur}
\end{center}

\begin{itemize}
	\item The residual representation associated with the orbit $E_8$ is the trivial one.
	This residue is irreducible, in accordance with Arthur's conjectures.
	\item The residual representation associated with the orbit $E_8\bk{a_1}$ is the minimal one.
	This residue is irreducible, in accordance with Arthur's conjectures.
	\item The residual representation associated with the orbit $E_8\bk{a_2}$ is the next-to-minimal one.
	This residue is irreducible, in accordance with Arthur's conjectures.
	\item In the cases of the orbits $E_8\bk{a_3}$ and $E_8\bk{a_4}$, one can carry an analysis similar to that of the orbit $E_6\bk{a_3}$ of $E_6$.
	\item In the cases of the orbits $E_8\bk{b_4}$ and $E_8\bk{a_5}$, one can carry an analysis similar to that of the orbit $E_7\bk{a_4}$ of $E_7$.
	\item In the cases of the orbits $E_8\bk{b_5}$ and $E_8\bk{a_6}$, one can carry an analysis similar to that of the orbit $E_7\bk{a_5}$ of $E_7$.
	\item For the residual representations associated with the orbits $E_8\bk{b_6}$ and $E_8\bk{a_7}$, one may carry similar calculations.
	Two issues should be pointed out though.
	One is due to the fact that the structure of $\ResRep{E_8}{2}{\frac12}{\Id}$ and $\ResRep{E_8}{5}{\frac12}{\Id}$ is not yet determined.
	The other is that for $E_8\bk{a_7}$ the component group $S_\psi$ is $S_5$; we will not carry an analysis of the expected structure of the associated space $L_\psi^2$ but since this is the only nilpotent orbit of algebraic groups whose component group is not $S_3$ or a product of copies of $S_2$, one would expect $L_\psi^2$ to have an exceptionally interesting structure.
\end{itemize}

\section[Non-trivial $\chi$]{Unipotent Arthur Parameters for Degenerate Residual Representations of Groups of Type $E_n$ - Non-trivial $\chi$}

An endoscopic group $H$ of $G$ is, roughly speaking, a quasi-split group such that $\ldual{H} = Cent\bk{s, \ldual{G}}^0$ for some semi-simple element $s\in\ldual{G}$.
Since we are dealing with split connected groups, we may replace replace $\ldual{G}$ and $\ldual{H}$ with $\check{G}$ and $\check{H}$ in the definition and following discussion.

For a square-integrable residue $\ResRep{G}{i}{s_0}{\chi}$, with $\chi\neq\Id$, we expect that the corresponding Arthur parameter $\psi$ will factor through $\check{H}$ for some endoscopic group $H$ of $G$.
\[
\xymatrix{
	\psi: W_F \times SL_2\bk{\C} \ar[dr]_{\psi'=\bk{\mu',\phi'}}  \ar[rr] & & \check{G} \\
	& \check{H} \ar[ur] & 
}
\]
In particular, we will describe the parameter so that $\phi'$ corresponds to a semi-regular unipotent orbit of $\check{H} \rmod \cZ\bk{\check{H}}$ (as $\check{H}$ might not be adjoint).
Thus, the centralizer of $\phi'$ will be central in $\check{H}$.

Therefore, in order to describe the parameter $\psi$, it is enough to describe the map to $\check{H}$ and the embedding $\check{H}\hookrightarrow\check{G}$.
For our purposes (with one exception), it would be enough to use subgroups $\check{H}$ which are pseudo-Levi subgroups of $\check{G}$ of maximal rank.

\begin{Rem}
	Comparing with \cite[page 219]{MR32659}, note that pseudo-Levi subgroup of maximal rank need not be a maximal subgroup.
	For example, $E_7$ admits a pseudo-Levi subgroup of type $A_3\times A_3\times A_1$ which has rank $8$ but can be realized as a pseudo-Levi subgroup of the pseudo Levi subgroup of type $D_6\times A_1$ (as groups of type $D_6$ admits a maximal subgroup of type $A_3\times A_3$.)
	Similarly, in $E_8$ we have the following inclusions (written in types) $A_7\times A_1\subset E_7\times A_1$, $A_5\times A_2\times A_1\subset E_7\times A_1$ and $D_5\times A_3\subset D_8$.
\end{Rem}

In what follows, we choose $\check{H}$ using the results of \cite[\S 4]{MR1631769}.
However, it should be pointed out that to study various factorial lifts from and to $E_n$, one would want to use other endoscopic subgroups $H$ to realize the parameter.
Other lists of subgroups of $\check{G}$ that meets its unipotent orbits can be found in other works such as \cite{MR294349,MR2044850,MR1048074,MR1264015}.

Write $\frakg=Lie\bk{\check{H}}$.
In \cite{MR1631769}, Sommers describes a correspondence
\[
\begin{split}
	&\set{\bk{N,C} \mvert \begin{matrix} \text{$N\in \frakg$ is nilpotent} \\ \text{$C$ is a conjugacy class in $A\bk{N}$} \end{matrix} } \\
	& \longleftrightarrow
	\set{\bk{\frakh,\mO} \mvert \begin{matrix}	 \text{$\frakh$ is a pseudo-Levi subalgebra of $\frakg$} \\ \text{$\mO$ is a distinguished unipotent orbit in $\frakh$} \end{matrix} },
\end{split}
\]
where $A\bk{N}$ denotes the component group $C_N\rmod C_N^0$ of the centralizer of $N$.
In fact, it is shown in \cite{MR1631769} that the component group $\cZ\bk{\check{H}}\rmod \cZ\bk{\check{H}}^0$ is generated by the image of a semi-simple element in $C$.

In order to describe the parameter, we proceed as follows:
\begin{itemize}
	\item To a residual representations $\ResRep{G}{i}{s_0}{\chi}$ we will associate a nilpotent orbit $\mO_G$ of $\check{G}$ using the Satake parameter at unramified places and \Cref{Eq:Satake_parameter_for_Arthur_parameter}.
	In fact, for this part, one only needs to consider the real part of the Satake parameter as in the case of trivial $\chi$.
	
	\item To $\mO_G$ we will associate a pseudo-Levi subgroup $\check{H}$ of $\check{G}$ so that the intersection between $\mO_G$ and $\check{H}$ is a semi-regular unipotent orbit of $\check{H}\rmod \cZ\bk{\check{H}}$.
	The group $\check{H}$ and the orbit $\mO_{\check{H}} = \mO_G\cap\check{H}$ are determined by the tables in \cite[\S 4]{MR1631769}; the choice is such that the corresponding element of $A\bk{N}$ is of order $ord\bk{\chi}$.
	This determines the map $\phi'$.
	
	\item The map $\mu'$ then factors through the center of $\check{H}$ and can be determined by class field theory.
	
	\item Lastly, we describe the isomorphism class of the component group $S_\psi$.
	Bellow, we have two cases :
	\begin{itemize}
		\item If $C_{\phi}$ is of type $S_n$ (see \cite[page 402-407]{MR1266626}), then $S_\psi$ is the subgroup of $A\bk{N}=S_\phi$ which centralizes $C$.
		These are well known for $S_2$, $S_3$ and $S_5$ (the relevant cases).
		
		\item If $C_{\phi}$ is of type $T_d\rtimes S_n$, then $\mu'$ is mapped to $T_d$ and $S_\psi$ is the stabilizer in $S_n$ of the image of $\mu'$.
	\end{itemize}
	
\end{itemize}

\begin{Rem}
	A list of semiregular orbits which are non-regular (non-principal that is) can be found in 	\cite[Lemma 4.4]{MR308228}.
	Such orbits only occur in types $D_n$ and $E_n$.
\end{Rem}

It is useful to recall that the types of pseudo-Levi subgroups of $G$ can be read from the extended Dynkin diagram (see \cite[\S 2]{MR1631769} for a full account).
The extended Dynkin diagram is the diagram attained by adding a node associated with the negative highest root $\alpha_0$ of $\check{G}$ to the usual Dynkin diagram.
The added root $\alpha_0$ is given by
\[
\alpha_0 = \piece{-\fun{2}, & G=E_6 \\ -\fun{1}, & G=E_7 \\ -\fun{8}, & G=E_8 .}
\]
The resulting diagrams for the various $E_n$ will be given bellow.
Maximal pseudo-Levi subgroups correspond to subdiagrams attained by removing one node from the extended Dynkin diagram.

\subsection{Case of $E_6$}

Recall the extended Dynkin diagram of type $E_6$:
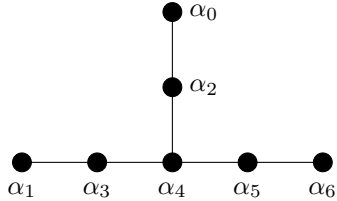
\[\begin{tikzpicture}[scale=0.5]
	\draw (-1,0) node[anchor=east]{};
	\draw (0 cm,0) -- (8 cm,0);
	\draw (4 cm, 0 cm) -- +(0,4 cm);
	\draw[fill=black] (0 cm, 0 cm) circle (.25cm) node[below=4pt]{$\alpha_1$};
	\draw[fill=black] (2 cm, 0 cm) circle (.25cm) node[below=4pt]{$\alpha_3$};
	\draw[fill=black] (4 cm, 0 cm) circle (.25cm) node[below=4pt]{$\alpha_4$};
	\draw[fill=black] (6 cm, 0 cm) circle (.25cm) node[below=4pt]{$\alpha_5$};
	\draw[fill=black] (8 cm, 0 cm) circle (.25cm) node[below=4pt]{$\alpha_6$};
	\draw[fill=black] (4 cm, 2 cm) circle (.25cm) node[right=3pt]{$\alpha_2$};
	\draw[fill=black] (4 cm, 4 cm) circle (.25cm) node[right=3pt]{$\alpha_0$};
\end{tikzpicture}\]

In the following table, we list the pseudo-Levi subgroups $\check{H}$ of $\check{G}$ of maximal rank (up to conjugacy), all of which are relevant to our calculation.
For each such pseudo-Levi subgroup we list its type, the root $\frakR$ that needs to be removed from the extended Dynkin diagram in order to attain it, its isomorphism class and its center.

\begin{center}
	\begin{longtable}[h]{|c|c|c|c|}
		\hline
		$Type\bk{\check{H}}$ & $\frakR$ & $\check{H}$ & $\cZ\bk{\check{H}}$ 
		\\ \hline \endhead \hline
		$A_5\times A_1$ & $\alpha_2$ & $\bk{SL_6\bk{\C}\rmod \mu_3 \times SL_2\bk{\C}}\rmod \mu_2$ & $\mu_2$ \\ \hline
		$A_2\times A_2\times A_2$ & $\alpha_4$ & $\bk{SL_3\bk{\C}}^3\rmod \mu_3^2$ & $\mu_3$ \\ \hline
	\end{longtable}
	\label{Table:E6_proper_maximal_cuspidal_endoscopic_groups}
\end{center}

In the following table, we sort the square-integrable residual representations $\ResRep{G}{i}{s_0}{\chi}$ with non-trivial $\chi$, sorted by their associated nilpotent orbit, together with the Satake parameter at unramified places.
Note that we write here $\chi$ for $\chi_v\bk{\unif_v}$.
We also list the complex dual $\check{H}$ of the associated endoscopic group.
In both cases, $\phi'$ is associated with the principal orbit in $\check{H}$.
Finally, 

\begin{center}
	\begin{longtable}[h]{|c|c|c|c|c|c|}
		\hline
		Orbit & $ord\bk{\chi}$ & $\bk{i,s_0}$ & $Type\bk{\check{H}}$ & $\mO_{\check{H}}$ & $S_\psi$
		\\ \hline \endhead \hline
		$E_6\bk{a_3}$ & $2$ & $\bk{2,\frac12}$, $\bk{3,\frac32}$, $\bk{4,\frac32}$, $\bk{5,\frac32}$ & $A_5\times A_1$ & $A_5\oplus A_1$ & $1$ \\ \hline
		$D_4\bk{a_1}$ & $3$ & $\bk{4,\frac12}^\ast$ & $A_2\times A_2\times A_2$ & $A_2\oplus A_2\oplus A_2$ & $\mu_3$ \\ \hline
	\end{longtable}
	\label{Table:E6_Arthur_Chi_Non-trivial}
\end{center}

The data in this table should be read as follows:
\begin{itemize}
	\item The unipotent orbit $E_6\bk{a_3}$ of $E_6$ corresponds to the principal orbit $A_5\oplus A_1$ in $\check{H}=\bk{SL_6\bk{\C}\rmod \mu_3 \times SL_2\bk{\C}}\rmod \mu_2$ of type $A_5\times A_1$.
	It holds that $S_\phi=\cZ\bk{\check{H}} = \cZ\bk{SL_6\bk{\C}\rmod \mu_3} \cong \mu_2$.
	Thus, for a quadratic Hecke character $\chi$, we take $\psi$ to be
	\[
	\psi\bk{w,g} 
	= \chi\bk{w} \bk{\Id_6, \Id_2 } \phi\bk{g},
	\]
	where, by abuse of notations, we identify between $\chi$ and the map $\chi: W_F\to \mu_2\subset \C^\times$ corresponding to it by class field theory.
	In particular, $\widehat{S_\psi} = \set{1}$.
	
	We point out that at places where $\chi_v=\Id_v$, the local Arthur parameter is the same as in the case where $\chi=\Id$ and the corresponding representation in the local Arthur packet is the unique irreducible quotient of $\DPS{P_i}\bk{s_0,\Id_v}$.
	If, on the other hand, $\chi_v\neq\Id_v$, then the corresponding representation in the local Arthur packet is the unique irreducible quotient of $\DPS{P_i}\bk{s_0,\chi_v}$.
	
	Taking $\eta=\Id$, \Cref{Eq:Arthurs_multiplicity_formula} implies that $m_\eta=1$ and hence
	\[
	\begin{split}
		L^2_\psi
		& = \ResRep{E_6}{2}{\frac12}{\chi} 
		= \ResRep{E_6}{3}{\frac32}{\chi} \\
		& = \ResRep{E_6}{4}{\frac32}{\chi} 
		= \ResRep{E_6}{5}{\frac32}{\chi}.
	\end{split}
	\]
		
	\item Note that the orbit $D_4\bk{a_1}$ of $E_6$ is not distinguished but it is quasi-distinguished, even and special.
	By \cite[pg. 402]{MR1266626}, we have $C_\phi^0$ is a 2-dimensional torus and the component group of the centralizer is $C_\phi\rmod C_\phi^0 \cong S_3$.
	This orbit corresponds to the principal orbit $A_2\oplus A_2\oplus A_2$ in $A_2\times A_2\times A_2$.
	Hence, the corresponding parameters will factor through $\check{H}=\bk{SGL_3\bk{\C}}^3\rmod \mu_3^2$.
	Thus, for a cubic Hecke character $\chi$, we take $\psi$ to be
	\[
	\psi\bk{w,g} 
	 = \bk{\coset{\begin{matrix}\chi\bk{w}&&\\&\chi\bk{w}&\\&&\chi\bk{w} \end{matrix}}, \coset{\begin{matrix}\chi\bk{w}&&\\&\chi\bk{w}&\\&&\chi\bk{w} \end{matrix}},\coset{\begin{matrix}\chi\bk{w}&&\\&\chi\bk{w}&\\&&\chi\bk{w}\end{matrix}}} \phi\bk{g}
	\]
	Hence, $C_\psi\rmod C_\psi^0 \cong \mu_3$.
	This implies that:
	\begin{itemize}
		\item For a place $v$ such that $\chi_v\neq\Id_v$, it holds that $\Pi_{\psi_v}=\set{\pi_{\Id,v}, \pi_{\kappa,v}, \pi_{\kappa^2,v}}s$, where $\set{\Id, \kappa, \kappa^2}$ are the three irreducible representations of $\mu_3$.
		\item For a place where $\chi_v=\Id_v$, it holds that $S_\psi \cong S_3$ and $\Pi_{\psi_v}=\set{\pi_{\Id,v}, \pi_{\sgn,v}, \pi_{\tau,v}}$.
		In the degenerate residual spectrum we encounter only $\pi_{\Id,v}$ though.
	\end{itemize}
	Let $S_\kappa$ denote the places where $\chi_v\neq\Id_v$ and $\eta_v=\kappa$ and let $S_{\kappa^2}$ denote the places where $\chi_v\neq\Id_v$ and $\eta_v=\kappa^2$.	
	The multiplicity formula \Cref{Eq:Arthurs_multiplicity_formula} then yields
	\[
	m_\eta = \piece{1,& \Card{S_\kappa}\equiv \Card{S_{\kappa^2}}\pmod{3} \\ 0,& \Card{S_\kappa}\not\equiv \Card{S_{\kappa^2}}\pmod{3},}
	\]
	which is consistent with \Cref{Prop:E6_4_1-2_3_Residue}.
	In particular,
	\[
	L^2_\psi
	= \ResRep{E_6}{4}{\frac12}{\chi} .
	\]

\end{itemize}

\subsection{Case of $E_7$}

The extended Dynkin diagram of type $E_7$ is 
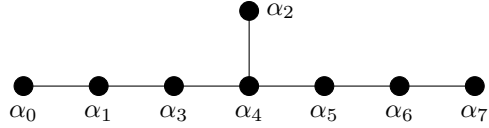
\[\begin{tikzpicture}[scale=0.5]
	\draw (-1,0) node[anchor=east]{};
	\draw (-2 cm,0) -- (10 cm,0);
	\draw (4 cm, 0 cm) -- +(0,2 cm);
	\draw[fill=black] (0 cm, 0 cm) circle (.25cm) node[below=4pt]{$\alpha_1$};
	\draw[fill=black] (2 cm, 0 cm) circle (.25cm) node[below=4pt]{$\alpha_3$};
	\draw[fill=black] (4 cm, 0 cm) circle (.25cm) node[below=4pt]{$\alpha_4$};
	\draw[fill=black] (6 cm, 0 cm) circle (.25cm) node[below=4pt]{$\alpha_5$};
	\draw[fill=black] (8 cm, 0 cm) circle (.25cm) node[below=4pt]{$\alpha_6$};
	\draw[fill=black] (10 cm, 0 cm) circle (.25cm) node[below=4pt]{$\alpha_7$};
	\draw[fill=black] (4 cm, 2 cm) circle (.25cm) node[right=3pt]{$\alpha_2$};
	\draw[fill=black] (-2 cm, 0 cm) circle (.25cm) node[below=4pt]{$\alpha_0$};
\end{tikzpicture}\]

In the following table, we list the pseudo-Levi subgroups $\check{H}$ of $\check{G}$ which are relevant to our calculation as in the $E_6$ case.

\begin{center}
	\begin{longtable}[h]{|c|c|c|c|}
		\hline
		$Type\bk{\check{H}}$ & $\frakR$ & $\check{H}$ & $\cZ\bk{\check{H}}$ \\ \hline \endhead \hline
		$D_6\times A_1$ & $\alpha_1$ & $\coset{\bk{Spin_{12}\bk{\C}\rmod \mu_2}\times SL_2\bk{\C}}\rmod \mu_2$ & $\mu_2$ \\ \hline
		$A_7$ & $\alpha_2$ & $SL_8\bk{\C}\rmod \mu_4$ & $\mu_2$ \\ \hline
		$A_5\times A_2$ & $\alpha_3$ & $\bk{\bk{SL_6\bk{\C}\rmod \mu_2} \times SL_3\bk{\C}}\rmod \mu_3$ & $\mu_3$ \\ \hline
		$A_1\times A_3\times A_3$ & $\alpha_4$ & $\coset{SL_4\bk{\C} \times \bk{SL_4\bk{\C}\times SL_2\bk{\C}} \rmod \mu_2 } \rmod \mu_4$ & $\mu_4$ \\ \hline
	\end{longtable}
	\label{Table:E7_proper_maximal_cuspidal_endoscopic_groups}
\end{center}
In the following tables, we sort data on square-integrable residues $\ResRep{G}{i}{s_0}{\chi}$, with non-trivial $\chi$, similar to the $E_6$ case.

\begin{center}
	\begin{longtable}[h]{|c|c|c|c|c|c|}
		\hline
		Orbit & $ord\bk{\chi}$ & $\bk{i,s_0}$ & $Type\bk{\check{H}}$ & $\mO_{\check{H}}$ & $S_\psi$
		\\ \hline \endhead \hline
		$E_7\bk{a_3}$ & $2$ &  $\bk{1,\frac12}$, $\bk{3,\frac52}$, $\bk{4,2}$, $\bk{6,\frac52}$ & $D_6\times A_1$ & $D_6\oplus A_1$ & $1$ \\ \hline
		$E_7\bk{a_4}$ & $2$ & $\bk{3,\frac32}$, $\bk{6,\frac12}$ & $D_6\times A_1$ & $D_6\bk{a_1}\oplus A_1$ & $1$ \\ \hline
		$E_7\bk{a_5}$ & $2$ & $\bk{3,\frac12}$, $\bk{4,1}$ & $D_6\times A_1$ & $D_6\bk{a_2}\oplus A_1$ & $1$ \\ 
		& $3$ & $\bk{3,\frac12}$, $\bk{4,1}$, $\bk{5,1}$ & $A_5\times A_2$ & $A_5\oplus A_2$ & $1$ \\ \hline 
		$E_6\bk{a_1}$ & $2$ & $\bk{2,2}^\ast$, $\bk{5,2}^\ast$ & $A_7$ & $A_7$ & $S_2$ \\ \hline 
		$A_4+A_1$ & $4$ & $\bk{4,\frac12}^\ast$ & $A_1\times A_3\times A_3$ & $A_1\oplus A_3\oplus A_3$ & $S_2$ \\ \hline 
	\end{longtable}
	\label{Table:E7_Arthur_Chi_Non-trivial-Part2}
\end{center}

The construction of the Arthur parameters for $E_7$ is analogues to the one explained in the $E_6$ case using the data from this table. 
It should be pointed out that for the orbits $E_6\bk{a_1}$ and $A_4+A_1$, the associated local degenerate principal series for non-trivial $\chi_v$ have a maximal semi-simple quotient of length $2$.
In the rest of the cases, the maximal semi-simple quotient is irreducible.

\subsection{Case of $E_8$}

The extended Dynkin diagram of type $E_7$ is 
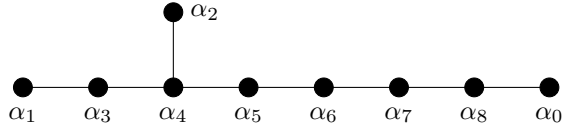
\[\begin{tikzpicture}[scale=0.5]
	\draw (-1,0) node[anchor=east]{};
	\draw (0 cm,0) -- (14 cm,0);
	\draw (4 cm, 0 cm) -- +(0,2 cm);
	\draw[fill=black] (0 cm, 0 cm) circle (.25cm) node[below=4pt]{$\alpha_1$};
	\draw[fill=black] (2 cm, 0 cm) circle (.25cm) node[below=4pt]{$\alpha_3$};
	\draw[fill=black] (4 cm, 0 cm) circle (.25cm) node[below=4pt]{$\alpha_4$};
	\draw[fill=black] (6 cm, 0 cm) circle (.25cm) node[below=4pt]{$\alpha_5$};
	\draw[fill=black] (8 cm, 0 cm) circle (.25cm) node[below=4pt]{$\alpha_6$};
	\draw[fill=black] (10 cm, 0 cm) circle (.25cm) node[below=4pt]{$\alpha_7$};
	\draw[fill=black] (12 cm, 0 cm) circle (.25cm) node[below=4pt]{$\alpha_8$};
	\draw[fill=black] (4 cm, 2 cm) circle (.25cm) node[right=3pt]{$\alpha_2$};
	\draw[fill=black] (14 cm, 0 cm) circle (.25cm) node[below=4pt]{$\alpha_0$};
\end{tikzpicture}\]

In the following table, we list the pseudo-Levi subgroups $\check{H}$ of $\check{G}$ which are relevant to our calculation as in the $E_6$ case.

\begin{center}
	\begin{longtable}[h]{|c|c|c|c|}
		\hline
		$Type\bk{\check{H}}$ & $\frakR$ & $\check{H}$ & $\cZ\bk{\check{H}}$ \\ \hline \endhead \hline
		$D_8$ & $\alpha_1$ & $Spin_{16}\bk{\C}\rmod \mu_2$ & $\mu_2$ \\ \hline
		$A_8$ & $\alpha_2$ & $SL_9\bk{\C}\rmod\mu_3$ & $\mu_3$ \\ \hline
		$A_7\times A_1$ & $\alpha_3$ & $\coset{\bk{SL_8\bk{\C}\rmod \mu_2} \times SL_2\bk{\C}} \rmod \mu_2$ & $\mu_4$ \\ \hline
		$A_5\times A_2 \times A_1$ & $\alpha_4$ & $\bk{ SL_6\bk{\C} \times SL_3\bk{\C} \times SL_2\bk{\C} } \rmod \mu_6$ & $\mu_6$ \\ \hline
		$A_4\times A_4$ & $\alpha_5$ & $\bk{SL_5\bk{\C}\times SL_5\bk{\C}}\rmod \mu_5$ & $\mu_5$ \\ \hline
		$D_5\times A_3$ & $\alpha_6$ & $\bk{Spin_{10}\bk{\C} \times SL_4\bk{\C}}\rmod \mu_4$ & $\mu_4$ \\ \hline		
		$E_6\times A_2$ & $\alpha_7$ & $\bk{E_6^{sc}\bk{\C}\times SL_3\bk{\C}} \rmod \mu_3$ & $\mu_3$ \\ \hline
		$E_7\times A_1$ & $\alpha_8$ & $\bk{E_7^{sc}\bk{\C}\times SL_2\bk{\C}}\rmod \mu_2$ & $\mu_2$ \\ \hline
		
	\end{longtable}
	\label{Table:E8_proper_maximal_cuspidal_endoscopic_groups}
\end{center}

In the following table, we sort data on square-integrable residues $\ResRep{G}{i}{s_0}{\chi}$, with non-trivial $\chi$, as in the $E_6$ case.
As in the case of trivial $\chi$, we mark the cases $\bk{i,s_0,\chi}$ with $s_0=\frac12$ and $i=2,5$ with a question mark superscript to mark the fact that its maximal semi-simple quotient might be reducible.

\begin{center}
	\begin{longtable}[h]{|c|c|c|c|c|c|}
		\hline
		Orbit & $ord\bk{\chi}$ & $\bk{i,s_0}$ & $Type\bk{\check{H}}$ & $\mO_{\check{H}}$ & $S_\psi$
		\\ \hline \endhead \hline
		$E_8\bk{a_3}$ & $2$ & $\bk{3,\frac72}$, $\bk{4,\frac52}$, $\bk{7,\frac92}$, $\bk{8,\frac12}$ & $E_7\times A_1$ & $E_7\oplus A_1$ & $1$ \\ \hline
		$E_8\bk{a_4}$ & $2$ & $\bk{1,\frac72}$, $\bk{2,\frac72}$, $\bk{5,\frac52}$, $\bk{6,3}$ & $D_8$ & $D_8$ & $1$ \\ \hline
		$E_8\bk{b_4}$ & $2$ & $\bk{3,\frac52}$, $\bk{7,\frac52}$ & $E_7\times A_1$ & $E_7\bk{a_1}\oplus A_1$ & $1$ \\ \hline
		$E_8\bk{a_5}$ & $2$ & $\bk{1,\frac12}$, $\bk{2,\frac52}$ & $D_8$ & $D_8\bk{a_1}$ & $1$ \\ \hline
		$E_8\bk{b_5}$ & $2$ & $\bk{4,\frac32}$, $\bk{7,\frac12}$ & $E_7\times A_1$ & $E_7\bk{a_2}\oplus A_1$ & $1$ \\ 
		& $3$ & $\bk{4,\frac32}$, $\bk{6,2}$, $\bk{7,\frac12}$ & $E_6\times A_2$ & $E_6\oplus A_2$ & $1$ \\ \hline
		$E_8\bk{a_6}$ & $2$ & $\bk{2,\frac32}$, $\bk{5,\frac32}$ & $D_8$ & $D_8\bk{a_2}$ & $1$ \\ 
		& $3$ & $\bk{2,\frac32}$, $\bk{3,\frac32}$, $\bk{5,\frac32}$ & $A_8$ & $A_8$ & $1$ \\ \hline 
		$E_8\bk{b_6}$ & $2$ & $\bk{2,\frac12}^{?}$, $\bk{6,1}$ & $D_8$ & $D_8\bk{a_3}$ & $1$ \\ 
		& $3$ & $\bk{6,1}$ & $E_6\times A_2$ & $E_6\bk{a_1}\oplus A_2$ & $1$ \\ \hline 
		$E_8\bk{a_7}$ & 
		$2$ & $\bk{5,\frac12}^{?}$ & $E_7\times A_1$ & $E_7\bk{a_5}+A_1$ & $S_3$ \\ 
		& $2$ & $\bk{4,\frac12}$ & $D_8$ & $D_8\bk{a_5}$ & $\mu_4$ \\ 
		& $3$ & $\bk{4,\frac12}$ & $E_6\times A_2$ & $E_6\bk{a_3} \oplus A_2$ & $S_2$ \\ 
		& $4$ & $\bk{5,\frac12}^{?}$ & $D_5\times A_3$ & $D_5\bk{a_1}\oplus A_3$ & $1$ \\ 
		& $5$ & $\bk{4,\frac12}$, $\bk{5,\frac12}^{?}$ & $A_4\times A_4$ & $A_4 \oplus A_4$ & $1$ \\ 
		& $6$ & $\bk{4,\frac12}$ & $A_5\times A_2 \times A_1$ & $A_5 \oplus A_2 \oplus A_1$ & $1$ \\ \hline 
		$D_5+A_2$ & $2$ & $\bk{3,\frac12}^\ast$ & $E_7\times A_1$ & $E_7\bk{a_4} \oplus A_1$ & $S_2$ \\ \hline 
		$E_6\bk{a_1}+A_1$ & $4$ & $\bk{3,1}$, $\bk{4,1}$ & $A_7\times A_1$ & $A_7\oplus A_1$ & $S_2$ \\ \hline 
		$D_7\bk{a_2}$ & $2$ & $\bk{5,1}$, $\bk{6,\frac12}$ & $D_5\times A_3$ & $D_5\oplus A_3$ & $S_2$ \\ 
		& $4$ & $\bk{5,1}$, $\bk{6,\frac12}$ & $D_5\times A_3$ & $D_5\oplus A_3$ & $1$ \\ \hline 
	\end{longtable}
	\label{Table:E8_Arthur_Chi_Non-trivial-Part2}
\end{center}

The construction of the Arthur parameters for $E_8$ is analogues to the one explained in the $E_6$ case using the data from this table.
It should be pointed out that for the orbit $D_5+A_2$, the associated local degenerate principal series for non-trivial $\chi_v$ have a maximal semi-simple quotient of length $2$.
In the rest of the cases, the maximal semi-simple quotient is irreducible.

There are however a couple of peculiarities that should be noted here:
\begin{itemize}
	\item In \cite{MR1631769}, other pseudo-Levi subgroups that intersect the orbit $E_8\bk{a_7}$ are listed for whose center has order $2$ or $3$.
	These are pseudo-Levi subgroups are not listed here as the intersection is not a semi-regular orbit of the pseudo-Levi subgroup.
	
	\item For the orbit $D_5+A_2$ above, we listed the distinguished orbit $E_7\bk{a_4}\oplus A_1$ for $\mO_{\check{H}}$.
	This is, in fact, not a semi-regular orbit as the semi-simple part of its centralizer is $S_2$.
	As was seen in the $E_7$ case, $E_7\bk{a_4}$ intersects the semi-regular orbit $D_6\bk{a_1}\oplus A_1$ of the pseudo-Levi $D_6\times A_1$.
	Thus, the Arthur parameter associated with the orbit $E_7\bk{a_4}\oplus A_1$ of $E_8$ should be factored here through the subgroup $D_6\times A_1^2$ using the orbit $D_6\bk{a_1}\oplus 2A_1$.
	This, however, is not a standard pseudo-Levi and thus it is simpler to realize this parameter by ``factoring in steps'' as explained here.
	
	\item Similarly, for the orbit $E_8\bk{a_7}$ with orders $2$ and $3$, one should further factor the orbit as follows:
	\begin{itemize}
		\item The case of $\bk{?,\frac12,2}$ factors through the orbit $E_7\bk{a_5}+A_1$ in $E_7\times A_1$. The orbit $E_7\bk{a_5}$, in turn, intersects the orbit $D_6\bk{a_2}+A_1$ in $D_6\times A_1$.

		\item In the case of $\chi$ of order $3$, the parameter factors through the orbit $E_6\bk{a_3}+A_2$ in $E_6\times A_2$. Since the orbit $E_6\bk{a_3}$ of $E_6$ intersects the principal orbit of $A_5\times A_1$, the parameter factors through $A_5\times A_2\times A_1$ via its principal orbit.
		
	\end{itemize}
	
	\item Note that the orbits $E_8\bk{b_6}$ and $E_8\bk{a_7}$ are associated with residual representations $\ResRep{G}{i}{s_0}{\chi}$ with $s_0=\frac12$ and $i=2$ or $5$ (among other data).
	In these cases, when $S_\psi=1$, we point out that this should be interpreted as follows:
	\begin{itemize}
		\item If $\chi_v\neq \Id_v$, then $S_{\psi,v}=1$.
		\item IF $\chi_v=\Id_v$, then $S_{\psi,v}=S_2$.
	\end{itemize}
\end{itemize}

\chapter{Non-Square-Integrable Degenerate Residual Spectrum}
\label{Chap:Non_square_Integrable}

\section{The Method}

\subsection{Siegel-Weil Identities}
For most cases, where $\chi=\Id$ and the residue of $\Eisen_{P_i}$ is not square-integrable, we are able to use the results of \Cref{Chap:Siegel_Weil} in order to compute this residue.

Assume that $\bk{P_i,s_0}$ and $\bk{P_j,t_0}$ are in the position described in \Cref{Chap:Siegel_Weil} and that $\Res{s=s_0} \Eisen_{P_i}$ (and hence also $\Res{s=t_0} \Eisen_{P_j}$) is not square-integrable.
In most such cases, we are able to write
\[
\Res{s=s_0} \Eisen_{P_i} = \Res{s=t_0} \Eisen_{P_j},
\]
where $\DPS{P_j}\bk{t_0}$ is irreducible and $\Eisen_{P_j}$ is holomorphic and non-vanishing at $t_0$.
It then follows that
\begin{equation}
\Res{s=s_0} \Eisen_{P_i} = \DPS{P_j}\bk{t_0} .
\end{equation}

In a number of cases, it happens that $\DPS{P_j}\bk{t_0}$ is reducible and $\Eisen_{P_j}$ admits a pole at $t_0$, yet $\DPS{P_i}\bk{s_0}$ and $\DPS{P_j}\bk{t_0}$ share a unique irreducible subquotient $\tau$. In such a case, $\tau$ is their irreducible spherical quotient and one can still use a Siegel-Weil identity which would yield:
\begin{equation}
\label{Eq:Residues_using_SW_Id}
\ResRep{G}{i}{s_0}{\Id} = \ResRep{G}{j}{t_0}{\Id} = \tau .
\end{equation}
This, for example, is used for the calculation of the residue in the case $\coset{E7,5,2,1}$ using the residue in the case of $\coset{E7,2,2,1}$ which can be easily calculated in a more direct manner.

\subsection{Analysis of Images of Intertwining Operators}

In the remaining cases, we make a more direct use of \Cref{eq::CT_second_form} and \Cref{eq::sum_over_eq_cl} in order to calculate the residue $\Res{s=s_0}\Eisen_{P}\bk{\chi}$.
More precisely, it follows from \Cref{eq::CT_second_form} that
\[
\Res{s=s_0}\Eisen_{P_i} \bk{f,s,\chi,g}_T = 
\Res{s=s_0}\bk{\suml_{\coset{w'}\in W^{M,T}\rmod\sim_{s_0,\chi}}  \kappa_{\coset{w'}}\bk{f,g}}.
\]
In particular, if the order of $\Eisen_{P_i}\bk{\chi}$ at $s_0$ is $k$, denote by $W^{M,T}_{s_0,\chi,res.}$ the set of equivalence classes $\coset{w'}$ such that $\kappa_{\coset{w'}}$ also have order $k$ there.
It follows that
\[
\Res{s=s_0}\Eisen_{P_i} \bk{f,s,\chi,g}_T = \suml_{\coset{w'}\in W^{M,T}_{s_0,\chi,res.}} \Res{s=s_0}\kappa_{\coset{w'}}\bk{f,g}
\]
and hence $\Res{s=s_0}\Eisen_{P_i} \bk{f,s,\chi}\neq 0$ if and only if $\Res{s=s_0}\kappa_{\coset{w'}}\bk{f_s,g}\neq 0$ for at least one $\coset{w'}\in W^{M,T}_{s_0,\chi,res.}$.
In particular, for a factorizable section $f_s=\otimes_{v\in \Places}f_{s,v}\in \DPS{P}\bk{s,\chi}$ and $\coset{w'}\in W^{M,T}_{s_0,\chi,res.}$ such that $\Res{s=s_0}\kappa_{\coset{w'}}\bk{f_s,g}\neq 0$ and hence, for at least one $w\in\coset{w'}$, it holds that $N_{w,v}\bk{s_0,\chi}f_{s,v}\neq 0$.
In other words, the image of $N_{w,v}\bk{s_0,\chi}$ contains $w\cdot\lambda_0$ as an exponent.

Let
\[
W^{shortest}_{\bk{P_i,s_0,\chi}} = \set{
w\in W^{M,T} \mvert 
\begin{array}{l}
\coset{w}\in W^{M,T}_{s_0,\chi,res.} \\
l\bk{w'w^{-1}} > l\bk{w}-l\bk{w'} \quad \forall \coset{w'}\in W^{M,T}_{s_0,\chi,res.},\quad w'\neq w
\end{array}}.
\]
That is, $W^{shortest}_{\bk{P_i,s_0,\chi}}$ denotes the set of Weyl elements in $W^{M,T}$ such that $\kappa_{\coset{w}}$ has order $k$ and there are no other elements $w'\in W^{M,T}$ with the same property such that $w$ can be written as a reduced expression of the form $u\cdot w'$.

In \Cref{Appendix:Local_data_for_global_NSI_resisudes}, we study the structure of $\DPS{P}\bk{s_0,\chi}_v$ and the images of $N_w\bk{\lambda_{s,\chi}^{P}}_v$ for $w\in W^{shortest}_{\bk{P_i,s_0,\chi}}$.
This is used in the calculation of such residues as the images of all $N_w\bk{s_0,\chi}_v$, for $w\in W^{M,T}_{s_0,\chi,res.}$, are given by the images of $N_w\bk{s_0,\chi}_v$ for $w\in W^{shortest}_{\bk{P_i,s_0,\chi}}$.

\begin{Rem}
	We note that while the method described in \Cref{Chap:Siegel_Weil} does not accommodate representation with non-trivial $\chi$'s, sometimes such identities could be simply attained for such cases.
	Consider the following situation:
	\begin{itemize}
		\item $\DPS{P_j}\bk{\chi^k,t_0}$ (for some $1\leq k<n$ with $\bk{k,n}=1$) is irreducible.
		\item $\DPS{P_i}\bk{s_0,\chi}$ admits a unique irreducible quotient isomorphic to $\DPS{P_j}\bk{\chi^k,t_0}$.
		\item $\Res{s=s_0 \Eisen_{P_i}\bk{\chi}}$ is isomorphic to the unique irreducible quotient of $\DPS{P_i}\bk{s_0,\chi}$.
	\end{itemize}
	In such a case, it follows, from these assumptions, that
	\[
	\ResRep{G}{i}{s_0}{ord\bk{\chi}} = \DPS{P_j}\bk{\chi^k,t_0} .
	\]
	
	The list of irreducible degenerate principal series representations of $G$ can be found in \Cref{Chap:Local_DPS}.
	
	On the other hand, triples $\coset{P_i,s_0,ord\bk{\chi}}$ and $\coset{P_j,t_0,ord\bk{\chi}}$ which share a common irreducible subquotient can be found using the method described in \cite[Section 3B]{SDPS_E7}. 
	In a nutshell, it happens when $\DPS{P_i}\bk{s_0,\chi}_v$ and $\DPS{P_j}\bk{\chi^k,t_0}_v$ have the same anti-dominant exponents and in such a case, there is a common irreducible subquotient with an anti-dominant exponent.
	A (non-comprehensive) list of such cases can be found in \cite[page 19-20]{SDPS_E6}, \cite[Tables 8-15]{SDPS_E7} and \cite[Tables 9-16]{SDPS_E8}.
\end{Rem}

\section{Results for $E_6$}

\begin{Prop}
	For $\bk{Pi,s_0,P_j,t_0}$ as listed in the following table, it holds that $\Eisen_{P_i}$ admits a simple pole at $s_0$ and $\Eisen_{P_j}$ is holomorphic at $t_0$.
	Furthermore, $\DPS{P_j}\bk{t_0}$ is spherical, irreducible and isomorphic to the unique irreducible quotient of $\DPS{P_i}\bk{t_0}$.
	In particular,
	\[
	\Res{s=s_0} \Eisen_{P_i} = \DPS{P_j}\bk{t_0} .
	\]
	\begin{longtable}[H]{|c|c|}
		\hline
		$\bk{P_i,s_0}$ & $\bk{P_j,t_0}$ \\ \hline
		$\bk{2,\frac{5}{2}}$ & $\bk{1,0}$ and $\bk{6,0}$ \\ \hline
		$\bk{3,\frac{5}{2}}$ & $\bk{1,1}$ \\ \hline
		$\bk{3,\frac{7}{2}}$ & $\bk{6,4}$ \\ \hline
		$\bk{4,1}$ & $\bk{3,0}$ and $\bk{5,0}$ \\ \hline
		$\bk{5,\frac{5}{2}}$ & $\bk{6,1}$ \\ \hline
		$\bk{5,\frac{7}{2}}$ & $\bk{1,4}$ \\ \hline
	\end{longtable}
\end{Prop}

\begin{Prop}
	Let $\chi$ be a quadratic Hecke character.
	\begin{enumerate}
		\item $\Res{s=\frac{1}{2}} \Eisen_{P_4} = \pi_1 = \bigotimes_{v\in\Places} \pi_{1,v}$.
		\item $\Res{s=\frac{1}{2}} \Eisen_{P_4}\bk{\chi} = \pi_1 = \bigotimes_{v\in\Places} \pi_{1,v}$. 
		\item $\Res{s=1} \Eisen_{P_4}\bk{\chi} = \pi_1 = \bigotimes_{v\in\Places} \pi_{1,v} = \DPS{P_3}\bk{0,\chi} = \DPS{P_5}\bk{0,\chi}$.
	\end{enumerate}
\end{Prop}

\begin{proof}
	The proof in all of these cases is simple.
	As explained above, it is enough to consider the images of $\DPS{P}\bk{s_0,\chi}_v$ under $N_w\bk{\lambda_{s,\chi}^{P}}_v$ for $w\in W^{shortest}_{\bk{P_i,s_0,\chi}}$ and $v\in\Places$.
	This is performed in \Cref{Appendix:Local_data_for_global_NSI_resisudes} to the following results:
	\begin{enumerate}
		\item $N_w\bk{\lambda_{\frac{1}{2}}^{P_4}}_v\bk{\DPS{P_4}\bk{\frac{1}{2}}_v} = \DPS{P_4}\bk{\frac{1}{2}}_v\rmod\pi_{0,v}$ for all $v\in\Places$.
		\item $N_w\bk{\lambda_{\chi_v,\frac{1}{2}}^{P_4}}_v\bk{\DPS{P_4}\bk{\frac{1}{2},\chi}_v} = \pi_{1,v}$ for all $v\in\Places$ and $\chi_v^2=\Id$.
		\item $N_w\bk{\lambda_{\chi_v,1}^{P_4}}_v\bk{\DPS{P_4}\bk{1,\chi}_v} = \pi_{1,v}$ for all $v\in\Places$ and $\chi_v^2=\Id$.
	\end{enumerate}
	
	The results then follow from the discussion above.
	
\end{proof}

\section{Results for $E_7$}

\begin{Prop}
	For $\bk{P_i,s_0,P_j,t_0}$ as listed in the following table, it holds that $\Eisen_{P_i}$ admits a simple pole at $s_0$ and $\Eisen_{P_j}$ is holomorphic at $t_0$.
	Furthermore, $\DPS{P_j}\bk{t_0}$ is spherical, irreducible and isomorphic to the unique irreducible quotient of $\DPS{P_i}\bk{t_0}$.
	In particular,
	\[
	\Res{s=s_0} \Eisen_{P_i} = \DPS{P_j}\bk{t_0} .
	\]
	\begin{longtable}[H]{|c|c|}
		\hline
		$\bk{P_i,s_0}$ & $\bk{P_j,t_0}$ \\ \hline
		$\bk{2,3}$ & $\bk{1,\frac{3}{2}}$ \\ \hline
		$\bk{2,4}$ & $\bk{7,2}$ \\ \hline
		$\bk{3,\frac{7}{2}}$ & $\bk{7,3}$ \\ \hline
		$\bk{3,\frac{9}{2}}$ & $\bk{1,\frac{13}{2}}$ \\ \hline
		$\bk{4,\frac{2}{3}}$ & $\bk{5,\frac{1}{3}}$ \\ \hline
		$\bk{4,\frac{3}{2}}$ & $\bk{6,1}$ \\ \hline
		$\bk{5,\frac{3}{2}}$ & $\bk{2,\frac{1}{2}}$ \\ \hline
		$\bk{5,4}$ & $\bk{7,6}$ \\ \hline
		$\bk{6,\frac{11}{2}}$ & $\bk{7,7}$ \\ \hline
	\end{longtable}
\end{Prop}

\begin{Prop} It holds that:
	\begin{enumerate}
		\item
		\begin{enumerate}
			\item $\Res{s=2} \Eisen_{P_2}= \Res{s=2} \Eisen_{P_5}= \pi_1 = \bigotimes_{v\in\Places} \pi_{1,v}$ .
			\item $\Res{s=\frac{1}{2}} \Eisen_{P_4}=\pi_1 = \bigotimes_{v\in\Places} \pi_{1,v}$.
		\end{enumerate}

		\item Let $\chi$ be a quadratic Hecke character.
		\begin{enumerate}
			\item $\Res{s=\frac{1}{2}} \Eisen_{P_4}\bk{\chi}=\pi_1 = \bigotimes_{v\in\Places} \pi_{1,v}$.
			\item $\Res{s=\frac{3}{2}} \Eisen_{P_4}\bk{\chi}=\pi_1 = \bigotimes_{v\in\Places} \pi_{1,v} = \DPS{P_6}\bk{1,\chi}$.
			\item $\Res{s=1} \Eisen_{P_5}\bk{\chi} = N_w\bk{\lambda_1^{P_5}}\bk{\DPS{P_i}\bk{s,\chi}}$, where $w_0=w_{[1, 3, 4, 5, 6, 2, 4, 5, 3, 4, 1, 3, 2, 4, 5]}$. 
			\item $\Res{s=\frac{3}{2}} \Eisen_{P_5}\bk{\chi} = \pi_1 = \bigotimes_{v\in\Places} \pi_{1,v} = \DPS{P_2}\bk{\frac{1}{2},\chi}$.
		\end{enumerate}

		\item Let $\chi$ be a cubic Hecke character.
		Then 
		\[
		\Res{s=\frac{2}{3}} \Eisen_{P_4}\bk{\chi}=\pi_1 = \bigotimes_{v\in\Places} \pi_{1,v} = \DPS{P_5}\bk{\chi^2,\frac{1}{3}}
		\]

	\end{enumerate}
\end{Prop}

\begin{proof}
	Cases (1)(b), (2)(a), (2)(b), (2)(d) and (3) are proven by a similar method to that of $E_6$, while the detailed local calculation can be found in \Cref{Appendix:Local_data_for_global_NSI_resisudes}.
	We now deal with the other cases.
	
	In case (1)(a), the equality $\Res{s=2} \Eisen_{P_2}= \pi_1$ also follows as above.
	The equality $\Res{s=2} \Eisen_{P_2}= \Res{s=2} \Eisen_{P_5}$ follows from \Cref{Eq:Residues_using_SW_Id}.
	
	We now turn to the case (2)(c).
	For a global quadratic character $\chi$, we recall that in half of the places $v\in\Places$ it holds that $\chi_v=\Id_v$.
	Thus, in such places $\DPS{{P_5}}\bk{1,\chi}_v=\DPS{{P_5}}\bk{1,\Id}_v$.
	In particular, the local component of $\ResRep{E_7}{5}{1}{2}$ at half of the places are quotients of $\DPS{{P_5}}\bk{1,\Id}_v$.
	Furthermore, $W^{shortest}_{\bk{P_5,1,\chi}}=\set{w_0}$ which implies our claim.
	As explained in \Cref{Appendix:Local_data_for_global_NSI_resisudes}, for places $v\in\Places$ such that $\chi_v$ is quadratic, $N_w\bk{\lambda_1^{P_5}}_v\bk{\DPS{P_i}\bk{s,\chi}_v}$ is the unique irreducible quotient $\pi_{1,v}$ of $\DPS{P_i}\bk{s,\chi}_v$.
	Also, if $\chi_v=\Id_v$, then $N_w\bk{\lambda_1^{P_5}}_v\bk{\DPS{P_i}\bk{s,\chi}_v}$ is reducible of length at least $2$, with a unique irreducible subrepresentation $\tau_{0,v}$ and a unique irreducible quotient $\pi_{1,v}$. 
\end{proof}

\section{Results for $E_8$}

\begin{Prop}
	For $\bk{Pi,s_0,P_j,t_0}$ as listed in the following table, it holds that $\Eisen_{P_i}$ admits a simple pole at $s_0$ and $\Eisen_{P_j}$ is holomorphic at $t_0$.
	Furthermore, $\DPS{P_j}\bk{t_0}$ is spherical, irreducible and isomorphic to the unique irreducible quotient of $\DPS{P_i}\bk{t_0}$.
	In particular,
	\[
	\Res{s=s_0} \Eisen_{P_i} = \DPS{P_j}\bk{t_0} .
	\]

	\begin{longtable}[H]{|c|c|}
		\hline
		$\bk{P_i,s_0}$ & $\bk{P_j,t_0}$ \\ \hline
		$\bk{1,\frac{11}{2}}$ & $\bk{8,\frac{5}{2}}$ \\ \hline
		$\bk{2,\frac{9}{2}}$ & $\bk{8,\frac{3}{2}}$ \\ \hline
		$\bk{2,\frac{11}{2}}$ & $\bk{8,\frac{13}{2}}$ \\ \hline
		$\bk{3,\frac{7}{6}}$ & $\bk{2,\frac{5}{6}}$ \\ \hline
		$\bk{3,2}$ & $\bk{7,1}$ \\ \hline
		$\bk{3,\frac{9}{2}}$ & $\bk{8,\frac{15}{2}}$ \\ \hline
		$\bk{3,\frac{11}{2}}$ & $\bk{1,\frac{19}{2}}$ \\ \hline
		$\bk{4,\frac{3}{10}}$ & $\bk{5,\frac{1}{10}}$ \\ \hline
		$\bk{4,\frac{3}{4}}$ & $\bk{3,\frac{1}{4}}$ \\ \hline
		$\bk{4,\frac{5}{6}}$ & $\bk{6,\frac{1}{3}}$ \\ \hline
		$\bk{4,\frac{7}{6}}$ & $\bk{6,\frac{4}{3}}$ \\ \hline
		$\bk{4,2}$ & $\bk{7,3}$ \\ \hline
		$\bk{5,\frac{5}{6}}$ & $\bk{3,\frac{1}{6}}$ \\ \hline
		$\bk{5,\frac{7}{6}}$ & $\bk{2,\frac{1}{6}}$ \\ \hline
		$\bk{5,2}$ & $\bk{1,1}$ \\ \hline
		$\bk{5,\frac{9}{2}}$ & $\bk{8,\frac{21}{2}}$ \\ \hline
		$\bk{6,\frac{5}{2}}$ & $\bk{1,2}$ \\ \hline
		$\bk{6,4}$ & $\bk{8,\frac{7}{2}}$ \\ \hline
		$\bk{6,5}$ & $\bk{7,\frac{13}{2}}$ \\ \hline
		$\bk{6,6}$ & $\bk{8,\frac{23}{2}}$ \\ \hline
		$\bk{7,\frac{17}{2}}$ & $\bk{8,\frac{25}{2}}$ \\ \hline
	\end{longtable}
\end{Prop}

\begin{Prop} It holds that:
	\begin{enumerate}
		\item
		\begin{enumerate}
			\item $\Res{s=\frac{1}{2}} \Eisen_{P_3}=\pi_1 = \bigotimes_{v\in\Places} \pi_{1,v}$.
		\end{enumerate}
		
		\item Let $\chi$ be a quadratic Hecke character.
		\begin{enumerate}
			\item $\Res{s=2} \Eisen_{P_3}\bk{\chi}=\pi_1 = \bigotimes_{v\in\Places} \pi_{1,v} = \DPS{P_7}\bk{1,\chi}$.
			\item $\Res{s=\frac34} \Eisen_{P_4}\bk{\chi}=\pi_1 = \bigotimes_{v\in\Places} \pi_{1,v} = \DPS{P_3}\bk{\frac14,\chi}$.
			\item $\Res{s=2} \Eisen_{P_4}\bk{\chi}=\pi_1 = \bigotimes_{v\in\Places} \pi_{1,v} = \DPS{P_7}\bk{3,\chi}$.
			\item $\Res{s=2} \Eisen_{P_5}\bk{\chi}=\pi_1 = \bigotimes_{v\in\Places} \pi_{1,v} = \DPS{P_1}\bk{1,\chi}$.
			\item $\Res{s=\frac52} \Eisen_{P_6}\bk{\chi}=\pi_1 = \bigotimes_{v\in\Places} \pi_{1,v} = \DPS{P_1}\bk{2,\chi}$.
		\end{enumerate}

		\item Let $\chi$ be a cubic Hecke character.
		\begin{enumerate}
			\item $\Res{s=\frac76} \Eisen_{P_3}\bk{\chi}=\pi_1 = \bigotimes_{v\in\Places} \pi_{1,v}$.
			\item $\Res{s=\frac76} \Eisen_{P_4}\bk{\chi}=\pi_1 = \bigotimes_{v\in\Places} \pi_{1,v}$.
			\item $\Res{s=\frac56} \Eisen_{P_5}\bk{\chi}=\pi_1 = \bigotimes_{v\in\Places} \pi_{1,v}$.
			\item $\Res{s=\frac76} \Eisen_{P_5}\bk{\chi}=\pi_1 = \bigotimes_{v\in\Places} \pi_{1,v} = \DPS{P_2}\bk{\frac{1}{6},\chi}$.
		\end{enumerate}
		
		\item Let $\chi$ be a quartic Hecke character.
		\begin{enumerate}
			\item $\Res{s=\frac34} \Eisen_{P_4}\bk{\chi}=\pi_1 = \bigotimes_{v\in\Places} \pi_{1,v}$.
		\end{enumerate}
		
		\item Let $\chi$ be a quintic Hecke character.
		Then $\Res{s=\frac{3}{10}} \Eisen_{P_4}\bk{\chi}=\pi_1 = \bigotimes_{v\in\Places} \pi_{1,v}$.
		
	\end{enumerate}
\end{Prop}

\begin{proof}
	In the cases of (2)-(5) the local degenerate principal series has length $2$ and the proof is similar to that of the $E_6$ cases.
	
	In the case of (1)(a), we note the Siegel-Weil like map
	\[
	\Res{s=\frac{1}{2}} \Eisen_{P_3} = \Res{t=0} \Eisen_{P_6} .
	\]
	This map does not appear in \Cref{Chap:Siegel_Weil} as $\DPS{P_6}\bk{0}$ is not generated by the spherical section.
	However, since $\DPS{P_3}\bk{\frac12}$ is, a similar calculation allows us to us the irreducible spherical constituent of $\DPS{P_6}\bk{0}$.
	
	More precisely, since $\DPS{P_6}\bk{0}$ is unitary, it follows from \cite[Proposition 6]{MR2767521}, that $\Eisen_{P_6}\bk{t}$ is holomorphic at $t_0=0$.
	We also note that, for any $v\in\Places_{fin.}$, we have $\DPS{P_6}\bk{0}_v = \pi_{1,v} \oplus \pi_{-1,v}$, where $\pi_{1,v}$ is spherical and irreducible and $\pi_{-1,v}$ is irreducible and non-spherical.
	Here, $\pi_{1,v}$ is a quotient of $\DPS{P_3}\bk{\frac{1}{2}}_v$ while $\pi_{-1,v}$ is not since
	\[
	\mult{\lambda_0}{r_T^G\bk{\DPS{P_3}\bk{\frac{1}{2}}_v}} = 1,
	\]
	where
	\[
	\lambda_0 = \eeightchar{-1}{-1}{-1}{-1}{-1}{6}{-1}{-1}
	\]
	is the initial exponent of $\DPS{P_6}\bk{0}_v$ and
	\[
	\mult{\lambda_0}{r_T^G\bk{\DPS{P_6}\bk{0}_v}} = 2,\quad
	\mult{\lambda_0}{r_T^G\bk{\pi_{1,v}}} = \mult{\lambda_0}{r_T^G\bk{\pi_{-1,v}}} = 1.
	\]
	
	It follows then that
	\[
	\bk{\Res{s=\frac{1}{2}} \Eisen_{P_3}}_v = \bk{\Res{t=0} \Eisen_{P_6}}_v = \pi_{1,v},
	\]
	from which the claim follows.
	
\end{proof}
	
The following cases remain to be determined:
\begin{itemize}
	\item $\Res{s=1} \Eisen_{P_4}$, $\Res{s=1} \Eisen_{P_3}$, $\Res{s=1} \Eisen_{P_5}$, $\Res{s=\frac{1}{2}} \Eisen_{P_6}$.
	\item $\Res{s=1} \Eisen_{P_3}\bk{\chi}$, $\Res{s=\frac32} \Eisen_{P_3}\bk{\chi}$, $\Res{s=1} \Eisen_{P_4}\bk{\chi}$, $\Res{s=2} \Eisen_{P_6}\bk{\chi}$ - For a quadratic Hecke character $\chi$.
	\item $\Res{s=\frac56} \Eisen_{P_4}\bk{\chi}$, $\Res{s=\frac12} \Eisen_{P_5}\bk{\chi}$ - For a cubic Hecke character $\chi$.
	\item $\Res{s=\frac12} \Eisen_{P_4}\bk{\chi}$ - For a quartic Hecke character $\chi$.
\end{itemize}

%
%


\appendix
\part{Appendices}
\chapter{Local Calculations}
\label{Appendix:Chap_Local_Calculations}

In this chapter, we study the structure of certain local degenerate principal series and examine the action of certain intertwining operators on them.
The results of this chapter were stated in \Cref{Chap:Local_DPS} and used in \Cref{Chap:Square_Integrable} or \Cref{Chap:Non_square_Integrable}.

There are two types of degenerate principal series $\DPS{P_i}\bk{s_0,\chi}$ of interest to us:
\begin{enumerate}
	\item Where the global residue $\ResRep{G}{i}{s_0}{ord\bk{\chi}}$ is square-integrable and the co-socle of $\DPS{P_i}\bk{s_0,\chi}$ is not irreducible.
	In this case, we are interested in the length of the co-socle and the interpretation of the irreducible quotients as images or eigenspaces of certain intertwining operators.
	\item Where the global residue $R\bk{P_i,s_0,\chi}$ is non-square-integrable and cannot be calculated using a Siegel-Weil identity.
	In this case, we are interested in the length and structure (i.e. a Jordan-H\"older series) of the representation and the images of certain singular intertwining operators.
\end{enumerate}

We point out that there is a certain amount of overlap between these two types of cases.
For example, when $G$ is of type $E_6$, the residue $R\bk{P_4,1/2,\chi}$ is square integrable, while at a third of the primes $v\in\Places_{fin.}$, the local character $\chi_v$ is trivial and $\DPS{P_4}\bk{\frac{1}{2},\chi}_v = \DPS{P_4}\bk{1/2}_v$ and the residue $R\bk{P_4,1/2,\Id}$ is non-square-integrable, even though it shares infinitely many local factors with $R\bk{P_4,1/2,\chi}$.
Thus, it would make sense to include the analysis of $\DPS{P_4}\bk{\frac{1}{2}}_v$ in both \Cref{Appendix:Local_data_for_global_SI_resisudes} and \Cref{Appendix:Local_data_for_global_NSI_resisudes}.
However, in the case of $R\bk{P_4,1/2,\chi}$, we only require that $\pi_v$ admits a unique irreducible quotient in places where $\chi_v=\Id_v$, this was already proven in \cite{SDPS_E6}.

Since this chapter deals with local degenerate principal series, we fix a place $v\in\Places_{fin.}$ and drop it from our notations.
At places where a global character is needed for the consistency of discussion, we let $\widetilde{\chi}$ denote a quadratic Hecke character of $\AA^\times$ such that $\widetilde{\chi}_v=\chi$.
We will also use $\widetilde{\Id}$ to denote the trivial Hecke character of $\AA^\times$, as opposed to $\Id$ which will denote the trivial character of $F_v^\times$.

\section{Branching Rule Calculations}
\label{Appendix:Branching_Rules}

The main tool used in this chapter is the branching rules calculation, as described in \cite[Subsection 3.3]{SDPS_E6} and \cite[Subsection 3.3]{SDPS_E7}.
In this manuscript, we will treat this process as a black-box, while the precise algorithm is detailed there.

Let
\[
\mathcal{S} = \set{f:\bX\bk{T}\to \N \mvert \text{The support of $f$ is finite}}
\]
and note the natural partial order on $\mathcal{S}$ induced from the order of $\N$.

To an admissible representation $\pi$ of $G$ we associate a function $f_\pi\in\mathcal{S}$ by
\[
f_\pi\bk{\lambda} = mult\bk{\lambda,r_T^G\pi}.
\]
The branching rule calculation is a process in which, given an irreducible representation $\pi_0$ of $G$ and an exponent $\lambda\leq r_T^G\pi$, we construct a function $f\leq f_{\pi_0}$ such that $f\bk{\lambda}>0$.
In most cases, this construction turns out to yield $f=f_{\pi_0}$, but this is not assumed here (and in some cases it is not true).

Of most interest to us is the case where $\pi=\DPS{P_i}\bk{s,\chi}$ is a degenerate principal series and $\pi_0$ is an irreducible subquotient of $\pi$.
We pick $\lambda\leq r_T^G\pi_0$ and construct $f$ by the recipe in \cite[Subsection 3.3]{SDPS_E7}.
By construction, it holds that $f\leq f_{\pi_0}\leq f_\pi$.

If one shows that $f\bk{\lambda'} = f_\pi\bk{\lambda'}>0$, for some $\lambda'\in W\cdot\lambda$, then:
\begin{itemize}
	\item $f\bk{\lambda'} = f_{\pi_0}\bk{\lambda'}=f_\pi\bk{\lambda'}$.
	\item $\pi_0$ is the unique irreducible subquotient of $\pi$ which contains $\lambda'$ as an exponent.
	\item For $w\in W^{M_i,T}$ such that $\lambda'=w\cdot\lambda_0$, the image of the intertwining operator $N_w:\pi\to i_T^G\bk{\lambda'}$ is either zero or contains $\pi_0$ as its unique irreducible subrepresentation.
\end{itemize}

Finally, we set $A_1$-equivalence relation on $\bX\bk{T}$ to be the minimal equivalence relation such that $\lambda$ and $w_i\cdot \lambda$ are $A_1$-equivalent if and only if 
\[
\gen{\lambda,\check{\alpha_i}}\neq \pm 1 .
\]
In particular, we will say that $\lambda$ and $\lambda'$ are $A_1$-equivalent if they lie in the same equivalence class under this relation.
This implies that:
\begin{itemize}
	\item There exist $w=w_{i_1}\cdot...\cdot w_{i_k}$ such that $w\cdot\lambda=\lambda'$ and 
	\[
	\gen{ w_{i_{l+1}}\cdot...\cdot w_{i_k}\cdot\lambda, \check{\alpha_l}} \neq \pm 1 \quad \forall 1\leq l<k.
	\]
	\item $N_w\bk{\lambda}:i_T^G\lambda \to i_T^G\lambda'$ is bijective and so are
	\[
	N_{w_{i_l}} \bk{w_{i_{l+1}}\cdot...\cdot w_{i_k}\cdot\lambda} \quad \forall 1\leq l<k.
	\]
	\item For any irreducible representation $\sigma$ of $G$, it holds that
	\[
	f_\sigma\bk{\lambda}=f_{\sigma}\bk{\lambda'}.
	\]
\end{itemize}

\section{Global Square-integrable Residues}

\label{Appendix:Local_data_for_global_SI_resisudes}

In this section, we deal with local degenerate principal series representations, where the global residue $R\bk{P_i,s_0,\chi}$ is square-integrable and the co-socle of $\DPS{P_i}\bk{s_0,\chi}$ is not irreducible.

The relevant cases are:
\begin{itemize}
	\item For $G$ of type $E_6$: $\bk{4,1/2,3}$.
	\item For $G$ of type $E_7$: $\bk{2,1,1}$, $\bk{2,2,2}$, $\bk{5,2,2}$ and $\bk{4,1/2,4}$.
	\item For $G$ of type $E_8$: $\bk{1,5/2,1}$, $\bk{3,1/2,2}$ and $\bk{7,3/2,1}$.
\end{itemize}

We note that some of the cases were dealt with in \cite{SDPS_E6} and \cite{SDPS_E7}, and hence will just be recalled here.

\subsection{$E_6$}

Assume that $G$ is of type $E_6$ and let $\chi$ be a character of $F^\times$ of order $3$ and let $\mu_3$ denote the group of cube roots of unity.
Then, the representation $\chi=\DPS{P_4}\bk{\frac{1}{2},\chi}$ has length $4$.
\begin{itemize}
	\item Its unique irreducible subrepresentation $\pi_0$ is the unique irreducible subquotient of $\pi$ such that 
	\[
	\esixchar{-1}{-1}{-1}{3+\chi}{-1}{-1} \leq \jac{G}{T}{\pi_0}.
	\]
	\item 
	The quotient $\pi\rmod\pi_0$ is semi-simple of length $3$.
	Each irreducible quotient of $\pi$ admits the exponent $\lambda_{a.d.}$, where
	\[
	\lambda_{a.d.}=\esixchar{\chi}{\chi}{\chi}{-1+\chi}{\chi}{\chi} .
	\]
	We note that $\Stab_W\bk{\lambda_{a.d.}}=\gen{w_{3165}}$ and use the intertwining operator $N_{w_{3165}}\bk{\lambda_{a.d.}}$ to distinguish between these irreducible quotients.
	
	More precisely, we write $\pi\rmod\pi_0=\oplus_{\eta\in\mu_3}\pi_\eta$, where
	\[
	N_{w_{3165}}\bk{\lambda_{a.d.}}v = \eta v \quad \forall v \in \pi_\eta \quad \forall \eta\in\mu_3 .
	\]
	We point out that for each $\eta\in\mu_3$, it holds that $\lambda_{a.d.} \leq \jac{G}{T}{\pi_\eta}$.
		
\end{itemize}
It holds that
\[
\dim_\C\bk{\jac{G}{T}{\pi}} = 720, \quad
\dim_\C\bk{\jac{G}{T}{\pi_0}} = 480, \quad \dim_\C\bk{\jac{G}{T}{\pi_{\eta}}} = 80.
\]

In particular, note that the Jacquet modules $r_T^G\pi_\eta$ are indistinguishable (and semi-simple) and do not contain any exponent in common with $r_T^G\pi_0$.
Furthermore, all of the exponents in $r_T^G\pi_\eta$ and $r_T^G\pi_0$ appear there with multiplicity $1$.

\subsection{$E_7$}

Assume that $G$ is of type $E_7$.
We separate between cases where the decomposition of the co-socle is performed using the Gelbart-Knapp theory of $R$-groups: $\bk{2,2,2}$, $\bk{5,2,2}$ and $\bk{4,1/2,4}$,
and a case, $\bk{2,1,1}$, where the decomposition of the co-socle is performed using the theory of Iwahori-Hecke algebras.

\begin{enumerate}
	\item Let $\chi$ be a character of order $2$ of $F^\times$.
	Then, the representations $\pi=\DPS{P_2}\bk{2,\chi}$ and $\pi=\DPS{P_5}\bk{2,\chi}$ have length $3$.
	We specify their irreducible constituents:
	\begin{itemize}
		\item Its unique irreducible subrepresentation $\pi_0$ is the unique irreducible subquotient of $\pi$ such that $\lambda_0\leq \jac{G}{T}{\pi_0}$, where
		\[
		\lambda_0=\piece{\esevenchar{-1}{8+\chi}{-1}{-1}{-1}{-1}{-1},& \bk{2,2,2}\\
			\esevenchar{-1}{-1}{-1}{-1}{6+\chi}{-1}{-1},& \bk{5,2,2} .}
		\]
		Note that these are distinct representations for the two cases.
		\item The quotient $\pi\rmod\pi_0=\oplus_{\eta\in\set{1,-1}}\pi_\eta$ is semi-simple of length $2$.
		Each of the $\pi_\eta$ satisfy the property that
		\[
		\lambda_{a.d.}=\esevenchar{-1}{\chi}{\chi}{-1}{\chi}{-1}{\chi} \leq \jac{G}{T}{\pi_\eta}.
		\]
		Indeed, both cases share the same irreducible quotients.

		Note that $\Stab_W\bk{\lambda_{a.d.}}=\gen{w_{257}}$.
		We distinguish between $\pi_1$ and $\pi_{-1}$ by their eigenvalues under the action of $N_{w_{257}}\bk{\lambda_{a.d.}}$.
		Namely,
		\[
		N_{w_{257}}\bk{\lambda_{a.d.}}v = \eta v \quad \forall v \in \pi_{\eta},\quad \eta\in\set{1,-1} .
		\]
		Also, we note that $\pi_1$ and $\pi_{-1}$ are two distinct representations but both appear as a quotient of $\DPS{P_2}\bk{2,\chi}$ as well as of $\DPS{P_5}\bk{2,\chi}$.
		That is, 
		\[
		\pi_1 \oplus \pi_{-1} = 
		\DPS{P_2}\bk{-2,\chi} \cap \DPS{P_5}\bk{-2,\chi} .
		\]
	\end{itemize}
	In conclusion, the quotient $\pi\rmod\pi_0=\pi_1\oplus\pi_{-1}$ is semi-simple of length $2$ and it holds that
	\[
	\dim_\C\bk{\jac{G}{T}{\pi_{\pm 1}}} = 36,\quad 
	\dim_\C\bk{\jac{G}{T}{\pi_0}} = \piece{504,& \bk{2,2,2}\\3,960, & \bk{5,2,2}}.
	\]

	\item Let $\chi$ be a character of order $4$ of $F^\times$.
	The representation $\pi=\DPS{P_4}\bk{\frac{1}{2},\chi}$ has length $3$.
	We denote the unique irreducible subrepresentation by $\pi_0$.
	
	\begin{itemize}
		\item Its unique irreducible subrepresentation $\pi_0$ is the unique irreducible subquotient of $\pi$ such that
		\[
		\lambda_0=\esevenchar{-1}{-1}{-1}{\frac{7}{2}+\chi}{-1}{-1}{-1} \leq \jac{G}{T}{\pi_0} .
		\]
		\item The quotient $\pi\rmod\pi_0=\oplus_{\eta\in\set{1,-1}}\pi_\eta$ is semi-simple of length $2$.
		Each of the $\pi_\eta$ satisfy the property that
		\[
		\lambda_{a.d.}=\esevenchar{-\frac{1}{2}+3\chi}{2\chi}{2\chi}{-\frac{1}{2}+\chi}{2\chi}{-\frac{1}{2}+3\chi}{2\chi} \leq \jac{G}{T}{\pi_\eta}.
		\]

		Note that $\Stab_W\bk{\lambda_{a.d.}}=\gen{w_{257}}$ and as above, we distinguish between $\pi_1$ and $\pi_{-1}$ by their eigenvalues under the action of $N_{w_{257}}\bk{\lambda_{a.d.}}$.
		Namely,
		\[
		N_{w_{257}}\bk{\lambda_{a.d.}}v = \eta v \quad \forall v \in \pi_{\eta},\quad \eta\in\set{1,-1} .
		\]
	\end{itemize}
	
	The quotient $\pi\rmod\pi_0=\pi_1\oplus\pi_{-1}$ is semi-simple of length $2$.
	
	It holds that
	\[
	\dim_\C\bk{\jac{G}{T}{\pi_0}} = 7,560, \quad \dim_\C\bk{\jac{G}{T}{\pi_{\eta}}} = 1,260.
	\]
	
	\vspace{0.3cm}
	The above cases follow from \cite{SDPS_E7} (where they were calculated using Knapp-Stein R-groups) and a branching rule calculation performed for each of $r_T^G\pi_\eta$ and $\pi_0$ as in the $E_6$ case.
	\vspace{0.3cm}
	
	\item The representation $\pi=\DPS{P_2}\bk{1}$ has length $3$.

	\begin{itemize}
		\item Its unique irreducible subrepresentation $\pi_0$ is its unique irreducible subquotient which satisfy
		\[
		\lambda_0 = \lambda_{1}^{P_2} = \esevenchar{-1}{7}{-1}{-1}{-1}{-1}{-1} \leq r_T^G\pi_0 .
		\]
		
		\item The quotient $\pi\rmod\pi_0=\pi_1\oplus\pi_{-1}$ is semi-simple of length $2$, where $\pi_1$ is the unique irreducible subquotient of $\pi$ which satisfy
		\[
		\lambda_{a.d.}=\esevenchar{-1}{0}{0}{-1}{0}{0}{-1} \leq r_T^G\pi_1 ,
		\]
		while $\pi_{-1}$ the unique irreducible subquotient of $\pi$ which satisfy
		\[
		\lambda_1\leq r_T^G\pi_{-1} \quad \lambda_{a.d.}\nleq r_T^G\pi_{-1},
		\]
		where
		\[
		\lambda_1 = \lambda_{-1}^{P_2} = \esevenchar{-1}{5}{-1}{-1}{-1}{-1}{-1} .
		\]
		
	\end{itemize}
	It holds that
	\[
	\dim_\C\bk{\jac{G}{T}{\pi_0}} = 105	, \quad 
	\dim_\C\bk{\jac{G}{T}{\pi_{1}}} = 456 , \quad
	\dim_\C\bk{\jac{G}{T}{\pi_{-1}}} = 15.
	\]
	Note that $\dim_\C\bk{\jac{G}{T}{\pi_{-1}}}$ was already determined in \cite{SDPS_E7}.

	\begin{Rem}
		\label{Rem:Missing_exps_E7_211}
		It should be noted that a simple branching rule would suggest that $\dim_\C\bk{\jac{G}{T}{\pi_{1}}} \geq 454$ and not $\dim_\C\bk{\jac{G}{T}{\pi_{1}}} = 456$.
		This is because the calculation brings about
		$4\times\mu_1, 4\times\mu_2 \leq r_T^G\pi_1$, where
		\[
		\mu_1 = \esevenchar{1}{-1}{-2}{1}{-1}{-1}{1}, \quad
		\mu_2 = \esevenchar{-1}{-1}{2}{-1}{-1}{-1}{1} .
		\]
		However, from the geometric lemma, it follows that $5\times\mu_1, 5\times\mu_2 \leq r_T^G\pi$.
		We wish to show that, actually, $5\times\mu_1, 5\times\mu_2 \leq r_T^G\pi_1$.
		Since $\mu_1$ and $\mu_2$ are $A_1$ equivalent, it is enough to show that $5\times\mu_1 \leq r_T^G\pi$.
		
		Consider the exponents
		\[
		\nu_1 = \esevenchar{-1}{-1}{2}{-1}{-1}{0}{-1}, \quad
		\nu_2 = \esevenchar{-1}{-1}{2}{-1}{-2}{1}{0} .
		\]
		We recall from \cite[Corollary A.2]{SDPS_E7} that $L=M_{6,7}$ admits unique irreducible representations $\sigma_1$ and $\sigma_2$ of $L=M_{\set{6,7}}$ such that $\nu_i \leq r_T^L \sigma_i$.
		Furthermore, it holds that
		\[
		r_T^L\sigma_1 = 2\times\nu_1+\mu_1, \quad r_T^L\sigma_2 = 2\times\nu_2+\mu_1 .
		\]
		On the other hand, The branching rule calculation above shows that $8\times \nu_1 \leq r_T^G\pi_1$ and $2\times \nu_2 \leq r_T^G\pi_1$.
		It follows that
		\[
		4\times\sigma_1 + 1\times \sigma_2 \leq r_L^G\pi_1
		\]
		and hence $5\times\mu_1 \leq r_T^G\pi$.
	\end{Rem}
\end{enumerate}

We now turn to calculate the eigenvalue of a certain intertwining operator on the irreducible quotients of $\pi=\DPS{P_2}\bk{1}$.

\begin{Lem}
	It holds that
	\[
	N_{u_0^{-1} w_0 u_0}\bk{\lambda_0} f_\epsilon \quad \forall f_\epsilon\in\pi_\epsilon,\ \epsilon=\pm 1 ,
	\]
	where
	\begin{itemize}
		\item $u_0=w\coset{1,3,4,2}$.
		\item $w_0=w\coset{5, 4, 3, 2, 4, 5, 6, 7, 5, 4, 3, 2, 4, 5, 6, 5, 4, 3, 2, 4, 5}$.
	\end{itemize}
\end{Lem}

\begin{proof}
	We denote
	\[
	\mu = u_0\cdot \lambda_0 = \esevenchar{-2}{-1}{-1}{-1}{3}{-1}{-1} 
	\]
	and note that $w_0\cdot\mu=\mu$.
	
	We further let:
	\begin{itemize}
		\item $\Omega'_0$ denote the $1$-dimensional representation of $M_2$ such that $r_T^{M_2}\Omega_0'=\lambda_{0}$.
		\item $\Omega_0 = r_{M_{2,5}}^{M_2} \Omega_0'$ denote the $1$-dimensional representation of $M_{2,5}=M_2\cap M_5$ such that $r_T^{M_{2,5}}\Omega_0=\lambda_{0}$.
		\item $\Omega_1$ denote the $1$-dimensional representation of $M_{1,5}=M_1\cap M_5$ such that $r_T^{M_{1,5}}\Omega_1=\mu$.
	\end{itemize}
	By induction in stages, it holds that
	\[
	\DPS{P_2}\bk{1}_v = i_{M_2}^G\Omega_0' 
	\hookrightarrow i_{M_2}^G\bk{i_{M_{2,5}}^{M_2} \Omega_0} 
	\cong i_{M_5}^G\bk{i_{M_{2,5}}^{M_5} \Omega_0} .
	\]
	We note that $i_{M_{2,5}}^{M_5} \Omega_0$ is an irreducible representation of $M_5$ which satisfy
	\[
	i_{M_{2,5}}^{M_5} \Omega_0 \cong i_{M_{1,5}}^{M_5} \Omega_1 .
	\]
	Here, $i_{M_{1,5}}^{M_5} \Omega_1$ is the invert representation of $i_{M_{2,5}}^{M_5} \Omega_0$ (see \cite{SDPS_E8} for the definition) and the isomorphism is realized by $N_{u_0,v}\bk{\lambda_{0}}$.
	It follows that
	\[
	i_{M_5}^G\bk{i_{M_{2,5}}^{M_5} \Omega_0} \cong
	i_{M_5}^G\bk{i_{M_{1,5}}^{M_5} \Omega_1} \cong i_{M_1}^G\bk{i_{M_{1,5}}^{M_1} \Omega_1} .
	\]
	We note here that $i_{M_{1,5}}^{M_1} \Omega_1$ is a unitary representation of length $2$ of $M_1$ and write
	\[
	i_{M_{1,5}}^{M_1} \Omega_1 = \sigma_1 \oplus \sigma_{-1} ,
	\]
	where $\sigma_1$ is the irreducible spherical subrepresentation.
	Since $w_0$ is the longest element in $W^{M_1,M_{1,5}}$, it follows from \cite{MR2363302}, that
	\[
	N_{w_0,v}^{M_1}\bk{\mu} f_\epsilon \quad \forall f_\epsilon\in\sigma_\epsilon,\ \epsilon=\pm 1 .
	\]
	It remains to show that
	\[
	\pi_\epsilon \subseteq i_{M_5}^G\sigma_\epsilon,\ \epsilon=\pm 1.
	\]
	For $\epsilon=1$, this follows from the fact that the antidominant weight
	\[
	\lambda_{a.d.} = \esevenchar{-1}{0}{0}{-1}{0}{0}{-1}
	\]
	is an exponent of $\pi_1$ and $i_{M_5}^G\sigma_1$ but not of $\pi_{-1}$ and $i_{M_5}^G\sigma_{-1}$.
	Otherwise, it can be shown by following the action of $N_{w_0,v}$ on the spherical section.
	
	In order to prove that $\pi_{-1,v}$ is an eigenspace of $N_{w_0,v}$ of eigenvalue $-1$, it is enough to show that $\pi_{-1,v}\subseteq \DPS{P_2}\bk{-1}_v \cap i_{M_1}^G\sigma_{-1}$.
	Indeed, one checks that:
	\begin{itemize}
		\item $\sigma_{-1}$ is the kernel of $N_{w\coset{7,6,4,5},v}^{M_1}\bk{\mu}$ (this can be done by a calculation in the Iwahori-Hecke model of this representation, see \cite[Proposition 4.9 and Appendix B]{SDPS_E7} for more details).
		Here,
		\[
		\begin{split}
			& \dim_\C\bk{\frakJ_{M_1,v}} = 160, \\
			& \dim_\C\bk{\sigma_{1}^{\frakJ_{M_1,v}}} = 155, \\
			& \dim_\C\bk{\sigma_{-1}^{\frakJ_{M_1,v}}} = 5, \\
			& \dim_\C \bk{\bk{ N_{w\coset{7,6,4,5},v}^{M_1} \bk{\mu}\bk{ i_{M_{1,5}}^{M_1} \Omega_1 }}^{\frakJ_{M_1,v}}} = 155.
		\end{split}
		\]
		\item $\pi_{-1}$ is the kernel of
		\[
		N_{w\coset{7,6,4,5}u_0}\bk{\lambda_{0}}\res{\DPS{P_2}\bk{-1}_v} .
		\]
		This can be verified by a calculation in the Iwahori-Hecke module, similar to the calculation mentioned above.
		One could also argue that 
		\[
		w\coset{7,6,4,5}u_0 = w\coset{4,1,3} \cdot w\coset{7,6,5,4,2},
		\]
		where $\pi_{-1,v}$ is the kernel of 
		\[
		N_{w\coset{7,6,5,4,2}}\bk{\lambda_{0}}\res{\DPS{P_2}\bk{-1}_v}
		\]
		as was demonstrated in \cite[Proposition 4.9 and Appendix B]{SDPS_E7}.
		One also checks that $N_{w\coset{4,1,3}}\bk{w\coset{7,6,5,4,2}\cdot\lambda_{0}}$ is an isomorphism. 
		Hence, indeed, $\pi_{-1,v}=\DPS{P_2}\bk{-1}_v \cap i_{M_1}^G\sigma_{-1}$.
		
		Alternatively, one checks, using a branching rule calculation (see \Cref{Appendix:Branching_Rules}), that
		\[
		mult\bk{\mu',r_T^G \bk{\DPS{P_2}\bk{-1}_v}} = mult\bk{\mu', r_T^G \bk{\pi_1}} = 2,
		\]
		where
		\[
		\mu'=w\coset{7,6,4,5}\cdot \mu =
		w\coset{7,6,4,5}u_0 \cdot \lambda_{0} =
		\esevenchar{-2}{1}{1}{-2}{1}{-1}{-1} .
		\]
		Hence $\pi_{-1,v} = \DPS{P_2}\bk{-1}_v \cap i_{M_1}^G\sigma_{-1}$.
	\end{itemize}
\end{proof}

\subsection{$E_8$}

Assume that $G$ is of type $E_8$.

\begin{enumerate}
	\item The representation $\pi=\DPS{P_1}\bk{\frac52}$ has length $3$.
	
	\begin{itemize}
		\item Its unique irreducible subrepresentation $\pi_0$ is its unique irreducible subquotient which satisfy
		\[
		\lambda_1 = \lambda_{\frac52}^{P_1} = \eeightchar{13}{-1}{-1}{-1}{-1}{-1}{-1}{-1} \leq r_T^G\pi_0 .
		\]
		
		\item The quotient $\pi\rmod\pi_0=\pi_1\oplus\pi_{-1}$ is semi-simple of length $2$, where $\pi_1$ is the unique irreducible subquotient of $\pi$ which satisfy
		\[
		\lambda_{a.d.}=\eeightchar{-1}{0}{0}{-1}{0}{0}{-1}{-1} \leq r_T^G\pi_0 ,
		\]
		while $\pi_{-1}$ the unique irreducible subquotient of $\pi$ which satisfy
		\[
		\lambda_0\leq r_T^G\pi_{-1} \quad \lambda_{a.d.}\nleq r_T^G\pi_{-1},
		\]
		where
		\[
		\lambda_0 = \lambda_{-\frac52}^{P_1} = \eeightchar{8}{-1}{-1}{-1}{-1}{-1}{-1}{-1} .
		\]
		We also point out that all exponents in $\pi_{-1}$ are $A_1$-equivalent and have multiplicity $1$.
		Furthermore, note that
		\[
		\mult{\lambda_i}{r_T^G\pi}=2,\quad 
		\mult{\lambda_i}{r_T^G\pi_{\pm 1}}=1,
		\]
		where
		\[
		\lambda_i = \eeightchar{-1}{5}{-1}{-1}{-1}{-1}{-1}{2}
		\]
	\end{itemize}
	It holds that
	\[
	\dim_\C\bk{\jac{G}{T}{\pi_0}} = 1,100	, \quad 
	\dim_\C\bk{\jac{G}{T}{\pi_{1}}} = 1,010 , \quad
	\dim_\C\bk{\jac{G}{T}{\pi_{-1}}} = 50.
	\]
	
	\begin{Rem}
		Here, a similar argument to that of \Cref{Rem:Missing_exps_E7_211} needs to be made.
		The branching rule calculation shows that $\dim_\C\bk{\jac{G}{T}{\pi_{1}}} \geq 1,006$.
		For the following exponents, we have $4\times\lambda \leq \jac{G}{T}{\pi_{1}}$, while $5\times\lambda \leq \jac{G}{T}{\pi}$
		\[
		\begin{split}
		& \eeightchar{-1}{-1}{2}{-1}{-1}{-1}{-1}{2},\ 
		\eeightchar{1}{-1}{-2}{1}{-1}{-1}{-1}{2},\ \\
		& \eeightchar{-1}{-1}{2}{-1}{-1}{-1}{1}{-2},\ 
		\eeightchar{1}{-1}{-2}{1}{-1}{-1}{1}{-2} .
		\end{split}
		\]
		Since they are $A_1$-equivalent, it is enough to show that $5\times\lambda \leq \jac{G}{T}{\pi_{1}}$ for one of them, in order to show it for all of them.
		Indeed, the same argument holds taking
		\[
		\mu_1 = \eeightchar{-1}{-1}{2}{-1}{-1}{-1}{1}{-2}, 
		\nu_1 = \eeightchar{-1}{-1}{2}{-1}{-1}{0}{-1}{-1},
		\nu_2 = \eeightchar{-1}{-1}{2}{-1}{-2}{1}{0}{-2}.
		\]
	\end{Rem}

	\item 
	Let $\chi$ be a character of order $2$ of $F^\times$.
	Then, the representation $\pi=\DPS{P_3}\bk{\frac12,\chi}$ has length $3$.
	\begin{itemize}
		\item Its unique irreducible subrepresentation $\pi_0$ is its unique irreducible subquotient which satisfy
		\[
		\lambda_0 = \lambda_{\frac12}^{P_3}  \eeightchar{-1}{-1}{6+\chi}{-1}{-1}{-1}{-1}{-1} \leq r_T^G\pi_0 .
		\]
		
		\item The quotient $\pi\rmod\pi_0=\pi_1\oplus\pi_{-1}$ is semi-simple of length $2$, where $\pi_1$ is the unique irreducible subquotient of $\pi$ which satisfy
		\[
		\lambda_{a.d.}=\eeightchar{0}{0}{0}{0}{-1+\chi}{0}{0}{-1} \leq r_T^G\pi_1 ,
		\]
		while $\pi_{-1}$ the unique irreducible subquotient of $\pi$ which satisfy
		\[
		\lambda_1\leq r_T^G\pi_{-1} \quad \lambda_{a.d.}\nleq r_T^G\pi_{-1},
		\]
		where
		\[
		\lambda_1 = \lambda_{-\frac12}^{P_3} = \eeightchar{-1}{-1}{5+\chi}{-1}{-1}{-1}{-1}{-1} .
		\]
		We also point out that all exponents in $\pi_{-1}$ are $A_1$-equivalent and have multiplicity $1$.
		Furthermore, note that
		\[
		\mult{\lambda_i}{r_T^G\pi}=2,\quad 
		\mult{\lambda_i}{r_T^G\pi_{\pm 1}}=1,
		\]
		where
		\[
		\lambda_i = \eeightchar{-1}{5}{-1}{-1}{-1}{-1}{-1}{4+\chi}
		\]
	\end{itemize}
	
	It holds that	
	\[
	\dim_\C\bk{\jac{G}{T}{\pi_0}} = 12,600	, \quad 
	\dim_\C\bk{\jac{G}{T}{\pi_{1}}} = 54,720 , \quad
	\dim_\C\bk{\jac{G}{T}{\pi_{-1}}} = 1,800.
	\]
	
	\begin{Rem}
		We note that using a direct branching rule calculation, we would get $\dim_\C\bk{\jac{G}{T}{\pi_{1}}} \geq 54,240$ and not $\dim_\C\bk{\jac{G}{T}{\pi_{1}}} = 54,720$.
		This is because the calculation shows $4\times \mu \leq r_T^G\pi_1$, where $\mu$ is $A_1$-equivalent to
		\[
		\mu_0 = \eeightchar{\chi}{-1}{-1}{1+\chi}{-1}{1+\chi}{-1}{-1}
		\]
		However, if we let $\tau$ denote an irreducible representation such that $\mu\leq r_T^G\pi_1$.
		Then, a branching calculation shows that $8\times \nu \leq r_T^G\pi, r_T^G\pi_1$,		while $2\times \nu\leq r_T^G\tau$, where
		\[
		\nu = \eeightchar{-1}{1}{1+\chi}{-1}{-1}{1+\chi}{-1}{-1}.
		\]
		Hence, $\tau=\pi_1$.
	\end{Rem}

	\item A similar analysis shows that the representation $\pi=\DPS{P_7}\bk{\frac32}$ has length $3$ with 
	
	\begin{itemize}
		\item Its unique irreducible subrepresentation $\pi_0$ is its unique irreducible subquotient which satisfy
		\[
		\lambda_0 = \lambda_{\frac32}^{P_7}  \eeightchar{-1}{-1}{-1}{-1}{-1}{-1}{10}{-1} \leq r_T^G\pi_0 .
		\]
		
		\item The quotient $\pi\rmod\pi_0=\pi_1\oplus\pi_{-1}$ is semi-simple of length $2$, where $\pi_1$ is the unique irreducible subquotient of $\pi$ which satisfy
		\[
		\lambda_{a.d.}=\eeightchar{-1}{0}{0}{-1}{0}{0}{-1}{0} \leq r_T^G\pi_1 ,
		\]
		while $\pi_{-1}$ the unique irreducible subquotient of $\pi$ which satisfy
		\[
		\lambda_1\leq r_T^G\pi_{-1} \quad \lambda_{a.d.}\nleq r_T^G\pi_{-1},
		\]
		where
		\[
		\lambda_1 = \lambda_{-\frac32}^{P_7} = \eeightchar{-1}{-1}{-1}{-1}{-1}{-1}{7}{-1} .
		\]
	\end{itemize}	
	
\end{enumerate}

\section{Global Non-square-integrable Residues}
\label{Appendix:Local_data_for_global_NSI_resisudes}

In this section, we deal with local degenerate principal series representations, where the global residue $R\bk{P,s_0,\chi}$ is non-square-integrable and cannot be calculated using a Siegel-Weil identity.
We study the structure of the local degenerate principal series $\DPS{P_i}\bk{s,\chi}$ in the following cases:
\begin{itemize}
	\item For $G$ of type $E_6$: $\bk{4,1/2,1}$, $\bk{4,1/2,2}$, $\bk{4,1,1}$ and $\bk{4,1,2}$.
	\item For $G$ of type $E_7$: $\bk{2,2,1}$, $\bk{4,1/2,1}$, $\bk{4,1/2,2}$, $\bk{4,2/3,1}$, $\bk{4,2/3,3}$, $\bk{4,3/2,1}$, $\bk{4,3/2,2}$, $\bk{5,1,2}$, $\bk{5,3/2,1}$ and $\bk{5,3/2,2}$.
	\item For $G$ of type $E_8$: 
	$\bk{3,\frac76,1}$, $\bk{3,\frac76,3}$, $\bk{3,2,1}$, $\bk{3,2,2}$, 
	$\bk{4,\frac{3}{10},1}$, $\bk{4,\frac{3}{10},5}$, $\bk{4,\frac12,4}$, $\bk{4,\frac34,1}$, $\bk{4,\frac34,2}$, $\bk{4,\frac34,4}$, $\bk{4,\frac76,1}$, $\bk{4,\frac76,3}$, $\bk{4,2,1}$, $\bk{4,2,2}$, 
	$\bk{5,\frac12,3}$, $\bk{5,\frac56,1}$, $\bk{5,\frac56,3}$, $\bk{5,\frac76,1}$, $\bk{5,\frac76,3}$, $\bk{5,2,1}$, $\bk{5,2,2}$,
	$\bk{6,2,2}$, $\bk{6,\frac52,1}$ and $\bk{6,\frac52,2}$.

\end{itemize}

In all of these cases, $\pi=\DPS{P}\bk{s_0,\chi}$ admits a unique irreducible subrepresentation, denoted $\pi_0$, and a unique irreducible quotient, denoted $\pi_1$.
In most cases, $\pi$ has length $2$.
We now explain the argument for showing that $\pi$ has length $2$ as the argument is identical for all such cases.
\begin{itemize}
	\item We recall that $\DPS{P}\bk{s_0,\chi}$ is reducible and hence of length at least $2$.
	Furthermore, in all relevant cases, $\pi$ admits a unique irreducible subrepresentation $\pi_0$ and a unique irreducible quotient $\pi_1$.

	\item Let $\lambda_0$ denote the initial exponent of $\pi$ and fix an anti-dominant exponent $\lambda_{a.d.}$ of $\pi$.
	
	\item Using the method of branching rules with respect to the exponents $\lambda_0$ and $\lambda_{a.d.}$ respectively, one constructs two functions $f_1, f_2\in\mathcal{S}$ such that $f_0\leq f_{\pi_0}$ and $f_1\leq f_{\pi_1}$.
	
	\item One checks that
	\[
	supp\bk{f_0}\cap supp\bk{f_1} =\emptyset, \qquad \sum_{\lambda}\bk{f_0\bk{\lambda} + f_1\bk{\lambda}} = \Card{W^{M,T}} = \dim_\C\bk{r_T^G\pi} .
	\]
	
	\item It follows that $\pi$ has no non-zero irreducible subquotients other than $\pi_0$ and $\pi_1$.
		
\end{itemize}

\begin{Rem}
	It was implicitly assumed in the above argument that $\lambda_{a.d.}\leq r_T^G\pi_1$.
	This follows easily from a standard ``central character argument'', see \cite[Lemma. 3.12 and pg. 22]{SDPS_E6} for more details.
\end{Rem}

\begin{Rem}
	We point out that under the above assumptions, the condition
	\[
	f_{\pi}\bk{w\cdot\lambda_0} = f_{\pi}\bk{w\cdot\lambda_0} \quad \forall w\in W^{shortest}_{\bk{P_i,s_0,\widetilde{\chi}}}
	\]
	also follows by the following argument:
	Since all exponent in $r_T^G\pi_0$ are $A_1$-equivalent to $\lambda_0$ and hence appear with multiplicity $1$ in $r_T^G\pi$, hence they do not appear in $r_T^G\pi_1$.
	On the other hand, $M_w\bk{\lambda_{s,\chi}^{P}} f_s\bk{g}$ is holomorphic at $s_0$ for any $w\in W^{M_i,T}$ such that $w\cdot\lambda_0 \leq r_T^G\pi_0$.
	The claim thus follows.
\end{Rem}

In what follows for a representation $\sigma$ of $G$, we denote 
\[
d\bk{\sigma} = \dim_\C\bk{\jac{G}{T}{\sigma}}.
\]

\subsection{$E_6$}

Assume that $G$ is of type $E_6$.
In some cases, $[i,s_0,k]$, it is possible to prove that $\pi=\DPS{P_i}\bk{s_0,\chi}$ is of length $2$ via the method describe in the beginning of \Cref{Appendix:Local_data_for_global_NSI_resisudes}.
These cases are listed in the following table:

\begin{center}
	\begin{longtable}[h]{|c|c|c|c|}
		\hline
		$[i,s_0,k]$ & $d\bk{\pi}$ &  $d\bk{\pi_1}$ & $d\bk{\pi_0}$ \\ \hline \hline
		$\bk{4,1,1}$ & 720 & 216 & 504 \\ \hline
		$\bk{4,1,2}$ & 720 & 216 & 504 \\ \hline
		$\bk{4,\frac12,2}$ & 720 & 540 & 180 \\ \hline
	\end{longtable}
	\label{Table:E6_local_length_2_NSI}
\end{center}

We now turn to the remaining case, $\coset{4,\frac12,1}$.
We claim that the representation $\pi=\DPS{P_4}\bk{\frac{1}{2}}_v$ is of length $3$.
Its composition series is given as follows:
\begin{itemize}
	
	\item $\pi$ admits a unique irreducible subrepresentation $\pi_0$.
	It is the unique subquotient of $\pi$ such that
	\[
	\lambda_0=\esixchar{-1}{-1}{-1}{3}{-1}{-1} \leq \jac{G}{T}{\pi_0}.
	\]
	It satisfies
	\[
	\dim_\C\bk{\jac{G}{T}{\pi_1}} = 10.
	\]
	\item $\pi\rmod\pi_0$ admits a unique irreducible subrepresentation $\sigma_1$.
	It is the unique irreducible subrepresentation of $\pi$ such that
	\[
	\lambda_1=\esixchar{-1}{3}{0}{-1}{0}{-1} \leq \jac{G}{T}{\sigma_1}.
	\]
	It satisfies
	\[
	\dim_\C\bk{\jac{G}{T}{\pi_2}} = 135.
	\]
	
	\item The quotient $\pi_1=\pi\rmod\bk{\pi_0+\sigma_1}$ is irreducible.
	It is the unique irreducible subrepresentation of $\pi$ such that
	\[
	\lambda_{a.d.}=\esixchar{0}{0}{0}{-1}{0}{0} \leq \jac{G}{T}{\pi_1}.
	\] 
	It satisfies
	\[
	\dim_\C\bk{\jac{G}{T}{\pi_1}} = 575.
	\]

\end{itemize}

We postpone the proof of this decomposition and first study the images of the intertwining operators $N_w\bk{\frac{1}{2}}\bk{\DPS{P_4}\bk{\frac{1}{2}}}$ for $w\in W^{shortest}_{\bk{P_i,s_0,\widetilde{\chi}}}$ (used in the calculation of the residue in the global case $\coset{4,1/2,2}$) and $w\in W^{shortest}_{\bk{P_i,s_0,\Id}}$ (used in the calculation of the residue in the global case $\coset{4,1/2,1}$).

For $w\in W^{shortest}_{\bk{P_i,s_0,\widetilde{\chi}}}$ (as described in the previous item), one checks that $w\cdot\lambda_0$ appears only in $r_T^G\pi_1$ and does not appear in $r_T^G\pi_0$ and $r_T^G\sigma_1$.

We now consider the elements of $W^{shortest}_{\bk{P_i,s_0,\widetilde{\Id}}}$ given by the following list
\[
\begin{array}{l}
	w_{[3, 2, 4, 5, 3, 2, 4]}, \\
	w_{[3, 4, 5, 4, 3, 2, 4]}, \\
	w_{[4, 1, 3, 2, 4]}, \\
	w_{[2, 4, 5, 4, 3, 2, 4]}, \\
	w_{[5, 6, 4, 5, 4, 3, 2, 4]}, \\
	w_{[6, 4, 5, 3, 4]}, \\
	w_{[6, 5, 1, 3, 4]}, \\
	w_{[6, 4, 5, 2, 4]}, \\
	w_{[4, 5, 1, 3, 4]}, \\
	w_{[1, 3, 4, 5, 1, 3, 2, 4]} .
\end{array}
\]

We separate this into two cases:
\begin{itemize}
	\item 
	If $w$ is one of $w_{[3, 2, 4, 5, 3, 2, 4]}$, $w_{[3, 4, 5, 4, 3, 2, 4]}$, $w_{[2, 4, 5, 4, 3, 2, 4]}$, $w_{[5, 6, 4, 5, 4, 3, 2, 4]}$ and $w_{[1, 3, 4, 5, 1, 3, 2, 4]}$, the equivalence class $\kappa_{\coset{w}}$ is a singleton and one checks that $w\cdot\lambda_0$ appears only in  $r_T^G\pi_1$ and does not appear in $r_T^G\pi_0$ and $r_T^G\sigma_1$.

	\item
	If $w$ is one of $w_{[4, 1, 3, 2, 4]}$, $w_{[6, 4, 5, 3, 4]}$, $w_{[6, 5, 1, 3, 4]}$, $w_{[6, 4, 5, 2, 4]}$ and $w_{[4, 5, 1, 3, 4]}$, then $\kappa_{\coset{w}}$ consists of two elements.
	One checks that the order of $\kappa_{\coset{w}}$ is lower than the order of the pole and hence does not participate in the constant term of the residue.
	However, there are only three equivalence classes $\coset{w'}$ such that $w'$ can be written as a reduced expression $w'=u\cdot w$ and this cannot be done with respect to the elements in the previous item.
	These equivalence classes are 
	\[
	\begin{array}{l}
		\set{w_{[4, 3, 1, 5, 4, 2, 3, 6, 5, 4]}, w_{[4, 2, 3, 4, 5, 4, 2, 3, 1, 4, 3, 6, 5, 4]}} \\
		\set{w_{[2, 4, 3, 1, 5, 4, 2, 3, 6, 5, 4]}, w_{[2, 4, 2, 3, 4, 5, 4, 2, 3, 1, 4, 3, 6, 5, 4]}} \\
		\set{w_{[2, 4, 2, 3, 1, 5, 4, 2, 3, 6, 5, 4]}, w_{[2, 4, 3, 5, 4, 2, 3, 1, 4, 3, 6, 5, 4]}, w_{[2, 3, 4, 2, 3, 1, 5, 4, 2, 3, 6, 5, 4]}, w_{[2, 3, 4, 3, 5, 4, 2, 3, 1, 4, 3, 6, 5, 4]}}
	\end{array}
	\]
	For each of these $w'$s, one checks that $w'\cdot\lambda_0$ appears only in  $r_T^G\pi_1$ and does not appear in $r_T^G\pi_0$ and $r_T^G\sigma_1$.
\end{itemize}
We conclude that $N_w\bk{\frac{1}{2}}\bk{\DPS{P_4}\bk{\frac{1}{2}}}=\pi_1$ for any $w$ such that $\kappa_{\coset{w}}$ admits a pole of order $2$ at $s_0=\frac{1}{2}$.

\mbox{}

We now turn to the proof that this is indeed the structure of $\pi$.
First of all, from \cite{SDPS_E6}, we know that $\pi$ is reducible with a unique irreducible subrepresentation and a unique irreducible quotient.
Hence, its length is at least $2$.

From a branching rule calculation, it follows that $\pi$ contains at most $4$ irreducible subquotients.
Namely,
\begin{itemize}
	\item Let $\pi_0$ denote the unique irreducible subquotient of $\pi$ such that
	\[
	\lambda_0=\esixchar{-1}{-1}{-1}{3}{-1}{-1} \leq \jac{G}{T}{\pi_0}.
	\]
	
	\item Let $\pi_1$ denote the unique irreducible subquotient of $\pi$ such that
	\[
	\lambda_{a.d.}=\esixchar{0}{0}{0}{-1}{0}{0} \leq \jac{G}{T}{\pi_1}.
	\]
	
	\item Let $\sigma_1$ denote the unique irreducible subquotient of $\pi$ such that
	\[
	\lambda_1=\esixchar{-1}{3}{0}{-1}{0}{-1} \leq \jac{G}{T}{\sigma_1}.
	\]
	
	\item Let $\sigma_2$ denote the unique irreducible subquotient of $\pi$ such that
	\[
	\lambda_2=\esixchar{1}{-1}{0}{0}{-1}{0} \leq \jac{G}{T}{\sigma_2}.
	\]
\end{itemize}

From a branching rule calculation, it follows that:
\[
\begin{array}{l}
	\dim_\C\bk{r_T^G\pi_0} \geq 10 \\
	\dim_\C\bk{r_T^G\sigma_1} \geq 135 \\
	\dim_\C\bk{r_T^G\sigma_2} \geq 335 \\
	\dim_\C\bk{r_T^G\pi_1} \geq 240 .
\end{array}
\]

Furthermore, it follows that $\pi_0$ is the unique irreducible subrepresentation of $\pi$ and $\pi_1$ is the unique irreducible quotient.
It remains to check whether some of these irreducible representations are, in fact, equal.

In order to determine the actual length of $\pi$, we study $\dim_\C\bk{\jac{G}{T}{N_w\bk{\frac{1}{2}}\bk{\pi}}}$ for certain $w\in W^{M,T}$.
In particular, consider
\[
w\in\set{
	w_{[1,2,3,4]}, 
	w_{[1,4,2,3,4]},
	w_{[4, 3, 2, 4, 5, 6, 4, 1, 3, 2, 4, 5, 4, 1, 3, 2, 4]}} .
\]
We indicate the in which of the $r_T^G\tau$, with $\tau\in\set{\pi_0,\pi_1,\sigma_1,\sigma_2}$, does $w\cdot\lambda_0$ appears:
\[
\begin{array}{l}
	w_{[1,2,3,4]}\cdot \lambda_0 = \esixchar{-1}{-2}{-1}{1}{2}{-1} \leq r_T^G \sigma_1 \\
	w_{[1,4,2,3,4]}\cdot \lambda_0 = \esixchar{-1}{-1}{0}{-1}{3}{-1} \leq r_T^G\sigma_1	\\
	w_{[1,2,4,5,4,2,3,4]}\cdot \lambda_0 = \esixchar{-1}{-1}{2}{-1}{-1}{2} \leq r_T^G\sigma_1,r_T^G\sigma_2 \\
	w_{[4, 3, 2, 4, 5, 6, 4, 1, 3, 2, 4, 5, 4, 1, 3, 2, 4]}\cdot\lambda_0 = \esixchar{0}{0}{0}{-1}{0}{0} \leq r_T^G\pi_1 .
\end{array}
\]
We list bellow what can be learned from the value of $\dim_\C\bk{\jac{G}{T}{N_w\bk{\frac{1}{2}}\bk{\pi}}}$ for the above elements $w$.
For more details regarding the calculations of such dimensions, the reader is encouraged to read \cite[Appendix A]{SDPS_E7}.
\begin{itemize}
	\item For $w=w_{[1,2,3,4]}$, $\dim_\C\bk{\jac{G}{T}{N_w\bk{\frac{1}{2}}\bk{\pi}}}=710$.
	Hence, $\dim_\C\bk{r_T^G\pi_0}=10$ and hence $\pi_0$ is disjoint from $\pi_1$, $\sigma_1$ and $\sigma_2$.
	
	\item For $w=w_{[1,4,2,3,4]}$, $\dim_\C\bk{\jac{G}{T}{N_w\bk{\frac{1}{2}}\bk{\pi}}}=710$.
	
	\item For $w=w_{[1,2,4,5,4,2,3,4]}$, $\dim_\C\bk{\jac{G}{T}{N_w\bk{\frac{1}{2}}\bk{\pi}}}=575$.
	Hence, $\dim_\C\bk{r_T^G\sigma_1}=135$ and $\sigma_1$ is disjoint from $\pi_0$, $\pi_1$ and $\sigma_2$.
	
	\item For $w=w_{[4, 3, 2, 4, 5, 6, 4, 1, 3, 2, 4, 5, 4, 1, 3, 2, 4]}$, $\dim_\C\bk{\jac{G}{T}{N_w\bk{\frac{1}{2}}\bk{\pi}}}=575$.
	Hence, $\dim_\C\bk{r_T^G\pi_1}=575$ and it turns out that $\pi_1=\sigma_2$.
\end{itemize}

We conclude that $\pi$ is of length $3$ with a Jordan-\"Holder series of the form
\[
0\subset \pi_0 \subset \tau \subset \pi,\quad \tau\rmod\pi_0=\sigma_1,\quad \pi\rmod\tau=\pi_1 .
\]

\subsection{$E_7$}

Assume that $G$ is of type $E_7$.
In the following cases, it is possible to prove that $\pi=\DPS{P_i}\bk{s_0,\chi}$ is of length $2$ via the method describe in the beginning of \Cref{Appendix:Local_data_for_global_NSI_resisudes}.
These cases are listed in the following table:

\begin{center}
	\begin{longtable}[h]{|c|c|c|c|}
		\hline
		$[i,s_0,k]$ & $d\bk{\pi}$ &  $d\bk{\pi_1}$ & $d\bk{\pi_0}$ \\ \hline \hline
		$\bk{2,2,1}$ & 576 & 232 & 344 \\ \hline \hline
		$\bk{4,\frac12,1}$ & 10,080 & 7,308 & 2,772 \\ \hline
		$\bk{4,\frac12,2}$ & 10,080 & 7,308 & 2,772 \\ \hline
		$\bk{4,\frac23,1}$ & 10,080 & 4,032 & 6,048 \\ \hline
		$\bk{4,\frac23,2}$ & 10,080 & 4,032 & 6,048 \\ \hline
		$\bk{4,\frac32,1}$ & 10,080 & 756 & 9,324 \\ \hline
		$\bk{4,\frac32,2}$ & 10,080 & 756 & 9,324 \\ \hline \hline
		$\bk{5,1,2}$ & 4,032 & 2,016 & 2,016 \\ \hline
		$\bk{5,\frac32,1}$ & 4,032 & 576 & 3,456 \\ \hline
		$\bk{5,\frac32,2}$ & 4,032 & 576 & 3,456 \\ \hline
	\end{longtable}
	\label{Table:E7_local_length_2_NSI}
\end{center}

We now turn to the remaining case, $\bk{5,1,1}$.
This case is of interest here in the context of calculating $\ResRep{E_7}{5}{1}{2}$.
Here, we first note that for a global quadratic character $\chi$, it holds that
\[
W^{shortest}_{\bk{P_5,1,\chi}} = \set{w_{[1, 3, 4, 5, 6, 2, 4, 5, 3, 4, 1, 3, 2, 4, 5]}}
\]
That is, for every $\coset{w}\in W^{M_5,T}_{1,\chi,res.}$ and every $w'\in\coset{w}$, we may write $w'=u\cdot w_0$, where
\[
w_0=w_{[1, 3, 4, 5, 6, 2, 4, 5, 3, 4, 1, 3, 2, 4, 5]} .
\]
In a place where $\chi_v$ is quadratic, $N_w\bk{\frac{1}{2}}\bk{\pi}=\pi_1$ and  $\dim_\C\bk{\jac{G}{T}{N_w\bk{\frac{1}{2}}\bk{\pi}}} = 2,016$.
Thus, we have that in the case $\chi_v$ is trivial, $\tau=N_w\bk{\frac{1}{2}}\bk{\pi}$ is a quotient of $\pi$ such that $\pi_1$ is the unique irreducible quotient of $\tau$ and $\dim_\C\bk{\jac{G}{T}{\tau}} = 2,016$.

On the other hand, $\tau$ is reducible.
This follows from the fact that $\pi_1$ is also the unique irreducible quotient of the reducible degenerate principal series $\sigma=\DPS{P_3}\bk{\frac12,\Id}$ (see \Cref{Chap:Siegel_Weil}).
Since $\Card{W^{M_3,T}}=2,016$, it follows that $\dim_\C\bk{\jac{G}{T}{\pi_1}} < 2,016$.

However, $\tau$ admits a unique irreducible subrepresentation.
To see this, denote
\[
\begin{array}{l}
	\lambda_0 = \esevenchar{-1}{-1}{-1}{-1}{5}{-1}{-1} \\
	\lambda_1=w_0\cdot\lambda_0 = \esevenchar{-1}{4}{-1}{-1}{-1}{-1}{4} .
\end{array}
\]
By the geometric lemma, it holds that
\[
mult\bk{\lambda_1,r_T^G\pi} = mult\bk{\lambda_1,r_T^G\tau} = 1
\]
and thus $\tau$ admits a unique irreducible subrepresentation $\tau_0$.
Since, by the geometric lemma,
\[
mult\bk{\lambda_1,r_T^G\sigma} = 0,
\]
it follows that $\tau_0$ is not a subquotient of $\sigma$ and, in particular, $\coset{\tau}\neq \coset{\sigma}$.

We were unable to completely determine the precise structure of $\tau$ but our calculations indicate that it is of length $2$ with a unique irreducible subrepresentation $\tau_0$ and a unique irreducible quotient $\pi_1$.

\subsection{$E_8$}

Assume that $G$ is of type $E_8$.
In the following cases, it is possible to prove that $\pi=\DPS{P_i}\bk{s_0,\chi}$ is of length $2$ via the method describe in the beginning of \Cref{Appendix:Local_data_for_global_NSI_resisudes}.
These cases are listed in the following table:

\begin{center}
	\begin{longtable}[h]{|c|c|c|c|}
		\hline
		$[i,s_0,k]$ & $d\bk{\pi}$ &  $d\bk{\pi_1}$ & $d\bk{\pi_0}$ \\ \hline \endhead \hline
		$\bk{3,\frac{7}{6},1}$ & 69,120 & 17,280 & 51,840 \\ \hline
		$\bk{3,\frac{7}{6},3}$ & 69,120 & 17,280 & 51,840 \\ \hline
		$\bk{3,2,1}$ & 69,120 & 6,720 & 62,400 \\ \hline
		$\bk{3,2,2}$ & 69,120 & 6,720 & 62,400 \\ \hline \hline
		$\bk{4,\frac{3}{10},1}$ & 483,840 & 241,920 & 241,920 \\ \hline
		$\bk{4,\frac{3}{10},5}$ & 483,840 & 241,920 & 241,920 \\ \hline
		$\bk{4,\frac{1}{2},4}$ & 483,840 & 241,920 & 241,920 \\ \hline
		$\bk{4,\frac{3}{4},1}$ & 483,840 & 69,120 & 414,720 \\ \hline
		$\bk{4,\frac{3}{4},2}$ & 483,840 & 69,120 & 414,720 \\ \hline
		$\bk{4,\frac{3}{4},4}$ & 483,840 & 69,120 & 414,720 \\ \hline
		$\bk{4,\frac{7}{6},1}$ & 483,840 & 60,480 & 423,360 \\ \hline
		$\bk{4,\frac{7}{6},3}$ & 483,840 & 60,480 & 423,360 \\ \hline
		$\bk{4,2, 1}$ & 483,840 & 6,720 & 477,120 \\ \hline
		$\bk{4,2, 2}$ & 483,840 & 6,720 & 477,120 \\ \hline \hline
		$\bk{5,\frac{1}{2},3}$ & 241,920 & 161,280 & 80,640 \\ \hline
		$\bk{5,\frac{5}{6},1}$ & 241,920 & 69,120 & 172,800 \\ \hline
		$\bk{5,\frac{5}{6},3}$ & 241,920 & 69,120 & 172,800 \\ \hline
		$\bk{5,\frac{7}{6},1}$ & 241,920 & 17,280 & 224,640 \\ \hline
		$\bk{5,\frac{7}{6},3}$ & 241,920 & 17,280 & 224,640 \\ \hline
		$\bk{5,2,1}$ & 241,920 & 2,160 & 239,760 \\ \hline
		$\bk{5,2,2}$ & 241,920 & 2,160 & 239,760 \\ \hline \hline
		$\bk{6,2, 2}$ & 60,480 & 15,120 & 45,360 \\ \hline
		$\bk{6,\frac{5}{2},1}$ & 60,480 & 2,160 & 58,320 \\ \hline
		$\bk{6,\frac{5}{2},2}$ & 60,480 & 2,160 & 58,320 \\ \hline
	\end{longtable}
	\label{Table:E8_local_length_2_NSI}
	
\end{center}

For the following cases, we were unable to show that the length of $\pi$ is $2$ using this method:
\[
\begin{split}
	&\bk{3,\frac12,1}, \bk{3,1,1}, \bk{3,1,2}, \bk{3,\frac32,1}, \bk{3,\frac32,2},
	\bk{4,\frac12,1}, \bk{4,\frac12,2}, \bk{4,\frac56,1}, \bk{4,\frac56,3}, \\
	&\bk{4,1,1}, \bk{4,1,2}, \bk{5,\frac12,3}, \bk{5,1,1}, \bk{6,\frac12,1}, \bk{6,\frac12,2}, \bk{6,2,1} .
\end{split}
\]



\backmatter
\bibliographystyle{amsalpha}
\bibliography{bib}

\newcommand{\etalchar}[1]{$^{#1}$}
\def\cprime{$'$}
\providecommand{\bysame}{\leavevmode\hbox to3em{\hrulefill}\thinspace}
\providecommand{\MR}{\relax\ifhmode\unskip\space\fi MR }
\providecommand{\MRhref}[2]{%
  \href{http://www.ams.org/mathscinet-getitem?mr=#1}{#2}
}
\providecommand{\href}[2]{#2}
\begin{thebibliography}{{The}YY}

\bibitem[Art79]{MR546601}
James Arthur, \emph{Eisenstein series and the trace formula}, Automorphic
  forms, representations and {$L$}-functions ({P}roc. {S}ympos. {P}ure {M}ath.,
  {O}regon {S}tate {U}niv., {C}orvallis, {O}re., 1977), {P}art 1, Proc. Sympos.
  Pure Math., XXXIII, Amer. Math. Soc., Providence, R.I., 1979, pp.~253--274.
  \MR{546601}

\bibitem[Art89a]{Arthur1989}
\bysame, \emph{Intertwining operators and residues. {I}. {W}eighted
  characters}, J. Funct. Anal. \textbf{84} (1989), no.~1, 19--84. \MR{999488}

\bibitem[Art89b]{MR1021499}
\bysame, \emph{Unipotent automorphic representations: conjectures}, no.
  171-172, 1989, Orbites unipotentes et repr\'{e}sentations, II, pp.~13--71.
  \MR{1021499}

\bibitem[BDS49]{MR32659}
A.~Borel and J.~De~Siebenthal, \emph{Les sous-groupes ferm\'{e}s de rang
  maximum des groupes de {L}ie clos}, Comment. Math. Helv. \textbf{23} (1949),
  200--221. \MR{32659}

\bibitem[BJ03]{MR2017065}
Dubravka Ban and Chris Jantzen, \emph{Degenerate principal series for
  even-orthogonal groups}, Represent. Theory \textbf{7} (2003), 440--480
  (electronic). \MR{2017065 (2004k:22020)}

\bibitem[BJ07]{MR2363302}
\bysame, \emph{{$R$}-groups and the action of intertwining operators in the
  nontempered case}, Int. Math. Res. Not. IMRN (2007), no.~20, Art. ID rnm059,
  29. \MR{2363302}

\bibitem[Bor66]{MR0207650}
Armand Borel, \emph{Introduction to automorphic forms}, Algebraic {G}roups and
  {D}iscontinuous {S}ubgroups ({P}roc. {S}ympos. {P}ure {M}ath., {B}oulder,
  {C}olo., 1965), Amer. Math. Soc., Providence, R.I., 1966, pp.~199--210.
  \MR{0207650}

\bibitem[Bou02]{Bourbaki:2002}
Nicolas Bourbaki, \emph{Lie groups and {L}ie algebras. {C}hapters 4--6},
  Elements of Mathematics (Berlin), Springer-Verlag, Berlin, 2002, Translated
  from the 1968 French original by Andrew Pressley. \MR{1890629}

\bibitem[BT71]{MR294349}
A.~Borel and J.~Tits, \emph{\'{E}l\'{e}ments unipotents et sous-groupes
  paraboliques de groupes r\'{e}ductifs. {I}}, Invent. Math. \textbf{12}
  (1971), 95--104. \MR{294349}

\bibitem[Bum05]{MR2192819}
Daniel Bump, \emph{The {R}ankin-{S}elberg method: an introduction and survey},
  Automorphic representations, {$L$}-functions and applications: progress and
  prospects, Ohio State Univ. Math. Res. Inst. Publ., vol.~11, de Gruyter,
  Berlin, 2005, pp.~41--73. \MR{2192819 (2006k:11097)}

\bibitem[BV85]{MR782556}
Dan Barbasch and David~A. Vogan, Jr., \emph{Unipotent representations of
  complex semisimple groups}, Ann. of Math. (2) \textbf{121} (1985), no.~1,
  41--110. \MR{782556}

\bibitem[Car93]{MR1266626}
Roger~W. Carter, \emph{Finite groups of {L}ie type}, Wiley Classics Library,
  John Wiley \& Sons, Ltd., Chichester, 1993, Conjugacy classes and complex
  characters, Reprint of the 1985 original, A Wiley-Interscience Publication.
  \MR{1266626}

\bibitem[CS80]{MR581582}
W.~Casselman and J.~Shalika, \emph{The unramified principal series of
  {$p$}-adic groups. {II}. {T}he {W}hittaker function}, Compositio Math.
  \textbf{41} (1980), no.~2, 207--231. \MR{581582 (83i:22027)}

\bibitem[Elk72]{MR308228}
Gordon~Bradley Elkington, \emph{Centralizers of unipotent elements in
  semisimple algebraic groups}, J. Algebra \textbf{23} (1972), 137--163.
  \MR{308228}

\bibitem[GG05]{MR2192822}
Wee~Teck Gan and Nadya Gurevich, \emph{Non-tempered {A}rthur packets of
  {$G_2$}}, Automorphic representations, {$L$}-functions and applications:
  progress and prospects, Ohio State Univ. Math. Res. Inst. Publ., vol.~11, de
  Gruyter, Berlin, 2005, pp.~129--155. \MR{2192822 (2006j:22019)}

\bibitem[GG06]{MR2262172}
\bysame, \emph{Nontempered {A}-packets of {$G_2$}: liftings from
  {$\widetilde{\rm SL}_2$}}, Amer. J. Math. \textbf{128} (2006), no.~5,
  1105--1185. \MR{2262172 (2008a:22024)}

\bibitem[GGJ02]{MR1918673}
Wee~Teck Gan, Nadya Gurevich, and Dihua Jiang, \emph{Cubic unipotent {A}rthur
  parameters and multiplicities of square integrable automorphic forms},
  Invent. Math. \textbf{149} (2002), no.~2, 225--265. \MR{1918673
  (2004e:11051)}

\bibitem[GGK{\etalchar{+}}21]{MR4273169}
Dmitry Gourevitch, Henrik P.~A. Gustafsson, Axel Kleinschmidt, Daniel Persson,
  and Siddhartha Sahi, \emph{Eulerianity of {F}ourier coefficients of
  automorphic forms}, Represent. Theory \textbf{25} (2021), 481--507.
  \MR{4273169}

\bibitem[GH15]{MR3359720}
David Ginzburg and Joseph Hundley, \emph{A doubling integral for {$G_2$}},
  Israel J. Math. \textbf{207} (2015), no.~2, 835--879. \MR{3359720}

\bibitem[GMV15]{MR3267116}
Michael~B. Green, Stephen~D. Miller, and Pierre Vanhove, \emph{Small
  representations, string instantons, and {F}ourier modes of {E}isenstein
  series}, J. Number Theory \textbf{146} (2015), 187--309. \MR{3267116}

\bibitem[GRS97]{MR1469105}
David Ginzburg, Stephen Rallis, and David Soudry, \emph{On the automorphic
  theta representation for simply laced groups}, Israel J. Math. \textbf{100}
  (1997), 61--116. \MR{1469105 (99c:11058)}

\bibitem[GS]{RallisSchiffmannPaper}
Nadya Gurevich and Avner Segal, \emph{Poles of the standard
  {$\mathcal{L}$}-function of {$G_2$} and the rallis-schiffmann lift},
  Preprint.

\bibitem[GS05]{MR2123125}
Wee~Teck Gan and Gordan Savin, \emph{On minimal representations definitions and
  properties}, Represent. Theory \textbf{9} (2005), 46--93. \MR{2123125}

\bibitem[GS19]{gurevich_segal_2019}
Nadya Gurevich and Avner Segal, \emph{Poles of the standard ${\mathcal{l}}$
  -function of $g_{2}$ and the rallis–schiffmann lift}, Canadian Journal of
  Mathematics \textbf{71} (2019), no.~5, 1127–1161.

\bibitem[Hal16]{HeziMScThesis}
Hezi Halawi, \emph{Poles of degenerate {E}isenstein series and {S}iegel-{W}eil
  identities for exceptional split groups}, Master's thesis, Ben-Gurion
  University, 2016.

\bibitem[Han15]{MR3437492}
Marcela Hanzer, \emph{Degenerate {E}isenstein series for symplectic groups},
  Glas. Mat. Ser. III \textbf{50(70)} (2015), no.~2, 289--332. \MR{3437492}

\bibitem[HSa]{SDPS_E6}
Hezi Halawi and Avner Segal, \emph{The degenerate principal series
  representations of exceptional groups of type $e_6$ over p-adic fields}, To
  appear in the Israel Journal of Mathematics.

\bibitem[HSb]{SDPS_E7}
\bysame, \emph{The degenerate principal series representations of exceptional
  groups of type $e_7$ over p-adic fields}, Preprint.

\bibitem[HSc]{SDPS_E8}
\bysame, \emph{The degenerate principal series representations of exceptional
  groups of type $e_8$ over p-adic fields}, In preparation.

\bibitem[HS19]{hanzer_savin_2019}
Marcela Hanzer and Gordan Savin, \emph{Eisenstein series arising from jordan
  algebras}, Canadian Journal of Mathematics (2019), 1–19.

\bibitem[Ike92]{MR1174424}
Tamotsu Ikeda, \emph{On the location of poles of the triple {$L$}-functions},
  Compositio Math. \textbf{83} (1992), no.~2, 187--237. \MR{1174424
  (94b:11042)}

\bibitem[Jac35]{MR1503258}
Nathan Jacobson, \emph{Rational methods in the theory of {L}ie algebras}, Ann.
  of Math. (2) \textbf{36} (1935), no.~4, 875--881. \MR{1503258}

\bibitem[Jan95]{MR1341660}
Chris Jantzen, \emph{On the {I}wahori-{M}atsumoto involution and applications},
  Ann. Sci. \'{E}cole Norm. Sup. (4) \textbf{28} (1995), no.~5, 527--547.
  \MR{1341660}

\bibitem[Kim96]{MR1426903}
Henry~H. Kim, \emph{The residual spectrum of {$G_2$}}, Canad. J. Math.
  \textbf{48} (1996), no.~6, 1245--1272. \MR{1426903}

\bibitem[Kim01]{MR1847140}
\bysame, \emph{Residual spectrum of split classical groups; contribution from
  {B}orel subgroups}, Pacific J. Math. \textbf{199} (2001), no.~2, 417--445.
  \MR{1847140}

\bibitem[Kna97]{MR1476501}
A.~W. Knapp, \emph{Introduction to the {L}anglands program}, Representation
  theory and automorphic forms ({E}dinburgh, 1996), Proc. Sympos. Pure Math.,
  vol.~61, Amer. Math. Soc., Providence, RI, 1997, pp.~245--302. \MR{MR1476501
  (99d:11123)}

\bibitem[Kos78]{Kostant1978}
Bertram Kostant, \emph{On {W}hittaker vectors and representation theory},
  Invent. Math. \textbf{48} (1978), no.~2, 101--184. \MR{507800}

\bibitem[KR88]{MR946349}
Stephen~S. Kudla and Stephen Rallis, \emph{On the {W}eil-{S}iegel formula}, J.
  Reine Angew. Math. \textbf{387} (1988), 1--68. \MR{946349}

\bibitem[KS88]{MR944102}
C.~David Keys and Freydoon Shahidi, \emph{Artin {$L$}-functions and
  normalization of intertwining operators}, Ann. Sci. \'Ecole Norm. Sup. (4)
  \textbf{21} (1988), no.~1, 67--89. \MR{944102 (89k:22034)}

\bibitem[KS96]{MR1385286}
Henry~H. Kim and Freydoon Shahidi, \emph{Quadratic unipotent {A}rthur
  parameters and residual spectrum of symplectic groups}, Amer. J. Math.
  \textbf{118} (1996), no.~2, 401--425. \MR{1385286}

\bibitem[Kud03]{MR1990377}
Stephen~S. Kudla, \emph{Tate's thesis}, An introduction to the {L}anglands
  program ({J}erusalem, 2001), Birkh\"auser Boston, Boston, MA, 2003,
  pp.~109--131. \MR{1990377}

\bibitem[Lan76]{MR0579181}
Robert~P. Langlands, \emph{On the functional equations satisfied by
  {E}isenstein series}, Lecture Notes in Mathematics, Vol. 544,
  Springer-Verlag, Berlin-New York, 1976. \MR{0579181 (58 \#28319)}

\bibitem[Lap11]{MR2767521}
Erez Lapid, \emph{On {A}rthur's asymptotic inner product formula of truncated
  {E}isenstein series}, On certain {$L$}-functions, Clay Math. Proc., vol.~13,
  Amer. Math. Soc., Providence, RI, 2011, pp.~309--331. \MR{2767521}

\bibitem[Lie02]{Lieberman_Hecle_K-functions}
Michael~J. Lieberman, \emph{{H}ecke {L}-functions on algebraic number fields},
  2002.

\bibitem[LS04]{MR2044850}
Martin~W. Liebeck and Gary~M. Seitz, \emph{The maximal subgroups of positive
  dimension in exceptional algebraic groups}, Mem. Amer. Math. Soc.
  \textbf{169} (2004), no.~802, vi+227. \MR{2044850}

\bibitem[Mg96]{MR1404331}
C.~M\oe~glin, \emph{Repr\'{e}sentations quadratiques unipotentes des groupes
  classiques {$p$}-adiques}, Duke Math. J. \textbf{84} (1996), no.~2, 267--332.
  \MR{1404331}

\bibitem[Mor42]{MR0007750}
V.~V. Morozov, \emph{On a nilpotent element in a semi-simple {L}ie algebra}, C.
  R. (Doklady) Acad. Sci. URSS (N.S.) \textbf{36} (1942), 83--86. \MR{0007750}

\bibitem[MS14]{MR3165233}
Jan M\"{o}llers and Benjamin Schwarz, \emph{Structure of the degenerate
  principal series on symmetric {$R$}-spaces and small representations}, J.
  Funct. Anal. \textbf{266} (2014), no.~6, 3508--3542. \MR{3165233}

\bibitem[Mui01]{Muic2001}
Goran Mui\'{c}, \emph{A proof of {C}asselman-{S}hahidi's conjecture for
  quasi-split classical groups}, Canad. Math. Bull. \textbf{44} (2001), no.~3,
  298--312. \MR{1847492}

\bibitem[MW95a]{MR1361168}
C.~M{\oe}glin and J.-L. Waldspurger, \emph{Spectral decomposition and
  {E}isenstein series}, Cambridge Tracts in Mathematics, vol. 113, Cambridge
  University Press, Cambridge, 1995, Une paraphrase de l'{\'E}criture [A
  paraphrase of Scripture]. \MR{1361168 (97d:11083)}

\bibitem[MW95b]{MW_AppendixIII}
\bysame, \emph{Spectral decomposition and {E}isenstein series}, Cambridge
  Tracts in Mathematics, vol. 113, ch.~Appendix III, pp.~298--313, Cambridge
  University Press, Cambridge, 1995, Une paraphrase de l'{\'E}criture [A
  paraphrase of Scripture]. \MR{1361168 (97d:11083)}

\bibitem[Neu99]{MR1697859}
J\"{u}rgen Neukirch, \emph{Algebraic number theory}, Grundlehren der
  Mathematischen Wissenschaften [Fundamental Principles of Mathematical
  Sciences], vol. 322, Springer-Verlag, Berlin, 1999, Translated from the 1992
  German original and with a note by Norbert Schappacher, With a foreword by G.
  Harder. \MR{1697859}

\bibitem[NSS]{SegalSingularities}
Taeuk Nam, Avner Segal, and Lior Silberman, \emph{Singularities of intertwining
  operators and certain decompositions of principal series representations},
  preprint, \url{https://arxiv.org/abs/1811.00803}.

\bibitem[Ral91]{MR1159270}
Stephen Rallis, \emph{Poles of standard {$L$} functions}, Proceedings of the
  {I}nternational {C}ongress of {M}athematicians, {V}ol.\ {I}, {II} ({K}yoto,
  1990), Math. Soc. Japan, Tokyo, 1991, pp.~833--845. \MR{1159270}

\bibitem[Rub94]{MR1264015}
Hubert Rubenthaler, \emph{Les paires duales dans les alg\`ebres de {L}ie
  r\'{e}ductives}, Ast\'{e}risque (1994), no.~219, 121. \MR{1264015}

\bibitem[Sah95]{MR1329899}
Siddhartha Sahi, \emph{Jordan algebras and degenerate principal series}, J.
  Reine Angew. Math. \textbf{462} (1995), 1--18. \MR{1329899}

\bibitem[Seg18]{MR3803152}
Avner Segal, \emph{The degenerate {E}isenstein series attached to the
  {H}eisenberg parabolic subgroups of quasi-split forms of {$Spin_8$}}, Trans.
  Amer. Math. Soc. \textbf{370} (2018), no.~8, 5983--6039. \MR{3803152}

\bibitem[Seg19a]{MR4024536}
\bysame, \emph{The degenerate residual spectrum of quasi-split forms of
  {$Spin_8$} associated to the {H}eisenberg parabolic subgroup}, Trans. Amer.
  Math. Soc. \textbf{372} (2019), no.~9, 6703--6754. \MR{4024536}

\bibitem[Seg19b]{SegalResiduesD4}
Avner Segal, \emph{The degenerate residual spectrum of quasi-split forms of
  {$Spin_8$} associated to the {H}eisenberg parabolic subgroup}, Trans. Amer.
  Math. Soc. \textbf{372} (2019), no.~9, 6703--6754.

\bibitem[Sei91]{MR1048074}
Gary~M. Seitz, \emph{Maximal subgroups of exceptional algebraic groups}, Mem.
  Amer. Math. Soc. \textbf{90} (1991), no.~441, iv+197. \MR{1048074}

\bibitem[Sha10]{MR2683009}
Freydoon Shahidi, \emph{Eisenstein series and automorphic {$L$}-functions},
  American Mathematical Society Colloquium Publications, vol.~58, American
  Mathematical Society, Providence, RI, 2010. \MR{2683009 (2012d:11119)}

\bibitem[Som98]{MR1631769}
Eric Sommers, \emph{A generalization of the {B}ala-{C}arter theorem for
  nilpotent orbits}, Internat. Math. Res. Notices (1998), no.~11, 539--562.
  \MR{1631769}

\bibitem[Ste68]{MR0466335}
Robert Steinberg, \emph{Lectures on {C}hevalley groups}, Yale University, New
  Haven, Conn., 1968, Notes prepared by John Faulkner and Robert Wilson.
  \MR{0466335 (57 \#6215)}

\bibitem[{The}YY]{sagemath}
{The Sage Developers}, \emph{{S}agemath, the {S}age {M}athematics {S}oftware
  {S}ystem ({V}ersion x.y.z)}, YYYY, {\tt http://www.sagemath.org}.

\bibitem[Vog78]{Vogan1978}
David~A. Vogan, Jr., \emph{Gel\cprime fand-{K}irillov dimension for
  {H}arish-{C}handra modules}, Invent. Math. \textbf{48} (1978), no.~1, 75--98.
  \MR{506503}

\bibitem[Was97]{MR1421575}
Lawrence~C. Washington, \emph{Introduction to cyclotomic fields}, second ed.,
  Graduate Texts in Mathematics, vol.~83, Springer-Verlag, New York, 1997.
  \MR{1421575 (97h:11130)}

\bibitem[Wat11]{MR2894271}
Mark Watkins, \emph{Computing with {H}ecke {G}r\"{o}ssencharacters}, Actes de
  la {C}onf\'{e}rence ``{T}h\'{e}orie des {N}ombres et {A}pplications'', Publ.
  Math. Besan\c{c}on Alg\`ebre Th\'{e}orie Nr., vol. 2011, Presses Univ.
  Franche-Comt\'{e}, Besan\c{c}on, 2011, pp.~119--135. \MR{2894271}

\bibitem[Win78]{MR517138}
Norman Winarsky, \emph{Reducibility of principal series representations of
  {$p$}-adic {C}hevalley groups}, Amer. J. Math. \textbf{100} (1978), no.~5,
  941--956. \MR{517138 (80f:22018)}

\bibitem[{\v{Z}}am97]{MR1485422}
Sini{\v{s}}a {\v{Z}}ampera, \emph{The residual spectrum of the group of type
  {$G_2$}}, J. Math. Pures Appl. (9) \textbf{76} (1997), no.~9, 805--835.
  \MR{1485422 (98k:22077)}

\bibitem[Zha97]{Zhang1997}
Yuanli Zhang, \emph{The holomorphy and nonvanishing of normalized local
  intertwining operators}, Pacific J. Math. \textbf{180} (1997), no.~2,
  385--398. \MR{1487571}

\end{thebibliography}
\printindex


\end{document}